\documentclass[11pt,a4paper,reqno]{amsart}
\usepackage[english]{babel}
\usepackage[utf8]{inputenc} 
\usepackage{fancyhdr}
\usepackage{indentfirst}
\usepackage{color}     
\usepackage{graphicx}   
\usepackage{newlfont}

\usepackage{commath}
\usepackage{amsfonts}
\usepackage{amsmath}
\usepackage{amssymb}
\usepackage{amsthm}
\usepackage{latexsym}
\usepackage{mathrsfs}
\usepackage{mathtools}
\mathtoolsset{showonlyrefs=true}
\usepackage{verbatim}
\usepackage[all]{xy}
\usepackage{lineno}
\usepackage{leftidx}
\usepackage[hidelinks]{hyperref}
\usepackage{units}
\usepackage{xfrac}
\usepackage{faktor}
\usepackage{float}
\usepackage{tikz}
\usetikzlibrary{arrows, positioning, automata}

\usepackage{chngcntr}
\usepackage{apptools}

\theoremstyle{plain} 
\newtheorem*{theo*}{Theorem}
\newtheorem*{con*}{Conjecture}
\newtheorem{theo}{Theorem}[section] 
\newtheorem{prop}[theo]{Proposition}
\newtheorem{cor}[theo]{Corollary}
\newtheorem{lem}[theo]{Lemma}
\newtheorem{conj}[theo]{Conjecture}

\theoremstyle{definition}
\newtheorem{defin}[theo]{Definition}
\newtheorem{ex}[theo]{Example}
\newtheorem{rem}[theo]{Remark}

\theoremstyle{definition}

\theoremstyle{remark}
\newtheorem*{notat}{Notation and convention}

\AtAppendix{\counterwithin{theo}{section}}

\newcommand{\Z}{\mathbb{Z}}
\newcommand{\Zplus}{\Z_{>0}}
\newcommand{\Zpluseq}{\Z_{\geq0}}
\newcommand{\R}{\mathbb{R}}
\newcommand{\C}{\mathbb{C}}
\newcommand{\T}{\mathbb{T}}

\newcommand{\J}{\mathcal{J}}
\newcommand{\A}{\mathcal{A}}
\newcommand{\B}{\mathcal{B}}
\newcommand{\F}{\mathfrak{F}}
\newcommand{\g}{\mathfrak{g}}

\newcommand{\Diff}{\operatorname{Diff}}
\newcommand{\Vir}{\mathfrak{Vir}}
\newcommand{\Mob}{\operatorname{M\ddot{o}b}}
\newcommand{\Aut}{\operatorname{Aut}}
\newcommand{\Hom}{\operatorname{Hom}}
\newcommand{\Ker}{\operatorname{Ker}}
\newcommand{\End}{\operatorname{End}}
\newcommand{\Rep}{\operatorname{Rep}}
\newcommand{\Repu}{\operatorname{Rep}^{\operatorname{u}}}

\newcommand{\parzero}{{\overline{0}}}
\newcommand{\parone}{{\overline{1}}}
\newcommand{\parzerone}{{\parzero,\parone}}
\newcommand{\paronezero}{{\parone,\parzero}}
\newcommand{\parzerozero}{{\parzero,\parzero}}
\newcommand{\paroneone}{{\parone,\parone}}
\newcommand{\scalar}{(\cdot|\cdot)}
\newcommand{\curlyscalar}{\{\cdot|\cdot\}}
\newcommand{\bilinear}{(\cdot\,,\cdot)}

\newcommand{\Bilinear}{B(\cdot\,,\cdot)}
\newcommand{\pairing}{\langle\cdot,\cdot\rangle}
\newcommand{\pSone}{S^1\backslash\{-1\}}
\newcommand{\half}{\frac{1}{2}}

\begin{document}

\author{Sebastiano Carpi}
\address{Dipartimento di Matematica, Universit\`a di Roma ``Tor Vergata'', Via della Ricerca Scientifica, 1, 00133 Roma, Italy\\
E-mail: {\tt carpi@mat.uniroma2.it}
}
\author{Tiziano Gaudio}
\address{Dipartimento di Matematica, Universit\`a di Roma ``Tor Vergata'', Via della Ricerca Scientifica, 1, 00133 Roma, Italy\\
E-mail: {\tt gaudio@mat.uniroma2.it}
}
\author{Robin Hillier} 
\address{Department of Mathematics and Statistics, Lancaster University, Lancaster LA1 4YF, UK\\
E-mail: {\tt r.hillier@lancaster.ac.uk}
}
\title[From Vertex Operator Superalgebras to Graded-Local Conformal Nets]{From Vertex Operator Superalgebras to Graded-Local Conformal Nets and Back}
\date{}

\begin{abstract}
We generalize the Carpi-Kawahigashi-Longo-Weiner correspondence between vertex operator algebras and conformal nets to the case of vertex operator superalgebras and graded-local conformal nets by introducing the notion of strongly graded-local vertex operator superalgebra. Then we apply our machinery to a number of well-known examples including superconformal field theory models. We also prove that all lattice VOSAs are strongly graded-local. Furthermore, we prove strong graded locality of the super-Moonshine VOSA, whose group of  automorphisms preserving the superconformal structure is isomorphic to Conway's largest sporadic simple group, and of the shorter Moonshine VOSA, whose automorphisms group is isomorphic to the direct product of the baby Monster with a cyclic group of order two.
\end{abstract}

\maketitle

\tableofcontents

\section{Introduction}

A remarkable aspect of conformal field theory (CFT) in two space-time dimensions is that it provides a bridge between many different areas of physics and mathematics such as critical phenomena in classical statistical mechanics, string theory, topological quantum computation, infinite dimensional Lie groups and Lie algebras, finite simple groups, tensor categories, 3-manifold invariants, modular forms, algebraic curves, quantum groups, the theory of subfactors, noncommutative geometry, see e.g. \cite{EK98,Gan06,DMS97,FST89,RW18}. 

Chiral CFTs, namely CFTs on the ``light-ray'', may be considered as the building blocks of two-dimensional CFTs, see e.g.  \cite{AGT23,KL04b}.

Two of the major mathematical approaches to chiral CFTs are based on vertex operator algebras (more generally vertex algebras) and conformal nets (on $S^1$). On the one hand, vertex operator algebras (VOAs) can be seen as an algebraic (not necessarily unitary) version of Wightman field theories \cite{SW64} in the special case of chiral CFTs, see \cite[Chapter 1]{Kac01}, \cite{RTT22} and \cite{CRTT25}. Their axioms were first formulated in connection with the famous moonshine conjecture by Conway and Norton \cite{Bor86,FLM88}. On the other hand, conformal nets are the chiral CFT versions of Haag-Kastler nets in the so called ``algebraic approach to quantum field theory" (AQFT) \cite{Haag96}.

Although it was rather clear from the first developments that these two approaches share many structural similarities and should be essentially equivalent (in the unitary case), a general direct connection between VOAs and conformal nets have been proposed  only recently by Carpi, Kawahigashi, Longo and Weiner \cite{CKLW18}, see also \cite{Car17,Kaw15,Kaw19}. In \cite{CKLW18} the authors use smeared vertex operators (which turn out to be chiral CFT versions of Wightman fields \cite{CRTT25,RTT22}) to define a conformal net $\A_V$ from every sufficiently nice unitary VOA $V$ called {\it strongly local}. Moreover, they prove that a strongly local VOA $V$ can be recovered from the the corresponding conformal net and conjecture that every unitary VOA is strongly local and that every conformal net comes from a strongly local VOA. The main 
technical problem in the proof of strong locality for specific models comes from the fact that the locality axiom of conformal nets does not follow in an obvious way from the locality property of VOAs. The reason is that the smeared vertex operators are typically unbounded \cite{Bau97} and the commutativity properties of unbounded operators are a  rather sophisticated matter (``a touchy business''), see \cite[Section 10]{Nel59},  \cite[Section VIII.5]{RS80}, \cite[Section 2.2]{CKLW18} and \cite{CTW22}. This problem has an old history in quantum field theory, see e.g.\ \cite{BZ63,LS65,DF77,DSW86,GJ87}.

The relevance of this correspondence is not limited to the understanding of the structural similarities between the two-approaches that, in a certain sense, could be explained by referring to their common physical roots. In fact, it is expected that this correspondence will provide more direct bridges between different important areas 
of mathematics such as the Jones theory of subfactors, the Tomita-Takesaki modular theory of von Neumann algebras and noncommutative geometry on the conformal net side (see e.g.\ \cite{BKLR15,Con94,EG11,EK98,EK23,FRS89,FRS92,GL96,Jon83,Lon89,Lon90,LR95,Was98}) and modular forms, algebraic curves and finite simple groups on the VOA side (see e.g.\ \cite{DM00,FB01,FLM88}). This also open the way to transfer results from the VOA setting to the conformal net setting and {\it viceversa}, cf.\ 
\cite{CGGH23,CWX}. 

Since the first step made in \cite{CKLW18}, the understanding of the above connection has been considerably improved in various directions: new examples of strongly local VOAs, representation theory aspects, relations with Segal CFT, see e.g.\ \cite{CGGH23,CTW22,CWX,Gui19I,Gui19II,Gui,Gui21,Gui22,Ten17,Ten19,Ten19b,Ten24}.

In this paper we extend the correspondence in \cite{CKLW18} in order to cover the unitary vertex operator superalgebras (unitary VOSAs) and graded-local conformal nets 
(also called Fermi conformal nets), cf.\ \cite{AL17,CKL08}. Besides the fact that this generalization is mathematically natural and physically relevant, there are various other reasons indicating that it is highly desirable.  

First of all, the free fermions give the most obvious example of vertex algebras generating graded-local conformal nets, but they are not covered by the analysis in \cite{CKLW18} (note however that the even (Bose) subalgebra of a free fermion VOSAs can be proved to be strongly local  in the the framework of \cite{CKLW18}, see \cite{CWX}). This is because the smeared free Fermi fields are bounded operators, cf.\ \cite{Boc96,Bau97} and Proposition \ref{prop:gen_by_V1/2_V1} in this paper. As a consequence the free fermions provide the easiest examples of {\it strongly graded-local} VOSAs (see Definition \ref{def:strongly_graded_local}), see Example \ref{ex:free_fermion} and Example \ref{ex:d_free_fermions}. Moreover, free fermion VOSAs and their unitary subalgebras play a central role in various works in the AQFT framework: the construction of finite index subfactors \cite{Was98}; the construction of operator algebraic Segal CFTs \cite{Ten17,Ten19,Ten19b}; the analysis of entropy  properties in CFT \cite{CHMP20,LX18,LX21,Xu22}.

A second motivation comes from the study of superconformal symmetry and its relations with Connes' noncommutative geometry and K-theory \cite{Con94}, which naturally leads to consider Fermi  fields and the corresponding graded-locality (or superlocality) \cite{CCH13,CCHW13,CH17,CHKL10,CHKLX15,CHL15,CKL08,Hil14,Hil15}.

Finally, we mention some important recent classification results for unitary (or potentially unitary) chiral fermionic CFTs 
\cite{BLTZ24,CGGH23,HM,JoF21,KMP23,Ray24}. We expect that these will give new and remarkable families of strongly graded-local  VOSAs and hence many new interesting examples of graded-local conformal nets.  Some of these examples are already covered by this paper,  see  e.g.\ Theorem \ref{theo:lattices_strong_graded_locality}, Theorem \ref{theo:super-moonshine}, Theorem \ref{theo:c_leq_12}, Theorem \ref{theo:J-F_N=1_models} and Theorem \ref{theo:baby_moonshine} and we hope to consider further examples in the future. 
In particular, Theorem \ref{theo:J-F_N=1_models} gives new models of graded-local conformal nets carrying the action of a finite group by automorphisms preserving the $N=1$ superconformal structure, including the super-Moonshine conformal net of Theorem \ref{theo:super-moonshine}, described in \cite{Kaw10}. Among them, we can find many sporadic simple groups: for example, we have an action of the semidirect product of the Mathieu group $M_{12}$ of order 95040 with the cyclic group of order two. Moreover, after a first draft of this paper appeared on the arXiv, one of us (T.G.) has proved that all the 969 holomorphic VOSAs with central charge 24 and non trivial odd part classified in \cite{HM} are unitary and that at least 910 of them are strongly graded-local \cite{Gau25}. Furthermore, two of us (S.C.\ and T.G.) proved that all the unitary minimal $W$-algebras classified in \cite{KMP23} are strongly graded-local \cite{CG25}.

We end this introduction by giving some more details on the content and on the main results of this paper. Concerning the VOSA -- graded-local conformal net correspondence, we mainly follow the exposition in \cite{CKLW18}. We are able to prove the VOSA version of essentially all the general results in \cite{CKLW18} such as the equivalence of various definitions of unitarity for VOSAs, the essential uniqueness of the unitary structure and its relations with the properties of the automorphisms group, the unitary subalgebras and their relations to subnets, criteria for strong graded locality based on VOSAs generators and the reconstruction of strongly graded-local VOSAs from the corresponding conformal nets. In various cases the generalization is rather straightforward, although it often requires a certain care in some points where the presence of fermionic fields could make some relevant difference. We give all the details when the arguments given in \cite{CKLW18} (potentially) need some substantial change.
In some cases, on the other hand, the arguments in this paper differ considerably from those of \cite{CKLW18}. A remarkable example is the proof of the Bisognano-Wichman property for smeared vertex operators (acting on the vacuum vector) that we give in Section  \ref{appendix:action_dilation_subgroup}.  This property plays a crucial role 
in one of the main technical results of this article namely,  Theorem  \ref{theo:gen_by_quasi_primary} which gives a central criterion for strong graded locality and in the reconstruction of strongly graded-local VOSAs from the corresponding graded-local conformal nets, discussed in Section \ref{section:back}. Differently from \cite[Appendix B]{CKLW18}, in the proof of this Bisognano-Wichmann property we avoid the use of the theory of covariant nets of real Hilbert spaces, described in \cite{Lon08b}, which is not directly applicable here and rely instead on direct computations in the ``real line picture", cf.\ \cite{Gui21b} for related results. 
As an example of novel result, in Section \ref{subsection:pct_theorem}, we derive a PCT theorem for VOSAs, generalizing \cite[Section 5.2]{CKLW18} not only from the VOA setting to the VOSA one, but also to cover the case of VOSAs (and thus VOAs too) not of CFT type. 

Finally, we mention the fact that we give many examples of strongly graded-local VOSAs with superconformal symmetry, such as e.g.\ the $N=1$ and $N=2$ unitary super-Virasoro 
VOSAs, the Duncan super-Moonshine VOSA and the remaining $N=1$ superconformal VOSAs classified by Johnson-Freyd \cite{JoF21}, see Section \ref{section:examples}. Furthermore, thanks to some recent results and ideas by Gui (see \cite{Gui,Gui21}) we prove  strong graded locality
of the  VOSAs associated to any positive-definite integral lattice and of the H{\"o}hn shorter Moonshine VOSA $VB^\natural$. 

Throughout the paper we will freely use some standard concepts from the theory of operator algebras in Hilbert spaces ($C^*$-algebras and von Neumann algebras) and refer the reader to standard textbooks on the subject such as \cite{Bla06} and \cite{BR02} (cf.\ also  \cite{Con94} and \cite{Ped89}).
\medskip

This work is based in part on T.G.'s Ph.D.\ thesis \cite{Gau21} at Lancaster University.

\section{Graded-local conformal nets}
\label{section:graded_local_conformal_nets}

We review the main facts about the orientation-preserving diffeomorphisms of the unit circle $S^1$ and about graded-local conformal nets in Section \ref{subsection:diff_group} and in Section \ref{subsection:graded-local_conformal_nets} respectively, giving specific references for a detailed exposition. Furthermore, we expose with proofs and references the fundamental facts about covariant subnets in Section \ref{subsection:covariant_subnets}.

\subsection{The diffeomorphisms of the circle} 
\label{subsection:diff_group}

Consider the unit circle:
\begin{equation}
S^1
:=\left\{z\in\C\mid\abs{z}=1\right\}
=\left\{e^{ix}\in \C\mid x\in(-\pi,\pi]\right\} 
\end{equation}
with its usual subspace topology. 
Consider also the Fréchet space $C^\infty(S^1)$ of complex-valued infinitely differentiable function on $S^1$ and its subset $C^\infty(S^1,\R)$ of real-valued functions.
Then we write a \textbf{smooth real vector field} on $S^1$ as $f\frac{\mathrm{d}}{\mathrm{d}x}$, where $f\in C^\infty(S^1,\R)$.
Recall that smooth real vector fields form a topological vector space
$\mathrm{Vect}(S^1)$ with the usual $C^\infty$ Fréchet topology.
$\mathrm{Vect}(S^1)$ can be used to model
$\Diff^+(S^1)$, which is the infinite dimensional Lie group of \textbf{orientation-preserving diffeomorphisms of $S^1$}, see e.g.\ \cite[Section 6]{Mil84} and \cite[Section I.4]{Ham82}, cf.\ also \cite[Section 3.3]{PS88} and \cite[Section 1.1]{Lok94}. Moreover, $\mathrm{Vect}(S^1)$ turns out to be the Lie algebra associated to $\Diff^+(S^1)$, where the Lie bracket $[\cdot,\cdot]$ is given by the negative of the usual bracket of vector fields, that is
\begin{equation}  \label{eq:lie_bracket_Vect}
	[f\frac{\mathrm{d}}{\mathrm{d}x}, g\frac{\mathrm{d}}{\mathrm{d}x}]
	=(f'g-g'f)\frac{\mathrm{d}}{\mathrm{d}x}
	\qquad\forall f,g\in C^\infty(S^1, \R) \,.
\end{equation}
Accordingly, we use the usual symbol $\exp:\mathrm{Vect}(S^1)\rightarrow \Diff^+(S^1)$ for the related exponential map and with an abuse of notation, we use $\exp(tf)$ for the one-parameter subgroup of $\Diff^+(S^1)$ generated by the smooth real vector field $f\frac{d}{dx}$ and $t\in \R$.
$\Diff^+(S^1)$ is connected, but not simply connected, see \cite[Section 10]{Mil84} and \cite[Example 4.2.6]{Ham82}, cf.\ also \cite[Section 1.1]{Lok94}. Thus, for every $n\in\Zplus$, we denote the $n$-cover of $\Diff^+(S^1)$ by $\Diff^+(S^1)^{(n)}$, and its universal cover by $\Diff^+(S^1)^{(\infty)}$ . 
We can identify every 
$\gamma\in\Diff^+(S^1)^{(\infty)}$ 
with a unique function $\phi_\gamma$ in the closed subgroup of diffeomorphisms of $\R$ defined by
\begin{equation}
\left\{
\phi\in C^\infty(\R,\R)\mid \phi(x+2\pi)=\phi(x)+2\pi
\,,\,\, \phi'(x)>0 
\,\,\forall x\in\R
\right\}
\,.
\end{equation}
Under this identification, the center of $\Diff^+(S^1)^{(\infty)}$ is given by
\begin{equation}
\{
\phi\in\Diff^+(S^1)^{(\infty)} \mid
\phi(x)=x+2\pi k \,\,\,\forall x\in\R \,,\,\,\, k\in\Z
\} \cong 2\pi\Z \,. 
\end{equation}
Accordingly,
for every $n\in\Zplus$, 
we have the following identification
\begin{equation}  \label{eq:quotien_diff+_sim_n}
\Diff^+(S^1)^{(n)}\cong
\faktor{\Diff^+(S^1)^{(\infty)}}{2n\pi\Z}
\,\,\,.
\end{equation}
For every $n\in\Zplus$, we use $\phi_\gamma\in\Diff^+(S^1)^{(\infty)}$ to denote a representative of any $\gamma\in\Diff^+(S^1)^{(n)}$ in its $2n\pi\Z$-class of diffeomorphisms, which is given by $\{\phi_\gamma+2n\pi k\mid k\in\Z\}$. Moreover, note that for all $\gamma,\gamma_1,\gamma_2\in\Diff^+(S^1)^{(n)}$, $\phi_{\gamma_1\gamma_2}$ and $\phi_{\gamma_1}\circ\phi_{\gamma_2}$ are in the same class of diffeomorphisms as well as $\phi_\gamma^{-1}$ and $\phi_{\gamma^{-1}}$.
For all $n\in\Zplus\cup\{\infty\}$ and all $\gamma\in\Diff^+(S^1)^{(n)}$, we use $\dot{\gamma}$ for the image of $\gamma$ under the covering map $p:\Diff^+(S^1)^{(\infty)}\mapsto\Diff^+(S^1)$ and $\exp^{(n)}(tf)$ for the lift of $\exp(tf)$ to $\Diff^+(S^1)^{(n)}$.

\begin{rem}  \label{rem:diffeomorphisms_generated_by_exp}
By the results of Epstein \cite{Eps70}, Herman \cite{Her71} and Thurston \cite{Thu74}, cf.\ \cite[Proposition (3.3.3)]{PS88}, we know that $\Diff^+(S^1)$ is a simple group. It follows that every element of $\Diff^+(S^1)$ is a finite product of exponentials of vector fields $\exp(tf)$, where $f\in C^\infty(S^1,\R)$ with $\mathrm{supp}f\subset I$ for some interval $I\subset S^1$, see \cite[Remark 1.7]{Mil84} and \cite[Remark 3.1]{CKLW18}. It also follows that every element of $\Diff^+(S^1)^{(2)}$ is a finite product of exponentials of vector fields $\exp^{(2)}(tf)$ with $f$ as above by the proof of \cite[Proposition 38]{CKL08}.
\end{rem}

Now, consider the double cover $S^{1(2)}$ of the circle $S^1$, which we  identify with the set
\begin{equation}
\left\{
e^{i\frac{x}{2}}\in\C \mid x\in(-2\pi,2\pi]
\right\} \,.
\end{equation}
Then every $\gamma\in\Diff^+(S^1)^{(2)}$ acts naturally on it by $\gamma(e^{i\frac{x}{2}})=e^{i\frac{\phi_\gamma(x)}{2}}$, where $\phi_\gamma$ is one of the representatives in its $4\pi\Z$-class of diffeomorphisms given by \eqref{eq:quotien_diff+_sim_n}. Note that the former action does not depend on the choice of the representative $\phi_\gamma$.

Consider the complexification $\mathrm{Vect}_\C(S^1)$ of $\mathrm{Vect}(S^1)$ (with the Lie bracket $[\cdot,\cdot]$ given by the natural extension of the bracket given in \eqref{eq:lie_bracket_Vect}) and its infinite dimensional Lie subalgebra spanned by generators $l_n:=-ie^{inx}\frac{d}{dx}$ for all $n\in\Z$, known as the \textit{complex Witt algebra} $\mathfrak{Witt}$. Then these generators satisfy the commutation relations:
\begin{equation}
[l_n,l_m]=(n-m)l_{n+m}
\qquad\forall n,m\in\Z \,.
\end{equation}
The \textbf{Virasoro algebra}, see \cite[Lecture 1]{KR87} for details, is defined as the non-trivial central extension $\mathfrak{Witt}\oplus \C k$ with commutation relations:
\begin{equation}
[l_n,l_m]=(n-m)l_{n+m}+\frac{n^3-n}{12}\delta_{n,-m}k
\,,\quad
[l_n,k]=0
\qquad\forall n,m\in\Z \,.
\end{equation}
Note that the Lie subalgebra of $\mathfrak{Witt}$ generated by $l_n$ with $n\in\{-1,0,1\}$ is isomorphic to the finite Lie algebra $\mathfrak{sl}(2,\C)$. Moreover, the smooth real vector fields $il_0$, $\frac{i(l_1+l_{-1})}{2}$ and $\frac{l_1-l_{-1}}{2}$ generate a Lie subalgebra of $\mathrm{Vect}(S^1)$ isomorphic to $\mathfrak{sl}(2,\R)\cong \operatorname{PSU}(1,1)\cong\operatorname{PSL}(2,\R)$, which exponentiates to the three-dimensional Lie subgroup $\Mob(S^1)$ of $\Diff^+(S^1)$ of \textbf{M\"{o}bius transformations} of $S^1$. 
In particular, with $t\in\R$,
\begin{equation} 		\label{eq:defin_rot_dilat_transl}
r(t):=\exp(itl_0)
\,,\quad
\delta(t):=\exp\left(t\frac{l_1-l_{-1}}{2}\right)
\,,\quad
w(t):=\exp\left(\frac{it}{2}l_0+it\frac{l_1+l_{-1}}{4}\right)
\end{equation}
are known as the \textbf{rotation}, \textbf{dilation} and \textbf{translation subgroups} of $\Mob(S^1)$ respectively. Note also that these transformations act on $S^1$ by

\begin{align}  
	r(t)z &= e^{it}z  \qquad\forall t\in\R \,,\quad \forall z\in S^1 
	\label{eq:rotation_action_circle}\\
	\delta(t)z &= \frac{z\cosh(t/2)-\sinh(t/2)}{-z\sinh(t/2)+\cosh(t/2)} 
	\qquad\forall t\in\R \,,\quad \forall z\in S^1 
	\label{eq:dilation_action_circle}\\
	w(t)z &= \frac{z(4+it)+it}{-zit+(4-it)} 
	\qquad\forall t\in\R \,,\quad \forall z\in S^1 
	\label{eq:translation_action_circle}\,. 
\end{align}
Note that equation \eqref{eq:rotation_action_circle} explains why $r(t)$ is called a rotation, whereas to better understand the remaining nomenclature and for further use, we introduce the \textbf{Cayley transform}, that is the diffeomorphism $C:S^1\backslash \{-1\} \to \R$ defined by

\begin{equation} \label{eq:cayley_transform}
	C(z):=2i\frac{1-z}{1+z} \quad \forall z\in S^1
	\,,\qquad
	C^{-1}(x)=\frac{1+\frac{i}{2}x}{1-\frac{i}{2}x}
	\quad \forall x\in\R \,.
\end{equation}
(Note that $\lim_{t\rightarrow\pm\pi^{\mp}}C(e^{it})=\pm\infty$ and $\lim_{x\rightarrow\pm\infty}C^{-1}(x)=-1$.)
Then $C\delta(t) C^{-1}(x)= e^tx$ and $Cw(t)C^{-1}(x)= x+t$ for all $t, x\in \R$.

For all $n\in\Zplus\cup\{\infty\}$, we use $\Mob(S^1)^{(n)}$ to denote the $n$-cover of $\Mob(S^1)$ in $\Diff^+(S^1)^{(n)}$. 
In particular, $\Mob(S^1)^{(2)}=\operatorname{SU}(1,1)\cong\operatorname{SL}(2,\R)$. Recall that the covering map $p$ restricts to the one of $\Mob(S^1)^{(\infty)}$ on $\Mob(S^1)$. 
Moreover, we have the identifications inherited from \eqref{eq:quotien_diff+_sim_n} for every $\Mob(S^1)^{(n)}$, namely
\begin{equation}  \label{eq:quotien_mob_sim_n}
\Mob(S^1)^{(n)}\cong
\faktor{\Mob(S^1)^{(\infty)}}{2n\pi\Z}\qquad
\forall n\in\Zplus\,.
\end{equation}
For every $n\in\Zplus\cup\{\infty\}$, we use $r^{(n)}(t)$, $\delta^{(n)}(t)$ and $w^{(n)}(t)$ for the lifts of the corresponding subgroups given by the exponential map.

Crucial for our analysis is the representation theory of $\Diff^+(S^1)^{(\infty)}$ and of $\Mob(S^1)^{(\infty)}$. We do not give here a complete treatment of such theory, which can be found in a resumed form e.g.\ in \cite[Section 3.2]{CKLW18}, but we present some key facts which can be useful to recall before proceeding further.

For a topological group $\mathscr{G}$, a \textbf{strongly continuous projective unitary representation} $U$ on an Hilbert space $\mathcal{H}$ is a strongly continuous homomorphism from $\mathscr{G}$ to the quotient $\operatorname{U}(\mathcal{H})/\T$ of the group of unitary operators on $\mathcal{H}$ by the circle subgroup $\T$ of operators with unit norm. 
Note that although for $\gamma\in \mathscr{G}$,  $U(\gamma)$ is defined only up to a phase as an operator
on $\mathcal{H}$, expressions like $U(\gamma)\in \mathcal{K}$ for a (complex) linear subspace $\mathcal{K}\subseteq B(\mathcal{H})$ or $U(\gamma)XU(\gamma)^*$ for any $X\in B(\mathcal{H})$  are unambiguous for all $\gamma\in\mathcal{G}$ and are frequently used
in this paper.
A strongly continuous projective unitary representation $U$ of $\Diff^+(S^1)^{(\infty)}$ restricts to one of the subgroup $\Mob(S^1)^{(\infty)}$, which can always be lift to a strongly continuous unitary representation $\widetilde{U}$ of $\Mob(S^1)^{(\infty)}$, see \cite{Bar54}. $U$ is said to be \textbf{positive-energy} if $\widetilde{U}$ is, that is, if the infinitesimal self-adjoint generator $L_0$ of the strongly continuous one-parameter group $t\mapsto e^{itL_0}:=\widetilde{U}(r^{(\infty)}(t))$ is a positive operator on $\mathcal{H}$, namely, it has a non-negative spectrum. Note that $U$ factors through a positive-energy strongly continuous projective unitary representation of $\Diff^+(S^1)^{(2)}$ whenever $e^{i4\pi L_0}=1_\mathcal{H}$. In that case, it restricts also to a positive-energy strongly continuous projective unitary representation of $\Mob(S^1)^{(2)}$.

A \textbf{positive-energy unitary representation} of the Virasoro algebra $\Vir$ is a Lie algebra representation $\pi$ of $\Vir$ on a complex vector space $V$ equipped with a scalar product $\scalar$ such that:
\begin{itemize}
\item
$\pi$ is \textbf{unitary}, that is, if $L_n:=\pi(l_n)$ for all $n\in\Z$, then $(a|L_nb)=(L_{-n}a|b)$ for all $n\in\Zpluseq$ and all $a,b\in V$;

\item 
$\pi$ has \textbf{positive energy}, which means that $L_0$ is diagonalizable on $V$ with non-negative eigenvalues, namely,
\begin{equation}
V=\bigoplus_{\lambda\in\R_{\geq 0}} V_\lambda
\,,\quad
V_\lambda:=\operatorname{Ker}(L_0-\lambda 1_V)
\,;
\end{equation}

\item 
$\pi$ has a \textbf{central charge} $c\in\C$, that is, if $K:=\pi(k)$, then $K=c1_V$.
\end{itemize}

Finally, we remark that there is a well-known correspondence between positive-energy strongly continuous projective unitary representations of $\Diff^+(S^1)^{(\infty)}$ and positive-energy unitary representations of $\Vir$, which is summarized in \cite[Theorem 3.4]{CKLW18}:

\begin{theo}   \label{theo:representations_Diff_Vir}
Every positive-energy unitary representation $\pi$ of the Virasoro algebra $\Vir$ on a complex inner product space $V$ integrates to a unique positive-energy strongly continuous projective unitary representation $U_\pi$ of $\Diff^+(S^1)^{(\infty)}$ on the Hilbert space completion $\mathcal{H}$ of $V$. 
Furthermore, every positive-energy strongly continuous projective unitary representation of $\Diff^+(S^1)^{(\infty)}$ on a Hilbert space $\mathcal{H}$ arises in this way, whenever the subspace
$\mathcal{H}^\mathrm{fin}:=\bigoplus_{\alpha\in\R_{\geq0}}
\mathrm{Ker}(L_0-\alpha 1_\mathcal{H})$
is dense in $\mathcal{H}$.
The map $\pi\mapsto U_\pi$ becomes one-to-one when restricting to representations $\pi$ on the complex inner product spaces $V$, whose Hilbert space completion $\mathcal{H}$ satisfies $\mathcal{H}^\mathrm{fin}=V$. These are exactly those complex inner product spaces $V$ such that $V_\lambda:=\operatorname{Ker}(L_0-\lambda 1_V)\subseteq V$ is complete (that is, it is a Hilbert space) for all $\lambda\in\R_{\geq 0}$. Finally, $U_\pi$ is irreducible if and only if $\pi$ is irreducible, i.e. if and only if the corresponding $\Vir$-module is $L(c,0)$ for some $c\geq 1$ or $c_m=1-\frac{6}{m(m+1)}$ with integers $m\geq 2$.
\end{theo}

The proofs of the facts stated by Theorem \ref{theo:representations_Diff_Vir} are essentially given by \cite{GW85} and \cite[Theorem 5.2.1, Proposition 5.2.4]{Tol99} for the integrability statement and by \cite[Chapter 1]{Lok94} and \cite[Appendix A]{Car04} for the remaining part, see \cite{GKO86} and \cite[Lecture 3]{KR87} for the definition of the $\Vir$-modules $L(c,0)$; see also \cite{NS15} and \cite{CDIT22} for related results.

\subsection{Preliminaries on graded-local conformal nets}
\label{subsection:graded-local_conformal_nets}

We recall some preliminaries about graded-local conformal nets from \cite[Section 2]{CKL08}, see also \cite[Section 2]{CHL15} and references therein.

Let $\J$ be the set of \textbf{intervals} of the unit circle $S^1$, that is, the open, connected, non-empty and non-dense subsets of $S^1$. 
For every interval $I\in\J$, we set $I':=S^1\backslash\overline{I}$, that is the complement of $I$ in $S^1$. Set 
\begin{equation}
\Diff(I):=\left\{\gamma\in\Diff^+(S^1)\mid \gamma(z)=z \,\,\forall z\in I'\right\}
\qquad\forall I\in\J \,.
\end{equation}
Accordingly to the notation of Section \ref{subsection:diff_group}, for all $I\in\J$ and all $n\in\Zplus\cup\left\{\infty\right\}$, 
$\Diff(I)^{(n)}$ is the connected component to the identity of the pre-image of 
$\Diff(I)$ in 
$\Diff^+(S^1)^{(n)}$ under the covering map $p:\Diff^+(S^1)^{(\infty)}\to \Diff^+(S^1)$.

A \textbf{graded-local M{\"o}bius covariant net on $S^1$} is a family of von Neumann algebras $\A:=(\A(I))_{I\in\mathcal{J}}$ on a fixed separable Hilbert space $\mathcal{H}$, also called \textbf{net}
of von Neumann algebras, with the following properties:
\begin{itemize}
\item[\textbf{(A)}] \textbf{Isotony}. For all $I_1, I_2\in\mathcal{J}$ such that $I_1\subseteq I_2$, then $\A(I_1)\subseteq\A(I_2)$.

\item[\textbf{(B)}] \textbf{M{\"o}bius covariance}.  There exists a strongly continuous unitary representation $U$ of $\Mob(S^1)^{(\infty)}$ on $\mathcal{H}$ such that
\begin{equation}
U(\gamma)\A(I)U(\gamma)^{-1}=\A(\dot{\gamma} I)
\qquad\forall \gamma\in\Mob(S^1)^{(\infty)},
\,\,\,\forall I\in\J \,.
\end{equation}

\item[\textbf{(C)}] \textbf{Positivity of the energy}. The infinitesimal generator $L_0$ of the (universal cover of) rotation one-parameter subgroup $t\mapsto U(r^{(\infty)}(t))$, called \textbf{conformal Hamiltonian}, is a positive operator on $\mathcal{H}$. This means that $U$ is a positive-energy representation.

\item[\textbf{(D)}] \textbf{Existence of the vacuum}. There exists a $U$-invariant unit vector $\Omega\in\mathcal{H}$, called \textbf{vacuum vector}, which is also cyclic for the von Neumann algebra $\A(S^1):=\bigvee_{I\in\J}\A(I)$, that is the one generated by $\A(I)$ for all $I\in\J$.

\item[\textbf{(E)}] \textbf{Graded locality}. There exists a self-adjoint unitary operator $\Gamma$ on $\mathcal{H}$, called \textbf{grading unitary}, such that $\Gamma\Omega=\Omega$ and
\begin{equation}   \label{eq:graded-locality_nets}
\Gamma \A(I)\Gamma=\A(I)\,,
\quad
\A(I')\subseteq Z\A(I)'Z^*
\qquad \forall I\in\J
\end{equation}
with 
\begin{equation}   \label{eq:defin_Z_net}
Z:=\frac{1_\mathcal{H}-i\Gamma}{1-i} \,.
\end{equation}
\end{itemize}

Note that every operator $A\in\A(I)$ for any $I\in\J$ can be written as the sum of the two operators
\begin{equation}
A_0:=\frac{A+\Gamma A\Gamma}{2} \,,
\qquad
A_1:=\frac{A-\Gamma A\Gamma}{2}
\end{equation} 
which have the property that $\Gamma A_k\Gamma=(-1)^kA_k$ for all $k\in\{0,1\}$. Then we call an operator $A$ \textbf{homogeneous} with \textbf{degree} $\partial A$ equal to either $0$ or $1$ if $\Gamma A\Gamma$ is equal to either $A$ or $-A$ respectively. Contextually, we say $A$ is either a \textbf{Bose} or a \textbf{Fermi} element respectively. With this notation, graded-locality \eqref{eq:graded-locality_nets} is equivalent to saying that $[\A(I_1),\A(I_2)]=0$ whenever $I_1\subseteq I_2'$, where the graded commutator here is defined on homogeneous elements $A$ and $B$ by 
\[
[A,B]:=AB-(-1)^{\partial A\partial B}BA
\]
and thus extended to arbitrary elements by linearity.

A graded-local M\"{o}bius covariant net $\A$ is said to be \textbf{irreducible} if $\A(S^1)=B(\mathcal{H})$. This is equivalent to say that $\Omega$ in axiom \textbf{(D)} above is the unique $U$-invariant vector up to a multiplication for a complex number and equivalently that all local algebras $\A(I)$ are type $III_1$ factors (provided that $\A(I)\not\cong\C$ for all $I\in\J$), see \cite[Proposition 7]{CKL08} and references therein. 

The axioms \textbf{(A)-(E)} above have important consequences.
 We   rewrite  here below the ones which we will use more often throughout the paper:
\begin{itemize}
\item
\textbf{Reeh-Schlieder theorem}. See \cite[Theorem 1]{CKL08}.
The vacuum $\Omega$ is cyclic and separating for every von Neumann algebra $\A(I)$ with $I\in\J$.

\item
\textbf{Bisognano-Wichmann property}. \label{bisognano-wichmann_property}
See \cite[Theorem 2]{CKL08}.
For every $I\in\J$, let $\delta_I(t)$ with $t\in\R$ be the one-parameter subgroup of $\Mob(S^1)$ of \textbf{dilations associated to $I$}, see \cite[p.\ 15]{GL96}, cf.\ \cite[Section 1.1]{Lon08}. Use also 
$j_I:S^1\rightarrow S^1$ to denote the \textbf{reflection associated to $I$}, mapping $I$ onto $I'$, see \cite[p.\ 15]{GL96}, cf.\ \cite[Section 1.6.2]{Lon08}.
Let $\Delta_I$ and $J_I$ be the modular operator and the modular conjugation respectively associated to the couple $(\A(I),\Omega)$ from the Tomita-Takesaki modular theory for von Neumann algebras, see \cite[Chapter 2.5]{BR02}. Then
\begin{equation}  \label{eq:B-W_dilations}
U(\delta_I^{(\infty)}(-2\pi t))=\Delta_I^{it}
\qquad \forall I\in \J
\,\,\,\forall t\in\R \,.
\end{equation}
Moreover, $U$ extends uniquely to an anti-unitary representation of $\Mob(S^1)^{(\infty)}\ltimes \Z_2$, where $\Z_2:=\Z/2\Z$ is the cyclic group of order two, given by
\begin{equation}   \label{eq:B-W_reflection}
U(j_I)=ZJ_I
\qquad\forall I\in\J
\end{equation}
and acting covariantly on $\A$, that is,
\begin{equation}   \label{B-W_covariance}
U(\gamma)\A(I) U(\gamma)^*
=\A(\dot{\gamma}I)
\qquad\forall \gamma\in \Mob(S^1)^{(\infty)}\ltimes \Z_2
\,\,\,\forall I\in\J \,.
\end{equation}

\item  \label{property:uniqueness_mob_repr}
\textbf{Uniqueness of the M{\"o}bius representation}. See \cite[Corollary 3]{CKL08}. The strongly continuous unitary representation $U$ of $\Mob(S^1)^{(\infty)}$ is unique.

\item 
\textbf{Twisted Haag duality}.  See \cite[Theorem 5]{CKL08}.
\begin{equation}  \label{eq:twisted_haag_duality}
\A(I')=Z\A(I)'Z^*=Z^*\A(I)'Z
\qquad \forall I\in\J \,.
\end{equation}

\item 
\textbf{Vacuum Spin-Statistics theorem}.       
See \cite[Proposition 8 and Corollary 9]{CKL08}. 
The grading unitary $\Gamma$ is unique and
\begin{equation}   \label{eq:spin-statistics_theorem}
U(r^{(\infty)}(2\pi))=e^{i2\pi L_0}=\Gamma\,.
\end{equation}

\item
\textbf{External continuity}. 
By M{\"o}bius covariance, see \cite[Section 2.2, Remark]{CKL08}, cf.\ \cite[p.\  48]{Lon08}, it holds that
\begin{equation}\label{eq:external_continuity}
\A(I)=\bigcap_{\J\ni J\supset \overline{I}} \A(J)
\qquad\forall I\in\J \,.
\end{equation}
\end{itemize}

Strictly speaking, in \cite{CKL08}, the graded-local M{\"o}bius covariant nets are assumed to be irreducible, but the above properties hold also without this assumption, see \cite{GL96} and \cite{DLR01}.

A \textbf{graded-local conformal net on $S^1$} is a graded-local M{\"o}bius covariant net on $S^1$ with the following additional property:
\begin{itemize}
\item[\textbf{(F)}] \textbf{Diffeomorphism covariance}. There exists a strongly continuous projective unitary representation $U^\mathrm{ext}$ of $\Diff^+(S^1)^{(\infty)}$, which extends $U$ and such that
\begin{align}
U^\mathrm{ext}(\gamma)\A(I)U^\mathrm{ext}(\gamma)^{-1} 
 &=
\A(\dot{\gamma} I) 
\qquad\forall \gamma\in\Diff^+(S^1)^{(\infty)}
\,\,\,\forall I\in\J \,;  \label{diff_covar} \\
U^\mathrm{ext}(\gamma)AU^\mathrm{ext}(\gamma)^{-1} 
 &=
A
\qquad\forall \gamma\in\Diff(I)^{(\infty)}
\,\,\,\forall A\in\A(I')
\,\,\,\forall I\in\J  \label{diff_covar_graded_local} 
\,.
\end{align}
With an abuse of notation, we use the same symbol $U$ to denote $U^\mathrm{ext}$.
\end{itemize}

Note that the Vacuum Spin-Statistics theorem above implies that the representation $U$ of $\Mob(S^1)^{(\infty)}$ and consequently its extension to $\Diff^+(S^1)^{(\infty)}$ factor through representations of $\Mob(S^1)^{(2)}$ and $\Diff^+(S^1)^{(2)}$ respectively, which we still denote by $U$. Moreover, we have the following uniqueness result:

\begin{itemize}
	\item \label{property:uniqueness_diff_repr} 
	\textbf{Uniqueness of the diffeomorphism representation}. See \cite[Corollary 11]{CKL08}. If $\A$ is an irreducible graded-local conformal net, then the strongly continuous projective unitary representation $U$ of $\Diff^+(S^1)^{(2)}$ (and thus as representation of $\Diff^+(S^1)^{(\infty)}$ too) is unique (up to a projective phase).
\end{itemize}

A graded-local M{\"o}bius covariant (or conformal) net such that $\Gamma=1_\mathcal{H}$ is said to be \textbf{local} and simply called \textbf{M{\"o}bius covariant (or conformal) net}.

Let $(\A_1,\mathcal{H}_1,\Omega_1,U_1)$ and $(\A_2,\mathcal{H}_2,\Omega_2,U_2)$ be two graded-local M{\"o}bius covariant nets. We say that $\A_1$ and $\A_2$ are \textbf{isomorphic} or \textit{unitarily equivalent} if there exists a unitary operator $\phi:\mathcal{H}_1\rightarrow\mathcal{H}_2$ such that $\phi(\Omega_1)=\Omega_2$ and $\phi\A_1(I)\phi^{-1}=\A_2(I)$ for all $I\in\J$. By the uniqueness of the M{\"o}bius representation, it follows that 
\begin{equation}			\label{eq:uniqueness_mob_symmetries}
	\phi U_1(\gamma)\phi^{-1}=U_2(\gamma) \qquad\forall \gamma\in \Mob(S^1)^{(\infty)} \,.	
\end{equation} 
We define the \textbf{automorphism group} $\Aut(\A)$ of the graded-local M{\"o}bius covariant net $\A$ as
\begin{equation}  \label{eq:defin_automorphism_group_net}
	\Aut(\A):=\left\{
	\phi\in U(\mathcal{H}) \mid 
	\phi(\Omega)=\Omega  \,,\,\,\,
	\phi\A(I)\phi^{-1}=\A(I) \,\,\,\forall I\in\J
	\right\}.
\end{equation}
Therefore, equation \eqref{eq:uniqueness_mob_symmetries} implies that every $\phi\in\Aut(\A)$ commutes with $U(\gamma)$ for all $\gamma\in\Mob(S^1)^{(\infty)}$. 
Note also that $\Aut(\A)$ equipped with the topology induced by the strong one of $B(\mathcal{H})$ is a topological group.
Furthermore, we have the following desirable result:

\begin{prop} \label{prop:automorphisms_and_symmetries}
	Let $\phi:\mathcal{H}_1\to\mathcal{H}_2$ be an isomorphism between two irreducible graded-local M\"{o}bius covariant nets $\A_1$ and $\A_2$. If $\A_1$ and $\A_2$ are also diffeomorphism covariant, then 
	$$
	\phi U_1(\gamma)\phi^*=U_2(\gamma)
	\qquad\forall \gamma\in\Diff^+(S^1)^{(\infty)}
	$$
	where the equalities are meant in the projective unitary group $U(\mathcal{H}_2)/\T$.
	Furthermore, if $\phi\in\Aut(\A)$ for an irreducible graded-local conformal net $\A$, then 
	$$
	\phi U(\gamma)\phi^*=U(\gamma)
	\qquad\forall \gamma\in\Diff^+(S^1)^{(\infty)}
	$$
	where the equalities are meant in the unitary group $U(\mathcal{H})$.
\end{prop}

\begin{proof}
	Thanks to the Vacuum Spin-Statistics theorem, see \eqref{eq:spin-statistics_theorem}, the representations $U_1$, $U_2$ and $U$ of $\Diff^+(S^1)^{(\infty)}$ factor through representations of $\Diff^+(S^1)^{(2)}$, which we still denote by the same symbols.
	Then the first part follows from the uniqueness of the diffeomorphism representation, see p.\ \pageref{property:uniqueness_diff_repr}. 
	
	Now, reasoning as in the proof of \cite[Corollary 5.8]{CW05},	we get that $\phi U(\gamma)\phi^*U(\gamma)^*$ is equal to $\alpha(\gamma)1_\mathcal{H}$ for all $\gamma\in\Diff^+(S^1)^{(2)}$, where $\alpha$ is a continous morphism from $\Diff^+(S^1)^{(2)}$ to the group of complex numbers with norm $1$. In other words, $\alpha$ is a character. Therefore, we have to prove that $\alpha(\gamma)=1$ for all $\gamma\in\Diff^+(S^1)^{(2)}$.
	To this aim, let $H$ be the kernel of the character $\alpha$. Then $p(H)$, that is the projection of $H$ under the covering map $p:\Diff^+(S^1)^{(2)}\rightarrow\Diff^+(S^1)$, is a non-trivial normal subgroup of $\Diff^+(S^1)$. It follows that $\Diff^+(S^1)= p(H)$ because $\Diff^+(S^1)$ is a simple group, see Remark \ref{rem:diffeomorphisms_generated_by_exp}. Hence, $\Diff^+(S^1)^{(2)}=H\cup H z$, where $z:=r^{(2)}(2\pi)$ is the rotation by $2\pi$ in $\Diff^+(S^1)^{(2)}$, see \eqref{eq:defin_rot_dilat_transl} for the notation. Suppose by contradiction that $z\not\in H$. Note that $H$ is closed by the continuity of $\alpha$ and thus $H z$ is closed too. This means that $\Diff^+(S^1)^{(2)}$ is the union of two disjoint closed subsets, which contradicts the connectedness of $\Diff^+(S^1)^{(\infty)}$. As a consequence, $z\in H$ and $\Diff^+(S^1)^{(2)}=H$, that is $\alpha(\gamma)=1$ for all $\gamma\in\Diff^+(S^1)^{(2)}$, as desired. 
\end{proof}

\subsection{Covariant subnets} \label{subsection:covariant_subnets}

A \textbf{M\"{o}bius covariant subnet} $\B$ of $\A$ is a family of von Neumann algebras $(\B(I))_{I\in\J}$ acting on $\mathcal{H}$ with the following properties:
\begin{itemize}
\item $\B(I)\subseteq\A(I)$ for all $I\in\J$;

\item if $I_1\subseteq I_2$ are intervals in $\J$, then $\B(I_1)\subseteq\B(I_2)$;

\item $U(\gamma)\B(I)U(\gamma)^*=\B(\dot{\gamma}I)$ for all $\gamma\in\Mob(S^1)^{(\infty)}$ and all $I\in\J$.
\end{itemize}
We use $\B\subseteq\A$ to say that $\B$ is a M{\"o}bius covariant subnet of $\A$.
Define 
$$
\B(S^1):=\bigvee_{\substack{I\in\J \\ I\subset S^1}} \B(I)
\,,\qquad
\mathcal{H}_\B:=\overline{\B(S^1)\Omega}\subseteq\mathcal{H} 
$$
and denote by $e_\B\in \B(S^1)'\cap U(\Mob(S^1)^{(\infty)})'$ the projection on the Hilbert space $\mathcal{H}_\B$. Then we have:
\begin{prop}  \label{prop:covariant_subnets}
	Let $\B$ be a M\"{o}bius covariant subnet of a graded-local M\"{o}bius covariant net $\A$. Then the restriction maps
	$$
	\J\ni I\mapsto \B(I)_{e_\B}:=\B(I)\restriction_{\mathcal{H}_\B}
	\,,\qquad
	\Mob(S^1)^{(\infty)}\ni\gamma\mapsto 
	U(\gamma)\restriction_{\mathcal{H}_\B}
	$$
	define a graded-local M\"{o}bius covariant net $\B_{e_\B}$ acting on $\mathcal{H}_\B$, which is irreducible if $\A$ is. 
	
	If $\A$ is also irreducible and conformal, then there is a strongly continuous projective unitary representation $U_\B$ of $\Diff^+(S^1)^{(\infty)}$ acting on the vacuum Hilbert space $\mathcal{H}$ of $\A$  such that for any $I \in \J$, $U_\B(\gamma) \in \B(I)$ for all $\gamma \in \Diff(I)^{(\infty)}$ and  $U_\B(\gamma)^*U(\gamma) \in \B(S^1)'$ for all $\gamma \in \Diff^+(S^1)^{(\infty)}$. As a consequence, $\B$ is a diffeomorphism covariant subnet of $\A$, that is, $U(\gamma)\B(I)U(\gamma)^*=\B(\dot{\gamma} I)$ for all $\gamma\in\Diff^+(S^1)^{(\infty)}$ and all $I\in\J$. Moreover, 
	${\mathcal{H}_\B}$ is globally invariant for $U_\B$ and  $U(\gamma)\restriction_{{\mathcal{H}_\B}} = U_\B(\gamma)\restriction_{{\mathcal{H}_\B}}$ for all $\gamma \in \Mob(S^1)^{(\infty)}$. 
	In particular, $U_\B \restriction_{{\mathcal{H}_\B}}$ makes $\B_{e_\B}$ into an irreducible graded-local conformal net.
\end{prop} 

\begin{rem}  \label{rem:cyclicity_covariant_subnets}
	Note that $\mathcal{H}_\B=\mathcal{H}$ if and only if $\B(I)=\A(I)$ for all $I\in\J$, see \cite[Proposition 2.30]{Ten19}. This means that $\Omega$ is generally not cyclic for $\B(S^1)$ and thus $\B$ does not determine a graded-local M\"{o}bius covariant net on $\mathcal{H}$.
\end{rem}

\begin{proof}
	If $\A$ is a graded-local M\"{o}bius covariant net, then properties \textbf{(A)}-\textbf{(D)} are clearly satisfied by $\B_{e_\B}$ and $U(\cdot)\restriction_{\mathcal{H}_\B}$ on the Hilbert space $\mathcal{H}_\B$ with vacuum $\Omega$. It remains to prove the graded-locality \textbf{(D)} of $\B_{e_\B}$. By the Vacuum Spin-Statistics theorem and the M\"{o}bius covariance of $\B$, it is clear that $\Gamma \B(I)_{e_\B}\Gamma= \B(I)_{e_\B}$ for all $I\in\J$. Moreover, by the graded-locality of $\A$, we get that 
	$$
	\B(I')\subseteq \A(I')\subseteq Z\A(I)' Z^*
	\subseteq Z\B(I)'Z^*
	\qquad\forall I\in\J \,,
	$$
	which implies the graded-locality of $B_{e_\B}$.
	If $\A$ is irreducible, then clearly $\Omega$ is the unique $U$-invariant vector in $\mathcal{H}_\B$.
	To sum up, we have proved that $\B_{e_\B}$ is a M\"{o}bius covariant net on $\mathcal{H}_\B$, which is irreducible if $\A$ is.

	The second part follows by a rather straightforward adaptation of the results in \cite[Section 6.2]{Wei05} to the graded-local case, cf.\ \cite[Section 3.4]{CKLW18}. For the convenience of the reader, we outline here some of the main results of \cite[Section 6.2]{Wei05} and we give few hints for their adaptation to the graded-local case. The starting point is the unitary representation of the Virasoro algebra on $\mathcal{H}^\mathrm{fin}$ with central charge $c \geq 0$ associated with the projective unitary  representation $U$ making $\A$ conformal, see Theorem \ref{theo:representations_Diff_Vir}. This representation is used to define another unitary representation of the Virasoro algebra on $\mathcal{H}^\mathrm{fin}$ by operators $L^\B_n$, $n\in \Z$, having non-negative central charge $c_\B \leq c$ and leaving  $\mathcal{H}_\B \cap \mathcal{H}^\mathrm{fin}$ globally invariant, see the paragraph before \cite[Eq.\ (6.31)]{Wei05}. Thanks to the estimates in \cite[Proposition 6.2.23]{Wei05} and to \cite[Theorem 5.2.1]{Tol99}, the latter representation integrates to a strongly continuous 
	projective unitary  representation $U_\B$ of $\Diff^+(S^1)^{(\infty)}$ on $\mathcal{H}$ having $\mathcal{H}^\B$ has an invariant subspace. 
	The fact that, for any $I \in \J$, $U_\B(\gamma) \in \B(I)$ for all $\gamma \in \Diff(I)^{(\infty)}$, follows from \cite[Eq.\ (6.19)]{Wei05} together with the fact that the group generated by the exponentials of smooth real vector fields on $S^1$ having support contained in $I$ is dense in 
	$\Diff(I)^{(\infty)}$, cf.\ \cite[Subsection 2.4]{Car04}. On the other hand, the fact that  $U_\B(\gamma)^*U(\gamma) \in \B(S^1)'$ for all $\gamma \in \Diff^+(S^1)^{(\infty)}$ follows from (the proof of) \cite[Proposition 6.2.28]{Wei05}. Note that the proof of \cite[Corollary 6.2.13]{Wei05} uses \cite[Proposition 3.7]{Car04}, which can be adapted to the graded-local case, as done e.g.\ in the proof of Theorem \ref{theo:diffeo_covariant_net_from_V}. Note also that in the proof of \cite[Proposition 6.2.23]{Wei05}, results from \cite{Kos02} are used. Anyway, these are expressed in a wider generality for von Neumann algebras and thus they can be used in the graded-local setting too, using the version of \cite{Kos02} (see the discussion just after the proof of \cite[Theorem 5]{Kos02}) given by \cite{Kos}.  
\end{proof}

Thanks to the above result,
we may denote the net $\B_{e_\B}$ by just $\B$, eventually specifying if it acts on $\mathcal{H}_\B$ or $\mathcal{H}$ respectively, whenever confusion can arise. Furthermore, we can refer to a M{\"o}bius covariant subnet simply as a \textbf{covariant subnet}.

We end with some examples and general constructions.
A first example of covariant subnet is the \textbf{trivial subnet}, which is defined by $\B(I):=\C1_{\C\Omega}$ for all $I\in\J$ with $\Omega:=1\in\C$. 

\begin{ex}  \label{ex:fixed_point_subnet}
Let $\A$ be a graded-local conformal net and $G$ be a compact subgroup of $\Aut(\A)$. We define the \textbf{fixed point subnet} $\A^G$ with respect to $G$ as
$$
\A^G(I):=\left\{
A\in\A(I) \mid gAg^{-1}=A  \,\,\,\forall g\in G
\right\}  \quad\forall I\in\J\,.
$$
It is easy to check that $\A^G$ is actually a graded-local conformal net because every automorphism $g$ is in $U(\Diff^+(S^1)^{(\infty)})'$ by Proposition \ref{prop:automorphisms_and_symmetries}. If $G$ is finite, we call $\A^G$ an \textbf{orbifold subnet}. $\A^\Gamma:=\A^{\{1_\mathcal{H}, \Gamma\}}$, the set of Bose elements of $\A$ is a conformal net, usually called \textbf{the Bose subnet} of $\A$.
\end{ex}

\begin{ex}
The \textbf{graded tensor product} of two graded-local conformal nets $\A_1$ and $\A_2$, see \cite[Section 2.6]{CKL08} for details, is the isotone map of von Neumann algebras $\J\ni I\mapsto(\A_1\hat{\otimes}\A_2)(I)$, acting on the tensor product Hilbert space $\mathcal{H}_1\otimes\mathcal{H}_2$ with grading unitary $\Gamma_1\otimes\Gamma_2$ and vacuum vector $\Omega_1\otimes\Omega_2$, defined by the following tensor product:
\begin{equation}  \label{eq:defin_graded_tensor_product}
\A_1(I)\hat{\otimes}\A_2(I):=
\left\{
A_1\otimes A_2 \,,\,\,
B_1\otimes 1_{\mathcal{H}_2} \,,\,\,
\Gamma_1\otimes B_2
\mid
\substack{ 
A_1,B_1\in\A_1(I) \,,\,\,
A_2,B_2\in\A_2(I) \\
\partial A_i=0 \,,\,\,
\partial B_i=1
}
\right\}''
\,.
\end{equation}
The definition above moved from the necessity that the Fermi elements on the right of the tensor product must anticommute with the ones on the left and vice versa. For a graded-local conformal net $\A$ and for every $I\in\J$, let us denote by $\A_f(I)$ the \textbf{Fermi subspace} of $\A(I)$, that is, the subspace given by the Fermi elements. Accordingly, $\A_1(I)\hat{\otimes}\A_2(I)$ is the vector space direct sum
\begin{equation}
\underbrace{
\A_1^{\Gamma_1}(I)\otimes\A_2^{\Gamma_2}(I)
\oplus
\A_{1f}(I)\Gamma_1\otimes\A_{2f}(I)
}_{\mbox{Bosons}}
\oplus
\underbrace{
\A_1^{\Gamma_1}(I)\Gamma_1\otimes\A_{2f}(I)
\oplus
\A_{1f}(I)\otimes\A_2^{\Gamma_2}(I)
}_{\mbox{Fermions}} \,.
\end{equation}
Moreover, we have the following isomorphisms of von Neumann algebras
\begin{equation}
\begin{split}
\hat{\A_1}(I)
  &:=
\A_1(I)\otimes 1_{\mathcal{H}_2}\cong \A_1(I) \\
\hat{\A_2}(I) 
  &:= 
(1_{\mathcal{H}_1}\otimes\A_2^{\Gamma_2}(I))\vee (\Gamma_1\otimes\A_{2f}(I))
\cong 1_{\mathcal{H}_1}\otimes\A_2(I)
\cong \A_2(I)
\end{split}
\end{equation}
for all $I\in\J$. Furthermore, we have that
\begin{equation}
\A_1(I)\hat{\otimes}\A_2(I)
=\hat{A_1}(I)\vee\hat{A_2}(I) \,,
\quad
[\hat{\A_1}(I),\hat{\A_2}(I)]=0 
\qquad \forall I\in\J\,.
\end{equation}
\end{ex}

\begin{ex}
The \textbf{coset subnet} $\B^c$ of a covariant subnet $\B$ of a graded-local conformal net $\A$ is the covariant subnet of $\A$ defined by the graded-local von Neumann algebras
\begin{equation}  \label{eq:defin_coset_subnet}
\B^c(I):=\left\{
A\in\A(I) \mid [A,B]=0 \,\,\,\forall B\in\B(S^1)
\right\}  
\qquad\forall I\in\J\,.
\end{equation}
\end{ex}

\section{Vertex operator superalgebras}
\label{section:superVOAs}

First, we introduce the basic theory of vertex superalgebras with their conformal and unitary structures and their module theory in Section \ref{subsection:basic_definitions_VOSA}. Section \ref{subsection:invariant_bilinear_forms} is about invariant bilinear forms. Here, we also characterize the unitary structure of vertex operator superalgebras in terms of their even and odd parts. 
After that, we generalize most of the results given in \cite[Chapter 4 and Chapter 5]{CKLW18} to the vertex superalgebra setting from Section \ref{subsection:invariant_bilinear_forms} to Section \ref{subsection:unitary_subalgebras}.
(Some of these generalizations have been already treated in \cite[Section 2.2]{Ten19}, cf.\ \cite[Section 2.1]{Ten19b}.)
For this reason, we will keep a notation consistent with the one used in \cite{CKLW18} as close as possible.
We also generalize \cite[Appendix A]{CGGH23} at the end of Section \ref{subsection:uniqueness_unitary_structure}.

\subsection{Unitary VOSAs and unitary modules}
\label{subsection:basic_definitions_VOSA}

We introduce unitary vertex operator superalgebras and correlated objects following mostly \cite{Kac01} for the general theory of vertex algebras
\footnote{Note that in \cite{Kac01}, the author uses the term ``vertex algebra'' for both of what we call ``vertex algebra'' and ``vertex superalgebra''.}, 
see also \cite{FLM88}, \cite{FHL93} and \cite{LL04} for the module theory,
whereas \cite{AL17} and \cite[Section 2.1]{Ten19b} for the unitary setting  (cf.\ also \cite{RTT22, KMP22} for more general notions of unitarity in the vertex algebra framework).

A \textbf{vertex superalgebra} is a quadruple $(V, \Omega, T, Y)$ such that:
\begin{itemize}
\item $V$ is a \textit{complex vector superspace}, namely a complex vector space equipped with an automorphism $\Gamma_V$ such that $\Gamma_V^2=1_V$ and $V$ decomposes as the direct sum of the following vector subspaces
\begin{equation}
V_\parzero := \left\{ a\in V \mid \Gamma_V (a)=a \right\} 
\quad\mbox{ and }\quad
V_\parone := \left\{ a\in V \mid \Gamma_V (a)=-a \right\} 
\end{equation}
where $p\in\{\parzero,\parone\}$ denotes an element of the cyclic group of order two $\Z_2:=\Z/2\Z$. Hence, we say an element $a\in V_\parzero$ is \textbf{even}, whereas $a\in V_\parone$ is \textbf{odd}. Moreover, we say that $a\in V$ has \textbf{parity} $p(a):=p\in\left\{\parzero, \parone\right\}$ if $a\in V_p$.

\item $\Omega\in V_\parzero$ is called the \textbf{vacuum vector} of $V$.

\item $T$ is an even endomorphism of $V$, called the \textbf{infinitesimal translation operator}.

\item $Y$ is a vector space linear map, called the \textbf{state-field correspondence}, defined by
\begin{equation}
V\ni a\longmapsto Y(a,z)
:=
\sum_{n\in\Z} a_{(n)}z^{-n-1}\in(\End(V))[[z,z^{-1}]]
\end{equation}
where $(\End(V))[[z,z^{-1}]]$ is the complex vector space of doubly-infinite formal Laurent series in the formal variables $z$ and $z^{-1}$ whose coefficients are endomorphisms of the vector space $V$. Furthermore, $Y$ must satisfy the following properties:

\begin{itemize}
\item \textbf{Parity-preserving}. $a_{(n)}V_{p}\subseteq V_{p+p(a)}$ for all $a\in V_{p(a)}$, all $n\in\Z$ and all $p\in\{\parzero,\parone\}$.

\item \textbf{Field}. For all $a, b \in V$, there exists a non-negative integer $N$ such that $a_{(n)}b=0$ for all $n\geq N$. Then we call $Y(a,z)$ a \textbf{field}.

\item \textbf{Translation covariance}. $[T, Y(a,z)]=\frac{\mathrm{d}}{\mathrm{d}z} Y(a,z)$ for all $a\in V$.

\item \textbf{Vacuum}. $Y(\Omega, z)=1_V$, $T\Omega=0$ and $a_{(-1)}\Omega=a$ for all $a\in V$.

\item \textbf{Locality}. For all $a, b\in V$ with given parities, there exists a positive integer $M$ such that (as formal distributions, see e.g.\ \cite[Section 2.3]{Kac01})
\begin{equation}
(z-w)^N[Y(a,z), Y(b,w)]=0 
\end{equation}
for all $N\geq M$, where
\begin{equation} \label{eq:defin_graded_commutator_fields}
[Y(a,z), Y(b,w)]:= Y(a,z)Y(b,w)-(-1)^{p(a)p(b)}Y(b,w)Y(a,z)
\end{equation}
is the \textbf{graded commutator} of the fields $Y(a,z)$ and $Y(b,w)$. Hereafter, with an abuse of notation, we use $(-1)^{p(a)p(b)}$ as in \eqref{eq:defin_graded_commutator_fields} to denote $(-1)^{p_ap_b}$, where $p_a,p_b\in\{0,1\}$ are representatives of the remainder class of $p(a)$ and $p(b)$ in $\Z_2$ respectively. 
\end{itemize}
Every $Y(a,z)$ with the properties above is called a \textbf{vertex operator}.
\end{itemize}

An important consequence of the above axioms, see \cite[Section 4.8]{Kac01}, is the so called \textbf{Borcherds identity} (or \textit{Jacobi identity}):
\begin{equation}  \label{eq:borcherds_id}
\begin{split}
\sum_{j=0}^\infty \binom{m}{j} 
\left(a_{(n+j)}b\right)_{(m+k-j)}c
&= \\
\sum_{j=0}^\infty(-1)^j \binom{n}{j}
&\left(
a_{(m+n-j)}b_{(k+j)}
-(-1)^{p(a)p(b)}(-1)^n b_{(n+k-j)} a_{(m+j)}
\right)c
\end{split}
\end{equation}
for all $a,b\in V$ with given parities, all $c\in V$ and all $m,n,k\in\Z$. Putting $n=0$ in the Borcherds identity, we get the famous \textbf{Borcherds commutator formula} for two arbitrary vectors $a,b\in V$ with given parities, see \cite[Eq.\  (4.6.3)]{Kac01}:
\begin{equation}  \label{eq:borcherds_commutator_formula}
[a_{(m)},b_{(k)}]c=
\sum_{j=0}^\infty\binom{m}{j}
\left(a_{(j)}b\right)_{(m+k-j)}c
\qquad\forall a,b,c\in V 
\,\,\,\forall m,k\in\Z \,.
\end{equation}
Instead, choosing $m=0$, we get the \textbf{Borcherds associative formula}:
\begin{equation} \label{eq:borcherds_associative_formula}
	(a_{(n)}b)_{(k)}c
	=\sum_{j=0}^\infty(-1)^j   \binom{n}{j}
	\left(
	a_{(n-j)}b_{(k+j)}
	-(-1)^{p(a)p(b)}(-1)^n b_{(n+k-j)} a_{(j)}
	\right)c 
\end{equation}
for all $a,b\in V$ with given parities, all $c\in V$ and all $n,k\in\Z$. 

A further property which we will use considerably is the \textbf{skew-symmetry}, see \cite[Eq.\ (4.2.1)]{Kac01}:
for any $a,b\in V$ with given parities,
\begin{equation}  \label{eq:skew-symmetry}
	Y(a,z)b= (-1)^{p(a)p(b)} e^{zL_{-1}} Y(b,-z)a
	\,.
\end{equation}

An even element $\nu$ of a vertex superalgebra $(V, \Omega, T, Y)$ is called a \textbf{Virasoro vector} of \textbf{central charge} $c\in\C$ if the following commutation relations hold
\begin{equation}  \label{eq:virasoro_cr}
[L_n, L_m]=(n-m)L_{n+m}+\frac{c(n^3-n)}{12}\delta_{n,-m}1_V
\qquad\forall n,m\in\Z
\end{equation} 
where $L_n:=\nu_{(n+1)}$. Then $Y(\nu,z)$ is called a \textbf{Virasoro field}. If $L_0$ is diagonalizable on $V$ and $L_{-1}=T$, then $\nu$ is called a \textbf{conformal vector}, $L_0$ a \textbf{conformal Hamiltonian} and $Y(\nu,z)$ an \textbf{energy-momentum field}. A vertex superalgebra with a fix conformal vector is called a \textbf{conformal vertex superalgebra}. 
Therefore, we define a \textbf{vertex operator superalgebra}, briefly called \textbf{VOSA}, as a conformal vertex superalgebra $(V, \Omega, T, Y, \nu)$ with the additional properties:
\begin{itemize}
\item $V_n:=\Ker\left( L_0-n1_V\right) =\{0\}$ if $n\not\in\half\Z$ with $V_\parzero=\bigoplus_{n\in\Z} V_n$ and $V_\parone=\bigoplus_{n\in\Z -\half}V_n$.

\item $\mathrm{dim}V_n<+\infty$ for all $n\in\half\Z$ and there exists an integer $N\leq0$ such that $V_n=\{0\}$ for all $n<N$.
\end{itemize}
With an abuse of notation, we will often use $V$ in place of the quintuple $(V, \Omega, T, Y, \nu)$.

A vertex operator superalgebra $V$ is said to be of \textbf{CFT type} if the corresponding conformal vector is \textit{of CFT type}, that is $V_n=\{0\}$ for all $n\not\in\half\Zpluseq$ and $V_0=\C\Omega$.

Every non-zero element $a\in V_n$ with $n\in\half\Z$ is called \textbf{homogeneous} with \textbf{conformal weight} $d_a:=n$. Accordingly, we rewrite
\begin{equation}
Y(a,z)=\sum_{n\in\Z-d_a}a_nz^{-n-d_a} 
\,,  \qquad
a_n:=a_{(n+d_a-1)}
\end{equation}
for all homogeneous $a\in V$. For our convenience, if $d_a\in \Z$, we set $a_n:=0$ for all $n\in\Z-\half$; whereas if $d_a\in\Z-\half$, we set $a_n:=0$ for all $n\in\Z$. Then for arbitrary $a\in V$ and $n\in\half\Z$, we denote by $a_n$ the finite sum of the coefficients of index $n$ of the fields corresponding to the homogeneous components of $a$, which is well-defined thanks to the linearity of the state-field correspondence. 
Furthermore, every homogeneous vector $a\in V$ is called \textbf{primary} if $L_na=0$ for all $n\in\Zplus$, whereas it is called \textbf{quasi-primary} if just $L_1a=0$. It is not difficult to prove that $\Omega$ is a primary vector in $V_0$ (see \cite[Remark 4.1]{CKLW18}), whereas $\nu$ is a quasi-primary vector in $V_2$ which cannot be primary if $c\not=0$ (see \cite[Theorem 4.10]{Kac01}). We have the following useful commutation relations, see e.g.\ \cite[Section 4.9]{Kac01}:
\begin{align}
[L_0,a_n] &=-na_n \label{eq:cr_l0}\\
[L_{-1}, a_n] &=(-n-d_a+1)a_{n-1}   \label{eq:cr_l-1}\\
[L_1, a_n] &=(-n+d_a-1)a_{n+1} + (L_1a)_{n+1}  \label{eq:cr_l1}
\end{align}
for all $a\in V$ and all $n\in\half\Z$. Moreover, if $a$ is primary, then, see e.g.\ \cite[Corollary 4.10]{Kac01},
\begin{equation} \label{eq:cr_l_m_a_n}
	[L_m,a_n]=((d_a-1)m-n)a_{m+n} \qquad\forall m\in\Z \,\,\,\forall n\in\half\Z\,.
\end{equation}

To introduce a unitary structure on a VOSA, we need to define \textbf{(antilinear) homomorphisms} between the VOSAs $(V, \Omega^V, T_V, Y_V, \nu^V)$ and $(W, \Omega^W, T_W, Y_W, \nu^W)$ as a vector space (antilinear) map $\phi$ between $V$ and $W$ with the additional conditions: $\phi(a_{(n)}b)=\phi(a)_{(n)}\phi(b)$ for all $a,b\in V$ and all $n\in\half \Z$ ($\phi$ respects the $(n)$-product), $\phi(\Omega_V)=\Omega_W$ and $\phi(\nu_V)=\nu_W$.
These imply that $\phi$ commutes with every $L_n$ for $n\in\Z$ and thus it is also parity-preserving. We also define \textbf{(antilinear) isomorphisms} and \textbf{(antilinear) automorphisms} of VOSAs in the obvious way.
If the condition $\phi(\nu_V)=\nu_W$ is removed, then $\phi$ is simply said a \textit{vertex superalgebra} (antilinear) homomorphism/isomorphism/automorphism.
We denote by $\Aut (V)$ the group of VOSA automorphisms of $V$. Note also that $\Aut(V)$ is a subset of the group $\prod_{n\in\half\Z}\operatorname{GL}(V_n)$ of grading-preserving vector space automorphisms of $V$. Therefore, $\Aut (V)$ can be turned into a metrizable topological group if it is equipped with the relative topology induced by the product topology of $\prod_{n\in\half\Z}\operatorname{GL}(V_n)$. We say that an (antilinear) automorphism $\phi$ is an \textbf{involution} if $\phi^2=1_V$. 

Note that $\Gamma_V$ is an involution of $V$. By the parity-preserving property, we deduce that $\Gamma_V$ commutes with every (antilinear) automorphism of $V$. 
We also set the following vector space automorphism of $V$:
\begin{equation}
Z_V:=\frac{1_V-i\Gamma_V}{1-i} 
\,,\qquad
Z_V^{-1}=\frac{1_V+i\Gamma_V}{1+i}
\,.
\end{equation}
Note that $Z_V$ preserves the vacuum and the conformal vectors, but it is not a VOSA automorphism of $V$ because it does not respect the $(n)$-product. Indeed, we have that
\begin{equation}  \label{eq:Z_(n)-product}
\begin{split}
Z_V(a_{(n)}b) 
  &=\left\{
\begin{array}{lr}
a_{(n)}b & \mbox{ if } \,\,\, p(a)=p(b)=\parzero \\
i a_{(n)}b & \mbox{ if } \,\,\, p(a)=\parone+p(b) \\
a_{(n)}b & \mbox{ if } \,\,\, p(a)=p(b)=\parone 
\end{array}  \right.  \\
(Z_Va)_{(n)}(Z_Vb)
  &=\left\{
\begin{array}{lr}
a_{(n)}b & \mbox{ if }  \,\,\,  p(a)=p(b)=\parzero \\
i a_{(n)}b & \mbox{ if }  \,\,\, p(a)=\parone+p(b) \\
-a_{(n)}b & \mbox{ if }  \,\,\, p(a)=p(b)=\parone 
\end{array}  \right.
\end{split}
\end{equation}
for all $a,b\in V$ with given parities and all $n\in\Z$. Hence, \eqref{eq:Z_(n)-product} can be rewritten in the useful formula
\begin{equation}  \label{eq:Z_(n)-product_formula}
Z_V(a_{(n)}b)=(-1)^{p(a)p(b)}(Z_Va_{(n)})(Z_Vb)
\end{equation}
for all $a,b\in V$ with given parities and all $n\in\Z$.

A \textbf{unitary vertex operator superalgebra} is a VOSA $V$ equipped with an antilinear VOSA automorphism $\theta$, which we call the \textbf{PCT operator}, associated with a scalar product (i.e., a positive-definite Hermitian form, linear in the second variable) $\scalar$ on $V$ such that:
\begin{itemize}
\item $\theta$ is an \textbf{involution}, that is, $\theta^2=1_V$ (see Remark \ref{rem:involution_not_necessary} below);
\item $\scalar$ is \textbf{normalized}, which means that $(\Omega| \Omega)=1$;
\item $\scalar$ is \textbf{invariant}, that is, (see \cite[Section 2.1]{AL17})
\begin{equation} \label{eq:invariant_scalar_product}
(Y(\theta(a),z)b| c)=(b| Y(e^{zL_1}(-1)^{2L_0^2+L_0}z^{-2L_0}a,z^{-1})c)
\qquad\forall a,b,c\in V
\end{equation}
or equivalently (see \cite[Definition 2.3]{Ten19b})
\begin{equation}   \label{eq:invariant_scalar_product_Z}
(Y(\theta(a),z)b| c)=(b| Y(e^{zL_1}(iz^{-1})^{2L_0}Z_Va,z^{-1})c)  
\qquad\forall a,b,c\in V
\end{equation}
where \eqref{eq:invariant_scalar_product} and \eqref{eq:invariant_scalar_product_Z} are to be understood as equalities between doubly-infinite formal Laurent series in the formal variable $z$ (which is not affected by the complex conjugation) with coefficients in $\C$. (The operators showing up in \eqref{eq:invariant_scalar_product} and \eqref{eq:invariant_scalar_product_Z} can be understood as defined on homogeneous elements in the obvious way and extended by linearity to $V$.)
Note that 
\begin{equation}  \label{eq:2d^2+d}
	(-1)^{2d^2+d}=\left\{
	\begin{array}{lc}
		(-1)^d & \mbox{ if } d\in\Z  \\
		(-1)^{d+\half } & \mbox{ if } d\in\Z-\half 
	\end{array}
	\right. 
\end{equation}
so that the equivalence of \eqref{eq:invariant_scalar_product} and \eqref{eq:invariant_scalar_product_Z} follows from $i^{2L_0}Z_V=(-1)^{2L_0^2+L_0}$.
\end{itemize}

\begin{rem}  \label{rem:involution_not_necessary}
In the references \cite[Section 2.1]{AL17} and \cite[Definition 2.3]{Ten19b} given above for the definition of unitary VOSAs, it is required that the PCT operator $\theta$ is an involution. For our convenience, we have kept this axiom. Nevertheless, we point out that it is not necessary because we will prove that it is actually a consequence of the other ones in (i) of Proposition \ref{prop:relation_pct_scalar_product}.
\end{rem}

If a (unitary) vertex (operator) superalgebra $V$ has $V_\parone=\{0\}$, then it is a \textbf{(unitary) vertex (operator) algebra}, briefly \textbf{V(O)A}, as defined in \cite{CKLW18}.

\begin{rem}  \label{rem:real_form}
Let $V$ be a unitary VOSA. The real subspace
\begin{equation}
V_\R:=\{a\in V\mid \theta(a)=a\}
\end{equation}
contains the vacuum and the conformal vectors because they are $\theta$-invariant. Moreover, $V_\R$ inherits the structure of a \textit{real} VOSA from the one of $V$. Note also that $V$ is the vector space complexification of $V_\R$ because $V=V_\R+ iV_\R$ and $V_\R\cap iV_\R=\{0\}$. Then $V_\R$ is known as a \textbf{real form} for $V$, see \cite[Remark 5.4]{CKLW18} and references therein. Restricting the scalar product on $V$ to $V_\R$, we obtain a positive-definite normalized real-valued $\R$-bilinear form with the invariant properties \eqref{eq:invariant_scalar_product} and \eqref{eq:invariant_scalar_product_Z}. Conversely, if $\widetilde{V}$ is a real VOSA with a positive-definite normalized real-valued invariant (in the meaning given just above) $\R$-bilinear form $\bilinear$, then its complexification $\widetilde{V}_\C$ is a VOSA with an invariant scalar product which extends $\bilinear$. Furthermore, $(\widetilde{V}_\C)_\R=\widetilde{V}$.
\end{rem}

We say that a vertex superalgebra homomorphism $\phi$ between unitary VOSAs $V$ and $W$ with scalar products $\scalar_V$ and $\scalar_W$ respectively is \textbf{unitary} if $(\phi(a)|\phi(b))_W=(a| b)_V$ for all $a,b\in V$. We denote the subgroup of $\Aut (V)$ of unitary automorphisms by $\Aut_{\scalar}(V)$.

In the following, we introduce some useful relations for a unitary VOSA $(V,\scalar)$ with PCT operator $\theta$. Due to the antilinearity of $\theta$, we have that 
\begin{equation}   \label{eq:theta_z}
Z_V\theta Z_V=\theta \,.
\end{equation}
If $a\in V$ is homogeneous of conformal weight $d_a$, then $L_1^la$ is still homogeneous of conformal weight $d_a-l$ for all $l\in\Zpluseq$ thanks to \eqref{eq:cr_l0}. Then by replacing $a$ with $\theta (a)$ in \eqref{eq:invariant_scalar_product} and \eqref{eq:invariant_scalar_product_Z} and remembering that $Z_V$ and $\theta$ commute with every $L_n$ for all $n\in\Z$, we can calculate the following equation:
\begin{equation}    \label{eq:an_star}
(a_nb| c)
=(-1)^{2d_a^2+d_a}\sum_{l\in\Zpluseq}\frac{(b| (\theta L_1^l a)_{-n}c)}{l!}
=i^{2d_a}\sum_{l\in\Zpluseq}\frac{(b| (Z_V\theta L_1^l a)_{-n}c)}{l!}
\end{equation}
for all $b,c\in V$, all homogeneous $a\in V$ and all $n\in\half\Z$. In particular, if $a$ is quasi-primary, we have that
\begin{align}   
(a_nb| c)
&=(-1)^{2d_a^2+d_a}(b| (\theta a)_{-n}c)
=i^{2d_a}(b| (Z_V\theta a)_{-n}c)
=i^{2d_a}(b| (\theta Z_V^{-1} a)_{-n}c) 
\label{eq:an_star_qp}\\
(b| a_nc)
&=(-1)^{2d_a^2+d_a}((\theta a)_{-n}b| c)
=i^{-2d_a}((\theta Z_V^{-1}a)_{-n}b| c)
=i^{-2d_a}((Z_V\theta a)_{-n}b| c) 
\label{eq:an_star_qp_2}
\,.
\end{align}
Then we say that the field $Y(a,z)$ for any quasi-primary vector $a\in V$, is \textbf{Hermitian} (with respect to $\scalar$) if $(a_nb| c)=(b| a_{-n}c)$ for all $n\in\half\Z$ and all $b,c\in V$.
Moreover, equation \eqref{eq:an_star_qp} says that 
\begin{equation} \label{eq:L_n_adj}
(L_na| b)=(a| L_{-n}b) 
\qquad\forall a,b\in V 
\,\,\,\forall n\in\Z
\end{equation}
which also implies, choosing $n=0$, that $(V_l| V_k)=0$ for all $l,k\in\half\Z$ such that $l\not= k$.

It is important to introduce the notion of \textbf{vertex subalgebra} of a vertex superalgebra $V$. This is given by a vector subspace $W$ of $V$ containing the vacuum vector $\Omega$ and such that $a_{(n)}W\subseteq W$ for all $a\in W$ and all $n\in\Z$. By the vertex superalgebra axioms, it follows that $Ta=a_{(-2)}\Omega$ for all $a\in V$. Accordingly, $W$ is automatically $T$-invariant and therefore $(W,\Omega, T\restriction_W, Y(\cdot,z)\restriction_W)$ is a vertex superalgebra. It is then clear that $\C\Omega$ and $V_\parzero$ are vertex subalgebras of $V$, which are also vertex algebras. Given a subspace $\F$ of $V$, we define $W(\F)$ as the smallest vertex subalgebra of $V$ containing $\F$ and we say that $W(\F)$ is \textbf{generated by} $\F$. 

An \textbf{ideal} of a vertex superalgebra $V$ is a $T$-invariant vector subspace $\mathscr{J}$ such that $a_{(n)}\mathscr{J}\subseteq\mathscr{J}$ for all $a\in V$ and all $n\in\Z$. From \cite[Eq.\  (4.3.1)]{Kac01}, we have that $a_{(n)}V\subseteq \mathscr{J}$ for all $a\in \mathscr{J}$ and all $n\in\Z$. As usual, we say that a vertex superalgebra $V$ is \textbf{simple} if the only ideals which it contains are $\{0\}$ and $V$ itself.

The \textbf{graded tensor product} $V^1\hat{\otimes}V^2$ of two VOSAs $(V^j,\Omega^j,Y_j,\nu^j)$ for $j\in\{1,2\}$, see \cite[Section 4.3]{Kac01}, is the VOSA given by 
$$
(V^1\otimes V^2, \Omega^1\otimes\Omega^2, Y_1\hat{\otimes}Y_2, \nu^1\otimes \Omega^2+\Omega^1\otimes\nu^2)
$$ 
where
\begin{equation}
	(Y_1\hat{\otimes}Y_2)(a^1\otimes a^2,z)
	:=Y_1(a^1,z)\Gamma_{V_1}^{p(a^2)}\otimes Y_2(a^2,z)
	\qquad\forall a^1\otimes a^2\in V^1\otimes V^2
	\,.
\end{equation}
Clearly, the parity operator is given by $\Gamma_{V_1}\otimes\Gamma_{V_2}$. Consider the operators
$$
Z_{V^1\hat{\otimes}V^2}=
\frac{1_{V^1}\otimes 1_{V^2}-i\Gamma_{V_1}\otimes\Gamma_{V_2}}{1-i}
\,,\qquad
Z_{V^1}=\frac{1_{V^1}-i\Gamma_{V_1}}{1-i}
\,,\qquad
Z_{V^2}=\frac{1_{V^2}-i\Gamma_{V_2}}{1-i}
\,.
$$ 
Then it is not difficult to show that 
\begin{equation}  \label{eq:graded_tensor_product_Z}
	\begin{split}
		Z_{V^1\hat{\otimes}V^2}(a^1\otimes a^2)
		&=
		(-1)^{p(a^1)p(a^2)}Z_{V^1}(a^1)\otimes Z_{V^2}(a^2) \\
		Z_{V^1\hat{\otimes}V^2}^{-1}(a^1\otimes a^2)
		&=
		(-1)^{p(a^1)p(a^2)}Z_{V^1}^{-1}(a^1)\otimes Z_{V^2}^{-1}(a^2)
	\end{split}
\end{equation}
for all vectors with given parities $a^j\in V^j$ with $j\in\{1,2\}$. Finally, by (2) of \cite[Proposition 2.4]{AL17}, see also \cite[Proposition 2.20]{Ten19}, if for $j\in\{1,2\}$, $(V^j,\scalar_j,\theta_j)$ are unitary, then $(V^1\hat{\otimes}V^2,\scalar_1\scalar_2,\theta_1\otimes\theta_2)$ is unitary too.

We end this section introducing some basic notions of the VOSA module theory with some examples. 
Let $V$ be a VOSA, then a \textbf{$V$-module} is a $\Z_2$-graded complex vector space $M=M_\parzero\oplus M_\parone$ with the following structure:
\begin{itemize}
	\item a vector space linear map $Y^M$ defined by
	\begin{equation}
		V\ni a\longmapsto 
		Y^M(a,z):=\sum_{n\in\Z}a^M_{(n)}z^{-n-1}
		\in (\End(M))[[z,z^{-1}]]
	\end{equation}
	and with the following properties: 
	\begin{itemize}
		\item \textbf{Parity-preserving}.
		$a^M_{(n)} M_p\subseteq M_{p+p(a)}$ for all $a\in V_{p(a)}$, all $p\in \{\parzero, \parone\}$ and all $n\in\Z$.
		
		\item \textbf{Field}.
		For all $a\in V$ and all $b\in M$, there exists a non-negative integer $N$ such that $a^M_{(n)}b=0$ for all $n\geq N$.
		
		\item \textbf{Vacuum}. $Y^M(\Omega,z)=1_M$.
		
		\item \textbf{Borcherds identity}. For all $a,b\in V$ with given parities, all $c\in M$ and all $m,n,k\in\Z$, it holds that
		\begin{equation}  \label{eq:borcherds_id_module}
			\begin{split}
				\sum_{j=0}^\infty & \binom{m}{j} 
				\left(a_{(n+j)}b\right)^M_{(m+k-j)}c
				= \\
				&\sum_{j=0}^\infty(-1)^j\binom{n}{j}
				\left(
				a^M_{(m+n-j)}b^M_{(k+j)}
				-(-1)^{p(a)p(b)}(-1)^n b^M_{(n+k-j)} a^M_{(m+j)}
				\right)c \,.
			\end{split}
		\end{equation}
	\end{itemize}
	
	\item Let $Y^M(\nu,z)=\sum_{n\in\Z}L^M_nz^{-n-2}$, then
	$M$ has a grading compatible with the parity given by the eigenspaces of $L^M_0$, that is for every $n\in\C$ set $M_n:=\mathrm{Ker}(L^M_0-n1_M)$, then
	$$
	M=\bigoplus_{n\in\C}M_n
	\,,
	\qquad
	M_p=\bigoplus_{n\in\C} (M_p\cap M_n)
	\quad\forall p\in\{\parzero,\parone\}
	$$ 
	with $\mathrm{dim}M_n<+\infty$ for all $n\in\C$ and $M_n\not=\{0\}$ at most for countably many $n$.
\end{itemize}

Any $V$-module $M$ satisfies the translation covariance property, see e.g.\ \cite[Proposition 2.11]{Ten19b}, cf.\ \cite[Proposition 4.1.5]{LL04}:
\begin{equation}  \label{eq:translation_covariance_modules}
	[L^M_{-1},Y^M(a,z)]=\frac{\mathrm{d}}{\mathrm{d}z}Y^M(a,z)
	=Y^M(L_{-1}a,z)
	\qquad\forall a\in V\,.
\end{equation}
Moreover, if $c\in\C$ is the central charge of $V$, then the following Virasoro algebra commutation relations hold, see \cite[Remark 2.3.2]{Li96}, cf.\  \cite[Proposition 4.1.5]{LL04}:
\begin{equation}  \label{eq:virasoro_cr_module}
	[L^M_n, L^M_m]=(n-m)L^M_{n+m}+\frac{c(n^3-n)}{12}\delta_{n,-m}1_M
	\qquad\forall n,m\in\Z \,.
\end{equation} 

A \textbf{submodule} $S$ of a $V$-module $M$ is a vector subspace of $M$ invariant for the action of $V$. $M$ is
called \textbf{irreducible} if there are no non-trivial submodules.

We define a \textbf{$V$-module (antilinear) homomorphism} $f$ between $M^1$ and $M^2$ as a vector space (antilinear) map from $M^1$ to $M^2$ such that $f(Y^{M^1}(a,z)b)=Y^{M^2}(a,z)f(b)$ for all $a\in V$ and all $b\in M^1$. Further definitions of \textbf{$V$-module (antilinear)} \textbf{endomorphisms}/\textbf{automorphisms}/\textbf{isomorphisms} are given in the obvious way.

\begin{ex} \label{ex:adjoint_contragredient_modules}
The basic example of a $V$-module is the \textbf{adjoint module}, which is obtained by choosing as vector space $V$ itself and setting $Y^V:=Y$. Then note that the VOSA $V$ is simple if and only if its adjoint module is irreducible. A further example is the \textbf{contragredient module} $M'$ of a $V$-module $M$, which is defined as the graded dual vector space 
\begin{equation}
	\label{eq:defin_contragredient_module}
	M':=\bigoplus_{n\in\C}(M_\parzero\cap M_n)^*\oplus
	\bigoplus_{n\in\C} (M_\parone\cap M_n)^*
\end{equation}
and the state-field correspondence $Y^{M'}$ defined by the formula
\begin{equation}   \label{eq:invariance_contragredient_module}
	\langle Y^{M'}(a,z)b',c\rangle=
	\langle b',Y^M(e^{zL_1}(-1)^{2L_0^2+L_0}z^{-2L_0}a,z^{-1})c\rangle
\end{equation}
for all $a\in V$, all $c\in M$ and all $b'\in M'$, 
where $\pairing$ is the natural pairing between $M$ and $M'$. Any $Y^{M'}(a,z)$ is called the \textbf{adjoint vertex operator} of $Y^M(a,z)$. The proof that $(M', Y^{M'})$ defines an actual $V$-module is an adaptation of the proof of \cite[Theorem 5.2.1]{FHL93}, see \cite[Lemma 2]{Yam14}.
We denote by $(V', Y')$ (instead of using $Y^{V'}$) the contragredient module of the adjoint module. 
\end{ex}

A \textbf{unitary $V$-module}, see \cite[Definition 2.10]{Ten19b}, cf.\ \cite[Section 2.1]{AL17} and \cite[Definition 2.4]{DL14}, of a unitary VOSA $V$ is a $V$-module $M$ equipped with a scalar product $\scalar_M$ such that 
\begin{equation}    \label{eq:defin_unitary_module}
	(Y^M(\theta(a),z)b| c)_M=(b| Y^M(e^{zL_1}(-1)^{2L_0^2+L_0}z^{-2L_0}a,z^{-1})c)_M
	\quad
	\forall a\in V
	\,\,\,
	\forall b,c\in M\,.
\end{equation}

\begin{ex}  \label{ex:even_simple_odd_irreducible}
	If $V$ is a (unitary) VOSA, then $V_\parzero$ is a (unitary) VOA with $V_\parone$ as (unitary) $V_\parzero$-module by restricting the (unitary) VOSA structure of $V$ in the obvious way. Suppose that $V$ is also simple and note that \cite[Proposition 4.5.6]{LL04} still works in the VOSA setting, cf.\ \cite[Lemma 6.1.1]{Li94b}. It implies that if $b$ is any non-zero vector in $V$, then $V$ must be linearly generated by elements of type $a_nb$ where $a\in V$ and $n\in\half\Z$. If we choose $b$ even, then $V_\parzero$ must be linearly generated by elements of type $a_nb$ for $a\in V_\parzero$ and $n\in\Z$. In other words, $V_\parzero$ is simple. Similarly, if we choose $b$ odd, then $V_\parone$ is generated by elements of type $a_nb$ for $a\in V_\parzero$ and $n\in\Z$, that is $V_\parone$ is irreducible. 
\end{ex}

\begin{ex} \label{ex:conjugate_module}
Given a module $M$ on a unitary VOSA $V$, we define the \textbf{conjugate module} $\overline{M}$ of $M$ as the complex vector superspace constituted by $M$ itself as a set and the scalar multiplication by $\lambda\in\C$, realized via the usual scalar multiplication on $M$ by its complex conjugate $\overline{\lambda}$. Vertex operators on $\overline{M}$ are defined by $Y^{\overline{M}}(a,z):=Y^M(\theta(a),z)$ for all $a\in V$, where $\theta$ is the PCT operator of $V$. It is not difficult to prove that $\overline{M}$ so defined is a $V$-module.
\end{ex}

\subsection{Invariant bilinear forms and unitarity}
\label{subsection:invariant_bilinear_forms}

We give the definition of invariant bilinear form for a VOSA and we prove some related results. Beyond the general interest, the current section presents results which are fundamental to develop the following remaining ones as for example, the equivalence between simplicity and CFT type condition for unitary VOSAs. As novel result, we show that the unitary structure of a simple unitary VOSA is determined by the unitary structure of its even and odd parts.

\begin{defin}   \label{defin:invariant_bilinear_form}
Let $V$ be a VOSA. An \textbf{invariant bilinear form} on $V$ is a bilinear form $\bilinear$ on $V$ such that
\begin{equation} \label{eq:invariant_bilinear_form}
(Y(a,z)b, c)=(b, Y(e^{zL_1}(-1)^{2L_0^2+L_0}z^{-2L_0}a,z^{-1})c)
\qquad\forall a,b,c\in V \,.
\end{equation}
Moreover, $\bilinear$ is said to be \textbf{normalized} if $(\Omega,\Omega)=1$.
\end{defin} 

\begin{rem}
Note that if $V$ is a unitary VOSA with scalar product $\scalar$ and PCT operator $\theta$, then $\bilinear:=(\theta(\cdot)|\cdot)$ is a normalized non-degenerate invariant bilinear form on $V$.
\end{rem}

Easily from the definition, for all homogeneous $a\in V$, we have that 
\begin{equation}  \label{eq:an_inverse_bilinear_form}
(a_n b,c)=(-1)^{2d_a^2+d_a}\sum_{l\in\Zpluseq}
\frac{(b,(L_1^la)_{-n}c)}{l!}
\qquad\forall n\in\half\Z
\,\,\,\forall b,c\in V \,.
\end{equation}
In particular,
\begin{equation}  \label{eq:bilinear_form_adjoint_Ln}
(L_na,b)=(a,L_{-n}b) 
\qquad\forall n\in\Z
\,\,\,\forall a, b\in V\,,
\end{equation}
which implies that $(V_l,V_k)=0$ for all $l,k\in\half\Z$ such that $l\not= k$.
Moreover, using the straightforward commutation identities (verified on homogeneous elements first and thus extended by linearity, cf.\  \cite[Eq.\  (5.3.1)]{FHL93})
\begin{equation}  \label{eq:commutation_identities}
\begin{split}
z^{-2L_0}e^{z^{-1}L_1} &= e^{zL_1}z^{-2L_0}  \\
(-1)^{2L_0^2+L_0}e^{-zL_1} &= e^{zL_1} (-1)^{2L_0+L_0}
\end{split}
\end{equation}
one can prove the \textbf{inverse invariance property}:
\begin{equation} \label{eq:inverse_invariance_property}
(c,Y(a,z)b)=(Y(e^{zL_1}(-1)^{2L_0^2+L_0}z^{-2L_0}a,z^{-1})c,b)
\qquad\forall a,b,c\in V \,.
\end{equation}

\begin{prop}   \label{prop:properties_bilinear_forms}
Let $V$ be a VOSA, then:
\begin{itemize}
\item[(i)]  Every invariant bilinear form on $V$ is symmetric, that is, 
$$
(a,b)=(b,a)
\qquad\forall a,b\in V\,.
$$ 

\item[(ii)] The map $\bilinear\mapsto (\Omega,\cdot)\restriction_{V_0}$ realizes a linear isomorphism from the space of invariant bilinear forms on $V$ to $\operatorname{Hom}_\C(V_0/L_1V_1,\C)=(V_0/L_1V_1)^*$.

\item[(iii)] If $V$ is simple, then every non-zero invariant bilinear form on $V$ is non-degenerate. Furthermore, if $\bilinear$ is a non-zero invariant bilinear form on $V$, then every other invariant bilinear form on $V$ is given by $\alpha\bilinear$ where $\alpha\in\C$.

\item[(iv)] If $V$ has a non-degenerate invariant bilinear form and $V_0=\C\Omega$, then $V$ is simple.
\end{itemize}
\end{prop}

\begin{proof}
The statements (iii) and (iv) are proved as \cite[(iii) and (iv) of Proposition 4.6]{CKLW18} respectively. (i) and (ii) are \cite[Proposition 1]{Yam14}, cf.\ \cite[Section 2.4]{Sch98}. In particular, (i) can be proved by adapting \cite[Proposition 5.3.6]{FHL93}, whereas (ii) by adapting \cite[Theorem 3.1]{Li94}, cf.\  also \cite[Theorem 1]{Roi04}.
\end{proof}

\begin{rem}   \label{rem:properties_bilinear_forms}
Let $V$ be a VOSA. Suppose that $V_0=\C\Omega$, then $V$ has a non-zero invariant bilinear form if and only if $L_1V_1=\{0\}$ as a consequences of (ii) of Proposition \ref{prop:properties_bilinear_forms}. Moreover, in the latter case, there is exactly one non-zero invariant bilinear form which is also normalized. Suppose instead that $V$ is simple, then there is at most one normalized non-degenerate invariant bilinear form by (iii) of Proposition \ref{prop:properties_bilinear_forms}.
\end{rem}

A first view on the link between simplicity and CFT type condition for VOSAs is given by Proposition \ref{prop:properties_bilinear_forms} and Remark \ref{rem:properties_bilinear_forms} in presence of an invariant bilinear form. More generally, simplicity turns out to be an important feature to develop interesting results in what follows. In this light, it is useful to have the following characterisation of simplicity:

\begin{prop}  \label{prop:characterisation_simplicity}
Let $V$ be a unitary VOSA. Then the following are equivalent:
\begin{itemize}
\item[(i)] $V$ is simple;

\item[(ii)] $V_0=\C\Omega$;

\item[(iii)] $V$ is of CFT type.
\end{itemize}
\end{prop}

\begin{proof}
(i) implies (ii) is a straightforward adaptation of the proof of \cite[Proposition 5.3]{CKLW18}. The statement (ii) implies (i) is (iv) of Proposition \ref{prop:properties_bilinear_forms}, noting that $(\theta(\cdot)|\cdot)$ realizes a non-degenerate invariant bilinear form on $V$. (iii) implies (ii) is trivial by definition. So the only remaining implication is (ii) implies (iii). Let $N\in\half\Z$ be such that $V_N\not=\{0\}$ and $V_n=\{0\}$ for all $n<N$. Pick $a\in V_N$, then using the Virasoro commutation relations \eqref{eq:virasoro_cr}, we have that
\begin{equation} \label{eq:positivity_conformal_weight}
	2N(a| a)=(a| 2L_0a)=(a| [L_1,L_{-1}]a) 
	=(a| L_1L_{-1}a) =(L_{-1}a| L_{-1}a) 
	\geq 0
\end{equation}
which implies that $N\geq 0$ because $(a|a)\geq0$ too. Therefore, $V$ is of CFT type.
\end{proof}

Proposition \ref{prop:properties_bilinear_forms} has a crucial role in the proof of the following result.
It is partially inspired by \cite[Theorem 3.3]{DL14}, where the authors construct the entire structure of unitary VOA supposing some extra hypothesis, such as rationality, $C_2$-cofiniteness, see e.g.\ \cite{ABD04} and references therein, and the existence of a map whose characteristics make it suitable for constructing a PCT operator. Instead, our focus is on characterizing the unitary structure on a VOSA in terms of the unitarity of its even and odd parts. 

\begin{theo}   \label{theo:unitarity_determind_by_even_part}
	Let $V$ be a simple VOSA. Then $V$ is unitary if and only if $V_\parzero$ is a unitary VOA and $V_\parone$ is a unitary $V_\parzero$-module. In both cases, we automatically have that $V_\parzero$ is simple and $V_\parone$ is irreducible.
\end{theo}

\begin{proof}
	The last claim and the ``only if'' part are given by Example \ref{ex:even_simple_odd_irreducible}.  
	
	Conversely, let $\scalar_\parzero$ and $\theta_\parzero$ realize the unitary structure on $V_\parzero$ and denote by $\scalar_\parone$ the scalar product on the unitary $V_\parzero$-module $V_\parone$. $\bilinear_\parzero:=(\theta_\parzero(\cdot)|\cdot)_\parzero$ is a non-degenerate invariant bilinear form on $V_\parzero$, thus there exists a non-zero invariant bilinear form $\bilinear$ on $V$ which extends $\bilinear_\parzero$ thanks to (ii) of Proposition \ref{prop:properties_bilinear_forms}. Moreover, $\bilinear$ is also non-degenerate by the simplicity of $V$, see (iii) of Proposition \ref{prop:properties_bilinear_forms}.
	Since $(V_\parzero,V_\parone)=0$ by \eqref{eq:bilinear_form_adjoint_Ln}, the restriction $\bilinear_\parone$ of $\bilinear$ to $V_\parone$ must be a non-degenerate bilinear form.
	
	Consider the conjugate module $\overline{V_\parone}$ of $V_\parone$, see Example \ref{ex:conjugate_module}. Then we have an isomorphism of $V_\parzero$-modules induced by the following pairings:
	\begin{equation} \label{eq:iso_V0-modules_pct}
		\xymatrix@R=0.3pc@C=5pc{
			V_\parone \ar[r]^f & V_\parone' \ar[r]^{g^{-1}} & \overline{V_\parone} \\
			c \ar@{|->}[r] & [f(c)](\cdot):=(c,\cdot)_\parone & \\
			& [g(d)](\cdot):=(d|\cdot)_\parone &  d \ar@{|->}[l]
		} \,.
	\end{equation}
	Indeed, it is easy to see that $f(Y(a,z)b)=g^{-1}(Y(\theta(a),z)b)$ for all $a\in V_\parzero$ and all $b\in V_\parone$, which implies the $V_\parzero$-module homomorphism property.
	Note that the $V_\parzero$-module isomorphism \eqref{eq:iso_V0-modules_pct} can be considered as a vector space antilinear automorphism of $V_\parone$, call it $\theta_\parone$, with the property that
	$$
	\theta_\parone(a_nb)=\theta_\parzero(a)_n\theta_\parone(b)
	\qquad
	\forall a\in V_\parzero
	\,\,\,
	\forall b\in V_\parone
	\,\,\,
	\forall
	n\in \Z \,.
	$$
	Moreover, we have that $(\theta_\parone(a)|b)_\parone=(a,b)$ for all $a,b\in V_\parone$ by construction. $\theta_\parone^2$ is an automorphism of the irreducible $V_\parzero$-module $V_\parone$ and thus it must be a multiple of the identity by Schur's Lemma, cf.\ \cite[Proposition 4.5.5]{LL04}. Let $r\in\C\backslash\{0\}$ such that $\theta_\parone^2=r 1_{V_\parone}$. 
	Then we have that
	$$
	\overline{r}(a|a)_\parone
	=(\theta_\parone^2(a)|a)_\parone
	=(\theta_\parone(a),a)
	=(a,\theta_\parone(a))
	=(\theta_\parone(a)|\theta_\parone(a))_\parone>0
	\qquad\forall a\in V_\parone\backslash\{0\}
	$$
	where we have used the symmetry of the bilinear form as in (i) of Proposition \ref{prop:properties_bilinear_forms}. It follows that $r$ must be a positive number. Accordingly, we can renormalize the scalar product $\scalar_\parone$ in such a way that $r=1$ and consequently the renormalized $\theta_\parone$ is an involution.
	
	Now, we prove that the scalar product $\scalar:=\scalar_\parzero\oplus\scalar_\parone$ and the vector space antilinear map $\theta:=\theta_\parzero\oplus\theta_\parone$ give a unitary structure on $V$. By construction $(\theta(\cdot)|\cdot)=\bilinear$, which assures us that $\scalar$ is normalized. Then to prove the invariance property, we need to show that $\theta$ is a well-defined VOSA antilinear automorphism. We have already proved that $\theta$ is an antilinear automorphism of the vector space $V$ preserving the vacuum and the conformal vectors. Moreover, we have also proved that $\theta^2=1_V$.
	So it remains to prove that $\theta$ respects the $(n)$-product. We already know that $\theta(a_nb)=\theta(a)_n\theta(b)$ whenever $a\in V_\parzero$, $n\in\Z$ and $b\in V$. 
	Suppose that $a\in V_\parone$ and $n\in\Z-\half$. Then we want to prove that $\theta(a_nb)=\theta(a)_n\theta(b)$ for all $b\in V$. First, we have that $\theta(L_mb)=\theta(\nu)_m\theta(b)=L_m\theta(b)$ for all $b\in V$ and all $m\in\Z$, that is, $\theta$ commutes with $L_m$ for all $m\in\Z$. It follows that if $b\in V_\parzero$, then we have that
	\begin{equation}
		\begin{split}
			\theta(Y(a,z)b)
			&=\theta((-1)^{p(a)p(b)}e^{zL_{-1}}Y(b,-z)a) \\
			&=(-1)^{p(a)p(b)}e^{zL_{-1}}Y(\theta(b),-z)\theta(a)
			=Y(\theta(a),z)\theta(b)
		\end{split}
	\end{equation}
	where we have used the skew-symmetry \eqref{eq:skew-symmetry} twice, proving that $\theta(a_nb)=\theta(a)_n\theta(b)$ whenever $b\in V_\parzero$. Hence, using the fact that $\theta^2=1_V$ and the invariance property of the bilinear form, we compute that
	\begin{equation}   \label{eq:compute_invariance_prop}
		\begin{split}
			(Y(a,z)u|v) &= (\theta(Y(\theta(a),z)\theta(u))|v) \\
			&= (Y(\theta(a),z)\theta(u),v) \\
			&= (\theta(u),Y(e^{zL_1}(-1)^{2L_0^2+L_0}z^{-2L_0}\theta(a),z^{-1})v) \\
			&= (u|Y(e^{zL_1}(-1)^{2L_0^2+L_0}z^{-2L_0}\theta(a),z^{-1})v)
			\qquad
			\forall u\in V_\parzero
			\,\,\,\forall v\in V
		\end{split}
	\end{equation}
	which is the invariance property of the scalar product in a specific case. Furthermore, substituting $a$ with $z^{2L_0}(-1)^{-2L_0^2-L_0}e^{-zL_1}a$ and using \eqref{eq:commutation_identities} in the two sides of \eqref{eq:compute_invariance_prop}, we get the following inverse invariance property
	\begin{equation}  
		(Y(\theta(a),z)v|u) = (v|Y(e^{zL_1}(-1)^{2L_0^2+L_0}z^{-2L_0}a,z^{-1})u)
		\qquad
		\forall u\in V_\parzero
		\,\,\,\forall v\in V \,.
		\label{eq:inverse_invariant_prop_for_proof} 
	\end{equation} 
	Consequently, if $b\in V_\parone$, then
	\begin{equation}
		\begin{split}
			(\theta(Y(a,z)b)|c)
			&= (Y(a,z)b,c) \\
			&=(b,Y(e^{zL_1}(-1)^{2L_0^2+L_0}z^{-2L_0}a,z^{-1})c) \\
			&=(\theta(b)|Y(e^{zL_1}(-1)^{2L_0^2+L_0}z^{-2L_0}a,z^{-1})c) \\
			&=(Y(\theta(a),z)\theta(b)|c)
			\qquad \forall c\in V_\parzero
		\end{split}
	\end{equation}
	where we have used the invariance property of the bilinear form for the second equality and the inverse invariance property \eqref{eq:inverse_invariant_prop_for_proof} for the last one. By the non-degeneracy of $\scalar$ on $V_\parzero$, it follows that $\theta(a_nb)=\theta(a)_n\theta(b)$ whenever $b\in V_\parone$, concluding the proof that $\theta$ respects the $(n)$-product. Then the invariance property of the scalar product $\scalar$ follows just proceeding as in \eqref{eq:compute_invariance_prop}.
\end{proof}

\begin{rem}
	Theorem \ref{theo:unitarity_determind_by_even_part} can be slightly generalized to cover the case where $V=V^+\bigoplus V^-$, where $V^\pm$ are the eigenspaces of a general automorphism of $V$ of order two, not only the parity operator $\Gamma_V$, see \cite[Section 2.7]{Gau21} for details.
\end{rem}

\subsection{Uniqueness of the unitary structure and automorphisms}
\label{subsection:uniqueness_unitary_structure}

In the present section, we investigate conditions which ensure the uniqueness of the unitary structure introduced in Section \ref{subsection:basic_definitions_VOSA}. For this reason, we show the relation between invariant scalar products and invariant bilinear forms, introduced in the preceding sections. We further describe and characterize the automorphism groups of unitary VOSAs.

Thanks to Proposition \ref{prop:properties_bilinear_forms}, we can prove the following result:
\begin{prop}  \label{prop:relation_pct_scalar_product}
Let $V$ be a VOSA with a normalized scalar product $\scalar$, which is invariant with respect to an antilinear automorphism $\theta$. Then we have that:
\begin{itemize}
\item[(i)]  $\theta$ is an involution and thus $(V,\scalar)$ is unitary with PCT operator $\theta$. Moreover, $\theta$ is antiunitary, that is, $(\theta(a)|\theta(b))=(b| a)$ for all $a,b\in V$.

\item[(ii)] $\theta$ is the unique PCT operator associated with $\scalar$. Moreover, if $V$ is simple, then $\scalar$ is the unique normalized invariant scalar product associated with $\theta$.
\end{itemize}
\end{prop}

\begin{proof}
We can prove (i) and the first claim of (ii) as in \cite[Proposition 5.1]{CKLW18}. Regarding the second claim of (ii),
$(\theta(\cdot)|\cdot)$ defines a normalized non-degenerate invariant bilinear form on $V$. According to (iii) of Proposition \ref{prop:properties_bilinear_forms}, if $V$ is simple, then $(\theta(\cdot)|\cdot)$ must be unique, which implies the remaining part of the proposition. 
\end{proof}

Claim (ii) of Proposition \ref{prop:relation_pct_scalar_product} says us that for simple unitary VOSAs, a PCT operator determines the associated normalized invariant scalar product and vice versa. Nevertheless, we can still have different choices for the PCT operator or equivalently for the normalized invariant scalar product, which can give us different unitary structures. From this point of view, it is even more interesting to introduce the following result, which is proved by straightforwardly adapting \cite[Proposition 5.19]{CKLW18} and its immediate consequences \cite[p.\  38]{CKLW18}.

\begin{prop} \label{prop:uniqueness_unitary_structure}
Let $V$ be a simple VOSA. Let $(\theta,\scalar)$ and $(\widetilde{\theta},\curlyscalar)$ be two unitary structures on $V$. Then there exists $h\in\Aut (V)$ such that:
\begin{itemize}
\item[(i)] $\{a| b\}=(h(a)| h(b))$ for all $a,b\in V$, that is, $h$ realizes a unitary isomorphism between $(V,\theta,\scalar)$ and $(V,\widetilde{\theta},\curlyscalar)$. Consequently, if $\Aut (V)=\Aut_{\scalar} (V)$, then the unitary structure on $V$ is unique.

\item[(ii)] $\widetilde{\theta}=h^{-1}\theta h$; $\theta h\theta=h^{-1}$; $(a| ha)>0$ for all non-zero $a\in V$.
\end{itemize}
\end{prop}

Proposition \ref{prop:uniqueness_unitary_structure} relates the uniqueness of the unitary structure of a unitary VOSA with its automorphism group. Then we give the following result to investigate the structure of the automorphism group of a unitary VOSA.

\begin{prop}   \label{prop:conformal_vector_bilinear_form}
Let $V$ be a vertex superalgebra with a conformal vector $\nu$ with respect to $V$ is a VOSA such that $V_0=\C \Omega$ and having a non-degenerate invariant bilinear form. Suppose that there exists another conformal vector $\nu'$ with respect to $V$ is a VOSA having a non-degenerate invariant bilinear form. If $\nu_{(1)}=L_0=L_0'=:\nu'_{(1)}$, then $\nu=\nu'$.
\end{prop}

\begin{proof}
The proof is obtained by applying \cite[Proposition 4.8]{CKLW18} to the vertex subalgebra $V_\parzero$.  
\end{proof}

The following corollary is proved by adapting the proof of \cite[Corollary 4.11]{CKLW18}.
\begin{cor} \label{cor:when_vosa_automorphism_preservs_nu}
Let $V$ be a VOSA with a non-degenerate invariant bilinear form $\bilinear$ and such that $V_0=\C \Omega$. Let $g$ be either a linear or an antilinear bijective map of the vector space $V$ preserving the vacuum vector and the $(n)$-product. Then the following are equivalent:
\begin{itemize}
\item[(i)] $g$ is grading-preserving, that is, $g(V_n)=V_n$ for all $n\in\half\Z$.

\item[(ii)] $g$ preserves $\bilinear$, that is: if $g$ is linear, then $(g(a),g(b))=(a,b)$ for all $a,b\in V$; if $g$ is antilinear, then $(g(a),g(b))=\overline{(a,b)}$ for all $a,b\in V$.

\item[(iii)] $g(\nu)=\nu$.
\end{itemize}
\end{cor}

At this stage, we have all the ingredients to prove the following statements by adapting the proofs of \cite[Lemma 5.20 and Theorem 3.21]{CKLW18} (cf.\  the proof of Theorem \ref{theo:characterization_aut_sc_group}).

\begin{theo}   \label{theo:characterization_aut_group}
Let $V$ be a unitary VOSA. Then $\Aut_{\scalar}(V)$ is a compact subgroup of $\Aut(V)$. Moreover, if $V$ is simple, then the following are equivalent:
\begin{itemize}
\item[(i)] $\scalar$ is the unique normalized invariant scalar product on $V$;

\item[(ii)] $\Aut_{\scalar}(V)=\Aut(V)$;

\item[(iii)] $\theta$ commutes with every $g\in\Aut(V)$;

\item[(iv)] $\Aut(V)$ is compact;

\item[(v)] $\Aut_{\scalar}(V)$ is totally disconnected.
\end{itemize}
\end{theo}

\begin{rem}   \label{rem:real_form_aut_group}
Consider a simple unitary VOSA $V$. Following the proof of \cite[Theorem 5.21]{CKLW18}, we have that $g\in\Aut(V)$ is unitary if and only if $(g(\theta( a))|g (b))=(\theta (a)| b)$ for all $a,b\in V$. On the other hand, $(\theta( g(a))|g(b))=(\theta (a)|b)$ for all $a,b\in V$ thanks to (ii)$\Leftrightarrow$(iii) of Corollary \ref{cor:when_vosa_automorphism_preservs_nu}. It follows that $g$ is unitary if and only if it commutes with $\theta$. Hence, $g$ is unitary if and only if it restricts to an automorphism of the real VOSA $V_\R$ defined in Remark \ref{rem:real_form}. Conversely, every VOSA automorphism of $V_\R$ extends to an automorphism of its complexification $V$. Thus, we can identify $\Aut_{\scalar}(V)$ and $\Aut(V_\R)$.
\end{rem}

We end this section investigating further the group of unitary automorphisms of unitary VOSAs, straightforwardly adapting \cite[Appendix A]{CGGH23}, which we refer to whenever details are not given.

From now onwards, let $V$ be a simple unitary VOSA. For any closed subgroup $G$ of $\Aut(V)$, recall the definition of the \textbf{fixed point subalgebra} of $V$ with respect to $G$:
$$
V^G:=\left\{a\in V\mid g(a)=a \,\,\forall g\in G\right\} \,.
$$
Note that $V_\parzero=V^{\{1_V,\Gamma_V\}}$ and thus if $G$ is a closed subgroup of $\Aut(V)$, the quotient group $G/\{1_V,\Gamma_V\}$ can be embeds into $\Aut(V_\parzero)$.

Define the subset of $\Aut(V)$ of \textbf{strictly positive} automorphisms by
\begin{equation}
	\Aut_+(V):=\left\{g\in\Aut(V)\mid (a|ga)>0 \,\,\forall a\in V\backslash\{0\}\right\}\,.
\end{equation}
We have that $1_V\in\Aut_+(V)$. Moreover, it can be shown that $\Aut_+(V)$ is path connected and thus it is contained in $\Aut(V)_0$, that is the connected component to the identity of $\Aut(V)$. Note that $g^{-1}\in\Aut_+(V)$ for all $g\in\Aut_+(V)$ and if $g$ and $h$ are any two commuting elements in $\Aut_+(V)$, then $gh=h^\frac{1}{2}gh^\frac{1}{2}\in\Aut_+(V)$. If $\Aut(V)_0$ is abelian, then $\Aut_+(V)$ is an abelian closed and path connected subgroup of $\Aut(V)$.

Now, for any $g\in \Aut(V)$, set $g^*:=\theta g^{-1}\theta$, so that $(a|g(b))=(g^*(a)|b)$ for all $a,b\in V$. Note that $g\in\Aut(V)$ is unitary if and only if $g^*=g^{-1}$. Furthermore, $g^*g\in\Aut_+(V)$ for all $g\in\Aut(V)$. Therefore, proceeding as in the proof of \cite[Proposition A.1]{CGGH23}, we have a polar decomposition for the automorphisms of $V$:

\begin{prop}
	Let $V$ be a simple unitary VOSA and let $g\in\Aut(V)$. Then there exist a unique $\abs{g}\in\Aut_+(V)$ and a unique $u\in\Aut_{\scalar}(V)$ such that $g=u\abs{g}$.
\end{prop}

The following two results can be proved as in \cite[Proposition A.2 and Proposition A.3]{CGGH23} respectively.

\begin{prop}
	Let $V$ be a simple unitary VOSA  and let $q:\Aut(V) \to \Aut(V)/\Aut(V)_0$ be the quotient map. 
	Then $q\restriction_{\Aut_{\scalar}(V)}: \Aut_{\scalar}(V) \to \Aut(V)/\Aut(V)_0$ is a surjective group homomorphism. 
	In particular, $\Aut(V)$ is almost connected, i.e.\  $\Aut(V)/\Aut(V)_0$ is compact. Furthermore, if $\Aut(V)$ is a finite dimensional Lie group, 
	then $\Aut(V)/\Aut(V)_0$ is a finite group and the map $q\restriction_{\Aut_{\scalar}(V)}: \Aut_{\scalar}(V) \to  \Aut(V)/\Aut(V)_0$ factors through an isomorphism of   
	$\Aut_{\scalar}(V) / \Aut_{\scalar}(V)_0$ onto  $\Aut(V)/\Aut(V)_0$.  
\end{prop}

\begin{prop}  \label{prop:Aut_fixes_as_Aut_scalar}
	 Let $V$ be a simple unitary VOSA. Then $V^{\Aut_{\scalar}(V)} = V^{\Aut(V)}$.
\end{prop}

We have the following characterization of $\Aut_{\scalar}(V)$:

\begin{theo}
	Let $V$ be a simple unitary VOSA and assume that $\Aut(V)$ is a finite dimensional Lie group. Then  $\Aut_{\scalar}(V)$ is a maximal compact subgroup of $\Aut(V)$. 
\end{theo}

\begin{proof}
Let $G$ be a compact subgroup of $\Aut(V)$ containing $\Aut_{\scalar}(V)$. Then by Proposition 
\ref{prop:Aut_fixes_as_Aut_scalar}, $V^G = V^{\Aut_{\scalar}(V)}$ and thus $V_\parzero^{G/\{1_V,\Gamma_V\}} = V_\parzero^{\Aut_{\scalar}(V)/\{1_V,\Gamma_V\}}$. Moreover, $G$ is a compact Lie group being a compact subgroup of the finite dimensional Lie group $\Aut(V)$. It follows that the quotient of $G$ by the its closed normal subgroup $\{1_V,\Gamma_V\}$ is a compact Lie group too. Hence, $G/\{1_V,\Gamma_V\} =  \Aut_{\scalar}(V)/\{1_V,\Gamma_V\}$ by the Galois correspondence in \cite[Theorem 3]{DM99}, which implies that $G =  \Aut_{\scalar}(V)$ as desired.
\end{proof}

The proof of \cite[Proposition A.5]{CGGH23} works also for the following one:

\begin{prop}  \label{prop:unitarize_compact_subgr_aut}
	Let $V$ be a simple unitary VOSA and assume that $\Aut(V)$ is a finite dimensional Lie group. If  $G$ is a compact subgroup of $\Aut(V)$, then there exists an invariant normalized scalar product $\curlyscalar$ on $V$ such that $G \subset \Aut_{\curlyscalar}(V)$, namely the automorphisms in $G$ are unitary with respect to $\curlyscalar$.
\end{prop}  

\begin{rem}  \label{rem:finitely_gen_VOSAs}
	Note that if $V$ is a finitely generated VOSA, then there is a $N \in \Zpluseq$ such that $V$ is generated by the finite dimensional vector space $V_{\leq N} := \bigoplus_{n \in \Z,\, n\leq N} V_n$ and the map $g \mapsto g\restriction_{V_{\leq N}}$ is a topological isomorphism of $\Aut(V)$ onto a closed subgroup of $\mathrm{GL}(V_{\leq N})$. It follows that if $V$ is finitely generated, then $\Aut(V)$ is a finite dimensional Lie group. 
	Finally, note that if $V$ is simple, then $V$ is finitely generated if $V_\parzero$ is by Example \ref{ex:even_simple_odd_irreducible}.
\end{rem}

\subsection{The PCT theorem in the VOSA setting}
\label{subsection:pct_theorem}

The main topic of this section is to characterise the unitarity of a VOSA equipped with a scalar product, whose invariance is unknown a priori. The convenience of this approach is twofold. On the one hand, it is natural from the point of view of Quantum Field Theory (QFT). Indeed, we obtain also the VOSA version of the PCT theorem, see \cite[Section 4.3]{SW64}. On the other hand, in the construction of examples of unitary VOSAs, we often have to find a PCT operator for a given scalar product, which we start with. It would be useful to have different methods to check the existence of the unitary structure, which avoid the direct search for a PCT operator. 
We also highlight that this section not only generalizes \cite[Section 5.2]{CKLW18} to the VOSA setting, but also cover the case of VOSAs (and thus VOAs too) not necessarily of CFT type, which is instead the case in the PCT theorem in the VOA setting \cite[Theorem 5.16]{CKLW18}.

As mentioned above, the Wightman axioms \cite[Section 3.1]{SW64} for QFT require for the unitarity at least that: (i) \textit{``the spacetime symmetries act unitarily''}; (ii) \textit{``the adjoints of local fields are local''}. We translate these requirements introducing the following definitions respectively:

\begin{defin} \label{defin:unitary_mob_symmetry}
Let $V$ be a VOSA with a normalized scalar product $\scalar$.
\begin{itemize}
\item[(i)]
The pair $(V,\scalar)$ is said to have \textbf{unitary M{\"o}bius symmetry} if
\begin{equation}
(L_na| b)=(a| L_{-n}b) 
\qquad\forall a,b\in V
\,\,\,\forall n\in\{-1,0,1\}\,.
\end{equation}

\item[(ii)] An element $A\in\End(V)$ has an \textbf{adjoint} on $V$ (with respect to $\scalar$) if there exists $A^+\in\End(V)$ such that 
\begin{equation}
(a| Ab)=(A^+a| b)
\qquad \forall a,b\in V\,.
\end{equation}
\end{itemize}
\end{defin}

It is clear that if an adjoint exists, then it will be unique and we will call it \textit{the} adjoint of $A$ on $V$. Of course, $A$ can be treated as a densely defined operator on the Hilbert space completion $\mathcal{H}_{(V,\scalar)}$ of $V$ and thus $A^+$ exists if and only if $V$ is contained in the domain of the Hilbert space adjoint $A^*$ of $A$, in which case $A^+\subseteq A^*$. It is not difficult to verify that the set of elements having an adjoint on $V$ is a unital subalgebra of $\End(V)$ closed with respect to the adjoint operation, that is
\begin{align}
(\alpha A+\beta B)^+ &= \overline{\alpha}A^++\overline{\beta}B^+ 
\qquad\forall A, B\in\End(V) \,\,\,\forall \alpha,\beta\in\C \\
(AB)^+ &= B^+A^+ \qquad\forall A, B\in\End(V)\\
A^{++}:= &(A^+)^+=A \qquad\forall A\in\End(V)\,.
\end{align}

By adapting the proof of \cite[Lemma 5.11]{CKLW18}, we get:

\begin{lem}  \label{lem:existence_an+}
Let $(V,\scalar)$ have unitary M\"{o}bius symmetry. Then for every homogeneous $a\in V$ and every $n\in\half\Z$, $a_n^+$ exists on $V$. Moreover, for all $b\in V$, there exists $N\in\half\Zpluseq$ such that $a_{-n}^+b=0$ for all $n\geq N$.
\end{lem}

If $(V,\scalar)$ has unitary M\"{o}bius symmetry, then Lemma \ref{lem:existence_an+} allows us to define for every $a\in V$, a parity-preserving field by
\begin{equation}  \label{eq:defin_adjoint_field}
Y(a,z)^+:=\sum_{n\in\Z} a_{(n)}^+z^{n+1}=
\sum_{n\in\Z}a_{(-n-2)}^+ z^{-n-1} \,.
\end{equation}
Hence, we adapt (ii) of Definition \ref{defin:unitary_mob_symmetry} to the field setting, saying that for any $a\in V$, the field $Y(a,z)$ has a \textbf{local adjoint} if for every $b\in V$, $Y(a,z)^+$ and $Y(b,z)$ are mutually local, that is, there exists $N\in\Zpluseq$ such that
$$
(z-w)^N[Y(a,z)^+,Y(b,w)]=0 \,.
$$
Accordingly, we set $V^{\scalar}$ as the subset of $V$ consisting of all the elements $a\in V$ such that $Y(a,z)$ has a local adjoint. One proves that $V^{\scalar}$ is a vertex subalgebra of $V$ retracing the proof of \cite[Proposition 5.14]{CKLW18} step by step. If $V = V^{\scalar}$, we say that $V$ \textit{has local adjoints}. Furthermore, the same proof of \cite[Lemma 5.15]{CKLW18} gives:

\begin{lem}  \label{lem:adjoint_quasi-primary_vectors}
	Let $(V,\scalar)$ have unitary M{\"o}bius symmetry and let $a\in V^{\scalar}$ be a quasi-primary vector. 
	 Then there is a quasi-primary vector  $\overline{a}\in V^{\scalar}$ with $d_{\overline{a}}=d_a$ and such that $z^{-2d_a}Y(a,z)^+=Y(\overline{a},z)$ or equivalently $a_n^+=\overline{a}_{-n}$ for all $n\in\half\Z$.
\end{lem}

As announced at the beginning, we prove the VOSA version of the PCT theorem:
\begin{theo}   \label{theo:VOSA_PCT}
Let $V$ be a VOSA equipped with a normalized scalar product $\scalar$ and such that $V_0=\C\Omega$. Then the following are equivalent:
\begin{itemize}
\item[(i)]  $(V,\scalar)$ is a unitary VOSA.

\item[(ii)] $(V,\scalar)$ has unitary M{\"o}bius symmetry and local adjoints, that is, $V^{\scalar}=V$. 
\end{itemize}
\end{theo}

\begin{proof}
The proof of \cite[Theorem 5.16]{CKLW18} does the job with just minor changes as we are going to explain.

(i) $\Rightarrow$ (ii). We can proceed as in the proof just cited, using the correct formulae with respect to our framework: for every homogeneous vector $a\in V$, we have from \eqref{eq:an_star} that
\begin{align}  
a_{-n}^+ 
 &= 
(-1)^{2d_a^2+d_a}\sum_{l\in\Zpluseq}
\frac{(L_1^l\theta (a))_n}{l!}
\qquad \forall n\in\Z -d_a  
\label{eq:an_cross}\\
Y(a,z)^+
 &=
(-1)^{2d_a^2+d_a}\sum_{l\in\Zpluseq}
\frac{Y(L_1^l\theta (a),z)z^{2d_a-l}}{l!}  
\end{align}
which are well-defined because $L_1b\in V_{d_b-1}$ for all homogeneous $b\in V$ thanks to \eqref{eq:cr_l0}.

(ii) $\Rightarrow$ (i). The first step for this proof in \cite[Theorem 5.16]{CKLW18} is to prove the simplicity of $V$. 
The argument there use the fact that if $W$ is any VOA with $W_0=\C\Omega$, then $W$ is of CFT type by \cite[Remark 4.5]{CKLW18}. In our setting, we cannot conclude that $V$ is of CFT type by adapting the same argument. Nevertheless, thanks to the fact that $V$ has unitary M{\"o}bius symmetry, we can anyway conclude that $V$ is of CFT type by reasoning as in \eqref{eq:positivity_conformal_weight}. 
It follows that $V$ is simple and the remaining part of the proof works as for \cite[Theorem 5.16]{CKLW18}.
\end{proof}

Thanks to the theorem above and by adapting the proof of \cite[Proposition 5.17]{CKLW18}, we have the desired characterization of the unitarity for VOSAs:
\begin{prop}   
\label{prop:simple_unitary_vosa_gen_by_qp_hermitian_fields}
Let $V$ be a VOSA equipped with a normalized scalar product $\scalar$ and such that $V_0=\C\Omega$. Then the following are equivalent:
\begin{itemize}
\item[(i)]  $(V,\scalar)$ is a unitary VOSA.

\item[(ii)] $Y(\nu,z)$ is a Hermitian field and $V$ is generated by a family of Hermitian quasi-primary fields.
\end{itemize}
\end{prop}

We point out that the condition that $V_0=\C\Omega$ in Theorem \ref{theo:VOSA_PCT} and in Proposition \ref{prop:simple_unitary_vosa_gen_by_qp_hermitian_fields} can be cut out thanks to the following result, which is of course applicable to the VOA case too.

First of all, we introduce some definitions. For a given $N\in\Zplus$ and a family $\{V^j\}_{j=1}^N$ of unitary VOSAs of CFT type with normalized invariant scalar products $\{\scalar_j\}_{j=1}^N$ respectively, the vector superspace $V:=\bigoplus_{j=1}^N V^j$ with the normalized invariant scalar product $\scalar:=\frac{1}{N}\oplus_{j=1}^N\scalar_j$ is a unitary VOSA. We say that $V$ is a \textit{direct sum of unitary VOSAs of CFT type}.
Similarly, if $\{V^j\}_{j=1}^N$ is a family of VOSAs of CFT type, equipped with normalized scalar products $\{\scalar_j\}_{j=1}^N$ with respect to $(V^j,\scalar_j)$ has unitary M{\"o}bius symmetry and $(V^j)^{\scalar_j}=V^j$ for all $j\in\{1,\dots, N\}$, then the vector superspace $V:=\bigoplus_{j=1}^N V^j$ has the M{\"o}bius symmetry with respect to the normalized scalar product $\scalar:=\frac{1}{N}\oplus_{j=1}^N\scalar_j$ and satisfies $V^{\scalar}=V$. Accordingly, we say that $(V,\scalar)$ is a \textit{direct sum of VOSAs of CFT type satisfying the M{\"o}bius symmetries and closed for local adjoints}. 

\begin{prop}  \label{prop:decomposing_unitary_VOSAs}
	Let $V$ be a VOSA. Then we have the following:
	\begin{itemize}
		\item[(i)] $V$ is unitary if and only if it is a direct sum of unitary VOSAs of CFT type.
		
		\item[(ii)] $V$ has a normalized scalar product $\scalar$ such that $(V,\scalar)$ has unitary M{\"o}bius symmetry and $V^{\scalar}=V$ if and only if $(V,\scalar)$ is a direct sum of VOSAs of CFT type  each of which has unitary M{\"o}bius symmetry and local adjoints. 
	\end{itemize}
\end{prop}

\begin{proof}
	\textit{Proof of (i)}.
	Note that the ``if'' part is the trivial one by definition. 
	Accordingly, suppose that $V$ is a unitary VOSA with normalized invariant scalar product $\scalar$ and PCT operator $\theta$. We first prove that $V_0$ can be turned into a finite dimensional commutative $C^*$-algebra. The argument as in \eqref{eq:positivity_conformal_weight} shows that $V_n=\{0\}$ whenever $n<0$. It follows also that if $a\in V_0$, then $L_{-1}a=0$. Hence, for all $a\in V_0$, $a_{(n-1)}=a_n=0$ whenever $n\not=0$ as $\frac{\mathrm{d}}{\mathrm{d}z}Y(a,z)=Y(L_{-1}a,z)=0$, see \cite[Proposition 4.8(a)]{Kac01}. By the Borcherds commutator formula \eqref{eq:borcherds_commutator_formula}, we get that
	$a_0b_kc=b_ka_0c$ for all $a\in V_0$, all $b,c\in V$ and all $k\in\half\Z$; whereas from the Borcherds associative formula \eqref{eq:borcherds_associative_formula}, we get that $(a_0b)_kc=a_0b_kc$ for all $a\in V_0$, all $b,c\in V$ and all $k\in\half\Z$. Therefore, we define a structure of commutative associative algebra on $V_0$ with unit $\Omega$ by setting $a b:= a_0b\in V_0$ for all $a,b\in V_0$. Moreover, $V_0$ is also a $*$-algebra with the involution $V_0\ni a\to \theta(a)\in V_0$ (recall that $\theta$ preserves $\nu$ and the $(n)$-product and thus the grading of $V$). Restricting the scalar product $\scalar$ of $V$ to $V_0$, we have that $V_0$ is also a finite dimensional Hilbert space. Then the left regular representation $L:V_0\to \End(V_0)$, given by $L(a)b:=ab=a_0b$ for all $a,b\in V_0$, is a faithful $*$-representation. Accordingly, the commutative $*$-algebra $V_0$ with the norm $|||\, a\, |||:=\norm{L(a)}$ for all $a\in V_0$ is a finite dimensional commutative $C^*$-algebra.
	
	Let $N:=\dim(V_0)$, then there exists an orthogonal basis $\{\Omega^j\}_{j=1}^N$ of $V_0$ of $\theta$-invariant vectors such that $\Omega=\sum_{j=1}^N\Omega^j$ and $V_0=\bigoplus_{j=1}^N \C \Omega^j$, so that $1_{V_0}=\sum_{j=1}^N\Omega^j_0$ and $\{\Omega^j_0\}_{j=1}^N$ are the corresponding central projections of $V_0$, see e.g.\ \cite[Section III.1]{Dav96}. 
	Thanks to the fact that $\Omega^j_0a_{(-1)}=a_{(-1)}\Omega^j_0$ for all $j\in\{1,\dots,N\}$ and all $a\in V$, we have that $\{\Omega^j_0\}_{j=1}^N$ is a family of orthogonal projections on $V$ too. Hence, it is easy to check that the vector subspaces of $V$
	$$
	V^j:=\Omega^j_0V=\{a_{(-1)}\Omega^j\mid a\in V\}
	\qquad \forall j\in\{1,\dots, N\}
	$$ 
	are VOSAs of CFT type, with of course $\Omega^j$ and $\nu^j:=\nu_{(-1)}\Omega^j$ as respective vacuum and conformal vector and state-field correspondence inherited from the one of $V$. Furthermore, $V$ is the direct sum of these VOSAs $\{V^j\}_{j=1}^N$, which are also unitary with the normalized invariant scalar product $\{\scalar_j:=(\Omega^j|\Omega^j)^{-1}\scalar\}_{j=1}^N$, satisfying the desired property.
		
	\textit{Proof of (ii)}. 
	Also in this case, the non-trivial statement is the ``only if'' part. So let us suppose that $V$ has a normalized scalar product $\scalar$, that $(V,\scalar)$ has unitary M{\"o}bius symmetry and that $V^{\scalar}=V$. 
	Thanks to the M{\"o}bius symmetry, we have that $V_n=\{0\}$ whenever $n<0$ and that $L_{-1}a=0$ for all $a\in V_0$ by the argument as in \eqref{eq:positivity_conformal_weight}. Again by the fact that $\frac{\mathrm{d}}{\mathrm{d}z}Y(a,z)=Y(L_{-1}a,z)=0$, see \cite[Proposition 4.8(a)]{Kac01}, $a_{(n-1)}=a_n=0$ whenever $n\not=0$ for all $a\in V_0$.  
	By Lemma \ref{lem:adjoint_quasi-primary_vectors}, there exists $\overline{a}\in V_0$ such that $Y(\overline{a},z)=Y(a,z)^+=a_0^+$, that is $\overline{a}_0=a_0^+$. Then the proof follows as for part (i), first proving that $V_0$ is a finite dimensional commutative $C^*$-algebra, with the difference that this time the $*$-structure is given by the involution $V_0\ni a\mapsto \overline{a}\in V_0$. 
\end{proof}

We finish with a further characterization of unitarity for VOSAs (and thus VOAs too), again in the spirit of the PCT theorem, cf.\ \cite[Section 2]{CT23}:

\begin{theo} \label{theo:characterization_unitarity_VOSAs}
	Let $V$ be a VOSA equipped with a normalized scalar product $\scalar$.  Then the following are equivalent:
	\begin{itemize}
		\item[(i)]  $(V,\scalar)$ is a unitary VOSA.
		
		\item[(ii)] There exists an antilinear vector space involution $V\ni a\mapsto \overline{a}\in V$ such that $\overline{\nu}=\nu$ and that $(a_nb|c)=(b|\overline{a}_{-n}c)$ for all $a,b,c\in V$ and all $n\in\half\Z$. 
	\end{itemize}
Furthermore, $\overline{\Omega}=\Omega$ and the PCT operator is related to the antilinear involution in (ii) by the following formula:
$$
\overline{a}=e^{L_1}(-1)^{2d_a^2+d_a}\theta(a)
\qquad\forall a\in V \,.
$$ 
\end{theo}

\begin{proof}
	The equivalence between (i) and (ii) follows from Theorem \ref{theo:VOSA_PCT} with Proposition \ref{prop:decomposing_unitary_VOSAs}. Regarding the last statement, it is easy to see that $\overline{\Omega}=\Omega$. Moreover, if $V$ is unitary, then there exists a PCT operator $\theta$ (which is also unique by (ii) of Proposition \ref{prop:relation_pct_scalar_product}) satisfying the desired relation, see \eqref{eq:an_cross}. 
\end{proof}

\subsection{Unitary subalgebras}
\label{subsection:unitary_subalgebras}

In what follows, we introduce unitary subalgebras, graded tensor products, cosets and fixed-point constructions, relying on some of the results of Section \ref{subsection:pct_theorem}.

\begin{defin}  \label{defin:unitary_subalgebras}
A \textbf{unitary subalgebra} $W$ of a unitary VOSA $V$ is a vertex subalgebra of $V$ such that:
\begin{itemize}
\item[(i)] $W$ is compatible with the grading, that is, $L_0W\subseteq W$ or equivalently $W=\bigoplus_{n\in\half\Z}(W\cap V_n)$;

\item[(ii)] $a_{(n)}^+b\in W$ for all $a, b\in W$ and all $n\in\Z$.
\end{itemize}
Note that $a_{(n)}^+$ exists for all $a\in W$ and all $n\in\Z$ by Lemma \ref{lem:existence_an+}.
Moreover, (ii) is equivalent to say that $a_n^+b\in W$ for all $a,b\in W$ and all $n\in\half\Z$, provided that (i) holds.
\end{defin}

We characterize a unitary subalgebra with the following result.
\begin{prop}   \label{prop:characterisation_unitary_subalgebras}
Let $W$ be a vertex subalgebra of a unitary VOSA $V$. Then $W$ is a unitary subalgebra of $V$ if and only if $L_1W\subseteq W$ and $\theta (W)\subseteq W$.
\end{prop}

\begin{proof}
The proof follows the one of \cite[Proposition 5.23]{CKLW18}, provided that we use \eqref{eq:an_star} instead of \cite[Eq.\  (84)]{CKLW18} whenever appropriate.
\end{proof}

\begin{ex}  \label{ex:fixed_point_subalgebra}
Let $V$ be a unitary VOSA and $G\subseteq\Aut_{\scalar}(V)$ be a closed subgroup. By Proposition \ref{prop:characterisation_unitary_subalgebras}, the fixed point subalgebra
\begin{equation}
V^G:=\{a\in V\mid g(a)=a \,\,\,\forall g\in G\}
\end{equation}
is a unitary subalgebra. When $G$ is finite, $V^G$ is known as \textbf{orbifold subalgebra}. For example, the even subspace $V_\parzero$ is the orbifold subalgebra $V^{\left\{1_V, \Gamma_V\right\}}$. 
\end{ex}

Now, note that the \textbf{projection operator} $e_W$ on a unitary subalgebra $W\subseteq V$ is a well-defined element of $\End(V)$ thanks to the grading compatibility of $W$ and the finite dimensions of the eigenspaces $V_n$ for all $n\in\half\Z$. We prove the following structural result for unitary subalgebras by adapting \cite[Lemma 5.28 and Proposition 5.29]{CKLW18}. 

\begin{prop}  \label{prop:unitary_structure_subalgebras}
Let $V$ be a simple unitary VOSA with conformal vector $\nu$, scalar product $\scalar$ and PCT operator $\theta$.
Let $W$ be a unitary subalgebra with the associated projector operator $e_W\in\End (V)$. Then 
$[Y(a,z),e_W]=0$ for all $a\in W$, $[L_n,e_W]=0$ for all $n\in\{-1,0,1\}$, $[\theta,e_W]=0$ and $e_WY(a,z)e_W=Y(e_Wa,z)e_W$ for all $a\in V$. 
If we define $\nu^W:=e_W\nu\in W$, then:
\begin{itemize}
	\item[(i)] 
	$Y(\nu^W,z)=\sum_{n\in\Z}L_n^Wz^{-n-2}$ is a Hermitian Virasoro field on $V$ and such that $L_n^W\restriction_W=L_n\restriction_W$ for all $n\in\{-1,0,1\}$;
	
	\item[(ii)] 
	$\nu^W$ is a conformal vector for the vertex superalgebra
	$(W, \Omega, T\restriction_W, Y(\cdot,z)\restriction_W)$, so that $W$ has a VOSA structure;
	
	\item[(iii)] 
	$(W, \Omega, T\restriction_W,  Y(\cdot,z)\restriction_W, \nu^W, \scalar, \theta\restriction_W)$ is a simple unitary VOSA.
\end{itemize}
\end{prop}

Of course, the \textbf{trivial VOA}, that is the vector space $\C$ with $\Omega:=1$, $Y(\Omega, z)=1_\C$, $\nu:=0$, $(a|a)=a$ and $\theta(a)=\overline{a}$ for all $a\in\C$, is a unitary subalgebra for any simple unitary VOSA.
Another example is given by the \textbf{coset subalgebra} of a vertex subalgebra $W$ of a vertex superalgebra $V$ (it is called centralizer in \cite[Remark 4.6b]{Kac01}). This is the vertex subalgebra of $V$ given by the vector subspace
\begin{equation}   \label{eq:defin_coset_subalgebra}
W^c:=\left\{
a\in V \mid [Y(a,z),Y(b,w)]=0 \,\,\,\forall b\in W
\right\}
=\left\{
a\in V \mid b_{(j)}a=0 \,\,\,
\substack{\forall b\in W  \\ \forall j\in\Zpluseq }
\right\}  
\end{equation}
where the last equality is obtained by the Borcherds commutator formula \eqref{eq:borcherds_commutator_formula}.
Then we have the following result:
\begin{prop}   \label{prop:properties_coset_subalgebra}
Let $V$ be a unitary VOSA with conformal vector $\nu$ and let $W$ be a unitary subalgebra. Then $W^c$ is a unitary subalgebra of $V$. If $V$ is simple, then we also have that
 $\nu=\nu^W+\nu^{W^c}$ and the operators $L_0^W=\nu_{(1)}^W$ and $L_0^{W^c}=\nu_{(1)}^{W^c}$ are simultaneously diagonalizable on $V$ with non-negative eigenvalues.
\end{prop}

\begin{proof}
A straightforward adaptation of \cite[Example 5.27]{CKLW18} proves that $W^c$ is a unitary subalgebra of $V$. For the remaining part, it is sufficient to proceed as in the proof of \cite[Proposition 5.31]{CKLW18}. We just clarify that for any $z_1\in\C\backslash\{0\}$, the operator $z_1^{L_0^W}:=e^{\mathrm{log}(z_1)L_0^W}$ is well-defined on the VOSA $V$ as $[L_0,L_0^W]=0$ by (i) and (ii) of Proposition \ref{prop:unitary_structure_subalgebras} and provided that a suitable branch of the complex logarithm is chosen. Nevertheless, note that for any $a\in W$, the expression $z_1^{L_0-L_0^W}Y(a,z)z_1^{L_0^W-L_0}$ is independent on the choice of the branch as $[L_0-L_0^W,a]=0$ by (i) of Proposition \ref{prop:unitary_structure_subalgebras}.
\end{proof}

\section{From VOSAs to graded-local conformal nets}
\label{section:construction_net}

In this section, we show how to define a graded-local conformal net from a given unitary VOSA under natural assumptions. This construction extends the one given in \cite[Chapter 6]{CKLW18} between unitary VOAs and conformal nets. To aid reading, we collect here below the objects which will be used throughout the exposition.

\begin{defin}  \label{defin:hilbet_space_from_V}
Let $(V,\scalar)$ be a unitary VOSA with PCT operator $\theta$. We set $\norm{\cdot}$ as the norm induced by the scalar product $\scalar$, that is $\norm{v}:=(v|v)^\half$ for all $v\in V$, and $\mathcal{H}:=\mathcal{H}_{(V,\scalar)}$ as the Hilbert space completion of $V$ with respect to $\scalar$. 
Moreover, we define $\Gamma$ and $Z$ as the extensions to $\mathcal{H}$ of the operators $\Gamma_V$ and $Z_V$ respectively. 
\end{defin}

Note that $\Gamma=\Gamma^{-1}=\Gamma^*$ and $Z^{-1}=Z^*$.

\subsection{Energy-bounded VOSAs and smeared vertex operators}
\label{subsection:energy_bounds}

In this first part, we introduce the concept of \textit{energy bounds} for unitary VOSAs and we construct certain operator-valued distributions associated to the vertex operators, called \textit{smeared vertex operators}. 

\begin{defin}
Let $k$ be a non-negative real number. Then $a\in V$, or equivalently the corresponding field $Y(a,z)$, satisfies \textbf{$k$-th order (polynomial) energy bounds} if there exist non-negative real numbers $M$ and $s$ such that 
\begin{equation}   \label{eq:energy-bounds}
\norm{a_nb}\leq M(\abs{n}+1)^s\norm{(L_0+1_V)^kb}
\qquad\forall n\in\half\Z \,\,\, \forall b\in V \,.
\end{equation}
More specifically, $a$ satisfies \textbf{linear energy bounds} if $k=1$; whereas if $k$ is not specified then it is simply said that $a$ satisfies \textbf{energy bounds}.
Accordingly, we say that $V$ is \textbf{energy-bounded} if every $a\in V$ satisfies energy bounds. Similarly, a unitary subalgebra $W$ of $V$ is said to be \textbf{energy-bounded} if every $a\in W$ satisfies energy bounds, cf.\ \cite[Remark 3.13]{CT23}.
\end{defin}

\begin{prop}  \label{prop:energy_boundedness_by_generators}
Let $V$ be a unitary VOSA generated by a family of homogeneous vectors satisfying the energy bounds. Then $V$ is energy-bounded too.
\end{prop}
\begin{proof}
The proof is the same as in \cite[Proposition 6.1]{CKLW18}, just replacing \cite[Eq.\  (103)]{CKLW18} there with the Borcherds identity for the superalgebra case \eqref{eq:borcherds_id}.
\end{proof}

\begin{cor}  \label{cor:energy_bound_graded_tensor_product}
If $V^1$ and $V^2$ are energy-bounded VOSAs, then $V^1\hat{\otimes}V^2$ is energy-bounded.
\end{cor}

The following proposition gives sufficient conditions of energy-boundedness and it will be useful in the production of examples in Section \ref{section:examples}.

\begin{prop}  \label{prop:gen_by_V1/2_V1}
Let $V$ be a simple unitary VOSA. Suppose that $V$ is generated by $V_\half\cup V_1\cup\F$ where $\F\subseteq V_2$ is a family of quasi-primary $\theta$-invariant Virasoro vectors. Then $V$ is energy-bounded.
\end{prop}

\begin{proof}
Suppose that $V_\half\not=\{0\}$, otherwise the proposition is \cite[Proposition 6.3]{CKLW18}. For $a,b\in V_\half$, the Borcherds commutator formula \eqref{eq:borcherds_commutator_formula} becomes
\begin{equation}  \label{eq:commutator_V_one_half}
[a_m,b_k]c=
\sum_{j=0}^\infty\binom{m-\half}{j}
\left(a_{(j)}b\right)_{(m+k-1-j)}c
\qquad\forall m,k\in\Z-\half 
\,\,\,\forall c\in V\,.
\end{equation}
Note that $a_{(j)}b\in V_{-j}$ for all $j\in\Zpluseq$ by \eqref{eq:cr_l0} and that $V$ is of CFT type by Proposition \ref{prop:characterisation_simplicity}. Hence, $a_{(j)}b=0$ for all $j\in\Zplus$ and $a_{(0)}b=\alpha\Omega$ for some $\alpha\in\C$. Thanks to \eqref{eq:an_star_qp_2}, we calculate that
$$
-(\theta (a)| b)= -((\theta (a))_{(-1)}\Omega| b)
=-((\theta (a))_{-\half}\Omega| b) = (\Omega| a_{\half}b)=(\Omega| a_{(0)}b)
=\alpha(\Omega|\Omega)=\alpha \,.
$$
Accordingly, \eqref{eq:commutator_V_one_half} is equivalent to
\begin{equation}  \label{eq:commutator_V_one_half_2}
[a_m,b_k]c
=-(\theta (a)| b)\delta_{m,-k}c
\qquad\forall m,k\in\Z-\half \,\,\,\forall c\in V
\end{equation}
which implies, choosing $b=\theta(a)$, $k=-m$ and recalling that $\theta$ is antiunitary by (i) of Proposition \ref{prop:relation_pct_scalar_product}, that 
$$
-[\theta(a)_{-m},a_m]c=-[a_m,\theta(a)_{-m}]c=\norm{\theta(a)}^2c=\norm{a}^2c 
\qquad \forall m\in\Z-\half \,\,\,\forall c\in V\,.
$$
Then for all $m\in\Z-\half$ and all $c\in V$, we have that
$$
\norm{a_mc}^2\leq \norm{a_mc}^2+\norm{\theta(a)_mc}^2=(c|-[a_m,\theta(a)_{-m}]c)
=\norm{a}^2\norm{c}^2
\,.
$$
It follows that, for any $a\in V_\half$ and any $m\in\Z-\half$, $a_m$ is \textit{bounded} on $V$ and thus it satisfies 0-th order energy bounds  as in \eqref{eq:energy-bounds} with $s=0$ and $M=\norm{a}^2$. Proceeding as in the proof of \cite[Proposition 6.3]{CKLW18}, based again on the argument in \cite[pp.\ 112--113]{BS90}, cf.\ also \cite[Proposition 3.4 and Proposition 3.6]{CT23}, we get that any $a\in V_1\cup \F$ satisfies linear energy bounds. Thus, $V$ is energy-bounded by Proposition \ref{prop:energy_boundedness_by_generators}.
\end{proof}

Now, two useful lemmata, which are the version for odd vectors of \cite[Lemma 3.5]{CT23} and of \cite[Lemma 3.7]{CT23}, cf.\ also \cite[Proposition 3.1]{CTW22}. Cf.\ also Theorem \ref{theo:characterization_unitarity_VOSAs} for some notations.

\begin{lem}  \label{lem:estimate_for_energy_bounds}
	Let $V$ be a simple unitary VOSA. Let $a\in V$ be a homogeneous vector in $V_\parone$ and set $\overline{a}:=e^{L_1}(-1)^{2L_0^2+L_0}\theta(a)$. Then $\norm{a_mb}^2\leq (m+d)^{-1}(b|(a_{-d-1}\overline{a})_0b)$ for all $b\in V$ and all $m\in\Zplus-d$.
\end{lem}

\begin{proof}
	If $b$ is any homogeneous vector of $V$, we can rewrite the Borcherds identity \eqref{eq:borcherds_id} as: for all $m\in \Z$, all $n\in \Z-\half$ and all $k\in\half\Z$, it holds that
	\begin{equation}  \label{eq:borcherds_id_homo}
		\begin{split}
				\sum_{j=0}^\infty \binom{m}{j} 
				\left(a_{n+j}b\right)_{m+k}
				&= \\
				\sum_{j=0}^\infty(-1)^j \binom{n+d-1}{j}
				&\left(
				a_{m+n-j}b_{k-n+j}
				-(-1)^{p(a)p(b)}(-1)^{n+d} b_{k-j+d-1} a_{m+j+1-d}
				\right)
		\end{split}
	\end{equation}
	which can be actually extended to all $b\in V$ by linearity. Choosing $m=0$, we can specialize it to the Borcherds commutator formula, cf.\ \eqref{eq:borcherds_commutator_formula}, so that it is easy to calculate that
	\begin{equation}
		(a_{-d-1}\overline{a})_0=\sum_{j=d+1}^{\infty}(j-d)a_{-j}\overline{a}_j
		+\sum_{j=1-d}^{\infty} (j+d)\overline{a}_{-j}a_j \,.
	\end{equation}
	Hence, we get that
	\begin{equation}
		(b|(a_{-d-1}\overline{a})_0b)=\sum_{j=d+1}^{\infty}(j-d)\norm{\overline{a}_jb}^2
		+\sum_{j=1-d}^{\infty} (j+d)\norm{a_jb}^2 
		\qquad\forall b\in V
	\end{equation}
and the result follows.
\end{proof}

\begin{lem}  \label{lem:k_energy_bounds}
	Let $V$ be a simple unitary VOSA. Let $a$ be a primary odd vector with conformal weight $d\in\Zpluseq+\half$ and let $k$ be any non-negative real number. Then we have three cases:
	if $d=\half$, then $a$ satisfies $0$-th order energy bounds; 
	if $d=\frac{3}{2}$ and $\norm{a_\half(L_0+1_V)^{-k}}<+\infty$, then $a$ satisfies $(k+2)$-th order energy bounds; 
	if $d\in\Zpluseq+\frac{5}{2}$ and $\norm{a_\half(L_0+1_V)^{-k}}<+\infty$, then $a$ satisfies $(k+1)$-th order energy bounds.
\end{lem}

\begin{proof}
	If $d=\half$, then we have already proved the result in the proof of Proposition \ref{prop:gen_by_V1/2_V1}. 
	If $d\in\Zpluseq+\frac{3}{2}$, then we have that $[L_m,a_\half]=((d-1)m-\half)a_{m+\half}$ for all $m\in\Z$. Now, if $d\in\Zpluseq+\frac{5}{2}$, then $(d-1)m-\half\not=0$ for all $m\in\Z$. Therefore, we can apply the argument in the proof of \cite[Lemma 3.7]{CT23} to conclude.
	Instead, if $d=\frac{3}{2}$, then $(d-1)m-\half=0$ if and only if $m=1$. Hence, for $m\not=1$, we can obtain $(k+1)$-th order energy bounds for every $a_n$ with $n\not=\frac{3}{2}$, applying the same argument as before. In particular, this implies that $\norm{a_{-\half}(L_0+1_V)^{-k-1}}<+\infty$. Moreover, $[L_2,a_{-\half}]=(2d-\frac{3}{2})a_\frac{3}{2}$, which allows us to obtain the $(k+2)$-th order energy bounds for $a_\frac{3}{2}$ by the same argument as before, so concluding the proof.
\end{proof}

Thanks to the above lemmata, we can prove the following:

\begin{theo}  \label{theo:V_energy-bounded_iff_V_0_is}
	Any simple unitary VOSA $V$ is energy-bounded if and only if $V_\parzero$ is an energy-bounded unitary subalgebra of $V$.
\end{theo}

\begin{proof}
	If $V$ is energy-bounded, then the claim is trivial. 
	By Example \ref{ex:even_simple_odd_irreducible}, $V$ is generated by $V_\parzero$ and any non-zero odd vector $a$. We can choose $a$ as the primary odd vector of lowest conformal weight. Suppose that $V_\parzero$ is an energy-bounded unitary subalgebra of $V$. Thus, $V$ will be energy-bounded if $a$ satisfy energy bounds by Proposition \ref{prop:energy_boundedness_by_generators}.
	If $d_a=\half$, then $a$ satisfies the $0$-th order energy bounds by the proof of Proposition \ref{prop:gen_by_V1/2_V1} and the result follows.
	If $d_a\not=\half$, we can adapt the argument in the proof of \cite[Theorem 4.1]{CT23}.
	Indeed, note that $x:=a_{-d_a-1}\overline{a}$ with 
	$$
	\overline{a}:=e^{L_1}(-1)^{2L_0^2+L_0}\theta(a)=(-1)^{d_a+\half}\theta(a)
	$$ 
	see \eqref{eq:2d^2+d}, is an even vector. Then $x$ satisfies energy bounds and thus there exists some non-negative real number $k$ such that $\norm{x_0(L_0+1_V)^{-2k}}<+\infty$. Hence, by Lemma \ref{lem:estimate_for_energy_bounds}, $x_0$ is a positive operator and for any homogeneous $b\in V$, we have that
	\begin{equation}
		\begin{split}
			\norm{a_\frac{1}{2}(L_0+1_V)^{-k}b}^2
			&\leq \left(d_a+\half\right)^{-1}((L_0+1_V)^{-k}b|(a_{-d_a-1}\overline{a})_0(L_0+1_V)^{-k}b) \\
			&=
			\left(d_a+\half\right)^{-1}((L_0+1_V)^{-k}b|x_0(L_0+1_V)^{-k}b)\\
			&=
			\left(d_a+\half\right)^{-1}(b|x_0(d_b+1)^{-2k}b)\\
			&=
			\left(d_a+\half\right)^{-1}(b|x_0(L_0+1_V)^{-2k}b)\\
			&\leq 
			\left(d_a+\half\right)^{-1}\norm{x_0(L_0+1_V)^{-2k}}<+\infty
			\,.
		\end{split}
	\end{equation}
Then the result follows by Lemma \ref{lem:k_energy_bounds}.
\end{proof}

\begin{rem}
	It is not difficult to generalize Theorem \ref{theo:V_energy-bounded_iff_V_0_is} in the following way, see again the proof of \cite[Theorem 4.1]{CT23}: if $G$ is a compact group of unitary automorphisms of a simple unitary VOSA $V$, then $V$ is energy-bounded if and only if $V^G$ is an energy-bounded unitary subalgebra of $V$.
\end{rem}

Our aim is to define some operator-valued distributions from energy-bounded vertex operators. Therefore, our next step is to introduce some test function spaces for these operator-valued distributions. Consider the Fréchet space $C^\infty(S^1)$ of complex-valued infinitely differentiable functions on $S^1$ and its subset $C^\infty(S^1,\R)$ of real-valued functions. Consider also the function 
\begin{equation}
	\chi:S^1\to \C \,,\quad z=e^{ix} \mapsto e^{i\frac{x}{2}} 
\end{equation}
where $x\in(-\pi,\pi]$. We define the following spaces of functions:
\begin{equation}  \label{eq:test_funct_spaces_chi}
	\begin{split}
		C_\chi^\infty(S^1) &:=\left\{\chi h \mid h\in C^\infty(S^1) \right\} \\
		C_\chi^\infty(S^1,\R) &:=\left\{g\in C_\chi^\infty(S^1) \mid g(z)\in\R \,\,\,\forall z\in S^1 \right\} .
	\end{split}
\end{equation}
Note that $C_\chi^\infty(S^1,\R)\not= \chi C^\infty(S^1,\R)$. 
For any $g=\chi h\in C_\chi^\infty(S^1)$, we also set $\overline{g}(z):=\chi(z)\overline{zh(z)}$ for all $z\in S^1$, so that $\overline{g}\in C_\chi^\infty(S^1)$.
The Fréchet topology on $C^\infty(S^1)$ induces a Fréchet topology on $C_\chi^\infty(S^1)$, making the latter a Fréchet space as well. 
With an abuse of notation, for every $f\in C^\infty(S^1)$ and $g\in C_\chi^\infty(S^1)$, we will write $f(x)$ and $g(x)$ to mean
the functions of real variable $x$ defined by $f(e^{ix})$ and $g(e^{ix})$ respectively. Then we write
\begin{equation}
f'(z):=\frac{\mathrm{d}f}{\mathrm{d}x}(x)
\,,\qquad
g'(z):=\frac{\mathrm{d}g}{\mathrm{d}x}(x)=\frac{i}{2}g(z)+\chi(z)h'(z)
\end{equation}
where $z=e^{ix}$ for some $x\in(-\pi,\pi]$ and $g=\chi h$ for some $h\in C^\infty(S^1)$.

Now, let $(V,\scalar)$ be a unitary VOSA and consider its Hilbert space completion $\mathcal{H}$, see Definition \ref{defin:hilbet_space_from_V}. Note that for every $a\in V$ and for every $n\in\Z$, $a_{(n)}$ is an operator on $\mathcal{H}$ with dense domain $V$. Furthermore, $a_{(n)}$ is closable as it has a densely defined adjoint thanks to the invariance property of the scalar product. By definition, for every $a\in V$ and every $n\in\half\Z$, $a_n$ is a closable operator on $\mathcal{H}$ with dense domain $V$. Suppose further that $V$ is energy-bounded and consider $f\in C^\infty(S^1)$ and $g=\chi h\in C^\infty_\chi(S^1)$ with their Fourier coefficients: 
\begin{equation} \label{eq:fourier_coeff_S1}
	\begin{split}
		\widehat{f}_n
		&:=\oint_{S^1} f(z)z^{-n}\,\frac{\mathrm{d}z}{2\pi i z}
		=\frac{1}{2\pi}\int_{-\pi}^\pi f(x)e^{-inx}\,\mathrm{d}x
			 \qquad\forall n\in\Z \\
		\widehat{g}_n 
		&:= \widehat{h}_\frac{2n-1}{2} 
		= \frac{1}{2\pi}\int_{-\pi}^\pi g(x)e^{-inx}\,\mathrm{d}x
			\qquad \forall n\in\Z-\half \,.
	\end{split}		
\end{equation}
For every non-negative real number $s$, define the following norms on $C^\infty(S^1)$ and $C^\infty_\chi(S^1)$ respectively:
\begin{equation}
	\norm{f}_s:=\sum_{n\in\Z}(1+\abs{n})^s\abs{\widehat{f}_n}
	\quad\mbox{ and }\quad
	\norm{g}_s:=\sum_{n\in\Z-\half}(1+\abs{n})^s\abs{\widehat{g}_n} 
\end{equation}
for all $f\in C^\infty(S^1)$ and all $g\in C^\infty_\chi(S^1)$.
Then for every $a\in V_\parzero$ and every $b\in V_\parone$, we define  the operators $Y_0(a,f)$ and $Y_0(b,g)$ both with domain $V$ by
\begin{equation}   \label{eq:def_smeared_vertex_operators}
Y_0(a,f)c:=\sum_{n\in\Z}\widehat{f}_n a_nc
\quad \forall c\in V
\,,\qquad
Y_0(b,g)c:=\sum_{n\in\Z-\half}\widehat{g}_n b_nc
\quad \forall c\in V \,.
\end{equation}
Note that the latter operators are densely defined on $\mathcal{H}$ because the series in \eqref{eq:def_smeared_vertex_operators} converge in $\mathcal{H}$ thanks to the energy bounds and the rapidly decaying of the Fourier coefficients $\widehat{f}_n$ and $\widehat{g}_n$. Moreover, both $Y_0(a,f)$ and $Y_0(b,g)$ have densely defined adjoints thanks to the invariance of the scalar product, and we give the following definition:

\begin{defin}
Let $(V,\scalar)$ be an energy-bounded unitary VOSA. For all $a\in V_\parzero$, $b\in V_\parone$ and all $f\in C^\infty(S^1)$, $g\in C^\infty_\chi(S^1)$, we define $Y(a,f)$ and $Y(b,g)$ as the closure on the Hilbert space $\mathcal{H}$ of the operators $Y_0(a,f)$ and $Y_0(b,g)$ respectively, as given in \eqref{eq:def_smeared_vertex_operators}. We call them \textbf{smeared vertex operators}.
\end{defin}

In order to define a net of von Neumann algebras from the smeared vertex operators, we need to find a common invariant core for them and their adjoints. This is what we do in the following, rewriting with some more details the argument in \cite[p.\  47]{CKLW18} and extending it to the super case.

For every $k\in\Zpluseq$, set $\mathcal{H}^k$ the domain in $\mathcal{H}$ of the positive self-adjoint operator $L_0^k$. Note that $\mathcal{H}^k$ is complete with respect to the scalar product $\scalar_k:=((1_\mathcal{H}+L_0)^k\cdot|(1_\mathcal{H}+L_0)^k\cdot)$. Define $V^k$ as the Hilbert space completion of $V$ with respect to $\scalar_k$ and consider the corresponding induced norm $\norm{\cdot}_k$. Then
we have that $V^k=\mathcal{H}^k$. Indeed, if $V$ were not $\norm{\cdot}_k$-dense in $\mathcal{H}^k$, there would exist a non-zero vector $v\in\mathcal{H}^k\backslash V^k$ such that
\begin{equation}
((1_\mathcal{H}+L_0)^kv|(1_\mathcal{H}+L_0)^ka)  =0
\quad\forall a\in V 
\quad\Leftrightarrow\quad
(v|(1_\mathcal{H}+L_0)^{2k}a)=0
\quad\forall a\in V
\end{equation}
which implies that $(v|b)=0$ for all $b\in V$,
which is impossible because $V$ is $\norm{\cdot}$-dense in $\mathcal{H}$ by construction. As usual, see \cite[Section 1]{Nel72} and \cite[Section 2]{GW85}, see also \cite[p.\ 481]{Tol99} and \cite[Section 1.5]{Lok94}, define the dense subspace of $\mathcal{H}$ of \textbf{smooth vectors} 
\label{defin:smooth_vectors_L_0}
for $L_0$ as $\mathcal{H}^\infty:=\bigcap_{k\in\Zpluseq}\mathcal{H}^k$ equipped with the Fréchet topology given by the norms $\norm{\cdot}_k$ for all  $k\in\Zpluseq$.
Then a routine proof gives us the following:
\begin{lem}   \label{lem:common_core}
$\mathcal{H}^\infty$ is a common core for all the smeared vertex operators $Y(a,f)$ and $Y(b,g)$ with $a\in V_\parzero$, $b\in V_\parone$, $f\in C^\infty(S^1)$ and $g\in C_\chi^\infty(S^1)$. Moreover, $Y(a,f)$ and $Y(b,g)$ are continuous in $\mathcal{H}^\infty$ satisfying, for some non-negative real numbers $M,s$ and $k$,
\begin{equation} \label{eq:bound_smeared_vertex_ops}
\begin{split}
\norm{Y(a,f)c} &\leq M\norm{f}_s\norm{(L_0+1_\mathcal{H})^kc}
\qquad \forall c\in \mathcal{H}^\infty \\
\norm{Y(b,g)c} &\leq M\norm{g}_s\norm{(L_0+1_\mathcal{H})^kc}
\qquad \forall c\in \mathcal{H}^\infty \,.
\end{split}
\end{equation}
Consequently, for every $c\in\mathcal{H}^\infty$, the two maps
\begin{equation}  \label{eq:smeared_vertex_operator_distribution}
	C^\infty(S^1)\ni f\longmapsto Y(a,f)c\in\mathcal{H} 
	\quad\mbox{ and }\quad
	C^\infty_\chi(S^1)\ni g\longmapsto Y(b,g)c\in\mathcal{H}
\end{equation}
are continuous and linear, that is, they are operator-valued distributions.
\end{lem}

The invariance of the common core $\mathcal{H}^\infty$ for all smeared vertex operators, is a consequence of the following standard argument.

\begin{lem}  \label{lem:h_infty_is_invariant}
Let $a\in V_\parzero$ and $f\in C^\infty(S^1)$. For all $t\in\R$, define the $C^\infty(S^1)$-function $f_t(z):=f(e^{-it}z)$. Then we have the following equalities:
\begin{align} 
e^{itL_0}Y(a,f)c &= Y(a,f_t)e^{itL_0}c
\qquad \forall c\in\mathcal{H}^\infty 
\,\,\,\forall t\in \R 
\label{eq:smeared_vertex_operator_f_t}\\
iL_0 Y(a,f)c &= -Y(a,f')c+iY(a,f)L_0c    
\qquad \forall c\in\mathcal{H}^\infty \,. \label{eq:derivation_smeared_vo_f_t} 
\end{align}
In particular, $Y(a,f)c\in\mathcal{H}^\infty$ for all $c\in\mathcal{H}^\infty$.

Let $b\in V_\parone$ and $g=\chi h$ with $h\in C^\infty(S^1)$. For all $t\in\R$, define the $C^\infty(S^1)$-function $h_t(z):=h(e^{-it}z)$. Then we have the following equalities:
\begin{align}
e^{itL_0}Y(b,g)c &= Y(b,e^{-i\frac{t}{2}}\chi  h_t)e^{itL_0}c
\qquad \forall c\in\mathcal{H}^\infty 
\,\,\,\forall t\in \R 
\label{eq:smeared_vertex_operator_f_t_odd}\\
iL_0 Y(b,g)c &= -Y(b,g')c+iY(b,g)L_0c    
\qquad \forall c\in\mathcal{H}^\infty \,.
\label{eq:derivation_smeared_vo_f_t_odd} 
\end{align}
In particular, $Y(b,g)c\in\mathcal{H}^\infty$ for all $c\in\mathcal{H}^\infty$.
\end{lem}

\begin{proof}
Equations \eqref{eq:smeared_vertex_operator_f_t} and \eqref{eq:derivation_smeared_vo_f_t} are stated in \cite[p.\  47]{CKLW18}, following an argument similar to the one here below, which we use to prove the second part.

Note that the bounded operator $e^{itL_0}$ on $\mathcal{H}$ preserves $V$ and $\mathcal{H}^\infty$, whereas $Y(b,g)c\in\mathcal{H}$ for all $c\in\mathcal{H}^\infty$. 
We prove \eqref{eq:smeared_vertex_operator_f_t_odd} for homogeneous vectors $b\in V_\parone$ and $c\in V$ first. On the one hand, for all $t\in\R$, we have that
\begin{equation}
e^{itL_0}Y(b,g)c=\sum_{n\in\Z-\half}\widehat{g}_ne^{itL_0}b_nc
=\sum_{n\in\Z-\half}e^{it(d_c-n)}\widehat{g}_nb_nc
\end{equation}
where we have used \eqref{eq:def_smeared_vertex_operators} and successively \eqref{eq:cr_l0} to prove that $d_{b_nc}=d_c-n$ for all $n\in\Z-\half$. 
On the other hand, just note that for all $t\in\R$, 
$$
\widehat{(\chi h_t)}_n=\widehat{(h_t)}_\frac{2n-1}{2}=e^{-i\frac{2n-1}{2}t}\widehat{h}_\frac{2n-1}{2}=e^{i\frac{t}{2}}e^{-int}\widehat{g}_n
\qquad\forall n\in\Z-\half \,.
$$
It follows that the two sides of \eqref{eq:smeared_vertex_operator_f_t_odd} are equal when $c\in V$, and the general case with $c\in\mathcal{H}^\infty$ follows by Lemma \ref{lem:common_core}.
Finally, equation \eqref{eq:derivation_smeared_vo_f_t_odd} can be proved by showing, with routine functional calculus arguments, that the right-hand side of \eqref{eq:smeared_vertex_operator_f_t_odd} is differentiable in $t=0$ with derivative $-Y(b,g')c+iY(b,g)L_0c$.
\end{proof}

It is useful to note that formula \eqref{eq:smeared_vertex_operator_f_t} implies the following equality between operators: for all $a\in V_\parzero$ and all $f\in C^\infty(S^1)$, we have that
\begin{equation}   \label{eq:rot_covariance}
e^{itL_0}Y(a,f)e^{-itL_0} = Y(a,f_t)
\qquad \forall t\in\R\,.
\end{equation} 
Similarly, formula \eqref{eq:smeared_vertex_operator_f_t_odd} gives: for all $b\in V_\parone$ and all $g=\chi h$ with $h\in C^\infty(S^1)$, we have that
\begin{equation}   \label{eq:rot_covariance_odd}
e^{itL_0}Y(b,g)e^{-itL_0} = Y(b,e^{-i\frac{t}{2}}\chi  h_t)
\qquad\forall t\in\R\,.
\end{equation}

Now, by \eqref{eq:an_star}, we can easily calculate that for all homogeneous $a\in V$, all $b,c\in V$ and all $f$ either in $C^\infty(S^1)$ or in $C^\infty_\chi(S^1)$ depending on $p(a)$,
\begin{equation}  
(Y(a,f)b | c)=(b| Y(a,f)^+ c) 
\end{equation}
where we have defined
\begin{equation} \label{eq:smeared_vertex_operator_adjoint}
Y(a,f)^+
:=(-1)^{2d_a^2+d_a}\sum_{l\in\Zpluseq} \frac{Y(\theta L_1^l a, \overline{f})}{l!}
=i^{2d_a}\sum_{l\in\Zpluseq} \frac{Y(Z\theta L_1^l a, \overline{f})}{l!}
\end{equation}
which is well-defined because $L_1^la\in V_{d_a-l}$ for all $l\in\Zpluseq$ and thus the sum in \eqref{eq:smeared_vertex_operator_adjoint} is finite.
This implies that $Y(a,f)^+\subseteq Y(a,f)^*$ and thus $\mathcal{H}^\infty$ is an invariant core for the adjoints of all smeared vertex operators.

To sum up, we have proven that:
\begin{prop}   \label{prop:common_invariant_core_smeared_vertex_operators}
	Let $V$ be an energy-bounded unitary VOSA. Then $\mathcal{H}^\infty$ is a common invariant core for all smeared vertex operators and their adjoints.
\end{prop}

\subsection{Strong graded-locality and the graded-local conformal net}
\label{subsection:strong_graded_locality}

In this second part, we introduce the central concept of \textit{strong graded locality} for a simple unitary VOSA $V$ and we define the graded-local conformal net on $S^1$ associated to it. We refer to Definition \ref{defin:hilbet_space_from_V} for some notations.

For any index set $\mathcal{I}$ and any family of closed densely defined operators 
$\{A_j\mid j\in\mathcal{I} \}$ on a given Hilbert space $\mathcal{G}$, define 
\begin{gather}  
	W^*(\{A_j\mid j\in\mathcal{I} \}) 
	:=
	\bigvee_{j\in\mathcal{I}} W^*(A_j)
	\label{eq:def_von_Neumann_alg_gen_family}\\
	W^*(A_j)
	:=
	\{
	B\in B(\mathcal{G})\mid BA_j\subseteq A_jB
	\,,\,\, B^*A_j\subseteq A_jB^* 
	\}' 
	\qquad\forall j\in\mathcal{I}
	\label{eq:def_von_Neumann_alg_gen}
\end{gather}
which are called the \textbf{von Neumann algebra generated by} 
$\{A_j\mid j\in\mathcal{I} \}$ and by $A_j$ respectively.
Recall that \eqref{eq:def_von_Neumann_alg_gen} is the smallest von Neumann algebra to which the operator $A_j$ is affiliated, see e.g.\ \cite[I.7.1.3]{Bla06} and \cite[Section 4.5 and Subsection 5.2.7]{Ped89} for details, see \cite[Section 2.2]{CKLW18} for a summary, see also \cite[Appendix B.1]{Gui19I}. Two closed densely defined operators 
$A$ and $B$ on a given Hilbert space $\mathcal{G}$ \textbf{commute strongly} if $W^*(A)\subset W^*(B)'$. \label{defin:operators_commute_strongly}
For later use, recall the following functional analytic fact:

\begin{prop}   \label{prop:commutation_von_neumann_algebras_generated}
	Let $A$ and $B$ be closed densely defined operators on a Hilbert space $\mathcal{H}$ and let $\mathcal{D}$ be a common invariant domain for them. Assume that $W^*(A)\subseteq W^*(B)'$. Then $ABa=BAa$ for all $a\in\mathcal{D}$.
\end{prop}

Thanks to Proposition \ref{prop:common_invariant_core_smeared_vertex_operators}, we can give the following:

\begin{defin} \label{defin:the_net_A_V}
Let $(V,\scalar)$ be an energy-bounded unitary VOSA. Define a family of von Neumann algebras $\A_{(V,\scalar)}$ on the Hilbert space completion $\mathcal{H}:=\mathcal{H}_{(V,\scalar)}$ of $V$ by
$$ 
\A_{(V,\scalar)}(I):=W^*\left(\left\{
Y(a,f),\,\, Y(b,g) \left| \
\begin{array}{l}
	a\in V_\parzero \,,\,\, f\in C^\infty(S^1) \,,\,\,\mathrm{supp}f\subset I\\ 
	b\in V_\parone\,,\,\, g\in C^\infty_\chi(S^1) \,,\,\,\mathrm{supp}g\subset I
\end{array}
\right. \right\} \right)
\,\,\, \forall I\in\J
\,.
$$
\end{defin}

We can now prove the cyclicity of the vacuum vector $\Omega\in V$ for the family $\A_{(V,\scalar)}$, which is part of the axiom \textbf{(F)} in Section \ref{subsection:graded-local_conformal_nets}.

\begin{prop}  \label{prop:Omega_is_cyclic}
Let $(V,\scalar)$ be an energy-bounded unitary VOSA. The vacuum vector $\Omega$ of $V$ is cyclic for the von Neumann algebra
\begin{equation}
A_{(V,\scalar)}(S^1):=\bigvee_{I\in\J}\A_{(V,\scalar)}(I) \,.
\end{equation}
\end{prop}

\begin{proof}
Let $E$ be the projection from $\mathcal{H}$ onto the closed subspace $\overline{\A_{(V,\scalar)}(S^1)\Omega}\subset\mathcal{H}$. We want to prove that $EV$ is dense in $\mathcal{H}$, which implies the cyclicity of $\Omega$. 
We split the proof of this fact into the following three steps.

First, we prove that $EV\subseteq V$.
To do this, note that
$$
L_0a=Y(\nu,f)a=Y(\nu,f_1)a+Y(\nu,f_2)a \qquad\forall a\in V
$$
where $f(z)=1$ for all $z\in S^1$ and $f_j\in C^\infty(S^1,\R)$ are such that $f=f_1+f_2$ and $\mathrm{supp}f_j\subset I_j$ for some $I_j\in\J$. Note that $E\in \A_{(V,\scalar)}(S^1)'$ and that $L_0$ is a self-adjoint operator.
Hence, using \eqref{eq:smeared_vertex_operator_adjoint}, we have that
\begin{equation}
	\begin{split}
		(L_0b|Ea)
		&
		=(Y(\nu,f_1)b|Ea)+(Y(\nu,f_2)b|Ea) \\
		&
		=(Eb|Y(\nu,f_1)a)+ (Eb|Y(\nu,f_2)a) \\
		&
		=(Eb|L_0a)=(b|EL_0a)
		\qquad\forall a,b\in V \,.
	\end{split}
\end{equation}
This implies that if $a\in V$ is homogeneous, then $Ea$ is an eigenvector of $L_0$ in $\mathcal{H}$ with eigenvalue $d_a$. Then $EV\subseteq V$.

Second, we prove that $EV=E\mathcal{H}\cap V$.
On the one hand, $EV$ is contained in both $E\mathcal{H}$ and $V$ and thus in their intersection. On the other hand, if $Ea\in E\mathcal{H}\cap V$ with $a\in\mathcal{H}$, then 
$$
Ea=E^2a=E(Ea)\in EV
$$
and thus $E\mathcal{H}\cap V$ is contained in $EV$, that is, $EV=E\mathcal{H}\cap V$.

Third, we prove that $EV=V$. 
Let $a\in EV$ and $b\in V$. Suppose initially that $b$ is even. If $f\in C^\infty(S^1)$ with $\mathrm{supp}f\subset I\in\J$, then $Y(b,f)$ is affiliated with $\A_{(V,\scalar)}(I)$. This implies that 
$$
EY(b,f)a=Y(b,f)Ea=Y(b,f)a 
\,\,\,\Longrightarrow\,\,\,
Y(b,f)a\in E\mathcal{H}
$$
for all $f\in C^\infty(S^1)$ with $\mathrm{supp}f\subset I\in\J$. 
Therefore, we obtain that
$$
V\ni b_na=Y(b,f_n)a=Y(b,g_n)a+Y(b,h_n)a\in E\mathcal{H} \qquad\forall n\in\Z
$$
with $f_n(z)=z^{n}$ and $g_n, h_n\in C^\infty(S^1)$ such that 
$f_n=g_n+h_n$, $\mathrm{supp}g_n\subset I_g$ and $\mathrm{supp}h_n\subset I_h$ for some $I_g,I_h\in\J$.
Therefore,
\begin{equation}  \label{bna_in_EV}
b_na\in E\mathcal{H}\cap V= EV 
\qquad\forall b\in V_\parzero \,\,\,\forall a\in EV\,.
\end{equation}
Now, the identity operator is in $\A_{(V,\scalar)}(S^1)$ and thus $E\Omega=\Omega$, which implies that $\Omega\in EV$. Hence, we can choose $a=\Omega$ and $n=d_b$  in \eqref{bna_in_EV} to obtain that $b_{(-1)}\Omega=b\in EV$ for all $b\in V_\parzero$. In a very similar way, we can prove that $b\in EV$ for all $b\in V_\parone$, which implies that $EV=V$. This concludes the proof of the cyclicity of $\Omega$ because $V$ is dense in $\mathcal{H}$. 
\end{proof}

The first step to discuss the covariance of the family $\A_{(V,\scalar)}$ is to prove that it is generated by the quasi-primary vectors of $V$.

\begin{lem}   \label{lem:gen_by_qp}
Let $V$ be an energy-bounded unitary VOSA. Let $A$ be a bounded operator on $\mathcal{H}$ and let $I\in\mathcal{J}$. Then we have two cases: 
\begin{itemize}
\item[($\parzero$)] Let $a\in V_\parzero$. Then
$AY(a,f)\subset Y(a,f)A$ for all $f\in C^\infty(S^1)$ with $\mathrm{supp}f\subset I$ if and only if $(A^*c| Y(a,f)d)=(Y(a,f)^*c | Ad)$ for all $c,d\in V$ and all $f\in C^\infty(S^1,\R)$ with $\mathrm{supp}f\subset I$.

\item[($\parone$)] Let $b\in V_\parone$. Then
$AY(b,g)\subset Y(b,g)A$ for all $g\in C^\infty_\chi(S^1)$ with $\mathrm{supp}g\subset I$ if and only if $(A^*c| Y(b,g)d)=(Y(b,g)^*c | Ad)$ for all $c,d\in V$ and all $g\in C^\infty_\chi(S^1,\R)$ with $\mathrm{supp}g\subset I$.
\end{itemize}
\end{lem}

\begin{proof}
\begin{itemize}
\item[($\parzero$)] This is exactly the proof of \cite[Lemma 6.5]{CKLW18}.

\item[($\parone$)] We can use the same proof of \cite[Lemma 6.5]{CKLW18} with minus changes. Indeed, it is sufficient to use formula \eqref{eq:rot_covariance_odd} instead of \eqref{eq:rot_covariance} whenever the latter occurs.
\end{itemize} 
\end{proof}

\begin{prop}  \label{prop:gen_by_qp}
Let $V$ be a simple energy-bounded unitary VOSA. Let $A$ be a bounded operator on $\mathcal{H}$ and let $I\in\mathcal{J}$. Then we have that: $A\in\A_{(V,\scalar)}(I)'$ if and only if 
$$
(A^*c| Y(a,f)d)=(Y(a,f)^*c | Ad) 
\qquad\mbox{ and }\qquad
(A^*c| Y(b,g)d)=(Y(b,g)^*c | Ad)
$$  
for all quasi-primary $a\in V_\parzero$ and all quasi-primary $b\in V_\parone$, for all $c,d\in V$, for all $f\in C^\infty(S^1,\R)$ with $\mathrm{supp}f\subset I$ and all $g\in C^\infty_\chi(S^1,\R)$ with $\mathrm{supp}g\subset I$. In particular, for all $I\in \J$, $\A_{(V,\scalar)}(I)$ is equal to
$$
W^*\left(\left\{
Y(a,f) \,,\,\, Y(b,g)
\left| 
\begin{array}{l}
	a\in \bigcup_{k\in\Z} V_k \,,\,\, L_1a=0 \,,\,\, f\in C^\infty(S^1,\R) \,,\,\,\mathrm{supp}f\subset I\\
	b\in\bigcup_{k\in\Z-\half} V_k\,,\,\,L_1b=0  \,,\,\, g\in C^\infty_\chi(S^1,\R) \,,\,\,\mathrm{supp}g\subset I
\end{array}
\right. \right\} \right) \,.
$$
\end{prop}

\begin{proof}
The proof of \cite[Proposition 6.6]{CKLW18} still works in this case, provided that we use the following precautions due to the presence of odd vectors. Indeed, we have to use the generalized formula for the adjoint smeared vertex operators, which we get by \eqref{eq:smeared_vertex_operator_adjoint}, namely
\begin{equation}
Y(a,f)c
=(-1)^{2d_a^2+d_a}Y(\theta a,f)^*c
=i^{-2d_a}Y(\theta Z^*a,f)^*c
\end{equation}
for all quasi-primary $a\in V$, for all $c\in V$ and all $f$ either in $C^\infty(S^1,\R)$ or $C^\infty_\chi(S^1,\R)$ depending on $p(a)$.
\end{proof}

From now on, suppose that the energy-bounded unitary VOSA $V$ is also \textit{simple}.
The next step is the definition of the representation of $\Diff^+(S^1)^{(\infty)}$ on $\mathcal{H}$. By Theorem \ref{theo:representations_Diff_Vir}, the positive-energy unitary representation of the Virasoro algebra on $V$, which we have from the conformal vector $\nu$ of the theory, gives rise to a positive-energy strongly continuous projective unitary representation of $\Diff^+(S^1)^{(\infty)}$, which factors through $\Diff^+(S^1)^{(2)}$ because $e^{i4\pi L_0}=1_\mathcal{H}$. In particular, as well-explained in \cite[Theorem 5.2.1 and Proposition 5.2.4]{Tol99}, for all $t\in\R$, all $f\in C^\infty(S^1,\R)$ and all $A\in B(\mathcal{H})$, we have that
\begin{equation}  \label{eq:repr_U_of_the_net_A_V}
U(\exp^{(2)}(tf))=e^{itY(\nu, f)}
\,, \quad
U(\exp^{(2)}(tf)) A
U(\exp^{(2)}(tf))^*
=
e^{itY(\nu, f)} A e^{-itY(\nu, f)}
\end{equation}
where $\exp^{(2)}(tf)$ is the lift to $\mathrm{Diff^+(S^1)^{(2)}}$ of the one-parameter subgroup $\exp(tf)$ generated by the real smooth vector field $f\frac{\mathrm{d}}{\mathrm{d}x}$.
Moreover, we have that $U(\gamma)\mathcal{H}^\infty=\mathcal{H}^\infty$ for all $\gamma\in\Diff^+(S^1)^{(2)}$ and that $U(r^{(2)}(t))=e^{it L_0}$, where $r^{(2)}(t)$ is the lift to $\mathrm{Diff^+(S^1)^{(2)}}$ of the rotation subgroup, where $t\in\R$ is the angle of rotation. In particular, $U(r^{(2)}(2\pi))=e^{i2\pi L_0}$ acts on $V$ as the parity operator $\Gamma_V$ and thus $U(r^{(2)}(2\pi))=\Gamma$.

\begin{rem}  \label{rem:differentiability_Diff_repr_H^infty}
For every real smooth vector field $f\frac{\mathrm{d}}{\mathrm{d}x}$, the unitary operators $U(\exp^{(2)}(tf))=e^{itY(\nu,f)}$ define a strongly continuous one-parameter unitary group on $\mathcal{H}$. Moreover, for any $c\in\mathcal{H}^\infty$, the function $\R\ni t\mapsto e^{itY(\nu,f)}c\in\mathcal{H}^\infty$ is differentiable with derivative $e^{itY(\nu,f)}iY(\nu,f)c$, see \cite[Corollary 2.2]{Tol99}. 
\end{rem}

A further step is to define a representation of $\Diff^+(S^1)$ on $C^\infty(S^1)$ and one of $\Diff^+(S^1)^{(2)}$ on $C^\infty_\chi(S^1)$.
Let $\gamma\in\Diff^+(S^1)$ and consider the function $X_\gamma:S^1\rightarrow\R$ defined as in \cite[Eq.\ (117)]{CKLW18}, cf.\ \cite[Eq.\ (40)]{CKL08}, that is
\begin{equation} \label{eq:def_X_gamma}
X_\gamma(z):=-i\frac{\mathrm{d}}{\mathrm{d}x}\log(\gamma(e^{ix}))
\qquad
z=e^{ix}\,, \,\,\, x\in(-\pi,\pi] \,.
\end{equation}
In other terms, $X_\gamma(z)=\frac{\mathrm{d}\phi_\gamma}{\mathrm{d}x}(x)$ for all $z=e^{ix}$ with $x\in(-\pi,\pi] $,
where $\phi_\gamma\in \Diff^+(S^1)^{(\infty)}$ is any representative of the $2\pi\Z$-class of $\gamma$, that is it satisfies $\gamma(e^{ix})=e^{i\phi_\gamma(x)}$ for all $x\in(-\pi,\pi]$, see Section \ref{subsection:diff_group}. It follows that $X_\gamma(z)>0$ for all $z\in S^1$. Furthermore, it is easy to prove that $X_\gamma\in C^\infty(S^1)$ and that
\begin{equation} \label{eq:X_gamma1_gamma2}
X_{\gamma_1\gamma_2}(z)=X_{\gamma_1}(\gamma_2(z))X_{\gamma_2}(z)
\qquad
\forall \gamma_1,\gamma_2\in\Diff^+(S^1)
\,\,\, \forall z\in S^1\,.
\end{equation}
For every $d\in\half\Zplus$, we define a family of continuous operators $\{\beta_d(\gamma)\mid\gamma\in\Diff^+(S^1)\}$ on the Fréchet space $C^\infty(S^1)$ by, cf.\ \cite[Eq.\  (43)]{CKL08} and \cite[Eq.\ (119)]{CKLW18}: 
\begin{equation}  \label{eq:family_cont_ops_diff}
(\beta_d(\gamma)f)(z):= \left[X_\gamma(\gamma^{-1}(z))\right]^{d-1}f(\gamma^{-1}(z))
\qquad
\forall f\in C^\infty(S^1)
\,\,\, \forall z\in S^1\,.
\end{equation}
Moreover, for all $d\in\half\Zplus$, $\beta_d$ defines a strongly continuous representation of $\Diff^+(S^1)$ on $C^\infty(S^1)$, preserving the subspace of real-valued functions $C^\infty(S^1, \R)$. It is not difficult to prove that:

\begin{lem}  \label{lem:beta_d_differentiable}
For all $f_1\in C^\infty(S^1,\R)$ and all $f_2\in C^\infty(S^1)$, the map 
$$
\R\ni t\longmapsto\beta_d(\exp(tf_1))f_2\in C^\infty(S^1)
$$ 
is differentiable and
\begin{equation}  \label{eq:derivative_beta_d_exp}
\frac{\mathrm{d}}{\mathrm{d}t}\left[\beta_d(\exp(tf_1))f_2\right]_{t=0}
=(d-1)f_1'f_2-f_1f_2' \,.
\end{equation}
\end{lem}

To define a representation of $\Diff^+(S^1)^{(2)}$ on $C^\infty_\chi(S^1)$, consider for every $\gamma\in\Diff^+(S^1)^{(2)}$, a representative of its $4\pi\Z$-class of diffeomorphisms $\phi_\gamma\in\Diff^+(S^1)^{(\infty)}$, so that $\gamma(e^{i\frac{x}{2}})=e^{i\frac{\phi_\gamma(x)}{2}}$ for all $x\in(-2\pi,2\pi]$, see Section \ref{subsection:diff_group}. We define the function $Y_\gamma:S^1\rightarrow\C$ by
\begin{equation}
Y_\gamma(z):=e^{i\frac{\phi_\gamma(x)-x}{2}}  \,,
\qquad z=e^{ix} \,, \,\,\, x\in(-\pi,\pi]
\end{equation}
which is well-defined because it does not depend on the choice of the representative in the class of diffeomorphisms $\{\phi_\gamma+4k\pi\mid k\in\Z\}$.

\begin{lem}  \label{lem:Y_gamma_cocycle}
$Y_\gamma\in C^\infty(S^1)$ for all $\gamma\in\Diff^+(S^1)^{(2)}$ and
\begin{equation}    \label{eq:Y_gamma_cocycle}
	Y_{\gamma_1\gamma_2}(z)
	=Y_{\gamma_1}(\dot{\gamma_2}(z))Y_{\gamma_2}(z)
	\qquad\forall \gamma_1,\gamma_2\in\Diff^+(S^1)^{(2)}
	\,\,\,\forall  z\in S^1
	\,.
\end{equation}
\end{lem}

\begin{proof}
It is useful to recall that every $\phi\in\Diff^+(S^1)^{(\infty)}$ satisfies
\begin{equation} \label{eq:2pi_condition_for_phi}
\phi(x+2k\pi)=\phi(x)+2k\pi 
\qquad
\forall x\in\R \,\,\, \forall k\in\Z \,.
\end{equation}
For all $\gamma\in\Diff^+(S^1)^{(2)}$, $\widetilde{Y_\gamma}(x):=Y_\gamma(e^{ix})$ for all $x\in\R$ is clearly an infinitely differentiable function. Moreover, it is $2\pi$-periodic, indeed
\begin{equation}
\widetilde{Y_\gamma}(x+2\pi)
=e^{i\frac{\phi_\gamma(x+2\pi)-(x+2\pi)}{2}}
=e^{i\frac{\phi_\gamma(x)+2\pi-x-2\pi}{2}}
=e^{i\frac{\phi_\gamma(x)-x}{2}}
=\widetilde{Y_\gamma}(x)
\qquad\forall x\in\R
\end{equation} 
where we have used \eqref{eq:2pi_condition_for_phi} for the second equality. Therefore, $Y_\gamma\in C^\infty(S^1)$ for all $\gamma\in\Diff^+(S^1)^{(2)}$.
To prove \eqref{eq:Y_gamma_cocycle}, let $\gamma_1,\gamma_2\in\Diff^+(S^1)^{(2)}$. Write $z=e^{ix}$ with $x\in(-\pi,\pi]$ and for every such $x$, consider $k_{\gamma_2}(x)\in 2\pi\Z$ such that $\phi_{\gamma_2}(x)+k_{\gamma_2}(x)$ is in $(-\pi,\pi]$. Hence, we have that
\begin{equation}
	\begin{split}
Y_{\gamma_1\gamma_2}(z)
&= e^{i\frac{\phi_{\gamma_1}(\phi_{\gamma_2}(x))-x}{2}} 
= e^{i\frac{\phi_{\gamma_1}(\phi_{\gamma_2}(x))-\phi_{\gamma_2}(x)}{2}} 
e^{i\frac{\phi_{\gamma_2}(x)-x}{2}} \\
&= e^{i\frac{\phi_{\gamma_1}(\phi_{\gamma_2}(x)+k_{\gamma_2}(x))-(\phi_{\gamma_2}(x)+k_{\gamma_2}(x))}{2}} Y_{\gamma_2}(z)
= Y_{\gamma_1}(e^{i(\phi_{\gamma_2}(x)+k_{\gamma_2}(x))}) Y_{\gamma_2}(z) \\
&= Y_{\gamma_1}(\dot{\gamma_2}(z)) Y_{\gamma_2}(z)
\qquad\forall z\in S^1
	\end{split}
\end{equation} 
where we have used the group properties for the representative, see Section \ref{subsection:diff_group}, for the first equality and \eqref{eq:2pi_condition_for_phi} for the third one.
\end{proof}  

The following definition is inspired by \cite[Appendix A]{Boc96} and \cite[Section 1.1]{Lok94} (cf.\ also \cite[Section 2 and Section 3]{Pal}).
For all $d\in\Zplus-\half$, we define a family of continuous operators $\{\alpha_d(\gamma)\mid \gamma\in\Diff^+(S^1)^{(2)}\}$ on the Fréchet space $C^\infty_\chi(S^1)$ in the following way: let $g=\chi h$ with $h\in C^\infty(S^1)$ and define
\begin{equation} \label{eq:family_cont_ops_diff^2}
(\alpha_d(\gamma)g)(z):=
\chi(z)Y_{\gamma^{-1}}(z)(\beta_d(\dot{\gamma})h)(z)
\qquad \forall z\in S^1 \,.
\end{equation}
By Lemma \ref{lem:Y_gamma_cocycle}, $\alpha_d$ defines a strongly continuous representation of $\Diff^+(S^1)^{(2)}$ on $C^\infty_\chi(S^1)$, which preserves $C^\infty_\chi(S^1,\R)$ as it is explained in the following Remark \ref{rem:epsilon_gamma_z}.

\begin{rem}  \label{rem:epsilon_gamma_z}
Let $\gamma$ be a generic element in $\Diff^+(S^1)^{(2)}$ and write every $z\in S^1$ as $e^{ix}$ for the unique $x\in(-\pi,\pi]$. Then we have that 
\begin{equation}   \label{eq:Y_gamma_times_chi}
Y_\gamma(z)\chi(z)
=e^{i\frac{\phi_\gamma(x)}{2}}
=\gamma(e^{i\frac{x}{2}})=\epsilon_\gamma(z) \chi(\dot{\gamma}(z))
\qquad \forall z\in S^1
\end{equation}
where $\epsilon_\gamma(z)\in\{\pm1\}$ is defined as follows: for all $z\in S^1$, there exist a unique $\phi_{\dot{\gamma}}$ in $\{\phi_\gamma+2k\pi\mid k\in\Z\}$ such that $\phi_{\dot{\gamma}}(x)$ is in $(-\pi,\pi]$; then
\begin{equation}  \label{eq:def_epsilon_gamma}
\epsilon_\gamma(z):=e^{i\frac{\phi_{\dot{\gamma}}(x)-\phi_\gamma(x)}{2}}
\qquad\forall z=e^{ix}\in S^1\,.
\end{equation}
Note that definition \eqref{eq:def_epsilon_gamma} does not depend on the representative $\phi_\gamma$. Thus, we can rewrite \eqref{eq:family_cont_ops_diff^2} as
\begin{equation}  \label{eq:family_cont_ops_diff^2_epsilon}
(\alpha_d(\gamma)g)(z)=
\left[X_{\dot{\gamma}}(\dot{\gamma}^{-1}(z))\right]^{d-1}
\epsilon_{\gamma^{-1}}(z)
g(\dot{\gamma}^{-1}(z))
\qquad\forall g\in C^\infty_\chi(S^1) \,\,\,  \forall z\in S^1 
\end{equation}
which clearly assures us that $\alpha_d(\gamma)$ preserves $C_\chi^\infty(S^1,\R)$ for all $\gamma\in\Diff^+(S^1)^{(2)}$.
Finally, note that \eqref{eq:family_cont_ops_diff^2_epsilon} corresponds to the one in \cite[Eq.\  (78)]{Boc96} for $\Mob(S^1)$.
\end{rem} 

With a standard argument, we can prove the following result:

\begin{lem} \label{lem:alpha_d_differentiable}
For all $f\in C^\infty(S^1,\R)$ and all $g\in C^\infty_\chi(S^1)$, the map 
$$
\R\ni t\longmapsto\alpha_d(\exp^{(2)}(tf))g
\in C^\infty_\chi(S^1)
$$ 
is differentiable and
\begin{equation}
\frac{\mathrm{d}}{\mathrm{d}t}\left[
\alpha_d(\exp^{(2)}(tf))g\right]_{t=0}
=(d-1)f'g-fg' \,.
\end{equation}
\end{lem}

The following proposition is the key result to prove the M{\"o}bius covariance of the family of von Neumann algebras $\A_{(V,\scalar)}$.
\begin{prop}   \label{prop:mob_covariance_qp}
Let $V$ be a simple energy-bounded unitary VOSA. Then we have that: for all $I\in\J$,
\begin{itemize}
\item[($\parzero$)] if $a\in V_\parzero$ is a quasi-primary vector then $U(\gamma)Y(a,f)U(\gamma)^*=Y(a,\beta_{d_a}(\gamma)f)$ for all $f\in C^\infty(S^1)$ with $\mathrm{supp}f\subset I$ and all $\gamma \in \Mob(S^1)$.

\item[($\parone$)] if $b\in V_\parone$ is a quasi-primary vector then $U(\gamma)Y(b,g)U(\gamma)^*=Y(b,\alpha_{d_b}(\gamma)g)$ for all $g\in C_\chi^\infty(S^1)$ with $\mathrm{supp}g\subset I$ and all $\gamma\in\Mob(S^1)^{(2)}$.
\end{itemize}
\end{prop}

\begin{proof}
The proof of $(\parone)$ is an adaptation of \cite[pp.\  1100--1103]{CKL08}, similar to what is done in the proof of \cite[Proposition 6.4]{CKLW18}, which is exactly the case $(\parzero)$. 

First, from equations \eqref{eq:cr_l0}--\eqref{eq:cr_l1}, we see that for every quasi-primary vector $b\in V$ 
\begin{equation}  \label{eq:cr_lm}
[L_m,b_n]=((d_b-1)m-n)b_{m+n}  
\qquad\forall m\in\{-1,0,1\}\,.
\end{equation}
Second, note that $\Mob(S^1)$ is generated by the one-parameter groups generated by the exponential map of the three real smooth vector fields $l_m\frac{\mathrm{d}}{\mathrm{d}x}$ given by
\begin{equation}  \label{eq:fm_gen_mob}
\begin{split}
l_0(z)=1 \,, &\qquad Y(\nu,l_0)=L_0\,,  \\
l_1(z)=\frac{z+z^{-1}}{2} \,, &\qquad
Y(\nu,l_1)=\frac{L_1+L_{-1}}{2}\,,  \\
l_{-1}(z)=\frac{z-z^{-1}}{2i}  \,, &\qquad
Y(\nu,l_{-1})= \frac{L_1-L_{-1}}{2i} \,.
\end{split}
\end{equation}
In other words, $\Mob(S^1)$ is generated by $\exp(tl)$ with $l\in C^\infty(S^1,\R)$ and such that $Y(\nu, l)$ is a linear combination of $L_m$ with $m\in\{-1,0,1\}$. Consequently, $\Mob(S^1)^{(2)}$ is generated by the corresponding one-parameter subgroups $\exp^{(2)}(tl)$. 

Fix a quasi-primary odd vector $b\in V$, a function $g\in C_\chi^\infty(S^1)$ and consider any $l\in C^\infty(S^1,\R)$ such that $\exp^{(2)}(tl)\in\Mob(S^1)^{(2)}$ as explained above. 
By Proposition \ref{prop:common_invariant_core_smeared_vertex_operators}, $\mathcal{H}^\infty$ is an invariant core for all the smeared vertex operators and thus we can write, using \eqref{eq:cr_lm} and \eqref{eq:fm_gen_mob},
\begin{equation}  \label{eq:cr_Y_fm_Ya_f}
i[Y(\nu,l), Y(b,g)]c=Y(b,(d_b-1)l'g-lg')c
\qquad
\forall c\in \mathcal{H}^\infty \,.
\end{equation}
By Lemma \ref{lem:alpha_d_differentiable} and Lemma \ref{lem:common_core}, we have that the map 
$$
\R\ni t\mapsto Y(b,\alpha_{d_b}(\exp^{(2)}(tl))g)c\in\mathcal{H}^\infty
$$ 
is differentiable on $\mathcal{H}$ for all $c\in\mathcal{H}^\infty$. Moreover
\begin{equation} \label{eq:cr_Y_and_derivative}
\frac{\mathrm{d}}{\mathrm{d}t}\left[ Y(b,\alpha_{d_b}(\exp^{(2)}(tl))g)c \right]_{t=0}= i[Y(\nu,l), Y(b,g)]c
\end{equation}
by equation \eqref{eq:cr_Y_fm_Ya_f}.

For every $c\in\mathcal{H}^\infty$ define
\begin{equation}
c(t):=Y(b,\alpha_{d_b}(\exp^{(2)}(tl))g)U(\exp^{(2)}(tl))c
\end{equation}
which is a well-defined vector in $\mathcal{H}^\infty$ because this core is invariant for the smeared vertex operators and the representation $U$. Using again Lemma \ref{lem:common_core}, Lemma \ref{lem:h_infty_is_invariant} and Remark \ref{rem:differentiability_Diff_repr_H^infty}, it is possible to prove the differentiability of $t\mapsto c(t)$ on $\mathcal{H}^\infty$ and using \eqref{eq:cr_Y_and_derivative}, that
\begin{equation}
\left. \frac{\mathrm{d}}{\mathrm{d}t}c(t)\right|_{t=0}=i Y(\nu, l)c \,.
\end{equation}
This means that $c(t)$ satisfies the Cauchy problem on $\mathcal{H}^\infty$
\begin{equation}
\left\{
\begin{array}{l}
\frac{\mathrm{d}}{\mathrm{d}t}c(t)=iY(\nu,l)c(t) \\
c(0)=Y(b,g)c
\end{array} \right.
\end{equation}
whose unique solution is given by
\begin{equation}
U(\exp^{(2)}(tl))Y(b,g)c \,.
\end{equation}
It follows by the definition of $c(t)$ that
\begin{equation}
Y(b,\alpha_{d_b}(\exp^{(2)}(tl))g)U(\exp^{(2)}(tl))c
=c(t)
=U(\exp^{(2)}(tl))Y(b,g)c 
\qquad\forall t\in\R\,.
\end{equation}
Hence, by the arbitrariness of $c$ in the core $\mathcal{H}^\infty$, we can conclude that
\begin{equation}
U(\exp^{(2)}(tl))Y(b,g)U(\exp^{(2)}(-tl))
=Y(b,\alpha_{d_b}(\exp^{(2)}(tl))g) 
\end{equation}
which is the desired result as $\Mob^{(2)}(S^1)$ is generated by the subset $\{\exp^{(2)}(tl)\mid t\in\R\}$.
\end{proof}

\begin{rem}  \label{rem:diff_covariance_primary_elems}
As in \cite[Proposition 6.4]{CKLW18}, substituting quasi-primary vectors with only primary ones, we can prove the covariance properties $(\parzero)$ and $(\parone)$ in Proposition \ref{prop:mob_covariance_qp} for all $\gamma\in\Diff^+(S^1)$ and all $\gamma\in\Diff^+(S^1)^{(2)}$ respectively.
\end{rem}

\begin{prop}  \label{prop:mob_covariance_of_A_V}
Let $V$ be a simple energy-bounded unitary VOSA and $\A_{(V,\scalar)}$ be the associated family of von Neumann algebras on $S^1$. Then $\A_{(V,\scalar)}$ is an irreducible graded-local M\"{o}bius covariant net on $S^1$ except for the graded-locality, which might not hold.
\end{prop}

\begin{proof}
We need to prove that $\A_{(V,\scalar)}$ satisfies properties \textbf{(A)}--\textbf{(D)} as in Section \ref{subsection:graded-local_conformal_nets}. The isotony \textbf{(A)} is clear from Definition \ref{defin:the_net_A_V} of the net. The M{\"o}bius covariance \textbf{(B)} with respect to the representation $U$ as defined at page \pageref{eq:repr_U_of_the_net_A_V} is proved thanks to Proposition \ref{prop:mob_covariance_qp} together with the fact that $\A_{(V,\scalar)}$ is generated by quasi-primary vectors as showed by Proposition \ref{prop:gen_by_qp}. As far $U$ is a positive-energy representation, as required by \textbf{(C)}, follows by simplicity as Proposition \ref{prop:characterisation_simplicity} implies that $V$ is of CFT type. 
Moreover, $\Omega$ is a vacuum vector for $\A_{(V,\scalar)}$ as it is $U$-invariant and cyclic by Proposition \ref{prop:Omega_is_cyclic}, so proving \textbf{(D)}. Since $V$ is of CFT type, $\Omega$ must be the unique vacuum vector and thus $\A_{(V,\scalar)}$ is irreducible, concluding the proof.
\end{proof}

As Proposition \ref{prop:mob_covariance_of_A_V} above suggests, it is unknown whether the net $\A_{(V,\scalar)}$ satisfies the axiom of graded-locality \textbf{(E)} as in Section \ref{subsection:graded-local_conformal_nets}, despite the locality of vertex operators. 
Indeed, by a result of E. Nelson, see \cite[Section 10]{Nel59} and see also \cite[Section VIII.5]{RS80}, there exist
two von Neumann algebras $W^*(A)$ and $W^*(B)$, generated by two unbounded self-adjoint operators $A$ and $B$, which \textit{commute} in a common invariant core, but such that $W^*(A)$ is not a subset of $W^*(B)'$. This implies that the locality of vertex operators does not automatically assures the graded-locality of the family of von Neumann algebras $\A_{(V,\scalar)}$. 
Therefore, we need to introduce the following important definition, see Definition \ref{defin:hilbet_space_from_V} for the notation.

\begin{defin}  \label{def:strongly_graded_local}
A unitary VOSA $(V,\scalar)$ is called \textbf{strongly graded-local} if it is energy-bounded and $\A_{(V,\scalar)}(I')\subseteq Z\A_{(V,\scalar)}(I)'Z^*$ for all $I\in\J$.
(Of course, a strongly graded-local VOA is a strongly local VOA in the sense of \cite[Definition 6.7]{CKLW18}.)
\end{defin}

Therefore, assuming the strong graded locality, we have the desired irreducible graded-local conformal net:
\begin{theo}  \label{theo:diffeo_covariant_net_from_V}
Let $(V,\scalar)$ be a simple strongly graded-local unitary VOSA. Then the family $\A_{(V,\scalar)}$ is an irreducible graded-local conformal net on $S^1$.
\end{theo}

\begin{proof}
We have already proved properties \textbf{(A)-(D)} and irreducibility as in Section \ref{subsection:graded-local_conformal_nets} by Proposition \ref{prop:mob_covariance_of_A_V}. Hence, it remains to prove properties \textbf{(E)} and \textbf{(F)}. 

It is easy to prove that $\Gamma\A_{(V,\scalar)}(I)\Gamma=\A_{(V,\scalar)}(I)$ for all $I\in\J$ and thus the graded-locality \textbf{(E)} is assured by Definition \ref{def:strongly_graded_local}. 

To prove the diffeomorphism covariance \textbf{(F)} of $\A_{(V,\scalar)}$, we could proceed as explained in Remark \ref{rem:diffeo_covariance_without_strong_graded-locality}, but it would require further results such as Theorem \ref{theo:gen_by_quasi_primary} and Remark \ref{rem:diff_covariance_primary_elems}. Therefore, as a matter of convenience, we present here below a simpler proof, which instead makes use of strong graded locality.
This alternative proof is obtained by adapting the argument in the proof of \cite[Proposition 3.7 (b)]{Car04}, cf.\ also the proof of \cite[Theorem 33]{CKL08}. The adaptation is as follows. By Remark \ref{rem:diffeomorphisms_generated_by_exp} and Definition \ref{defin:the_net_A_V}, we have that, for any interval $I\in\J$, $U(\gamma)\in \A_{(V,\scalar)}(I)$ for all $\gamma\in \Diff(I)^{(2)}$. Now, fix an interval $I\in\J$ and let $\gamma\in\Diff^+(S^1)^{(2)}$ be such that $\gamma I=I$ (for the sake of readability we omit the upper dot on $\gamma$). For all $J\in \J$ containing $\overline{I}$, it is possible to find $\gamma^J\in\Diff(J)^{(2)}$ such that $\gamma^J\restriction_I=\gamma\restriction_I$ and $\gamma^{-1}\gamma^J\in \Diff(I')^{(2)}$. The latter implies that $U(\gamma^{-1}\gamma^J)\in \A_{(V,\scalar)}(I')\subseteq \A_{(V,\scalar)}(I)'$ by the argument above, the strong graded locality of $V$ and the fact that every operator $U(g)$ commutes with the operator $Z$. We then have that
\begin{equation}
	\begin{split}
		\A_{(V,\scalar)}(J)
		&\supseteq
		U(\gamma^J)\A_{(V,\scalar)}(I)U(\gamma^J)^* \\
		&=
		U(\gamma)U(\gamma^{-1}\gamma^J)
		\A_{(V,\scalar)}(I)U(\gamma^{-1}\gamma^J)^*U(\gamma)^* \\
		&=
		U(\gamma)\A_{(V,\scalar)}(I)
		U(\gamma^{-1}\gamma^J)
		U(\gamma^{-1}\gamma^J)^*U(\gamma)^* \\ 
		&=
		U(\gamma)\A_{(V,\scalar)}(I)U(\gamma)^*
		\,.
	\end{split}
\end{equation}
By the external continuity \eqref{eq:external_continuity}, we have that 
\begin{equation}
	\label{eq:proof_diffeo_covariance_one_direction}
	U(\gamma)\A_{(V,\scalar)}(I)U(\gamma)^*\subseteq 
	\A_{(V,\scalar)}(I)
	\qquad
	\forall
	\gamma\in\Diff^+(S^1)^{(2)}
	\,\,\,\mbox{ s.t. } \,\,\,
	\gamma I=I \,.
\end{equation}
Now, let $\gamma$ be an arbitrary diffeomorphism in $\Diff^+(S^1)^{(2)}$. We can always find $\widehat{\gamma}\in\Mob(S^1)^{(2)}$ such that $\widehat{\gamma}I=\gamma I$. Then we have that
\begin{equation}
	\begin{split}
		U(\gamma)\A_{(V,\scalar)}(I) U(\gamma)^*
		&=
		U(\widehat{\gamma})U(\widehat{\gamma}^{-1}\gamma) 
		\A_{(V,\scalar)}(I)
		U(\widehat{\gamma}^{-1}\gamma)^*
		U(\widehat{\gamma})^* \\
		&\subseteq
		U(\widehat{\gamma})
		\A_{(V,\scalar)}(I)
		U(\widehat{\gamma})^* \\
		&=
		\A_{(V,\scalar)}(\widehat{\gamma}I)
		=\A_{(V,\scalar)}(\gamma I)
		\,,
	\end{split}
\end{equation}
where we have used: \eqref{eq:proof_diffeo_covariance_one_direction} for the second step because $\widehat{\gamma}^{-1}\gamma I=I$; the M\"{o}bius covariance of $\A_{(V,\scalar)}$ for the second equality. By the arbitrariness of $I\in\J$, we can conclude that $\A_{(V,\scalar)}$ is diffeomorphism covariant.
\end{proof}

\begin{theo}  \label{theo:unicity_irr_graded-local_cn}
	Let $(V,\scalar_V)$ and $(W,\scalar_W)$ be two simple strongly graded-local unitary VOSAs. Then $V$ and $W$ are isomorphic VOSAs if and only if $\A_{(V,\scalar_V)}$ and $\A_{(W,\scalar_W)}$ are isomorphic graded-local conformal nets. In particular, if $\scalar$ and $\curlyscalar$ are two unitary structures on $V$, then $\A_{(V,\scalar)}$ and $\A_{(V,\curlyscalar)}$ are isomorphic. 
\end{theo}

\begin{proof}
In the following, we use $V$ and $W$ as upper and lower indices to denote the structural mathematical objects related to the VOSAs $V$ and $W$ respectively.	

Let $\varphi$ be an isomorphism between $V$ and $W$. It is not difficult to check that $\curlyscalar_W:=(\varphi^{-1}(\cdot)|\varphi^{-1}(\cdot))_V$ with PCT operator $\varphi\theta_V \varphi^{-1}$ is a unitary structure for $W$. By Proposition \ref{prop:uniqueness_unitary_structure}, there exists $h\in\Aut(W)$	such that $\curlyscalar_W=(h(\cdot)|h(\cdot))_W$. This implies that $h \varphi$ is a unitary isomorphism between $(V,\scalar_V)$ and $(W,\scalar_W)$, and $h \varphi$ uniquely extends to a unitary operator $\phi$ between $\mathcal{H}_{(V,\scalar_V)}$ and $\mathcal{H}_{(W,\scalar_W)}$. Moreover, $\phi(\Omega^V)=\Omega^W$ and $\phi L_0^V=L_0^W\phi$. The latter implies that $\phi(\mathcal{H}_{(V,\scalar_V)}^\infty)=\mathcal{H}_{(W,\scalar_W)}^\infty$ and thus $\phi Y_V(a,f)\phi^{-1}=Y_W(\phi (a),f)$ for all $a\in V_\parzero\cup V_\parone$ and all $f$ in either $C^\infty(S^1)$ or $C_\chi^\infty(S^1)$ depending on $p(a)$. This implies that $\phi\A_{(V,\scalar_V)}(I)\phi^{-1}=\A_{(W,\scalar_W)}(I)$ for all $I\in\J$. To sum up, $\phi$ realizes an isomorphism between $\A_{(V,\scalar_V)}$ and $\A_{(W,\scalar_W)}$. In particular, we have that $\A_{(V,\scalar)}$ and $\A_{(V,\curlyscalar)}$ are isomorphic.
	
Vice versa, let $\phi:\mathcal{H}_{(V,\scalar_V)}\to \mathcal{H}_{(W,\scalar_W)}$ be an isomorphism between $\A_{(V,\scalar_V)}$ and $\A_{(W,\scalar_W)}$. Then $\phi(\Omega^V)=\Omega^W$ and as a consequence of \eqref{eq:uniqueness_mob_symmetries}, $\phi L_n^V=L_n^W\phi$ for all $n\in\{-1,0,1\}$. It follows that, $\phi(V)=\phi(\mathcal{H}_{(V,\scalar_V)}^{\textrm{fin}})=\mathcal{H}_{(W,\scalar_W)}^{\textrm{fin}}=W$, where $\mathcal{H}_{(V,\scalar_V)}^{\textrm{fin}}$ and $\mathcal{H}_{(W,\scalar_W)}^{\textrm{fin}}$ denote the subspaces of $\mathcal{H}_{(V,\scalar_V)}$ and $\mathcal{H}_{(W,\scalar_W)}$ respectively of finite energy vectors, i.e., linear combinations of eigenvectors of $L_0^V$ and $L_0^W$ respectively.
We also have that 
	$$
	\phi Y_V(a,z)\phi^{-1}(\Omega^W)=\phi Y_V(a,z)\Omega^V=\phi e^{zL_{-1}^V}a
	=e^{zL_{-1}^W}\phi(a)
	\qquad\forall a\in V \,.
	$$
Now, note that for any $a\in V$, the formal series $\phi Y_V(a,z)\phi^{-1}$ is a field on $W$. Moreover, for any $b\in W$, $\phi Y_V(a,z)\phi^{-1}$ and $Y_W(b,z)$ are mutually local in the Wightman sense thanks to the graded locality of $\A_{(V,\scalar_V)}$ and of $\A_{(W,\scalar_W)}$ with Proposition \ref{prop:commutation_von_neumann_algebras_generated}.
Thus, they are mutually local in the vertex superalgebra sense thanks to Proposition \ref{prop:wightman_locality}. By the uniqueness theorem for vertex superalgebras \cite[Theorem 4.4]{Kac01}, it follows that $\phi Y_V(a,z)\phi^{-1}=Y_W(\phi(a),z)$ for all $a\in V$, that is $\phi$ respects the $(n)$-product. 
As a consequence, we have that $\phi(\nu^V)$ is a conformal vector for $W$ and since $\phi(V_n)=W_n$ for all $n\in\half\Zpluseq$, we have that $(\phi(\nu^V))_{(1)}=\nu^W_{(1)}$. Moreover, $(\phi(\cdot),\phi(\cdot))_W$ is a non-degenerate invariant bilinear form on $W$ with respect to $\phi(\nu^V)$. By Proposition \ref{prop:conformal_vector_bilinear_form},  $\phi(\nu^V)=\nu^W$ and thus $\phi$ restricts to a VOSA isomorphism between $V$ and $W$. (Note also that as $\phi$ is unitary, it restricts to a unitary VOSA isomorphism.)
\end{proof}

\begin{defin}
$\A_V$ denotes the unique, up to isomorphism, irreducible graded-local conformal net arising from a simple strongly graded-local unitary VOSA $V$ as given by Theorem \ref{theo:diffeo_covariant_net_from_V} and Theorem \ref{theo:unicity_irr_graded-local_cn}.
\end{defin}

\begin{rem}
	If $(V,\scalar)$ is a (not necessarily simple) unitary VOSA, then it is a direct sum of unitary VOSAs of CFT type $\{(V^j,\scalar_j)\}_{j=1}^N$ by (i) of Proposition \ref{prop:decomposing_unitary_VOSAs}. Moreover, if $V$ is also strongly graded-local, then $\A_{(V,\scalar)}$ is the direct sum $\bigoplus_{j=1}^N\A_{V^j}$ of irreducible graded-local conformal nets and it is also independent, up to isomorphism, of the choice of the scalar product $\scalar$, cf.\ also \cite[Lemma 2.1]{KL04}.
\end{rem}

We also get that:
\begin{theo}  \label{theo:AutV_AutNet}
Let $V$ be a simple strongly graded-local unitary VOSA. Then $\Aut(\A_V)=\Aut_{\scalar}(V)$. If $\Aut(V)$ is compact, then $\Aut(\A_V)=\Aut_{\scalar}(V)=\Aut(V)$.
\end{theo}

\begin{proof}
The proof of \cite[Theorem 6.9]{CKLW18} can be used with the following prescriptions. We have to replace \cite[Proposition A.1, Corollary 4.11 and Theorem 5.21]{CKLW18} there
with Proposition \ref{prop:wightman_locality}, Corollary \ref{cor:when_vosa_automorphism_preservs_nu} and Theorem \ref{theo:characterization_aut_group} respectively, whenever the former occur, cf.\ the proof of Theorem \ref{theo:unicity_irr_graded-local_cn}.
\end{proof}

We end the current section with the following conjecture:
\begin{conj}  
	\label{conj:every_vosa_is_stongly_graded-local}
	Every simple unitary VOSA $V$ is strongly graded-local and thus it gives rise to a unique, up to isomorphism, irreducible graded-local conformal net $\A_V$.
\end{conj}

\section{A Bisognano-Wichmann property for smeared vertex operators}
\label{appendix:action_dilation_subgroup}

This whole section is dedicated to the proof of Theorem \ref{theo:delta_one_half}, which is crucial in the proof of Theorem \ref{theo:gen_by_quasi_primary} and in the development of the theory in Section \ref{section:back}. 

\begin{notat} \label{notations_conventions}
We identify the spaces of square-integrable functions $L^2(S^1)$ and $L^2([-\pi,\pi])$ through the isometric isomorphism $f(z)\mapsto f^\vee(x):=f(e^{ix})$ with $z=e^{ix}$ for some $x\in(-\pi,\pi]$.
With an abuse of notation, we still use $f$ in place of $f^\vee$. 
The convention for the convolution product $\ast$ for $L^2([-\pi,\pi])$-functions is
\begin{equation}
	(f\ast g)(x):=\frac{1}{2\pi}\int_{-\pi}^{\pi} f(y)g(x-y)\,\mathrm{d}y
	\qquad \forall x\in[-\pi,\pi] 
\end{equation} 
and thus for the Fourier coefficients as in \eqref{eq:fourier_coeff_S1}, we have that
\begin{equation}  \label{eq:convolution_fourier}
	\widehat{(f\ast g)}_n=\widehat{f}_n\widehat{g}_n 
	\qquad\forall n\in\frac12\Z \,.
\end{equation}
The convention for the Fourier transform for $L^2(\R)$-functions is
 \begin{equation}
 	\widehat{f}(p):=\frac{1}{\sqrt{2\pi}} \int_\R 
 	f(x)e^{-ipx}\,\mathrm{d}x  
 	\qquad \forall p\in\R \,.
 \end{equation}
where $\int_\R$ stands for $\int_{-\infty}^{+\infty}$.
Finally, $\norm{\cdot}_1$ will denote the usual norm on $L^1(\R)$.
\end{notat}

Preliminarily, we prove some facts about the relationship between the real and the complex picture.
To do that, recall the notation in \eqref{eq:test_funct_spaces_chi} and consider the dense subspace $C_c^\infty(\pSone)$ of $C^\infty(S^1)\cap C_\chi^\infty(S^1)=C^\infty(\pSone)$ of complex-valued functions with compact support in the punctured circle $\pSone$.
Moreover, let $C_c^\infty(\R)$ be the space of infinitely differentiable complex-valued functions with compact support in $\R$ respectively.
Then we define isomorphisms between $C_c^\infty(\pSone)$ and $C_c^\infty(\R)$ through the Cayley transform \eqref{eq:cayley_transform}, that is the diffeomorphism $C:S^1\backslash\left\{-1\right\}\longrightarrow\R$ defined by
$$
C(z):=2i\frac{1-z}{1+z} \,, \qquad C^{-1}(x)=\frac{1+\frac{i}{2}x}{1-\frac{i}{2}x}  \,.
$$
Then for all $d\in\half\Z$, there is an isomorphism given by:
\begin{equation}
	\begin{split}
		C_c^\infty(S^1\backslash\left\{-1\right\})  
		&\xleftrightarrow{\hspace{1cm}}
		C_c^\infty(\R) \\
		h(z) 
		&\xrightarrow{\hspace{1cm}}
		h^\R(x):=\left(1+\frac{x^2}{4}\right)^{d-1}h\left( C^{-1}(x)\right) \\
		\left(\frac{(1+z)^2}{4z}\right)^{d-1}h\left(C(z)\right) 
		=: h^\C(z)
		&\xleftarrow{\hspace{1cm}}
		h(x) \,.
	\end{split}
\end{equation}

\begin{rem}   \label{rem:nets_on_R_and_cover}
	We highlight that $C^\infty_c(\pSone)\cong C^\infty_c(\R)$ would have been a perfectly fine choice as test function space to define smeared vertex operators as in \eqref{eq:def_smeared_vertex_operators}. Accordingly, following the argument in Section \ref{section:construction_net}, we would have obtained a family of von Neumann algebras on $\R$ (instead of $S^1$) associated to every simple strongly graded-local unitary VOSA $V$. Of course, this family turns out to be the \textit{restriction} of $\A_V$ to a net on $\R$ in the sense of \cite[Section 3.1]{CKL08}. Similarly, choosing the space of infinite differentiable complex-valued functions on the double cover $S^{1(2)}$ as test function space to define smeared vertex operators, it would have brought to the \textit{promotion} of $\A_V$ to a net on $S^{(1)2}$ in the sense of \cite[Section 3.2]{CKL08}. See \cite[Section 3.3]{Gau21} for details.
\end{rem}

Therefore, we have the following result.

\begin{prop}  \label{prop:from_complex_to_real}
	Let $(V,\scalar)$ be an energy-bounded unitary VOSA. Let $a$ and $b$ be  quasi-primary vectors in $V$ and $f,g\in C_c^\infty(S^1\backslash\left\{-1\right\})$. We have the following formula for the ``two-point function'':
	\begin{equation}  \label{eq:2-point_funct}
		(Y(a,f)\Omega| Y(b,g)\Omega) = 
		\frac{(a| b)\delta_{d_a,d_b}}{2\pi(2d_a-1)!}
		\int_0^{+\infty} \overline{\widehat{f^\R}(-p)}\widehat{g^\R}(-p)p^{2d_a-1}\mathrm{d}p \,.
	\end{equation}
\end{prop}

\begin{proof}[Proof of Proposition \ref{prop:from_complex_to_real}]
	
	The proof is given by the following sequence of equalities \textbf{[A]}-\textbf{[D]}, which we prove separately below:
	\begin{eqnarray}
		(Y(a,f)\Omega | Y(b,g)\Omega) 
		&\overset{\textbf{[A]}}{=}& 
		(a| b)\delta_{d_a,d_b}\sum_{\substack{n\in\Z-d_a \\ n\leq -d_a}}
		\left(\begin{array}{c}
			d_a-n-1\\ -n-d_a
		\end{array}\right) \overline{\widehat{f}_n} \widehat{g}_n \\
		&\overset{\textbf{[B]}}{=}&
		(a| b)\delta_{d_a,d_b}\lim_{\epsilon\rightarrow 0^+}\oint_{S^1}\oint_{S^1}\frac{z^{d_a}w^{d_a}}{(z-(1-\epsilon)w)^{2d_a}}g(w)\overline{f(z)}\frac{\mathrm{d}w}{2\pi i w}\frac{\mathrm{d}z}{2\pi i z} \\
		&\overset{\textbf{[C]}}{=}& 
		\frac{(a| b)\delta_{d_a,d_b}}{(2\pi)^2i^{2d_a}}\lim_{\epsilon\rightarrow 0^+}\int_\R\int_\R\frac{g^\R(y)\overline{f^\R(x)}}{[(1-\frac{\epsilon}{2}-i\frac{\epsilon}{4}y)x-y(1-\frac{\epsilon}{2})-i\epsilon]^{2d_a}}\,\mathrm{d}y\,\mathrm{d}x \\
		&\overset{\textbf{[D]}}{=}& 
		\frac{(a| b)\delta_{d_a,d_b}}{2\pi(2d_a-1)!}\int_0^{+\infty}\overline{\widehat{f^\R}(-p)}\widehat{g^\R}(-p)p^{2d_a-1}\,\mathrm{d}p \,.
	\end{eqnarray}
	
	\textit{Proof of} \textbf{[A]}. From \cite[Eq.\  (4.1.2)]{Kac01}, we have that $Y(a,z)\Omega=e^{zL_{-1}}a$ for all $a\in V$, which means that
	$$
	\sum_{n\in\Z-d_a}a_n\Omega z^{-n-d_a}=\sum_{l=0}^{+\infty}\frac{(L_{-1})^la}{l!}z^l= \sum_{\substack{n\in\Z-d_a \\ n\leq -d_a}} \frac{(L_{-1})^{-n-d_a}a}{(-n-d_a)!}z^{-n-d_a} \,.
	$$
	This implies that for any $a\in V$, we have that 
	$$
	a^{-n-d_a}:=a_n\Omega=\left\{\begin{array}{lr}
		\frac{(L_{-1})^{-n-d_a}a}{(-n-d_a)!} &  n\leq-d_a \\
		0 & n>-d_a
	\end{array}\right. \,.
	$$
	It follows that
	\begin{equation}
		\begin{split}
		(Y(a,f)\Omega | Y(b,g)\Omega) &=
		(\sum_{n\in\Z-d_a}\widehat{f}_na_n\Omega | \sum_{m\in\Z-d_b}\widehat{g}_mb_m\Omega) \\
		&= \sum_{n\in\Z-d_a}\sum_{m\in\Z-d_b} \overline{\widehat{f}_n}\widehat{g}_m (a_n\Omega | b_m\Omega) \\
		&= \sum_{\substack{n\in\Z-d_a \\ n\leq -d_a}} \sum_{\substack{m\in\Z-d_b \\ m\leq -d_b}} \overline{\widehat{f}_n}\widehat{g}_m (a^{-n-d_a}| b^{-m-d_b}) \\
		&= (a| b)\delta_{d_a,d_b}\sum_{\substack{n\in\Z-d_a \\ n\leq -d_a}} \left(\begin{array}{c}
			d_a-n-1\\ -n-d_a
		\end{array}\right) \overline{\widehat{{f}}_n} \widehat{g}_n
		\end{split}
	\end{equation} 
	where we have used \eqref{eq:L_n_adj} and the formula $L_1c^{-m-d_c}=(d_c-m-1)c^{-m-d_c-1}$ for all quasi-primary $c\in V$ and all $m\in\Z-d_c$, which can be easily deduced by an induction argument, for the last equality.
	
	\textit{Proof of} \textbf{[B]}.
	For all $1>\epsilon>0$ and $d_a\in\half\Zplus$, consider the integral 
	\begin{equation} \label{integral_[B]}
		\oint_{S^1}\oint_{S^1}\frac{z^{d_a}w^{d_a}}{(z-(1-\epsilon)w)^{2d_a}}g(w)\overline{f(z)} \, \frac{\mathrm{d}w}{2\pi i w}\frac{\mathrm{d}z}{2\pi i z} \,.
	\end{equation}
	where, for $d_a\in\Zplus-\half$, the integrand has been extended to zero to $S^1\times S^1$ thanks to the smoothness and the compact support of $f$ and $g$.
	By the change of variables $e^{ix}:=z$ and $e^{iy}:=w$ with $x,y\in[-\pi,\pi]$, we rewrite \eqref{integral_[B]} as
	\begin{equation}   \label{int_[B]_change_var}
		\int_{-\pi}^\pi\int_{-\pi}^\pi\frac{e^{i(x-y)d_a}}{(e^{i(x-y)}-(1-\epsilon))^{2d_a}}g(e^{iy})\overline{f(e^{ix})} \, \frac{\mathrm{d}y}{2\pi}\frac{\mathrm{d}x}{2\pi} \,.
	\end{equation}
	For all $1>\epsilon>0$ and $d_a\in\half\Zplus$, define the following functions in $L^2(S^1)\cong L^2([-\pi,\pi])$
	$$
	h_{\epsilon,d_a}(z):=\frac{z^{d_a}}{(z-(1-\epsilon))^{2d_a}} \,.
	$$
	Recall that $\ast$ denotes the convolution between functions in $L^2([-\pi,\pi])$. Then we can rewrite the integral \eqref{int_[B]_change_var} as
	\begin{equation} \label{int_[B]_parseval}
		\int_{-\pi}^\pi (g\ast h_{\epsilon,d_a})(x)\overline{f(x)} \, \frac{\mathrm{d}x}{2\pi} 
		= \sum_{n\in\Z-d_a} \widehat{(g\ast h_{\epsilon,d_a})}_n\overline{\widehat{f}_n} 
		= \sum_{n\in\Z-d_a} \widehat{g}_n \widehat{(h_{\epsilon,d_a})}_n\overline{\widehat{f}_n} 
	\end{equation}
	where we have used Parseval's Theorem \cite[(6) of Section 4.26]{Rud87} first and the property \eqref{eq:convolution_fourier} of the convolution after. 
	It remains to calculate the Fourier coefficients of the function $h_{\epsilon,d_a}$, which are
	\begin{equation}   \label{eq:fourier_coeff_h_epsilon_d_a}
		\begin{split}  
			\widehat{(h_{\epsilon,d_a})}_n 
			&= \oint_{S^1}
			 \frac{z^{d_a-n-1}}{(z-(1-\epsilon))^{2d_a}}\frac{\mathrm{d}z}{2\pi i} \\
			&=
			\left\{\begin{array}{lr}
				\mathrm{Res}(h_{\epsilon,d_a} z^{-n-1},1-\epsilon) & n<d_a \\
				\mathrm{Res}(h_{\epsilon,d_a} z^{-n-1},0)+\mathrm{Res}(h_{\epsilon,d_a} z^{-n-1},1-\epsilon) & n\geq d_a
			\end{array}\right.
		\quad \forall n\in\Z-d_a \,.
		\end{split}
	\end{equation}
	Next, we calculate that
	\begin{equation}  \label{res_1-epsilon}
		\begin{split}
			\mathrm{Res}(h_{\epsilon,d_a} z^{-n-1} &, 1-\epsilon) 
			= 
			\frac{1}{(2d_a-1)!}\lim_{z\rightarrow 1-\epsilon}\left[\frac{\mathrm{d}^{2d_a-1}}{\mathrm{d}z^{2d_a-1}} \left(z^{d_a-n-1}\right)\right] \\
			&= 
			\left\{\begin{array}{lr}
				\left(\begin{array}{c}
					d_a-n-1 \\ -n-d_a
				\end{array}\right)
				(1-\epsilon)^{-n-d_a} & n\leq-d_a \\
				0 & -d_a<n<d_a \\
				\left(\begin{array}{c}
					d_a+n-1 \\ n-d_a 
				\end{array}\right) 
				(-1)^{2d_a-1}(1-\epsilon)^{-n-d_a} & n\geq d_a 
			\end{array}\right.  
		\end{split}
	\end{equation}
	whereas, for $n\geq d_a$, we have that
	\begin{equation}   \label{res_0}
		\begin{split}
			\mathrm{Res}(h_{\epsilon,d_a} z^{-n-1}, 0) 
			&= 
			\frac{1}{(n-d_a)!}\lim_{z\rightarrow 0}\left[\frac{\mathrm{d}^{n-d_a}}{\mathrm{d}z^{n-d_a}} \left(\frac{1}{(z-(1-\epsilon))^{2d_a}}\right)\right] \\
			&= (-1)^{2d_a}(1-\epsilon)^{-n-d_a}
			\left(\begin{array}{c}
				n+d_a-1 \\ n-d_a 
			\end{array}\right).
		\end{split}
	\end{equation}
	From \eqref{res_1-epsilon} and \eqref{res_0}, we obtain that
	\begin{equation}   \label{int_[B]_h_epsilon}
		\widehat{(h_{\epsilon,d_a})}_n=
		\left\{\begin{array}{lr}
			\left(\begin{array}{c}
				d_a-n-1 \\ -n-d_a
			\end{array}\right)
			(1-\epsilon)^{-n-d_a} & n\leq-d_a \\
			0  & n> -d_a 
		\end{array}\right.
	\end{equation}
	Putting \eqref{int_[B]_h_epsilon} in \eqref{int_[B]_parseval}, we have that the right hand side of equality \textbf{[B]} is equal to
	\begin{equation}  \label{int_[B]_sum}
		(a| b)\delta_{d_a,d_b}\lim_{\epsilon\rightarrow 0^+}
		\sum_{\substack{n\in\Z-d_a \\ n\leq -d_a}}(1-\epsilon)^{-n-d_a}
		\left(\begin{array}{c}
			d_a-n-1 \\ -n-d_a 
		\end{array}\right) \overline{\widehat{f}_n}\widehat{g}_n \,.
	\end{equation}
	Therefore, noting that the binomial product in \eqref{int_[B]_sum} is a polynomial in the variable $n$ with degree $2d_a-1$ and that $\widehat{f}_n$ and $\widehat{g}_n$ are rapidly decaying, we can swap the limit with the series in \eqref{int_[B]_sum} to get the proof of equality \textbf{[B]}.
	
	\textit{Proof of} \textbf{[C]}. Just apply the change of variables given by the Cayley transform, that is $x:=C(z)$ and $y:=C(w)$.
	
	\textit{Proof of} \textbf{[D]}. Fix $d_a\in\half\Zplus$. For all $0<\epsilon<1$ and all $y\in\R$, set the $L^2(\R)$-function
	$$
	q_{\epsilon,y}(x):=\frac{
	\beta_\epsilon(y)}{[x-\alpha_\epsilon(y)]^{2d_a}}
	\qquad\forall x\in\R
	$$
	where
	$$
	\alpha_\epsilon(y):=\frac{y(1-\frac{\epsilon}{2})+i\epsilon}{1-\frac{\epsilon}{2}-i\frac{\epsilon}{4}y} 
	\qquad\mbox{ and }\qquad
	\beta_\epsilon(y):=\left(1-\frac{\epsilon}{2}-i\frac{\epsilon}{4}y\right)^{-2d_a}\,.
	$$
	Now we can use the Plancherel's Theorem, see \cite[Theorem IX.6]{RS75}, so that
	\begin{equation}  \label{[D]:plancherel_thrm}
		\int_\R q_{\epsilon,y}(x)\overline{f^\R(x)} \, \mathrm{d}x
		= \int_\R \widehat{q_{\epsilon,y}}(p)\overline{\widehat{f^\R}(p)} \,\mathrm{d}p \,.
	\end{equation}
	To calculate $\widehat{q_{\epsilon,y}}$, note that $\mathrm{Im}(\alpha_\epsilon(y))>0$ for all $0<\epsilon<1$ and all $y\in\R$. Thus, we can use the well-known procedure by Jordan's Lemma and Residue Theorem (see e.g.\ \cite[Lemma 4.2.2]{ AF03}), to obtain that
	\begin{equation} \label{[D]:res_q}
		\widehat{q_{\epsilon,y}}(p)=
		\left\{\begin{array}{lr}
			\sqrt{2\pi}i\mathrm{Res}(q_{\epsilon,y}e^{-ipz},\alpha_\epsilon(y)) & p<0 \\
			0 & p>0
		\end{array}\right.
	\end{equation}
	where we have that for all $p<0$,
	\begin{equation}  \label{[D]:res_alpha_epsilon_y}
		\begin{split}
			\mathrm{Res}(q_{\epsilon,y}e^{-ipz},\alpha_\epsilon(y))
			&=
			\frac{\beta_\epsilon(y)}{(2d_a-1)!}\lim_{z\rightarrow\alpha_\epsilon(y)}\left[
			\frac{\mathrm{d}^{2d_a-1}}{\mathrm{d}z^{2d_a-1}}(e^{-ipz}) \right] \\
			&=
			\beta_\epsilon(y)
			\frac{(-ip)^{2d_a-1}e^{-ip\alpha_\epsilon(y)}}{(2d_a-1)!} \,.
		\end{split}
	\end{equation}
	Therefore, using \eqref{[D]:res_q} and \eqref{[D]:res_alpha_epsilon_y} in \eqref{[D]:plancherel_thrm}, we have that
	$$
	\int_\R q_{\epsilon,y}(x)\overline{f^\R(x)} \, \mathrm{d}x
	= \frac{(-i)^{2d_a+2}\sqrt{2\pi}}{(2d_a-1)!}
	\int_{-\infty}^0 \beta_\epsilon(y)
	e^{-ip\alpha_\epsilon(y)}\overline{\widehat{f^\R}(p)}p^{2d_a-1} \,\mathrm{d}p \,.
	$$
	Thus, the left hand side of \textbf{[D]} is equal to
	\begin{equation} \label{[D]:int_plancherel_thrm}
		\frac{(a| b)\delta_{d_a,d_b}(-1)^{2d_a+1}}{(2d_a-1)!(2\pi)^{3/2}}
		\lim_{\epsilon\rightarrow 0^+}
		\int_\R\int_{-\infty}^0  
		e^{-ip\alpha_\epsilon(y)}p^{2d_a-1} \overline{\widehat{f^\R}(p)}\,\mathrm{d}p \, \beta_\epsilon(y) g^\R(y)\,\mathrm{d}y \,.
	\end{equation}
	We can swap the limit with the integrals in \eqref{[D]:int_plancherel_thrm} because $f^\R, g^\R$ have compact support in $\R$ and $\abs{e^{-ip\alpha_\epsilon(y)}}\leq 1$ for all $0<\epsilon<1$, all $y\in\R$ and all $p\in(-\infty, 0)$. Moreover, for all $y\in\R$, $\alpha_\epsilon(y)\rightarrow y$ and $\beta_\epsilon(y)\rightarrow 1$ as $\epsilon\rightarrow 0^+$. Therefore, \eqref{[D]:int_plancherel_thrm} is equal to
	\begin{equation} 
		\frac{(a| b)\delta_{d_a,d_b}(-1)^{2d_a+1}}{2\pi(2d_a-1)!}
		\int_{-\infty}^0 
		\overline{\widehat{f^\R}(p)} \widehat{g^\R}(p) p^{2d_a-1}\,\mathrm{d}p   
	\end{equation}
	which completes the proof after exchanging the variable $p$ with $-p$.
\end{proof}

For all $d\in\half\Zplus$ and all $\gamma\in\mathrm{M\ddot{o}b}(S^1)$, we rewrite the map in \eqref{eq:family_cont_ops_diff} as:
\begin{equation}   \label{eq:def_action_mob_function}
	(\beta_d(\gamma)f)(z)=
	\left[\frac{\mathrm{d}\gamma}{\mathrm{d}z}(\gamma^{-1}(z))\frac{\gamma^{-1}(z)}{z}\right]^{d-1}f(\gamma^{-1}(z))
	\qquad \forall f\in C^\infty(S^1) \,.
\end{equation}  
Note that every $\beta_d(\gamma)$ preserves the subspace $C_c^\infty(\pSone)$.
Moreover, if for all $d\in\half\Zplus$, $\alpha_d$ is as in \eqref{eq:family_cont_ops_diff^2}, then $\alpha_d(\gamma)(f)=\beta_d(\dot{\gamma})(f)$ for all $\gamma\in\Mob(S^1)^{(2)}$ and all $f\in C^\infty_c(\pSone)$, cf.\ Remark \ref{rem:nets_on_R_and_cover}.
Now, a dilation $\delta(\lambda)\in\mathrm{M\ddot{o}b}(S^1)$ of parameter $\lambda\in\R$, see \eqref{eq:dilation_action_circle}, is given by the formulae: for all $z\in S^1$,
\begin{equation} \label{eq:def_dilatation}
	\delta(\lambda)(z):=\frac{z\cosh(\lambda/2)-\sinh(\lambda/2)}{-z\sinh(\lambda/2)+\cosh(\lambda/2)}
	\quad\mbox{ and }\quad
	\delta(\lambda)^{-1}(z)=\frac{z\cosh(\lambda/2)+\sinh(\lambda/2)}{z\sinh(\lambda/2)+\cosh(\lambda/2)}  \,.
\end{equation}
Note that every dilation $\delta(\lambda)$ preservers the point $-1\in S^1$ and thus it preserves $\pSone$ too.
Moreover, a straightforward calculation give us that
\begin{equation}
	(\beta_d(\delta(\lambda))f)(z)=
	\left[\frac{1+z^2}{2z}\sinh(\lambda)+\cosh(\lambda)\right]^{d-1}f(\delta(\lambda)^{-1}(z))
\end{equation}
for all $d\in\Zplus$, all $\lambda\in\R$, all $f\in C^\infty(S^1)$ and all $z\in S^1$.
Thus, we can state the following:

\begin{cor}    \label{cor:2-point_funct_dilatation}
	Let $(V,\scalar)$ be an energy-bounded unitary VOSA. Let $a$ and $b$ be quasi-primary vectors in $V$ and $f,g\in C_c^\infty(S^1\backslash\left\{-1\right\})$. Then for all $\lambda\in\R$,
	\begin{equation}  
		(Y(a,\beta_{d_a}(\delta(\lambda))f)\Omega| Y(b,g)\Omega)= 
		\frac{(a| b)\delta_{d_a,d_b}}{2\pi(2d_a-1)!}
		\int_0^{+\infty} \overline{e^{\lambda d_a}\widehat{f^\R}(-e^\lambda p)}\widehat{g^\R}(-p)p^{2d_a-1}\mathrm{d}p  \,.
	\end{equation}
\end{cor}

\begin{proof}
	A straightforward calculation gives that
	$$
	\left(\beta_{d_a}(\delta(\lambda))f\right)^\R(x)=e^{\lambda(d_a-1)}f^\R(e^{-\lambda}x)
	\qquad
	\lambda\in\R\,\,\,
	\forall x\in\R \,.
	$$
	Thus, the result follows by Proposition \ref{prop:from_complex_to_real} and observing that
	$$
	\widehat{\left[\left(\beta_{d_a}(\delta(\lambda))f\right)^\R\right]}(p)=e^{\lambda(d_a-1)}\widehat{\left[f^\R(e^{-\lambda}\,\cdot\,)\right]}(p)=e^{\lambda d_a}\widehat{f^\R}(e^\lambda p) 
	\qquad
	\lambda\in\R\,\,\,
	\forall p\in\R \,.
	$$
\end{proof}

Now, we are ready to formulate the main theorem of this section:

\begin{theo}   \label{theo:delta_one_half}
	Let $a$ be a quasi-primary vector of a simple energy-bounded unitary VOSA $V$. Define the self-adjoint operator $K:=i\pi\overline{(L_1-L_{-1})}$ and let $f\in C_c^\infty(S^1 \backslash \left\{-1\right\})$ with $\mathrm{supp}f\subset S^1_+$. Then $Y(a,f)\Omega$ is in the domain of the operator $e^\frac{K}{2}$ and 
	\begin{equation}
		e^\frac{K}{2}Y(a,f)\Omega=i^{2d_a}Y(a,f\circ j)\Omega
	\end{equation}
	where $j(z)=\overline{z}=z^{-1}$ for all $z\in S^1$. 
\end{theo}

Until the end of the current section, $a$ and $b$ are two fixed quasi-primary vectors in $V$ and we use the notation as in Definition \ref{defin:hilbet_space_from_V}.
Furthermore, let $U$ be the positive-energy strongly continuous unitary representation of $\Mob(S^1)^{(2)}$ on $\mathcal{H}$ induced by the conformal vector $\nu$ of $V$ as in p.\  \pageref{eq:repr_U_of_the_net_A_V}. 
Then see Section \ref{subsection:diff_group} for notation, we have that 
\begin{equation} \label{eq:K_and_U}
	e^{itK}=U(\delta^{(2)}(-2\pi t)) 
	\qquad\forall t\in\R 
	\,.
\end{equation} 
Since $\beta_d$ and $\alpha_d$ coincide on $C^\infty_c(\pSone)$, we obtain by Proposition \ref{prop:mob_covariance_qp} and Corollary \ref{cor:2-point_funct_dilatation}, that for all $f,g\in C_c^\infty(S^1\backslash\left\{-1\right\})$
\begin{equation}  \label{eq:K_and_U_scalar_prod}
	\begin{split}
		(Y(b,g)\Omega| e^{itK}Y(a,f)\Omega )
		&=(Y(b,g)\Omega| U(\delta^{(2)}(-2\pi t))Y(a,f)\Omega) \\
		&=C_{b,a}
		\int_0^{+\infty} 
		\overline{\widehat{g^\R}(-p)}
		e^{-2\pi td_a}\widehat{f^\R}(-e^{-2\pi t} p)
		p^{2d_a-1}\mathrm{d}p 
	\end{split}
\end{equation}
for all $t\in\R$ and $C_{b,a}:=\frac{(b| a)\delta_{d_a,d_b}}{2\pi(2d_a-1)!}$. Our aim is to extend \eqref{eq:K_and_U_scalar_prod} by analyticity to some domain in $\C$ to get the desired expression for $e^\frac{K}{2}Y(a,f)\Omega$ by a limit procedure.
To further simplify the notation, for all $x\in V$ set the function
\begin{equation}
	\phi_x(\cdot)\Omega: C_c^\infty(\R)\ni h \xmapsto{\hspace{1.5cm}}
	\phi_x(h)\Omega:=Y(x,h^\C\circ j)\Omega \in\mathcal{H} \,.
\end{equation}
Note that for all $h\in C_c^\infty(\pSone)$, $(h\circ j)^\R=h^\R\circ j_C$, where $j_C(t):=(CjC^{-1})(t)=-t$ for all $t\in\R$. Therefore, excluding the case $\delta_{d_a,d_b}=0$, we obtain the following formula: for all $f,g\in C_c^\infty(\R)$ and $t\in\R$
\begin{equation}   \label{eq:expitK_action_fourier_prod_scal}
	\begin{split}
		F_{g,f}^{d_a}(t)
		:=C_{b,a}^{-1}(\phi_b(g)\Omega| e^{itK}\phi_a(f)\Omega) 
		& =C_{b,a}^{-1}
		(Y(b,g^\C\circ j)\Omega| e^{itK}Y(a,f^\C\circ j)\Omega) \\
		& = \int_0^{+\infty} 
		\overline{\widehat{g}(p)} 
		e^{-2\pi td_a} \widehat{f}(e^{-2\pi t}p)
		p^{2d_a-1} \mathrm{d} p \,.      
	\end{split}
\end{equation}
The aim is to find a suitable range for $z\in\C$, where the following expression makes sense as analytic function:
\begin{equation}  \label{eq:defin_F_g,f(z)}
	F_{g,f}^{d_a}(z):=  \int_0^{+\infty} 
	\overline{\widehat{g}(p)} 
	e^{-2\pi zd_a} \widehat{f}(e^{-2\pi z}p)
	p^{2d_a-1} \mathrm{d} p    
\end{equation}
for all $f,g\in C_c^\infty(\R)$ with $\mathrm{supp}f\subset (-\infty,0)$ and
\begin{equation}  \label{eq:holo_four_trans}
	\widehat{f}(\zeta):=\frac{1}{\sqrt{2\pi}}
	\int_\R f(x)e^{-i\zeta x}\mathrm{d} x
	\qquad\forall \zeta\in\C
\end{equation}
which is an entire function by \cite[Section 19.1(b)]{Rud87}. 	
Set $D:=\left\{z\in\C \mid -\half<\mathrm{Im}(z)< 0\right\}$ and consider its closure $\overline{D}$, then we have the following: 

\begin{prop}   \label{prop:properties_F_g,f(z)}
For all $f,g\in C_c^\infty(\R)$ with $\mathrm{supp}f\subset (-\infty,0)$, $F_{g,f}^{d_a}:\overline{D}\to \C$ is a well-defined continuous function on $\overline{D}$ and it is the unique analytic extension on $D$ of $F_{f,g}^{d_a}:\R\to \C$.
	Moreover, for all $z\in\overline{D}$, $F_{g,f}^{d_a}(z)$ is linear in $f$ and antilinear in $g$. 
\end{prop}

\begin{proof}
		Throughout the proof, $f$ and $g$ are functions in $C_c^\infty(\R)$ with $\mathrm{supp}f\subset (-\infty,0)$. Furthermore, we set $A(z):=e^{-2\pi zd_a}$.
	
First, we prove that for all $z\in \overline{D}$, $\widehat{f}_z(p):=\widehat{f}(e^{-2\pi z}p)$ is a rapidly decreasing function of $p$ on $(0,+\infty)$: 
\begin{equation}  \label{eq:f_z_rap_decr}
	\begin{split}
		\norm{\widehat{f}_z}_{\alpha,\beta} 
		& :=
		\sup_{p\in(0, +\infty)} \abs{p^\alpha\frac{\mathrm{d}^\beta \widehat{f}_z}{\mathrm{d}x^\beta}(p)} \\
		& =
		\sup_{p\in(0, +\infty)}\abs{\frac{1}{\sqrt{2\pi}}
			\int_\R f(x)p^\alpha \frac{\mathrm{d}^\beta}{\mathrm{d}x^\beta}e^{-ie^{-2\pi z} px}\mathrm{d} x} \\
		&= 
		\sup_{p\in(0, +\infty)}\abs{\frac{1}{\sqrt{2\pi}}
			\int_{-\infty}^0 x^\beta f(x)p^\alpha(-ie^{-2\pi z})^\beta e^{-ie^{-2\pi z} px}\mathrm{d} x} \\
		&= 
		\sup_{p\in(0, +\infty)}\abs{\frac{(-1)^\alpha}{\sqrt{2\pi}}
			\int_{-\infty}^0\frac{\mathrm{d}^\alpha (x^\beta f)}{\mathrm{d}x^\alpha}(x)(-ie^{-2\pi z})^{\beta-\alpha} e^{-ie^{-2\pi z} px}\mathrm{d} x } \\
		& = 
		\sup_{p\in(0, +\infty)}\abs{\frac{e^{2\pi(\alpha-\beta)z}}{\sqrt{2\pi}}
			\int_{-\infty}^0\frac{\mathrm{d}^\alpha (x^\beta f)}{\mathrm{d}x^\alpha}(x)e^{-ie^{-2\pi z} px}\mathrm{d} x}\\
		& = 
		\sup_{p\in(0, +\infty)}\abs{ \frac{e^{2\pi(\alpha-\beta)z}}{\sqrt{2\pi}}
			\int_{-\infty}^0 \frac{\mathrm{d}^\alpha (x^\beta f)}{\mathrm{d}x^\alpha}(x) e^{-ie^{-2\pi\mathrm{Re}(z)}e^{-i2\pi\mathrm{Im}(z)} px}\mathrm{d} x }  \\
		& \leq 
		\sup_{p\in(0, +\infty)} \frac{\abs{ e^{2\pi(\alpha-\beta)z}}}{\sqrt{2\pi}}
		\int_{-\infty}^0 \abs{\frac{\mathrm{d}^\alpha (x^\beta f)}{\mathrm{d}x^\alpha}(x)} e^{-e^{-2\pi\mathrm{Re}(z)}\sin\left(2\pi \mathrm{Im}(z)\right) px}\mathrm{d} x \\
		& \leq
		\frac{\abs{e^{2\pi(\alpha-\beta)z}}}{\sqrt{2\pi}}
		\int_{-\infty}^0 \abs{\frac{\mathrm{d}^\alpha (x^\beta f)}{\mathrm{d}x^\alpha}(x)} \mathrm{d} x  \\
		&=
		\frac{\norm{ \frac{\mathrm{d}^\alpha (x^\beta f)}{\mathrm{d}x^\alpha}}_1}{\sqrt{2\pi}}\abs{ e^{2\pi(\alpha-\beta)z}}<+\infty
		\qquad\forall \alpha,\beta\in\Zpluseq
	\end{split}
\end{equation}
where we have used the integration by parts $\alpha$ times for the fourth equality. This assures us the convergence of the integral in \eqref{eq:defin_F_g,f(z)} for all $z\in \overline{D}$, proving that $F_{g,f}^{d_a}$ is a well-defined function. Moreover, for all $z\in\overline{D}$, $F_{g,f}^{d_a}(z)$ is linear in $f$ and antilinear in $g$ by the properties of the scalar product. 
	
	Now, we prove that for all $f,g$, $F_{g,f}^{d_a}$ is continuous on $\overline{D}$. Let $\zeta\in\overline{D}$ and let $D_\zeta^r$ be a closed disk of radius $r>0$ centred in $\zeta$. Without lose of generality, consider as range for $z$ the bounded region $D_\zeta:=D_\zeta^r\cap \overline{D}$. Our aim is to prove that for all $\epsilon>0$ there exist $\delta>0$ such that
	\begin{equation}   \label{eq:continuity_F_g,f}
		\abs{ F_{g,f}^{d_a}(z)- F_{g,f}^{d_a}(\zeta) } 
		=\abs{ \int_0^{+\infty} 
			\overline{\widehat{g}(p)}
			\left(A(z) \widehat{f}_z(p)-A(\zeta)\widehat{f}_\zeta(p)\right) 
			p^{2d_a-1} \mathrm{d} p  }
		<\epsilon 
	\end{equation}
	whenever $\abs{z-\zeta}<\delta$. On the one hand, note that there exists a $p_0>0$ such that
	\begin{equation}  \label{eq:continuity_F_g,f_part1}
		\abs{ \int_{p_0}^{+\infty} 
			\overline{\widehat{g}(p)} 
			\left(A(z) \widehat{f}_z(p)-A(\zeta)\widehat{f}_\zeta(p)\right)
			p^{2d_a-1} \mathrm{d} p  } 
		<\frac{\epsilon}{2} 
		\qquad \forall z\in D_\zeta \,.
	\end{equation}
	Indeed, since $g\in C_c^\infty(\R)$, $\widehat{g}$ is a rapidly decreasing function and thus $L^1$-integrable. Accordingly, we can choose a $p_0>0$ such that 
	\begin{equation}  
		\int_{p_0}^{+\infty}\abs{\widehat{g}(p)}p^{2d_a-1} \mathrm{d} p
		<
		\left[ 2 \norm{ f }_1
		\max_{z\in D_\zeta}\abs{A(z)} \right]^{-1} \frac{\epsilon}{2} \,.
	\end{equation}
	Thus, the left hand side of \eqref{eq:continuity_F_g,f_part1} is bounded by
	\begin{equation}
		\begin{split}
			\int_{p_0}^{+\infty} 
			&
			\abs{\widehat{g}(p)}
			\int_{-\infty}^0 \abs{f(x)}
			\abs{ A(z)e^{-ie^{-2\pi z}px} -A(\zeta)e^{-ie^{-2\pi \zeta}px} }\mathrm{d}x \,
			p^{2d_a-1} \mathrm{d} p    \\
			&\leq      
			\int_{p_0}^{+\infty}
			\abs{\widehat{g}(p)}
			\int_{-\infty}^0 \abs{f(x)}
			\left(\,\abs{ A(z)} 
			+\abs{ A(\zeta)} \right)
			\mathrm{d}x \,
			p^{2d_a-1} \mathrm{d} p \\
			&\leq 2 \norm{ f }_1
			\max_{z\in D_\zeta} \abs{ A(z)}
			\int_{p_0}^{+\infty} \abs{\widehat{g}(p)}p^{2d_a-1} \mathrm{d} p< \frac{\epsilon}{2} 
			\qquad \forall z\in D_\zeta \,.
		\end{split}
	\end{equation}
	On the other hand, let $x_0<0$ be such that $\mathrm{supp}f\subset(x_0,0)$, then
	\begin{equation}   \label{eq:continuity_F_g,f_part2.1}
		\begin{split}
			& \abs{
				\int_0^{p_0} 
				\overline{\widehat{g}(p)}
				\left( A(z) \widehat{f}_z(p)-A(\zeta)\widehat{f}_\zeta(p)\right) 
				p^{2d_a-1} \mathrm{d} p  } \\
			&\leq  
			\int_0^{p_0} 
			\abs{\widehat{g}(p)}
			\int_{x_0}^0 \abs{f(x)}
			\abs{ A(z)e^{-ie^{-2\pi z}px} -A(\zeta)e^{-ie^{-2\pi \zeta}px} }\mathrm{d}x \,
			p^{2d_a-1} \mathrm{d} p    \\
			&\leq    
			M_\zeta(z)
			\int_0^{p_0}\abs{\widehat{g}(p)}p^{2d_a-1} \mathrm{d} p 
			\int_\R \abs{f(x)}\mathrm{d}x
			\qquad \forall z\in D_\zeta
		\end{split}
	\end{equation}
	where $M_\zeta(z)$ is the function on $D_\zeta$ defined by
	$$
	M_\zeta(z):=
	\max_{\substack{x_0\leq x\leq 0 \\ 0\leq p\leq p_0}}
	\abs{ A(z)e^{-ie^{-2\pi z}px} -A(\zeta)e^{-ie^{-2\pi \zeta}px} }.
	$$
	Note that $\abs{ A(z)e^{-ie^{-2\pi z}px} -A(\zeta)e^{-ie^{-2\pi \zeta}px}}$ is a continuous function in the three variables $(z,p,x)$ on the compact domain $D_\zeta\times[0,p_0]\times[x_0,0]$, which assures us that $M_\zeta(z)$ is non-negative, well-defined and continuous on $D_\zeta$.
	Thus, $M_\zeta(\zeta)=0$ and by continuity, there exists $\delta>0$ such that
	\begin{equation}  \label{eq:continuity_F_g,f_part2.2}
		M_\zeta(z) <
		\left[ \norm{ f }_1 \int_0^{p_0} \abs{ \widehat{g}(p) }p^{2d_a-1} \mathrm{d} p \right]^{-1} \frac{\epsilon}{2}
	\end{equation}
	for all $z\in D_\zeta$ such that $\abs{z-\zeta}<\delta$. Hence, using equation \eqref{eq:continuity_F_g,f_part2.2} in \eqref{eq:continuity_F_g,f_part2.1}, we obtain that
	\begin{equation}  \label{eq:continuity_F_g,f_part2.3}
		\abs{
			\int_0^{p_0} \overline{\widehat{g}(p)}
			\left( A(z) \widehat{f}_z(p)-A(\zeta)\widehat{f}_\zeta(p)\right) 
			p^{2d_a-1} \mathrm{d} p  
		} <\frac{\epsilon}{2}
	\end{equation}
	for all $z\in D_\zeta$ such that $\abs{z-\zeta}<\delta$.
	Therefore, the continuity of $F_{g,f}^{d_a}$ on $\overline{D}$ is proved estimating \eqref{eq:continuity_F_g,f} by \eqref{eq:continuity_F_g,f_part1} and \eqref{eq:continuity_F_g,f_part2.3}.
	
	To prove analyticity, note that 
	\begin{equation}   \label{eq:partial_f_z_rapid_decr}
		\begin{split}
			\frac{\mathrm{d}}{\mathrm{d} z}\left[\widehat{f}(e^{-2\pi z}p)\right]
			&= -2\pi e^{-2\pi z}p  \left. \frac{\mathrm{d} \widehat{f}}{\mathrm{d} \zeta}\right|_{\zeta=e^{-2\pi z}p}\\
			&=-2\pi e^{-2\pi z}p \frac{1}{\sqrt{2\pi}}
			\int_\R f(x)(-ix)e^{-i e^{-2\pi z}p x}\mathrm{d} x \\
			&=-2\pi p \frac{1}{\sqrt{2\pi}}
			\int_\R f(x)(-ie^{-2\pi z}x)e^{-i e^{-2\pi z}p x}\mathrm{d} x \\
			&=-2\pi p\frac{\mathrm{d}\widehat{f}_z}{\mathrm{d}p}(p)
		\end{split}
	\end{equation}
	for all $p\in(0,+\infty)$, which is a rapidly decreasing function of $p$ on $(0,+\infty)$ by \eqref{eq:f_z_rap_decr}. Therefore, for all $f,g$, $\frac{\partial}{\partial z}F_{g,f}^{d_a}$ exists on $D$, i.e., for all $f,g$, $F_{g,f}^{d_a}$ is analytic on $D$. 
	It remains to show that $F_{g,f}^{d_a}(z)$ is the unique analytic extension on $D$ of $F_{g,f}^{d_a}(t)$ for $t\in\R$. Let $G$ an analytic extension on $D$ and set $H(z):=F_{g,f}^{d_a}(z)-G(z)$. $H(t)=0$ for all $t\in\R$, then by the Schwarz reflection principle (cf.\  \cite[p.\  76]{SW64}), $H$ extends to $D\cup\R\cup D^\mathrm{conj}$, where $D^\mathrm{conj}$ are the set of complex conjugates of elements of $D$. Then $H(z)=0$ in $D$ by the identity theorem (cf.\  \cite[p.\  122]{ AF03}), i.e. for all $f,g$, $F_{g,f}^{d_a}(z)$ is the unique analytic extension on $D$ of $F_{g,f}^{d_a}(t)$. 
\end{proof}

A further step in the proof of Theorem \ref{theo:delta_one_half} is to use the function $F_{g,f}^{d_a}(z)$ to define an antilinear functional on $\mathcal{H}$. To this end, define the Hilbert spaces
\begin{equation}  \label{eq:defin_hilbert_spaces}
	\mathcal{H}_a :=\overline{\left\{\phi_a(g)\Omega \mid g\in C_c^\infty(\R) \right\}}^{\norm{\cdot}}\subseteq\mathcal{H} \,,
	\qquad
	\widehat{\mathcal{H}}_a :=L^2((0,+\infty),p^{2d_a-1}\mathrm{d}p)
\end{equation}
and denote by $\norm{\cdot}_{2,a}$ the norm of $\widehat{\mathcal{H}}_a$. Note that $\norm{\phi_a(g)\Omega}=\abs{C_{a,a}}\norm{\widehat{g}}_{2,a}$ by \eqref{eq:expitK_action_fourier_prod_scal}. In particular, if $\left\{\phi_a(g_n)\Omega\right\}_{n\in\Zpluseq}$ is a convergent sequence in $\mathcal{H}$, then $\left\{ \widehat{g_n}\right\}_{n\in\Zpluseq}$ will be a convergent sequence in $\widehat{\mathcal{H}}_a$. Let $\varphi\in\mathcal{H}_a$ and pick any sequence $\left\{ \widehat{g_n}\right\}_{n\in\Zpluseq}$ convergent to some $g_\varphi\in\widehat{\mathcal{H}}_a$ and such that $\varphi=\lim_n \phi_a(g_n)\Omega$. Define
\begin{equation}   \label{eq:defin_F_varphi,f(z)}
	F_{\varphi,f}^{d_a}(z):=
	\int_0^{+\infty} 
	\overline{g_\varphi(p)} 
	e^{-2\pi zd_a} \widehat{f}(e^{-2\pi z}p)
	p^{2d_a-1} \mathrm{d} p
\end{equation}
where $f\in C_c^\infty(\R)$ with $\mathrm{supp} f\subset (-\infty,0)$ and $z\in\overline{D}$. Then the content of the following result is an extension of the properties proved in Proposition \ref{prop:properties_F_g,f(z)} to $F_{\varphi,f}^{d_a}(z)$.

\begin{prop}  \label{prop:properties_F_varphi,f(z)}
	For all $f\in C_c^\infty(\R)$ with $\mathrm{supp} f\subset (-\infty,0)$ and all $\varphi\in\mathcal{H}_a$, $F_{\varphi,f}^{d_a}:\overline{D}\to\C$ is a well-defined function, analytic on $D$ and continuous on $\overline{D}$.
 	Moreover, for all $z\in\overline{D}$, $F_{\varphi,f}^{d_a}(z)$ is linear in $f$, antilinear in $\varphi$ and $F_{\varphi,f}^{d_a}(z)=F_{g,f}^{d_a}(z)$ whenever $\varphi=\phi_a(g)\Omega$ for some $g\in C_c^\infty(\R)$.
\end{prop}

\begin{proof}
	Throughout the proof, $f\in C_c^\infty(\R)$ with $\mathrm{supp} f\subset (-\infty,0)$, $z\in\overline{D}$ and $\varphi\in\mathcal{H}_a$, unless it is differently specified. Furthermore, we set $A(z):=e^{-2\pi zd_a}$.
	
	Let $\left\{ \widehat{g_n}\right\}_{n\in\Zpluseq}$ be a sequence convergent to some $g_\varphi\in\widehat{\mathcal{H}}_a$ and such that $\varphi=\lim_n \phi_a(g_n)\Omega$. To prove that $F_{\varphi,f}^{d_a}$ is a well-defined function, we have to show that the integral in \eqref{eq:defin_F_varphi,f(z)} exists and that for all $\varphi\in\mathcal{H}_a$, $g_\varphi$ is independent of the choice of the sequence $\left\{ \widehat{g_n}\right\}_{n\in\Zpluseq}$. Equation \eqref{eq:f_z_rap_decr} says us that for all $f,\varphi,z$, $F_{\varphi,f}^{d_a}(z)$ exists. To prove the independence from the sequence, let $\left\{ \widehat{l_n}\right\}_{n\in\Zpluseq}$ be another sequence convergent to $l_\varphi\in\widehat{\mathcal{H}}_a$ and such that $\varphi=\lim_n \phi_a(l_n)\Omega$. Then for all $\epsilon>0$, there exists $N>0$ such that
	\begin{equation}
		\begin{split}
			\norm{ g_\varphi-l_\varphi}_{2,a}
			&\leq
			\norm{ g_\varphi- \widehat{g_n}}_{2,a}+
			\norm{ \widehat{g_n-l_n}}_{2,a}+
			\norm{ \widehat{l_n}- l_\varphi}_{2,a}  \\
			&=
			\norm{ g_\varphi- \widehat{g_n}}_{2,a}+
			\abs{C_{a,a}}^{-1}\norm{ \phi_a(g_n-l_n)\Omega}+
			\norm{ \widehat{l_n}- l_\varphi}_{2,a}  \\
			&=
			\norm{ g_\varphi- \widehat{g_n}}_{2,a}+
			\abs{C_{a,a}}^{-1}\norm{ \phi_a(g_n)\Omega-\phi_a(l_n)\Omega}+
			\norm{ \widehat{l_n}- l_\varphi}_{2,a}  \\
			&<\frac{\epsilon}{3}+\frac{\epsilon}{3}+\frac{\epsilon}{3}=\epsilon  \qquad\forall n\geq N
		\end{split}
	\end{equation}
	where we have used \eqref{eq:expitK_action_fourier_prod_scal} in the second raw. It follows that $F_{\varphi, f}^{d_a}$ is well-defined.
	
	Equation \eqref{eq:partial_f_z_rapid_decr} implies that for all $f,\varphi$, $F_{\varphi,f}^{d_a}$ is analytic and thus continuous on $D$.
	To prove that for all $f,\varphi$, $F_{\varphi,f}^{d_a}$ is continuous on $\overline{D}$, let $\zeta\in\overline{D}$ and without loss of generality, consider $D_\zeta:=D_\zeta^r\cap\overline{D}$ as range for $z$, where $D_\zeta^r$ is a closed disk of radius $r>0$ centred in $\zeta$. Our aim is to prove that for all $\epsilon>0$ there exist $\delta>0$ such that
	\begin{equation}   \label{eq:continuity_F_varphi,f}
		\abs{ F_{\varphi,f}^{d_a}(z)- F_{\varphi,f}^{d_a}(\zeta) } 
		<\epsilon 
	\end{equation}
	whenever $\abs{z-\zeta}<\delta$. 
	We have that for all $f$
	\begin{equation}  \label{eq:continuity_F_varphi,f_part1}
		\begin{split}
			\sup_{z\in D_\zeta}\norm{ \widehat{f}_z}_{2,a}
			&=
			\sup_{z\in D_\zeta}\left(\int_0^{+\infty}\abs{ \widehat{f}_z(p)}^2p^{2d_a-1}\mathrm{d}p\right)^\frac{1}{2}  \\
			&\leq 
			\sup_{z\in D_\zeta}\frac{\norm{ \frac{\mathrm{d}^{\alpha}f}{\mathrm{d}x^\alpha}}_1}{\sqrt{2\pi}} \abs{ e^{2\pi\alpha z} }
			\left(\int_0^{+\infty} p^{2d_a-1-2\alpha}\mathrm{d}p\right)^\frac{1}{2}
			<+\infty 
		\end{split}
	\end{equation}
	where we have used \eqref{eq:f_z_rap_decr} with $\beta=0$ and a positive integer $\alpha>d_a$. Therefore, using H{\"o}lder's inequality, there exists $N>0$ such that
	\begin{equation}  \label{eq:continuity_F_varphi,f_part2}
		\abs{ F_{g_n,f}^{d_a}(z)-F_{\varphi,f}^{d_a}(z) }\leq
		\norm{ \widehat{g_n}-g_\varphi}_{2,a}
		\abs{A(z)} \norm{ \widehat{f}_z}_{2,a}
		<\frac{\epsilon}{3}  
		\qquad\forall z\in D_\zeta
	\end{equation}
	for all $n>N$. Then choosing $n>N$, we can conclude that
	\begin{equation} \label{eq:continuity_F_varphi,f_part3}
		\begin{split}
			\abs{F_{\varphi,f}^{d_a}(z) -F_{\varphi,f}^{d_a}(\zeta)}
			&\leq 
			\abs{ F_{\varphi,f}^{d_a}(z)- F_{g_n,f}^{d_a}(z)} +
			\abs{ F_{g_n,f}^{d_a}(z)-F_{g_n,f}^{d_a}(\zeta)} + 
			\abs{ F_{g_n,f}^{d_a}(\zeta)-F_{\varphi,f}^{d_a}(\zeta) } \\
			&< 
			\frac{\epsilon}{3}+\frac{\epsilon}{3}+\frac{\epsilon}{3}
			=\epsilon
		\end{split}
	\end{equation} 
	whenever $\abs{z-\zeta}<\delta$ with $\delta>0$ given by the continuity \eqref{eq:continuity_F_g,f}.
	
	The remaining part is a straightforward consequence of the definition of $F_{\varphi,f}^{d_a}(z)$ and of what above.
	
\end{proof}

Finally, we can prove Theorem \ref{theo:delta_one_half}:

\begin{proof}[Proof of Theorem \ref{theo:delta_one_half}]
The following argument is inspired by \cite[Lemma 3.5]{Ara76}.
	Let $a\in V$ be a quasi-primary vector with conformal weight $d_a$. Let $z\in\overline{D}$, $\varphi\in\mathcal{H}_a$ as defined in \eqref{eq:defin_hilbert_spaces} and $f\in C_c^\infty(\R)$ with $\mathrm{supp}f\subset(-\infty,0)$. Consider $F_{\varphi,f}^{d_a}(z)$ as in \eqref{eq:defin_F_varphi,f(z)}, which is well-defined as stated in Proposition \ref{prop:properties_F_varphi,f(z)}.
	
	First, note that $\mathcal{H}=\mathcal{H}_a\oplus(\mathcal{H}_a)^\perp$ with respect to $\scalar$ because $\mathcal{H}_a$ is a closed subspace of $\mathcal{H}$. Accordingly, we set $F_{\varphi^\perp,f}^{d_a}(z)=0$ for all $\varphi^\perp\in(\mathcal{H}_a)^\perp$ and we extend it to $\mathcal{H}$ by antilinearity.
	Consider the group of unitary operators on $\mathcal{H}$ given by $e^{itK}$ for all $t\in\R$ and looking at their action on $\phi_a(f)\Omega$, note that by Proposition \ref{prop:mob_covariance_qp} and the $\Mob(S^1)^{(2)}$-invariance of the vacuum vector, it preserves $\mathcal{H}_a$. Consequently, it preserves its orthogonal complement $(\mathcal{H}_a)^\perp$ too.  Consider the set of vectors  
	$D(K)^\mathrm{fin}\subseteq \mathcal{H}$ whose spectral measure with respect to $K$ has compact support. Then $\psi \in  D(K)^\mathrm{fin}$
	if and only if $\psi$ belongs to the range of some spectral projection of the self-adjoint operator $K$ associated with a compact interval of $\R$. Then it follows by the spectral theorem that $D(K)^\mathrm{fin}$	is a dense invariant subspace of entire analytic vectors for $K$ and for $e^{\frac{K}{2}}$, see the beginning of \cite[Section X.6]{RS80}. Accordingly, it follows from Nelson's analytic vector theorem \cite[Theorem X.39]{RS80} that the restrictions to  $D(K)^\mathrm{fin}$ of $K$ and of $e^{\frac{K}{2}}$ are essentially self-adjoint operators on $\mathcal{H}$.  Hence, since self-adjoint operators are maximally symmetric, $D(K)^\mathrm{fin}$ is a core for both $K$ and $e^{\frac{K}{2}}$. Note also that 
	$z \mapsto (e^{-i\overline{z}}\varphi|\tilde{\varphi}) $	is an entire analytic function for all $\varphi \in D(K)^\mathrm{fin}$ and all $\tilde{\varphi} \in \mathcal{H}$.
	
	Since  $e^{itK}\mathcal{H}_a = \mathcal{H}_a $ for all $t \in \R$, the projection onto $\mathcal{H}_a $ commutes with all the spectral projections of 
	$K$. As a consequence, we have the direct sum decomposition 
	$$D(K)^\mathrm{fin} =D(K)^\mathrm{fin}_a \oplus D(K)^\mathrm{fin}_{a, \perp}$$
	where $D(K)^\mathrm{fin}_a := D(K)^\mathrm{fin} \cap \mathcal{H}_a$ and 
	$D(K)^\mathrm{fin}_{a,\perp} := D(K)^\mathrm{fin} \cap (\mathcal{H}_a )^\perp$.

	Now, let $\varphi \in D(K)^\mathrm{fin}$. We have $\varphi =\psi+\psi^\perp$ where $\psi \in D(K)^\mathrm{fin}_a$ and 
	$\psi^\perp \in D(K)^\mathrm{fin}_{a,\perp}$. 	We can find a sequence $\left\{\phi_a(g_n)\Omega\right\}_{n\in\Zpluseq}$  in $D(K)^\mathrm{fin}_a$ such that $\psi=\lim_n\phi_a(g_n)\Omega$. Then for all $t\in\R$, we have that
	\begin{equation}  \label{eq:equality_on_R}
		\begin{split}
			(e^{-itK}\varphi| \phi_a(f)\Omega)
			&=
			(e^{-itK}\psi| \phi_a(f)\Omega)+
			(e^{-itK}\psi^\perp| \phi_a(f)\Omega)\\  
			&=
			\lim_{n\rightarrow +\infty} 
			(e^{-itK}\phi_a(g_n)\Omega| \phi_a(f)\Omega) \\
			&=
			\lim_{n\rightarrow +\infty}
			C_{a,a} F_{\phi_a(g_n)\Omega,f}^{d_a}(t) \\
			&=
			C_{a,a} \left(F_{\psi,f}^{d_a}(t)+F_{\psi^\perp,f}^{d_a}(t)\right) \\
			&=
			C_{a,a} F_{\varphi,f}^{d_a}(t)  
		\end{split}
	\end{equation}
	where we have used \eqref{eq:expitK_action_fourier_prod_scal} ($K$ is self-adjoint), the orthogonality condition and the continuity given by \eqref{eq:continuity_F_varphi,f_part2}. Thus, combine the Schwarz reflection principle and the identity theorem as in the proof of Proposition \ref{prop:properties_F_g,f(z)}, we have that for all $f$ and all $\varphi\in D(K)^\mathrm{fin}$,
	\begin{equation}  \label{eq:F_and_exp_K}
		(e^{-i\overline{z}K}\varphi| \phi_a(f)\Omega)
		= C_{a,a} F_{\varphi, f}^{d_a}(z)  
		\qquad \forall z\in D\cup\R
	\end{equation}
	Now, Eq.\ \eqref{eq:F_and_exp_K} extends by continuity to $\overline{D}$. Thus, taking the values at $z= - i/2$  we obtain that
	\begin{equation}  \label{eq:equality_-i/2}
		\begin{split}
			(e^\frac{K}{2}\varphi| \phi_a(f)\Omega)
			&=  
			C_{a,a} F_{\varphi, f}(-i/2) \\
			&=
			\lim_{n\rightarrow +\infty}
			C_{a,a} F_{\phi_a(g_n)\Omega,f}^{d_a}(-i/2)   \\
			&=
			\lim_{n\rightarrow +\infty}
			C_{a,a} \int_0^{+\infty} 
			\overline{\widehat{g_n}(p)} 
			e^{i\pi d_a} \widehat{f}(-p)
			p^{2d_a-1} \mathrm{d} p \\
			&=
			\lim_{n\rightarrow +\infty}
			(\phi_a(g_n)\Omega| e^{i\pi d_a}Y(a,f^\C)\Omega) \\
			&= 
			(\varphi| e^{i\pi d_a}Y(a,f^\C)\Omega)
			\qquad\forall \varphi\in D(K)^\mathrm{fin}  
		\end{split}
	\end{equation}
	where we have used again the orthogonality condition and the continuity given by \eqref{eq:continuity_F_varphi,f_part2} as in \eqref{eq:equality_on_R}. Being $D(K)^\mathrm{fin}$ a core for the self-adjoint operator $e^\frac{K}{2}$, \eqref{eq:equality_-i/2} implies that $\phi_a(f)\Omega$ is in the domain of $e^\frac{K}{2}$ and that
	$$
	e^\frac{K}{2}Y(a,f^\C\circ j)\Omega
	=e^\frac{K}{2}\phi_a(f)\Omega
	=e^{i\pi d_a}Y(a,f^\C)\Omega
	=i^{2d_a}Y(a,f^\C)\Omega
	$$
	for all $f\in C_c^\infty(\R)$ with $\mathrm{supp}f\subset(-\infty,0)$. This is equivalent to
	$$
	e^\frac{K}{2}Y(a,f\circ j)\Omega
	=i^{2d_a}Y(a,f)\Omega
	$$
	for all $f\in C^\infty_c(\pSone)$ with $\mathrm{supp}f\subset S^1_-$, which implies the desired result.
\end{proof}

\section{Further results on the correspondence}
\label{section:further_results}

We present some further results on the correspondence $V\mapsto \A_V$ as shown in Section \ref{section:construction_net}. Beyond the general interest, these results will be useful in the production of examples in  Section \ref{section:examples}.

We fix the following notation: 
let $p\in\Z_2$, $d\in\half\Zpluseq$ and $\mathcal{S}$ be any subset of a vertex superalgebra $V$, then
\begin{align}  \label{eq:notation_even_odd}
	\begin{array}{c}
		C_p^\infty(S^1) 
		:=\left\{\begin{array}{lcl}
			C^\infty(S^1) & \mbox{if} & p=\parzero \\
			C_\chi^\infty(S^1) & \mbox{if} & p=\parone
				  \end{array}\right. 
	\quad
		\iota_d(\gamma) 
		:=\left\{ \begin{array}{lcl}
			\beta_d(\dot{\gamma})  & \mbox{if} &  d\in\Z   \\
			\alpha_d(\gamma)  & \mbox{if} &  d\in\Z-\half 
					\end{array}\right. 
\,\,\,\forall \gamma\in G \vspace{0.2cm}\\
\mathcal{S}_\parzerone:=(\mathcal{S}\cap V_\parzero)\cup(\mathcal{S}\cap V_\parone)
\end{array}
\end{align}
where $G$ can be $\Mob(S^1)^{(2)}$ or $\Diff^+(S^1)^{(2)}$, see Section \ref{subsection:diff_group}, \eqref{eq:test_funct_spaces_chi}, \eqref{eq:family_cont_ops_diff} and \eqref{eq:family_cont_ops_diff^2} for notation.

\subsection{Unitary subalgebras and covariant subnets}
\label{subsection:unitary_subalgebras_covariant_subnets}
In the current section, we show that there is a one-to-one correspondence between unitary subalgebras of a simple strongly graded-local unitary VOSA $V$ and covariant subnets of the associated irreducible graded-local conformal net $\A_V$. As a consequence, we also prove that the correspondence $V\mapsto \A_V$ preserves the fixed point and the coset constructions. See Section \ref{subsection:covariant_subnets} and Section \ref{subsection:unitary_subalgebras} for notation.

\begin{theo}  \label{theo:covariant_subnets_unitary_subalgebras}
Let $V$ be a simple strongly graded-local unitary VOSA. If $W$ is a unitary subalgebra, then the simple unitary VOSA $W$ is strongly graded-local and $\A_W$ embeds canonically as a covariant subnet of $\A_V$. Conversely, if $\B$ is a M\"{o}bius covariant subnet of $\A_V$, then $W:=\mathcal{H}_\B\cap V$ is a unitary subalgebra of $V$ such that $\A_W=\B_{e_\B}$. In other terms, the map $W\mapsto\A_W$ gives a one-to-one correspondence between unitary subalgebras of $V$ and covariant subnets of $\A_V$.
\end{theo}

\begin{proof}
The first and the second statement are obtained by adapting the proofs of \cite[Theorem 7.1 and Theorem 7.4]{CKLW18} respectively. 
We use the notation as given in Definition \ref{defin:hilbet_space_from_V} and in \eqref{eq:notation_even_odd}.

Suppose that $W$ is a unitary subalgebra of $V$. 
By proposition \ref{prop:unitary_structure_subalgebras}, 
$W$ with the restriction of $\scalar$ is a simple unitary VOSA.
Moreover, by the definition of vertex subalgebra, for all $a\in W$, the vertex operator $Y_W(a,z)$ of $W$ is equal to the restriction of the vertex operator $Y(a,z)$ of $V$ to $W$. It follows that $W$ is a simple energy-bounded unitary VOSA. 

If $e_W$ is the orthogonal projection onto the Hilbert space closure $\mathcal{H}_W$ of $W$, then $W=e_WV=\mathcal{H}_W\cap V$. Hence, for all $a\in W_\parzerone$, all $f\in C_{p(a)}^\infty(S^1)$ and all $b\in V$, we have that
$$
Y(a,f)e_Wb\in\mathcal{H}_W
\,,\qquad
Y(a,f)^*e_Wb \in\mathcal{H}_W \,.
$$
Therefore,
\begin{equation}
	\begin{split}
		(b|e_WY(a,f)c)
		&
		=(Y(a,f)^*e_Wb|c)
		=(Y(a,f)^*e_Wb|e_Wc) \\
		&
		=(e_Wb|Y(a,f)e_Wc)
		=(b|Y(a,f)e_Wc)
		\qquad\forall a\in W_\parzerone
		\,\,\,\forall b,c\in V \,.
	\end{split}
\end{equation}
Recall that $V$ is a core for every smeared vertex operator $Y(a,f)$ and thus the calculation above implies that every $Y(a,f)$ with $a\in W_\parzerone$ and $f\in C_{p(a)}^\infty(S^1)$ commutes with $e_W$. By the definition of vertex subalgebra and by Proposition \ref{prop:characterisation_unitary_subalgebras}, $L_{-1}, L_0$ and $L_1$ preserve $W$ and thus we can prove that they commute with $e_W$, just proceeding as above. This implies that $e_W\in U(\Mob(S^1)^{(\infty)})'$.
Therefore, the isotonous family $\B_W$ of $\A_V$ defined by
$$
\B_W(I):=\A_V(I)\cap\{e_W\}'
\qquad\forall I\in\J
$$
is a M\"{o}bius covariant subnet of $\A_V$. Moreover, for all $I\in \J$, every smeared vertex operator $Y(a,f)$ with $a\in W_\parzerone$ and $f\in C_{p(a)}^\infty(S^1)$ such that $\mathrm{supp} f\subset I$ is affiliated with $\B_W(I)$.
Consequently, $\mathcal{H}_{\B_W}=\mathcal{H}_W$ and thus $\B_W$ becomes irreducible when restricted to $\mathcal{H}_W$ by Proposition \ref{prop:covariant_subnets}. By twisted Haag duality \eqref{eq:twisted_haag_duality}, we have that
$$
Z(\B_W(I)_{e_W})'Z^*=Z^*(\B_W(I)_{e_W})'Z=\B_W(I')_{e_W}
\qquad \forall I\in\J
\,.
$$
Moreover,
$$
\mathcal{D}(Y_W(a,f))=e_W\mathcal{D}(Y(a,f))
=\mathcal{D}(Y(a,f))\cap\mathcal{H}_W
\qquad\forall a\in W_\parzerone
\,\,\,\forall f\in C_{p(a)}^\infty(S^1)
$$
because every $Y(a,f)$ commutes with $e_W$ and it coincides with $Y_W(a,f)$ on $W$. 
Accordingly, for every $a\in W_\parzerone$ and every $I\in\J$, $Y_W(a,f)$ must be affiliated with $(Z\B_W(I')_{e_W}Z^*)'=\B_W(I)_{e_W}$ whenever $f\in C_{p(a)}^\infty(S^1)$ has $\mathrm{supp}f\subset I$.
Thus, for all $I\in\J$, the von Neumann algebra on $\mathcal{H}_W$
$$
\A_W(I):=W^*\left(\left\{
Y_W(a,f) \mid a\in W_\parzerone \,,
\,\,\, f\in C_{p(a)}^\infty(S^1) 
\,\,\, \mathrm{supp}f\subset I
\right\}\right)
$$
is contained in $\B_W(I)_{e_W}$. 
This means that $W$ defines a simple strongly graded-local unitary VOSA because $\B_W(\cdot)_{e_W}$ is a graded-local M\"{o}bius covariant net.
Thus,  
by Theorem \ref{theo:diffeo_covariant_net_from_V}, 
$\A_W$ is an irreducible graded-local conformal net and by twisted Haag duality \eqref{eq:twisted_haag_duality},
$$
\A_W(I')\subseteq\B_W(I')_{e_W}=Z\B_W(I)_{e_W}'Z^*
\subseteq
Z\A_W(I)'Z^*=\A_W(I')
\qquad \forall I\in\J
$$
which  implies that $\A_W(I)=\B_W(I)_{e_W}$ for all $I\in\J$, so concluding the proof of the first part.

Conversely, suppose that $\B$ is a M\"{o}bius covariant subnet of $\A_V$ and set $W:=\mathcal{H}_\B\cap V$.
Both $\Omega\in W$   
and $L_nW\subseteq W$ for all $n\in\{-1,0,1\}$ 
because $\mathcal{H}_\B$ is globally invariant for the unitary representation $U$ of $\Mob(S^1)^{(\infty)}$ on $\mathcal{H}$. 
Now, if $a\in W$, then $a_{(-1)}\Omega=a\in W$ and 
$$
a_{(-n-1)}\Omega
=\frac{1}{n}[L_{-1},a_{(-n)}]\Omega
=\frac{1}{n}L_{-1}a_{(-n)}\Omega
\qquad\forall n\in\Z
$$ 
thanks to the commutation relation \eqref{eq:cr_l-1}. Using the inductive step, we can deduce that $a_{(n)}\Omega$ is in $W$ for all $n\in\Z$. 
It follows that $Y(a,f)\Omega\in \mathcal{H}_\B$ for all $a\in W_\parzerone$ and all $f\in C_{p(a)}^\infty(S^1)$. 
Let   $e_\B$  be the orthogonal projection of $\mathcal{H}$ onto $\mathcal{H}_\B$ and for any $I\in\J$, let $\epsilon_{I'}$ be the unique vacuum-preserving normal conditional expectation of $\A_V(I')$ onto $\B(I')$, see e.g. \cite[Lemma 13]{Lon03}. If $a\in W_\parzerone$, $f\in C_{p(a)}^\infty(S^1)$ with $\mathrm{supp}f\subset I$ and $A\in\A_V(I')$, we have that
\begin{equation}
	\begin{split}
		Y(a,f)e_\B ZAZ^*\Omega 
		&=
		Y(a,f)Ze_\B AZ^*\Omega 
		=Y(a,f)Z\epsilon_{I'}(A)Z^*\Omega \\
		&=
		Z\epsilon_{I'}(A)Z^* Y(a,f)\Omega 
		=Ze_\B AZ^* Y(a,f)\Omega  \\
		&=
		e_\B ZAZ^* Y(a,f)\Omega 
		=e_\B Y(a,f)ZAZ^*\Omega
	\end{split}
\end{equation}
where we have used  that $\epsilon_{I'}(A)e_\B=e_\B Ae_\B$ and that $\epsilon_{I'}(A)\in\B(I')$. 
By a slight modification of \cite[Proposition 7.3]{CKLW18}, $Z\A_V(I')Z^*\Omega$ is a core for every $Y(a,f)$ with $a\in W_\parzerone$ and $f\in C_{p(a)}^\infty(S^1)$ such that $\mathrm{supp}f\subset I$, so that every such $Y(a,f)$ commutes with $e_\B$. 
Consequently, $Y(a,f)$ and $Y(a,f)^*$ with $\mathrm{supp}f\subset I$ are affiliated with $\B(I)=\A_V(I)\cap\{e_\B\}'$. 
If $f\in C_{p(a)}^\infty(S^1)$ (that is without any restriction on the support), we can find $g,h\in C_{p(a)}^\infty(S^1)$ such that $f=g+h$, $\mathrm{supp}g\subset I_g$ and $\mathrm{supp}h\subset I_h$ for some $I_g,I_h\in\J$. Then $Y(a,f)c=Y(a,g)c+Y(a,h)c$ for all $c\in\mathcal{H}^\infty$ thanks to the linearity of the maps in \eqref{eq:smeared_vertex_operator_distribution}. Hence, $Y(a,f)$ commutes with $e_\B$ for all $a\in W_\parzerone$ and all $f\in C_{p(a)}^\infty(S^1)$ because $\mathcal{H}^\infty$ is a core for all smeared vertex operators. 
Now, proceeding similarly to the proof of Proposition \ref{prop:Omega_is_cyclic},
we can prove that $a_nb\in W$ and $a_n^+b\in W$ for all $a,b\in W$. Recalling that $L_0W\subseteq W$,   we conclude that $W$ is a unitary subalgebra of $V$. From the first part, it follows also that $\A_W=\B_{e_\B}$, so concluding the proof.
\end{proof}

Implementing the same proof of \cite[Proposition 7.6]{CKLW18}, we get that:
\begin{prop}   \label{prop:fixed_point_correspondence}
Let $V$ be a simple strongly graded-local unitary VOSA. If $G$ is a closed subgroup of $\Aut_{\scalar}(V)=\Aut(\A_V)$, then $\A_{V^G}=\A_V^G$.
\end{prop}

Theorem \ref{theo:covariant_subnets_unitary_subalgebras} allows us to prove the following result about the coset construction.
\begin{prop}   \label{prop:correspondence_coset_constructions}
Let $V$ be a simple strongly graded-local unitary VOSA and let $W$ be a unitary subalgebra. Then $\A_{W^c}=\A_W^c$. 
\end{prop}

\begin{proof}
This proof is adapted from the one of \cite[Proposition 7.8]{CKLW18}. We use the notation specified in \eqref{eq:notation_even_odd}.
 
First note that $\A_W$ and $\A_{W^c}$ are covariant subnets of $\A_V$ thanks to Proposition \ref{prop:properties_coset_subalgebra} and Theorem \ref{theo:covariant_subnets_unitary_subalgebras}. By Proposition \ref{prop:unitary_structure_subalgebras}, we have two positive-energy unitary representations of the Virasoro algebra on $V$ via the operators $L_n^W$ and $L_n^{W^c}$ with $n\in\Z$. Thus, we can extend them to two positive-energy strongly continuous projective unitary representations $U_W$ and $U_{W^c}$ respectively on $\mathcal{H}$ of $\Diff^+(S^1)^{(\infty)}$ in the usual way, cf.\  p.\  \pageref{eq:repr_U_of_the_net_A_V} and references therein. In particular, we have that $U_X(\exp^{(\infty)}(tf))=e^{itY(\nu^X,f)}$ for all $f\in C^\infty(S^1,\R)$ with $X\in \{W, W^c\}$. Now, we give some properties of these representations. For every $f\in C^\infty(S^1,\R)$, both $Y(\nu^W,z)$ and $Y(\nu^{W^c},z)$ satisfy linear energy bounds ($k=1$ in \eqref{eq:bound_smeared_vertex_ops}). Furthermore, by Proposition \ref{prop:properties_coset_subalgebra}, the vertex operators $Y(\nu^W,z)$ and $Y(\nu^{W^c},z)$ commute. Thanks to these two properties, we can apply the argument given in \cite[p.\ 113]{BS90}, based on \cite[Theorem 3.1]{DF77}, see also \cite[Theorem 19.4.4]{GJ87}, \cite[Theorem C.2]{Tan16} and \cite[Theorem 3.4]{CTW22}, to conclude that also $U_W(\gamma)$ and $U_{W^c}(\gamma)$ commute for all $\gamma\in\Diff^+(S^1)^{(\infty)}$. Moreover, we have that 
\begin{equation}  \label{eq:decomposition_U}
U(\gamma)=U_W(\gamma)U_{W^c}(\gamma) 
\qquad\forall \gamma\in\Diff^+(S^1)^{(\infty)}\,.
\end{equation}
Let $Y(a,g)$ be any smeared vertex operator with $a\in W_\parzerone$ and $g\in C_{p(a)}^\infty(S^1)$, $f\in C^\infty(S^1,\R)$ and $c\in\mathcal{H}^\infty$. For every $t\in\R$, define the vector in $\mathcal{H}^\infty$ $$
c(t):=Y(a,g)U_{W^c}(\exp^{(\infty)}(tf))c
=Y(a,g)e^{itY(\nu^{W^c},f)}c \,.
$$  
Due to Proposition \ref{prop:properties_coset_subalgebra}, $[Y(\nu^{W^c},f),Y(a,g)]=0$ and thus, from Lemma \ref{lem:common_core}, Lemma \ref{lem:h_infty_is_invariant} and Remark \ref{rem:differentiability_Diff_repr_H^infty}, it follows that $t\mapsto c(t)$ is differentiable on $\mathcal{H}^\infty$ and satisfies the Cauchy problem on $\mathcal{H}^\infty$:
$$
\left\{
\begin{array}{l}
\frac{\mathrm{d}}{\mathrm{d}t}c(t)=iY(\nu^{W^c},f)c(t) \\
c(0)=Y(a,g)c
\end{array} \right.
$$
which has $e^{itY(\nu^{W^c},f)}Y(a,g)c$ as unique solution. Therefore, for all $t\in\R$,
$$
Y(a,g)U_{W^c}(\exp^{(\infty)}(tf))c= c(t)
=e^{itY(\nu^{W^c},f)}Y(a,g)c
=U_{W^c}(\exp^{(\infty)}(tf))Y(a,g)c \,.
$$
By the arbitrariness of the choices made above, we have that $Y(a,g)$ and $U_{W^c}(\exp^{(\infty)}(tf))$ commute for all $a\in W_\parzerone$, all $g\in C_{p(a)}^\infty(S^1)$, all $f\in C^\infty(S^1,\R)$ and all $t\in\R$, that is, $U_{W^c}(\exp^{(\infty)}(tf))\in\A_W(I)'$ for all $I\in\J$. By \eqref{eq:decomposition_U}, what above implies that
$$
U(\exp^{(2)}(tf))AU(\exp^{(2)}(tf))^*
=U_W(\exp^{(\infty)}(tf))AU_W(\exp^{(\infty)}(tf))^*
$$
for all $t\in\R$, all $f\in C^\infty(S^1,\R)$, all $A\in\A_W(I)$ and all $I\in\J$.
Similarly, we can prove that $U_W(\exp^{(\infty)}(tf))\in\A_{W^c}(I)'$ for all $t\in\R$, all $f\in C^\infty(S^1,\R)$ and all $I\in\J$. Pick $I\in\J$ and $A\in\A_{W^c}(I)$. For any $I_1\in\J$, let $B\in\A_W(I_1)$ and choose a composition $\gamma$ of exponentials of vector fields such that $\dot{\gamma} I_1=I'$. Thus, we have that
\begin{equation}
\begin{split}
AZBZ^*
 &=U_W(\gamma)^*AU_W(\gamma)ZBZ^*
=U_W(\gamma)^*AZU_W(\gamma)BZ^* \\
 &=U_W(\gamma)^*AZB^\gamma U_W(\gamma)Z^*
=U_W(\gamma)^*AZB^\gamma Z^* U_W(\gamma) \\
 &=U_W(\gamma)^*ZB^\gamma Z^* AU_W(\gamma)
=ZBZ^*U_W(\gamma)^* AU_W(\gamma) \\ 
 &=ZBZ^* A
\end{split}
\end{equation}
where 
$$
B^\gamma:=U_W(\gamma)BU_W(\gamma)^*
=U(\dot{\gamma})BU(\dot{\gamma})^*
\in\A_W(I')\subseteq\A_V(I') \,.
$$
The above means that for every $I\in\J$, every $A\in\A_{W^c}(I)$ commutes with $\A_W(I_1)$ for all $I_1\in\J$, that is $\A_{W^c}\subseteq \A_W^c$. 

Conversely, there exists a unitary subalgebra $\widetilde{W}$ of $V$ such that $\A_W^c=\A_{\widetilde{W}}$ by Theorem \ref{theo:covariant_subnets_unitary_subalgebras}. As a consequence, for every $a\in\widetilde{W}_\parzerone$ and every $f\in C_{p(a)}^\infty(S^1)$, $Y(a,f)$ is affiliated with $Z\A_W(S^1)'Z^*$. This implies that $[Y(a,z),Y(b,w)]=0$ for all $b\in W$ by Proposition \ref{prop:commutation_von_neumann_algebras_generated}, that is, $a\in W^c$. Therefore, we can conclude that $\widetilde{W}\subseteq W^c$ and thus $\A_W^c\subseteq \A_{W^c}$, so finishing the proof.
\end{proof}

\subsection{Strong graded locality through generators}
\label{subsection:strong_graded-local_quasi-prim_generators}
In the current section, we prove that a simple energy-bounded unitary VOSA is strongly graded-local if the family of von Neumann algebras affiliated to a set of quasi-primary generators of the VOSA is graded-local. One of the consequences is that tensor products of simple strongly graded-local unitary VOSAs are strongly graded-local.

Let $V$ be a simple energy-bounded unitary VOSA and let $\F$ be a subset of $V$. Then for every $I\in\J$ we define a von Neumann subalgebra of $\A_V(I)$ by
\begin{equation}   \label{eq:defin_net_gen_by_subset_F}
\A_\F(I):=W^*\left(\left\{
Y(a,f),\,\, Y(b,g) \left| 
\begin{array}{l}
	a\in V_\parzero\cap\F \,,\,\, f\in C^\infty(S^1) \,,\,\,\mathrm{supp}f\subset I\\ 
	b\in V_\parone\cap\F\,,\,\, g\in C^\infty_\chi(S^1) \,,\,\,\mathrm{supp}g\subset I
\end{array}
\right.\right\} \right)
\,.
\end{equation}

\begin{theo}  \label{theo:gen_by_quasi_primary}
Let $\F$ be a subset of a simple energy-bounded unitary VOSA $V$ and assume that $\F$ contains only quasi-primary vectors. Moreover, suppose that $\F$ generates $V$ and that there exists an $I\in\J$ such that $\A_\F(I')\subseteq Z\A_\F(I)'Z^*$. Then $V$ is strongly graded-local and $\A_\F(I)=\A_V(I)$ for all $I\in\mathcal{J}$.
\end{theo}

\begin{proof}
The proof is an adaptation of the one of \cite[Theorem 8.1]{CKLW18}, which we are going to organize in five parts. We use the notation fixed in Definition \ref{defin:hilbet_space_from_V} and in \eqref{eq:notation_even_odd}.

\textit{Step 1}. 
We prove that $\A_\F(I)=\A_{\F\cup\theta(\F)}(I)$ for all $I\in\J$. Indeed, let $a\in\F$ and $f\in C_{p(a)}^\infty(S^1)$  with $\mathrm{supp}f\subset I\in \J$, then $Y(a,f)$ is affiliated to $\A_\F(I)$. Let $A\in B(\mathcal{H})$ be such that $Y(a,f)$ commutes with $A$ and $A^*$. Then we have
\begin{equation} 
	\begin{split}
		(A^*c| Y(\theta (a),f)d)
		&=
		\overline{(Y(\theta (a),f)d| A^*c)}
		=(-1)^{2d_a^2+d_a}\overline{(Y(a,f)^*d| A^*c)} \\
		&=
		(-1)^{2d_a^2+d_a}\overline{(Ad| Y(a,f)c)}
		=(-1)^{2d_a^2+d_a}(Y(a,f)c| Ad) \\
		&=
		(Y(\theta (a),f)^*c| Ad)
		\qquad\forall c,d\in V
	\end{split}
\end{equation}
where we have used \eqref{eq:smeared_vertex_operator_adjoint} for the second and the last equality, whereas Lemma \ref{lem:gen_by_qp} for the third one. The equality above is equivalent to say that $Y(\theta (a), f)$ commutes with $A$ and $A^*$ by Lemma \ref{lem:gen_by_qp}. Thus, we can conclude that $Y(a,f)$ and $Y(\theta(a), f)$ are affiliated to the same von Neumann algebra, that is, $\A_\F(I)=\A_{\F\cup\theta(\F)}(I)$ as desired. From now onwards, we suppose, without loss of generality, that $\F=\theta(\F)$. 

\textit{Step 2}. We prove that $\A_\F$ defines an irreducible graded-local M\"{o}bius covariant net on $S^1$ acting on $\mathcal{H}$.
First of all, $\A_\F$ is clearly isotonous and it is also M\"{o}bius covariant thanks to Proposition \ref{prop:mob_covariance_qp}. It follows that 
\begin{equation}
	\A_\F(I')\subseteq Z\A_\F(I)'Z^* \qquad\forall I\in\J\,.
\end{equation}
Now, define for all $I\in\J$, $\mathcal{P}_\F(I)$ as the algebra generated by all the smeared vertex operators $Y(a,f)$ where $a\in\F$ and $f\in C_{p(a)}^\infty(S^1)$ with $\mathrm{supp}f\subset I$. From the $\theta$-invariance of $\F$, we deduce that all $\mathcal{P}_\F(I)$ are $*$-algebras and moreover they have $\mathcal{H}^\infty$ as invariant domain. Thus, it makes sense to define for all $I\in\J$, $\mathcal{H}_\F(I)$ as the closure of $\mathcal{P}_\F(I)\Omega$. By Proposition \ref{prop:mob_covariance_qp} and recalling that $\mathcal{H}^\infty$ is an invariant domain for all $U(\gamma)$ with $\gamma\in\Mob(S^1)^{(2)}$, we obtain that
\begin{equation}
	\begin{split}
		U(\gamma)Y(a^n,f^n)\cdots Y(a^1,f^1)\Omega 
		&=U(\gamma)Y(a^n,f^n)U(\gamma)^*U(\gamma)\cdots Y(a^1,f^1)U(\gamma)^*\Omega  \\
		&=Y(a^n, \iota_{d_{a^n}}(\gamma)f^n)\cdots Y(a^1, \iota_{d_{a^1}}(\gamma)f^1)\Omega
	\end{split}
\end{equation}
for all $\gamma\in\Mob(S^1)^{(2)}$, where $\{a^m,f^m\}$ is any finite collection  of vectors in $\F$ and functions in $C_{p(a^m)}^\infty(S^1)$ respectively. Passing to the closure, it follows that $U(\gamma)\mathcal{H}_\F(I)=\mathcal{H}_\F(\dot{\gamma} I)$ for all $\gamma\in\Mob(S^1)^{(2)}$ and all $I\in\J$. Then
\begin{equation}   \label{eq:defin_H_F}
	\overline{
		\bigcup_{\gamma\in\Mob(S^1)^{(2)}} U(\gamma)\mathcal{H}_\F(I)
	}=
	\mathcal{H}_\F:
	=\overline{\mathcal{P}_\F\Omega}
\end{equation}
where $\mathcal{P}_\F$ is the $*$-algebra with $\mathcal{H}^\infty$ as invariant core, generated by all smeared vertex operators $Y(a,f)$ with $a\in\F$ and $f\in C_{p(a)}^\infty(S^1)$ without any restriction on the support. Moreover, noting that for all $a\in\F$ there exists a $f\in C_{p(a)}^\infty(S^1)$ such that
$$
a=a_{(-1)}\Omega=Y(a,f)\Omega  
$$
we can conclude that $V\subset \mathcal{H}_\F$ by the fact that $\F$ is generating. Therefore, $\mathcal{H}_\F$ is equal to $\mathcal{H}$. Now, we adapt the proof of \cite[Theorem 1]{Bor68}, cf. also \cite[Theorem 3.2.1]{Lon08}, to prove that $\mathcal{H}_\F(I)=\mathcal{H}$ for all $I\in\J$, that is, a Reeh-Schlieder property for fields. Fix an $I\in\J$ and consider $v\in\mathcal{H}$ orthogonal to $\mathcal{H}_\F(I)$. We want to prove that $v=0$. Let $I_0\in\J$ whose closure is contained in $I$, thus there exists a neighbourhood $N$ of $0$ small enough such that $r(t)I_0\subset I$ for all $t\in N$. We have
\begin{equation}
	F(t):=(v| U(r^{(2)}(t))w)=0 
	\qquad \forall t\in N \,\,\,\forall w\in\mathcal{H}_\F(I_0)
\end{equation}
because $U(r^{(2)}(t))w\in\mathcal{H}_\F(I)$ as we have proved above. The generator of the rotation subgroup is positive by construction and therefore by \cite[Theorem 11.4.1]{HP57}, $F$ can be extended to a continuous function, still called $F$, on the upper half-plane $\{z\in\C\mid\mathrm{Im}(z)\geq 0\}$, which is analytic on $\{z\in\C\mid\mathrm{Im}(z)>0\}$. By the identity theorem \cite[p.\ 122]{AF03}, $F=0$, which implies that $v$ is orthogonal to $U(r^{(2)}(t))\mathcal{H}_\F(I)$ for all $t\in\R$. Now, the generator of the translation subgroup of $\Mob(S^1)^{(2)}$ is also positive, see e.g.\ \cite[Lemma 2.6]{GF93}, cf.\ \cite[Proposition 1.4.1]{Lon08}. Hence, we can apply the same argument above for the generator of rotations to the generator of translations and consequently conclude that $v$ is orthogonal to $U(\gamma)\mathcal{H}_\F(I)$ for all $\gamma\in\Mob(S^1)^{(2)}$, just recalling that $\Mob(S^1)^{(2)}$ is generated by rotations and translations. By \eqref{eq:defin_H_F}, $v$ is orthogonal to $\mathcal{H}_\F=\mathcal{H}$, that is, $v=0$. This shows that $\mathcal{H}_\F(I)=\mathcal{H}$ for all $I\in\J$. 
Now, let $a\in\F$ and $Y(a,f)$ be affiliated with $\A_\F(I)$ for some $I\in\J$. There exists a sequence of operators $\{A_n\}\subset \A_\F(I)$ such that $\lim_{n\rightarrow+\infty}A_nc=Y(a,f)c$ for all $c\in\mathcal{H}^\infty$, see e.g.\ \cite[Section 2.2]{CKLW18}. Therefore, it is easy to see that $\overline{\A_\F(I)\Omega}\cap\mathcal{H}^\infty$ is invariant for the left action of $\mathcal{P}_\F(I)$. It follows that $\mathcal{P}_\F(I)\Omega\subseteq \overline{\A_\F(I)\Omega}$ and thus $\overline{\A_\F(I)\Omega}=\mathcal{H}$ for all $I\in\J$. This means that the vacuum vector $\Omega$ is cyclic with respect to the net $\A_\F$, which then defines an irreducible graded-local M\"{o}bius covariant net on $S^1$ acting on $\mathcal{H}$.

\textit{Step 3}.
We use the last three parts to prove that $\A_V(I)\subseteq \A_\F(I)$ for all $I\in\J$, which will conclude the proof. Thanks to the M\"{o}bius covariance of the net $\A_V$ given by Proposition \ref{prop:mob_covariance_of_A_V}, it is sufficient to prove the inclusion above for the upper semicircle $S^1_+\in\J$ only. Accordingly, consider the Tomita-Takesaki modular theory associated to the von Neumann algebra $\A_\F(S^1_+)$ and the vacuum vector $\Omega$, that is, the modular operator $\Delta$, the modular conjugation $J$ and the Tomita operator $S=J\Delta^\half$, see \cite[Definition 2.5.10]{BR02}. We are going to prove that $ZJ$ defines an antilinear VOSA automorphism of $V$.
Let $j$ be the orientation-reversing isometry of $S^1$, that is, $j(z)=\overline{z}=z^{-1}$ for all $z\in S^1$. By the Bisognano-Wichmann property, see p.\ \pageref{bisognano-wichmann_property}, $U$ extends to an antiunitary representation of $\Mob(S^1)^{(\infty)}\ltimes\Z_2$, which we still denote by $U$, such that $U(j)=ZJ$ and
\begin{align}  
	ZJ\A_\F(I)ZJ &=\A_\F(j(I)) 
	\qquad \forall I\in\J
	\label{eq:action_j_on_A_F}\\
	ZJU(\gamma)ZJ &=U(j\circ\dot{\gamma}\circ j)  
	\qquad \forall \gamma\in\Mob(S^1)^{(\infty)} \,.
	\label{eq:repr_Mob_ltimes_j}
\end{align}
We obtain that $L_n$ commutes with $ZJ$ for every $n\in\{-1,0,1\}$ (consider the one-parameter subgroups $\exp(tl)$ as in the proof of Proposition \ref{prop:mob_covariance_qp} in \eqref{eq:repr_Mob_ltimes_j} and derive them for $t=0$). Proceeding with an argument as in the first part of the proof of Proposition \ref{prop:Omega_is_cyclic}, we have that $ZJV\subseteq V$. This implies that for all $a\in V$, the formal series
\begin{equation}
	\Phi_a(z):=\sum_{n\in\Z}ZJa_{(n)}ZJ z^{-n-1}
\end{equation}
is a well-defined field on $V$ such that
\begin{equation}
	[L_{-1},\Phi_a(z)]=\frac{\mathrm{d}}{\mathrm{d}z}\Phi_a(z)
	\,,\quad
	\Phi_a(z)\Omega|_{z=0}=ZJa
	\,,\quad
	\Phi_a(z)\Omega=e^{zT}ZJa \,.
\end{equation}
If $a\in V$ is homogeneous, define the following operators acting on $V$:
$$
\Phi_a(f)c:=\sum_{n\in\Z-d_a} \widehat{f}_n ZJa_nZJ c
\qquad
\forall f\in C_{p(a)}^\infty(S^1)
\,\,\,\forall c\in V\,.
$$  
Note that every $\Phi_a(f)$ is closable (and we use the same symbol for its closure) thanks to the energy bounds and that they have $\mathcal{H}^\infty$ as common invariant core. Moreover, $ZJ$ preserves $\mathcal{H}^\infty$ too as it is an antilinear isometry which commutes with $L_0$.
Thus, we have that
\begin{equation} 
	\Phi_a(f)c=ZJY(a,\overline{f}\circ j)ZJc
	\qquad
	\forall f\in C_{p(a)}^\infty(S^1)
	\,\,\,\forall c\in \mathcal{H}^\infty\,.
\end{equation}
Now, $Y(a,\overline{f}\circ j)$ is affiliated to $\A_\F(j(I))$ whenever $\mathrm{supp}f\subset I$ for some $I\in\J$. Therefore, if $\mathrm{supp}f\subset I$ for some $I\in\J$, then $ZJY(a,\overline{f}\circ j)ZJ=\Phi_a(f)$ is affiliated to $ZJ\A_\F(j(I))ZJ$, which is exactly $\A_\F(I)$ thanks to \eqref{eq:action_j_on_A_F}. 
Now, by the graded locality of $\A_\F$ and by Proposition \ref{prop:commutation_von_neumann_algebras_generated}, we can prove that for all $a,b\in\F$, $\Phi_a(z)$ and $Y(b,z)$ are mutually local in the Wightman sense as defined in Appendix \ref{appendix:wightman_locality}. By Proposition \ref{prop:wightman_locality}, for all $a,b\in\F$, $\Phi_a(z)$ and $Y(b,z)$ are mutually local in the vertex superalgebra sense.
Combining the generating property of $\F$ and Dong's Lemma \cite[Lemma 3.2]{Kac01}, we can then prove that $\Phi_a(z)$ and $Y(b,z)$ are mutually local (in the vertex superalgebras sense) for all $a\in\F$ and all $b\in V$. Noting that $\Phi_a(z)=ZJY(a,z)ZJ$, we have that also $Y(a,z)$ and $\Phi_b(z)$ are mutually local for all $a\in \F$ and all $b\in V$. Thus, using again that $\F$ generates $V$ and Dong's Lemma \cite[Lemma 3.2]{Kac01}, we can conclude that $\Phi_a(z)$ and $Y(b,z)$ are mutually local for all $a, b\in V$. Using the uniqueness theorem for vertex superalgebras \cite[Theorem 4.4]{Kac01}, it follows that $\Phi_a(z)=Y(ZJa,z)$ for all $a\in V$ and thus $ZJ$ defines an antilinear automorphism of the VOSA $V$.

\textit{Step 4}.
Now, we want to prove the useful formula:
\begin{equation} \label{eq:S_acts_smeared_vertex_ops}
	SY(a,f)\Omega=Y(a,f)^*\Omega
\end{equation}
for all quasi-primary element $a\in V$ and all $f\in C_{p(a)}^\infty(S^1)$ with $\mathrm{supp}f\subset S^1_+$.
To this aim, fix $a\in\F$ and $f\in C_{p(a)}^\infty(S^1)$ with $\mathrm{supp}f\subset S^1_+$. We have that
\begin{equation}
	\theta Z\Delta^\frac12 Y(a,f)\Omega
	=\theta Ze^\frac{K}{2} Y(a,f)\Omega
	=\theta Z i^{2d_a} Y(a,f\circ j)\Omega
	=Y(a,f)^*\Omega
	=J\Delta^\frac12 Y(a,f)\Omega
\end{equation}
where we have used: the Bisognano-Wichmann property \eqref{eq:B-W_dilations} for $\A_\F$ for the first equality; Theorem \ref{theo:delta_one_half} (see p.\ \pageref{notations_conventions} for notations) for the second one; equation \eqref{eq:smeared_vertex_operator_adjoint} for the third one; the fact that $Y(a,f)$ is affiliated with $\A_\F(S^1_+)$ for the last one, cf.\ \cite[Proposition 2.5.9]{BR02}.
Consequently, using \eqref{eq:theta_z}, we have that
\begin{equation} \label{eq:ZJDelta_thetaDelta}
	ZJ\Delta^\frac12 Y(a,f)\Omega = \theta \Delta^\frac12 Y(a,f)\Omega \,.
\end{equation} 
Using that $ZJ$ and $\theta$ both commute with $L_n$ for all $n\in\{-1,0,1\}$ and again the Bisognano-Wichmann property \eqref{eq:B-W_dilations} for $\A_\F$, we prove that $ZJ\Delta^\frac12 ZJ$ and $\theta\Delta^\frac12\theta$ are both equal to $\Delta^{-\frac12}$. This implies from \eqref{eq:ZJDelta_thetaDelta} that $ZJY(a,f)\Omega$ is equal to $\theta Y(a,f)\Omega$. Again $ZJ$ and $\theta$ both commute with $L_0$ and thus $ZJY(a,\iota_{d_a}(r^{(2)}(t))f)\Omega$ is equal to $\theta Y(a,\iota_{d_a}(r^{(2)}(t))f)\Omega$ for all $t\in\R$. Therefore, using a partition of unity, we prove that $ZJY(a,f)\Omega$ is equal to $\theta Y(a,f)\Omega$ for all $f\in C_{p(a)}^\infty(S^1)$ and consequently it must be $(ZJ)(a)=\theta (a)$. From the arbitrariness of $a\in\F$, which generates $V$ and the fact that $ZJ$ and $\theta$ are antilinear automorphism, it follows that $ZJ=\theta$. By Theorem \ref{theo:delta_one_half}, we can conclude that for every quasi-primary element $a\in V$ and every $f\in C_{p(a)}^\infty(S^1)$ with $\mathrm{supp}f\subset S^1_+$, $Y(a,f)\Omega$ is in the domain of the operator $S$ and equation \eqref{eq:S_acts_smeared_vertex_ops} holds.

\textit{Step 5}.
In this last step, we are going to prove that $Y(a,f)$ is affiliated to $\A_\F(I)$ whenever $a\in V$ is any quasi-primary vector, $f\in C_{p(a)}^\infty(S^1)$ with $\mathrm{supp}f\subset S^1_+$ and $I\in\J$ containing the closure of $S^1_+$. This will bring us to conclude that $\A_V(I)\subseteq \A_\F(I)$ for all $I\in\J$, proving the theorem.
Let $I\in\J$ containing the closure of $S^1_+$ and let $A\in\A_\F(I')\subseteq\A_\F(S^1_-)$. Then by the M\"{o}bius covariance of $\A_\F$, there exists $\delta>0$ such that $e^{itL_0}Ae^{-itL_0}\in \A_\F(S^1_-)$ for all $t\in(-\delta,\delta)$. Proceeding similarly to the proof of \cite[Lemma 6.5]{CKLW18} (cf.\ the proof of Lemma \ref{lem:gen_by_qp}), it is possible, for all $s\in(0,\delta)$, to construct an operator $A(\varphi_s)$ such that (cf.\ \cite[Lemma 5.3]{DSW86} for a similar argument)
\begin{equation}   \label{eq:properties_A(varphi_s)}
	A(\varphi_s)\in\A_\F(S^1_-)
	\,,\quad
	A(\varphi_s)c\in\mathcal{H}^\infty
	\quad\forall c\in \mathcal{H}^\infty
	\,,\quad
	\lim_{s\rightarrow0}A(\varphi_s)c=Ac  
	\quad\forall c\in\mathcal{H}
	\,.
\end{equation}
Now, let $X_1, X_2\in\mathcal{P}_\F(S^1_-)$ and $B\in\A_\F(S^1_+)\subseteq Z\A_\F(S^1_-)'Z^*$ by graded locality, we have:
\begin{equation}
	\begin{split}
	(ZX_1^*A(\varphi_s)X_2\Omega| SB\Omega) 
	&= (ZX_1^*A(\varphi_s)X_2\Omega| B^*\Omega)  \\
	&= (ZZ^*BZX_1^*A(\varphi_s)X_2\Omega| \Omega) \\
	&= (ZX_1^*A(\varphi_s)X_2Z^*B\Omega| \Omega) \\
	&= (B\Omega| ZX_2^*A(\varphi_s)^*X_1\Omega) \,. 
	\end{split}
\end{equation}
where we have used \eqref{eq:properties_A(varphi_s)}, the properties of $Z$ as in Section \ref{subsection:basic_definitions_VOSA} and the commutation relation between smeared vertex operators and affiliated bounded operators. Because $S$ is antilinear, the equation above implies that $ZX_1^*A(\varphi_s)X_2\Omega$ is in the domain of $S^*$ and
\begin{equation}   \label{eq:action_Sstar}
	S^*ZX_1^*A(\varphi_s)X_2\Omega= ZX_2^*A(\varphi_s)^*X_1\Omega\,.
\end{equation}
Fix a quasi-primary vector $a\in V$ and a function $f\in C_{p(a)}^\infty(S^1)$ with $\mathrm{supp}f\subset S^1_+$. For every $X_1, X_2\in\mathcal{P}_\F(S^1_-)$ and every $s\in(0,\delta)$, we have the following equalities
\begin{equation}
	\begin{split}
	(X_1\Omega| A(\varphi_s)Z^*Y(a,f)ZX_2\Omega)
	&= (X_1\Omega| A(\varphi_s)X_2Z^*Y(a,f)\Omega) \\
	&= (ZX_2^*A(\varphi_s)^*X_1\Omega| Y(a,f)\Omega) \\
	&= (S^*ZX_1^*A(\varphi_s)X_2\Omega| Y(a,f)\Omega) \\
	&= (SY(a,f)\Omega| ZX_1^*A(\varphi_s)X_2\Omega) \\
	&= (Y(a,f)^*\Omega| ZX_1^*A(\varphi_s)X_2\Omega) \\
	&= (X_1Z^*Y(a,f)^*Z\Omega| A(\varphi_s)X_2\Omega) \\
	&= (Z^*Y(a,f)^*ZX_1\Omega| A(\varphi_s)X_2\Omega) \\
	&= (X_1\Omega| Z^*Y(a,f)ZA(\varphi_s)X_2\Omega)
	\end{split}
\end{equation}
where we have used: the twisted Haag duality \eqref{eq:twisted_haag_duality} for $\A_\F$ with Proposition \ref{prop:commutation_von_neumann_algebras_generated} for the first and for the seventh equality; \eqref{eq:action_Sstar} for the third one; equation \eqref{eq:S_acts_smeared_vertex_ops} for the fifth one; the usual properties of $Z$ as in Section \ref{subsection:basic_definitions_VOSA} and \eqref{eq:properties_A(varphi_s)} in general. As we have proved in \textit{Step 2} that $\mathcal{P}_\F(S^1_-)\Omega$ is dense in $\mathcal{H}$ and therefore
\begin{equation}
	A(\varphi_s)Z^*Y(a,f)ZX\Omega=Z^*Y(a,f)ZA(\varphi_s)X\Omega
	\qquad\forall X\in\mathcal{P}_\F(S^1_-)
	\,\,\,\forall s\in(0,\delta)
\end{equation}
which implies, thanks to the properties in \eqref{eq:properties_A(varphi_s)}, that $AX\Omega$ is in the domain of $Z^*Y(a,f)Z$ for all $X\in\mathcal{P}_\F(S^1_-)$ and that
\begin{equation} \label{eq:A_commutes_Zstar_Y_Z}
	AZ^*Y(a,f)ZX\Omega=Z^*Y(a,f)ZAX\Omega
	\qquad\forall X\in\mathcal{P}_\F(S^1_-)\,.
\end{equation}
Moreover, $\mathcal{P}_\F(S^1_-)\Omega$ is a core for all $(L_0+1_\mathcal{H})^k$ with $k\in\Zplus$. Indeed, choose $I_1\in\J$ such that $\overline{I_1}\subset S^1_-$. Then there exists a positive real number $\delta$ such that $e^{it}I_1\subset S^1_-$ for all $t\in(-\delta,\delta)$. By the M\"{o}bius covariance of the vertex operators, $U(r^{(2)}(t))\mathcal{P}_\F(I_1)\Omega\subseteq\mathcal{P}_\F(S^1_-)\Omega$ for all $t\in(-\delta,\delta)$, see \textit{Step 2}. By \cite[Lemma 7.2]{CKLW18}, it follows that $\mathcal{P}_\F(S^1_-)\Omega$ is a core for every $(L_0+1_\mathcal{H})^k$ with $k\in\Zplus$ and thus it is also a core for $Z^*Y(a,f)Z$ thanks to \eqref{eq:bound_smeared_vertex_ops}.
Therefore, \eqref{eq:A_commutes_Zstar_Y_Z} implies that $AZ^*Y(a,f)Z\subseteq Z^*Y(a,f)ZA$. 
By the arbitrariness of the choices made above, it follows that $Y(a,f)$ is affiliated with $Z\A_\F(I')'Z^*=\A_\F(I)$ (twisted Haag duality \eqref{eq:twisted_haag_duality}) for all $I\in\J$ which contains the closure of $S^1_+$, all quasi-primary $a\in V$ and all $f\in C_{p(a)}^\infty(S^1)$ with $\mathrm{supp}f\subset S^1_+$.  Applying Proposition \ref{prop:gen_by_qp}, we get that $\A_V(S^1_+)\subseteq \A_\F(I)$ for all $I\in\J$ which contains the closure of $S^1_+$. The external continuity \eqref{eq:external_continuity} for $\A_\F$ concludes the proof as:
\begin{equation}
	\A_V(S^1_+)\subseteq\bigcap_{\overline{S^1_+}\subset I}\A_\F(I)=\A_\F(S^1_+) \,.
\end{equation}
\end{proof}

\begin{rem}   \label{rem:diffeo_covariance_without_strong_graded-locality}
Let $V$ be a simple strongly graded-local VOSA. As a representation of the Virasoro algebra, $V$ is generated by the conformal vector and by primary elements, see e.g.\ \cite[Remark 3.9]{CTW22}. Then thanks to Theorem \ref{theo:gen_by_quasi_primary} and Remark \ref{rem:diff_covariance_primary_elems}, we can deduce the diffeomorphism covariance of the net $\A_{(V,\scalar)}$ by an alternative argument to the one given in the proof of Theorem \ref{theo:diffeo_covariant_net_from_V}.
\end{rem}

\begin{cor}   
\label{cor:strongly_graded-local_tensor_product}
Let $V^1$ and $V^2$ be simple strongly graded-local unitary VOSAs. Then $V^1\hat{\otimes}V^2$ is strongly graded-local and $\A_{V^1\hat{\otimes}V^2}=\A_{V^1}\hat{\otimes}\A_{V^2}$.
\end{cor}

\begin{proof}
$V^1\hat{\otimes}V^2$ is energy-bounded by Corollary \ref{cor:energy_bound_graded_tensor_product}. Call $\F^1$ and $\F^2$ the sets of all quasi-primary vectors of $V^1$ and $V^2$ respectively. $V^1\hat{\otimes}V^2$ is generated by the set of quasi-primary vectors $\F:=(\F^1\otimes\Omega^2)\cup(\Omega^1\otimes \F^2)$. Moreover, $\A_\F(I)=\A_{V^1}(I)\hat{\otimes}\A_{V^2}(I)$ for all $I\in\J$. In particular, if we call $Z_1,Z_2$ and $Z$ the operators \eqref{eq:defin_Z_net} for $\A_{V^1}, \A_{V^2}$ and $\A_{V^1}\hat{\otimes}\A_{V^2}$ respectively, then we have that
\begin{equation}
\begin{split}
\A_\F(I')
  &=
\A_{V^1}(I')\hat{\otimes}\A_{V^2}(I') 
=\left(Z_1\A_{V^1}(I)'Z_1^*\right)\hat{\otimes}
\left(Z_2\A_{V^2}(I)'Z_2^*\right)  \\
  &\subseteq 
Z\left(\A_{V^1}(I)\hat{\otimes}\A_{V^2}(I)\right)'Z^*  
=Z\A_\F(I)'Z^*
\qquad\forall I\in\J
\end{split}
\end{equation}
where we have used the twisted Haag duality \eqref{eq:twisted_haag_duality} and \eqref{eq:graded_tensor_product_Z}. Then we get the desired result just applying Theorem \ref{theo:gen_by_quasi_primary}.
\end{proof}

Here below, another useful consequence of Theorem \ref{theo:gen_by_quasi_primary}:
\begin{theo}   \label{theo:gen_by_V1/2_V1_F}
Let $V$ be a simple unitary VOSA generated by $V_\half\cup V_1\cup\F$ where $\F\subseteq V_2$ is a family of quasi-primary $\theta$-invariant Virasoro vectors. Then $V$ is strongly graded-local and $\A_V=\A_{V_\half\cup V_1\cup\F}$.
\end{theo}

\begin{proof}
We use the notation introduced in \eqref{eq:notation_even_odd}.
First, 
by Proposition \ref{prop:gen_by_V1/2_V1} and its proof, $V$ is energy-bounded and in particular, every smeared vertex operator $Y(a,f)$ with $a\in V_\frac12\cup V_1\cup\F$ and $f\in C_{p(a)}^\infty(S^1)$ satisfies linear energy bounds.
Second,
using Proposition \ref{prop:wightman_locality}, we have that the vertex operators $Y(a,z)$ with $a\in V_\frac12\cup V_1\cup\F$ are mutually local in the Wightman sense.
Therefore, we can use the argument in \cite[p.\ 113]{BS90}, based on \cite[Theorem 3.1]{DF77}, see also \cite[Theorem 19.4.4]{GJ87}, \cite[Theorem C.2]{Tan16} and \cite[Theorem 3.4]{CTW22}, to show that the family of von Neumann algebras $\A_{V_\frac12\cup V_1\cup\F}(I)$ with $I\in\J$ satisfies the graded-locality condition required by Theorem \ref{theo:gen_by_quasi_primary}. Therefore, we can conclude that $V$ is strongly graded-local and that $\A_{V_\frac12\cup V_1\cup\F}=\A_V$ by the same theorem.
\end{proof}

The following result is inspired by the proof of \cite[Theorem 5.5]{CGGH23}.

\begin{theo}  \label{theo:V_strongly_graded-local_iff_V_0_and_v_are}
	Let $V$ be a simple energy-bounded unitary VOSA.
	Then $V$ is strongly graded-local if and only if $V_\parzero$ is a simple strongly local unitary VOA and there exists a quasi-primary vector $v\in V_\parone$ such that $\A_{\{v\}}(I')\subseteq Z\A_{V_\parzero\cup\{v\}}(I)'Z^*$ for some $I\in\J$.
\end{theo}

\begin{proof}
We use the notation introduced in Definition \ref{defin:hilbet_space_from_V} for $V$, as well as in \eqref{eq:notation_even_odd} and in \eqref{eq:defin_net_gen_by_subset_F}.	
	
If $V$ is strongly graded-local, then the result follows from Theorem \ref{theo:covariant_subnets_unitary_subalgebras}.

Vice versa, set $\F:=\{v, \theta(v)\}$.	By Step 1 in the proof of Theorem \ref{theo:gen_by_quasi_primary}, $\A_{\{v\}}(I)=\A_{\F}(I)$ for all $I\in\J$. Moreover, $\A_\F$ is isotonous and it is also M\"{o}bius covariant by Proposition \ref{prop:mob_covariance_qp}. Similarly, $\A_{V_\parzero\cup\F}$ is isotonous and M\"{o}bius covariant too. Thus, by M\"{o}bius covariance, $\A_\F(I')\subseteq Z\A_{V_\parzero\cup \F}(I)'Z^*$ for all $I\in\J$.	
Call $W$ the subalgebra $W(\F)$ generated by $\F$. Using \eqref{eq:cr_l1}, it is not difficult to show that $W$ is $L_1$-invariant and thus $W$ is a unitary subalgebra of $V$ by Proposition \ref{prop:characterisation_unitary_subalgebras}.
	
Let $\A_{V_\parzero}$ be the irreducible (local) conformal net arising from the simple strongly local unitary VOA $V_\parzero$. Moreover, as $V$ is energy-bounded, we can define an isotonous family of von Neumann algebras on $\mathcal{H}$ from the smeared vertex operators $Y(a,f)$ for all $a\in V_\parzero$ and all $f\in C^\infty(S^1)$. 
To avoid confusion, we call it $\B$ instead of using the notation fixed in \eqref{eq:defin_net_gen_by_subset_F}.
By Proposition \ref{prop:mob_covariance_qp}, we have that $\B$ is also M\"{o}bius coviant. In the following, we are going to prove that it is (graded-)local too.

Let $E$ be the projection from $\mathcal{H}$ onto the Hilbert space completion $\mathcal{H}_\parzero$ of $V_\parzero$ with respect to $\scalar$. Proceeding as in the third step of the proof of \cite[Theorem 5.5]{CGGH23}, we prove that $E\B(I)E=\A_{V_\parzero}(I)$ for all $I\in\J$. 

Fix $I\in\J$ and let $I_1, I_2\in\J$ be such that $\overline{I_1}\subset I$ and $\overline{I_2}\subset I'$. 
Consider $A_1\in\B(I_1)$ and $A_2\in\B(I_2)$ and set $M:=V\cap\Ker([A_1,A_2])$. Since $EA_jE\in\A_{V_\parzero}(I_j)$ for all $j\in\{1,2\}$, we have that $[A_1,A_2]b=0$ for all $b\in V_\parzero$, that is $V_\parzero\subseteq M$. Furthermore, by hypotheses, if $J\in\J$ is such that $J\cap I_j=\emptyset$ for all $j\in\{1,2\}$, then $[A_1,A_2]Y(a,f)b=Y(a,f)[A_1,A_2]b=0$ for all $a\in \F$, all $f\in C^\infty_\chi(S^1)$ with $\mathrm{supp}f\subset J$ and all $b\in M$.
Now, fix any $J\in \J$ as above and for all $c\in V$, let $\mathcal{K}_c$ be the closed subspace generated by vectors of type $Y(a,f)c$ where $a\in \F$ and $f\in C^\infty_\chi(S^1)$ with $\mathrm{supp}f\subset J$. By a Reeh-Schlieder argument as in the Step 2 of the proof of Theorem \ref{theo:gen_by_quasi_primary}, we can prove that $\mathcal{K}_c$ contains all the vectors of type $Y(a,f)c$ with $a\in \F$ and $f\in C^\infty_\chi(S^1)$. This implies that $a_{(n)}c\in \mathcal{K}_c$ for all $a\in \F$ and all $n\in \Z$, so that $a_{(n)}b\in M$ for all $a\in \F$, all $n\in \Z$ and all $b\in M$. It follows from the Borcherds identity \eqref{eq:borcherds_id} that $M$ is a $W$-submodule of $V$ containing $V_\parzero$ and thus $M$ contains $W\cdot V_\parzero$, that is the linear span of vectors of type $a_{(n)}b$ with $a\in W$, $n\in\Z$ and $b\in V_\parzero$. By skew-symmetry \eqref{eq:skew-symmetry}, $V_\parzero\cdot W\subseteq M$, that is the linear span of vectors of type $b_{(n)}a$ with $b\in V_\parzero$, $n\in\Z$ and $a\in W$. Hence, $V_\parone\subset M$, see Example \ref{ex:even_simple_odd_irreducible}, so that $V\subseteq M$. This implies that $[A_1,A_2]=0$ and thus $\B(I_2)\subseteq \B(I_1)'$. By M\"{o}bius covariance, it follows that $\B(I')\subseteq \B(I)'$, that is $\B$ is local.

We know $V_\parzero$ is generated by a set of primary vectors and the conformal vector $\nu$ (as it is a sum of irreducible positive-energy unitary representation of the underlying Virasoro algebra, see e.g.\ \cite[Remark 3.9]{CTW22}), say $\F_\parzero$.
Recall from Example \ref{ex:even_simple_odd_irreducible} that $V$ is generated by $V_\parzero$ and $v$ as $V_\parzero$-module.
Hence, we can conclude that $V$ is strongly graded-local by applying Theorem \ref{theo:gen_by_quasi_primary} to the subset of quasi-primary generators $\F_\parzero\cup\F$.
\end{proof}

\begin{rem}
	With no major difficulties, Theorem \ref{theo:V_strongly_graded-local_iff_V_0_and_v_are} can be proved replacing $V_\parzero$ and $v$ with a generic unitary subalgebra $W$ and a generic subset of quasi-primary vectors $\F$ of $V$ respectively, and asking in addition that $V$ is generated by $W\cup \F$ as $W$-module (this last condition is automatic when $W=V_\parzero$ and $\F=\{v\}$ by Example \ref{ex:even_simple_odd_irreducible}).
\end{rem}

\section{Examples and superconformal structures}  \label{section:examples}

At the present, we are ready to construct the most famous examples of graded-local conformal nets starting from VOSAs through the machinery which we have developed throughout the present paper. Most of the VOSA examples are constructed starting from Lie superalgebras (see \cite{Kac77}), following the general theory developed in \cite[Section 4.7 and Section 4.10]{Kac01}; further useful references will usually be given for every specific example. Moreover, we introduce the superconformal structure in both frameworks, showing how they are related. We conclude by proving the strong graded locality of VOSAs, which gives us new models of irreducible graded-local conformal nets, but relying also on tools and results coming from the representation theory of VOAs.

\begin{ex}[\textit{The free fermion}]
\label{ex:free_fermion}
Let $A$ be a one-dimensional $\C$-vector space equipped with a non-degenerate symmetric bilinear form $\Bilinear$. We can always identify $A$ with $\C\varphi$, where $\varphi\in A$ is such that $\B(\varphi,\varphi)=1$.
As in \cite[Section 2.5]{Kac01}, we consider the \textit{Clifford affinization} of $A$, i.e., the Lie superalgebra
\begin{equation}
\mathfrak{f}:=\overbrace{A\otimes_\C \C[t,t^{-1}]}^\mathrm{odd}\oplus \overbrace{\C K}^\mathrm{even}
\end{equation}
with commutation relations:
\begin{equation}   \label{eq:cr_fermion_algebra}
	[\mathfrak{f},K] =0 \,,\qquad
[\alpha\varphi_n,\beta\varphi_m] =\alpha\beta\delta_{n,-m}K
\qquad
\forall \alpha,\beta\in\C
\,\,\,
\forall n,m\in\Z-\half 
\end{equation}
where $\varphi_n:=\varphi\otimes t^{n-\half}$ for all $n\in\Z-\half$. $\mathfrak{f}$ is known as \textit{fermion algebra} and we sometimes call $\varphi$ the \textit{free fermion vector} of $A$.
Let $U(\mathfrak{f})$ be the universal enveloping algebra associated to $\mathfrak{f}$. We denote the generalized Verma module $V^1(\mathfrak{f})$ associated to $\mathfrak{f}$ by
\begin{equation}
F:=V^1(\mathfrak{f})=U(\mathfrak{f})/
\langle 
K-1 
\,, \,\,\,
\varphi_{n_s}\cdots\varphi_{n_1}
\mid n_j\geq\half 
\rangle
\,.
\end{equation} 
$V^1(\mathfrak{f})$ has a structure of simple VOSA, see e.g.\ \cite[Theorem 3.6, Theorem 4.7 and Proposition 4.10(b)]{Kac01}, called \textit{neutral} or \textbf{real free fermion VOSA}, which we sum up in the following. $F$ is linearly generated by the vacuum vector $\Omega$, which is the class of $1\in\C$ in $F$, and the following ordered strings of vectors in $U(\mathfrak{f})$ with related parities:
$$
\left\{
\varphi_{n_s}\cdots\varphi_{n_1}\Omega \mid 
-\half\geq n_1\geq\cdots\geq n_s
\right\},
\quad
p(\varphi_{n_s}\cdots\varphi_{n_1}\Omega)=\overline{s}\in\Z_2\,.
$$ 
The state-field correspondence can be derived from
\begin{equation}
Y(\varphi_{-\half}\Omega,z)
=\sum_{n\in\Z}(\varphi_{-\half}\Omega)_{(n)}z^{-n-1}  
\,,\qquad
(\varphi_{-\half}\Omega)_{(n)}:=\varphi_{n+\half} \quad\forall n\in\Z\,.
\end{equation}
The vector $\nu:=\half\varphi_{-\frac{3}{2}}\varphi_{-\half}\Omega$ gives a Virasoro algebra with central charge $c=\half$, where
\begin{equation}
\begin{split}
L_0\varphi_{-n_s}\cdots\varphi_{-n_1}\Omega
 &=
\left(\sum_{j=1}^s n_j\right)
\varphi_{-n_s}\cdots\varphi_{-n_1}\Omega 
\qquad\forall s\in\Zplus \,\,\,\forall n_1,\dots, n_s\in\Zplus-\half\\
L_{-1} \varphi_{-n}\Omega
 &=
n\varphi_{-n-1}\Omega
\quad \forall n\in\Zplus-\half\,. 
\end{split}
\end{equation}
Clearly, $F$ is generated by the $L_0$-eigenspace $F_\half=\C\varphi_{-\half}\Omega$, where $\varphi_{-\half}\Omega$ is a primary vector.
It is known, see \cite[Lecture 4]{KR87} and \cite[Section 4]{KT85}, that $F$ has a scalar product $\scalar$ such that $(\Omega|\Omega) =1$ and
\begin{equation}
\begin{split}
(\varphi_na|b) 
 &= 
(a|\varphi_{-n}b) 
\qquad
\forall n\in\Z-\half
\,\,\,
\forall a,b\in F\\
(L_ma|b) 
 &= 
(a|L_{-m}b) 
\qquad
\forall m\in\Z
\,\,\,
\forall a,b\in F
\,.
\end{split}
\end{equation}
Therefore, $F$ has a unitary structure by Proposition \ref{prop:simple_unitary_vosa_gen_by_qp_hermitian_fields}, which is even unique, up to unitary isomorphism, by simplicity, see Proposition \ref{prop:uniqueness_unitary_structure}. Alternatively, the unitarity of $F$ can be established as in \cite[Section 2.3]{AL17}.
To sum up, $F$ is a simple unitary VOSA generated by a primary vector of conformal weight $\half$ and thus it is strongly graded-local thanks to Theorem \ref{theo:gen_by_V1/2_V1_F}. This means that $\mathcal{F}:=\A_F$ is an irreducible graded-local conformal net, actually known as the \textit{neutral} or \textbf{real free fermion net}. Finally, note that $\mathcal{F}$ is exactly the net of graded-local algebras of the chiral Ising model affiliated with the Majorana field, which gives the generators of the Araki's self-dual CAR-algebra, see \cite[Section 2.1]{Boc96}
and \cite[Section 2]{Ara70}. In particular, the Majorana field $\psi(z)$ in \cite[Section 2.1]{Boc96} is our vertex operator $Y(\varphi_{-\half}\Omega,z)$.
\end{ex}

\begin{ex}[$d$ \textit{free fermions}] \label{ex:d_free_fermions}
By Corollary \ref{cor:strongly_graded-local_tensor_product}, $F^d:=\hat{\bigotimes}_{j=1}^dF$ for $d\in\Zplus$ is a simple strongly graded-local unitary VOSA with central charge $c=\frac{d}{2}$, which we call the $d$ \textbf{free fermions VOSA}. Furthermore, we have an irreducible graded-local conformal net:
\begin{equation}
\mathcal{F}^d:=\hat{\bigotimes}_{j=1}^d\mathcal{F}
=\mathcal{A}_{F^d} \,,
\end{equation}
called the $d$ \textbf{free fermions net}. In particular, $\mathcal{F}^2$ is known as the \textit{charged} or \textbf{complex free fermion net}. It is worthwhile to note that $F^d$ can be alternatively realized through a procedure similar to the one in Example \ref{ex:free_fermion}, starting with a $d$-dimensional $\C$-vector space $A$ equipped with a non-degenerate symmetric bilinear form, see e.g.\ \cite[Section 4.3]{Li96}, \cite[Proposition 4.10(b)]{Kac01} and \cite[Section 2.3]{AL17}. 
\end{ex}

\begin{rem} It follows from Example \ref{ex:d_free_fermions} and Theorem \ref{theo:covariant_subnets_unitary_subalgebras} that every unitary subalgebra $W$ of the 
free fermion VOSA $F^d$ is strongly graded-local. In particular, if $W$  is contained in the even subalgebra $F^d_\parzero$ of $F^d$ then $W$ is a strongly local \textit{VOA}. Actually, the latter fact can be proved directly in the setting of \cite{CKLW18} without using the theory developed in this paper, see \cite{CWX}.
\end{rem}

\begin{rem}
	In \cite{Ten17}, see also \cite{Ten19, Ten19b, Ten24}, the author construct a chiral CFT model of the charged free fermion using the Graeme Segal approach to CFT \cite{Seg04}, called \textit{Segal CFT}. 
	Without going into details, it is shown that Segal's approach allows to construct graded-local conformal nets from simple unitary VOSAs through the so called \textit{bounded localized vertex operators}, see \cite[Definition 4.7 and Proposition 4.8]{Ten19}, and a charged free fermion net is given.
	This also allows to construct several other chiral CFT models in the Haag-Kastler approach, as free bosons, some of WZW, lattice, Virasoro and super-Virasoro models, see \cite[Example 4.14 -- Example 4.17]{Ten19}, using that some of the VOSA versions of those models can be included, as unitary subalgebras, in some $d$ free fermions VOSA $F^d$ of Example \ref{ex:d_free_fermions}.
	Thanks to Example \ref{ex:d_free_fermions}, Theorem \ref{theo:covariant_subnets_unitary_subalgebras}, Corollary \ref{cor:strongly_graded-local_tensor_product} and their Segal's CFT counterparts, see \cite[Remark 4.18]{Ten19}, the graded-local conformal nets constructed via the two different approaches (from the same simple unitary VOSAs) actually coincide.
\end{rem}

\begin{ex}[$N=1$ \textit{super-Virasoro models}]  \label{ex:N=1_super-Virasoro}
The \textit{Neveu-Schwarz} or $N=1$ \textit{super-Virasoro algebra} $NS$ is the Lie superalgebra
\begin{equation}
NS:=
\overbrace{
\bigoplus_{m\in\Z}\C L_m\oplus
\C C
}^\mathrm{even}
\oplus
\overbrace{
\bigoplus_{n\in\Z-\half}\C G_n 
}^\mathrm{odd}
\end{equation}
with the commutation relations:
\begin{equation}  \label{eq:NS_cr}
\begin{split}
[L_m,L_n]  &:=
(m-n)L_{m+n}+\frac{m^3-m}{12}\delta_{m,-n}C
\qquad
\forall m,n\in\Z \\
[L_m,G_n] &:= 
\left(\frac{m}{2}-n\right) G_{m+n}
\qquad
\forall m\in\Z
\,\,\,
\forall n\in\Z-\half \\
[G_m,G_n] &:=
2L_{m+n}+\frac{1}{3}\left(m^2-\frac{1}{4}\right)\delta_{m,-n}C
\qquad
\forall m,n\in\Z-\half \\
[NS,C] &:= 0 \,.
\end{split}
\end{equation}
Let $U(NS)$ be the universal enveloping algebra of $NS$ and define, for every $c\in\C$, the $NS$-module
\begin{equation}
\widetilde{V}^c(NS):=U(NS)/
\langle
\bigoplus_{m\in\Zpluseq-1}\C L_m\oplus
\bigoplus_{n\in\Zpluseq-\half}\C G_n 
\,,\,\,\,
C-c
\rangle  \,.
\end{equation}
Let $J^c$ be the maximal $NS$-submodule of $\widetilde{V}^c(NS)$, then the quotient $NS$-module $V^c(NS):=\widetilde{V}^c(NS)/J^c$ has a unique structure of simple VOSA with central charge $c\in\C$ by \cite[Theorem 4.7 and Lemma 5.9]{Kac01}, see also \cite[Section 4.2]{Li96} and \cite[Section 3.1]{KW94}. They are known as \textit{Neveu-Schwarz} or \textbf{$N=1$ super-Virasoro VOSAs}. More specifically, every $V^c(NS)$ is generated by the conformal vector $\nu=L_{-2}\Omega$ and the so called \textbf{superconformal vector} $\tau:=G_{-\frac{3}{2}}\Omega$ with state-field correspondence:
\begin{equation}
Y(\nu,z)=\sum_{n\in\Z}L_nz^{-n-2}
\,,\qquad
Y(\tau,z)=\sum_{n\in\Z-\half}G_nz^{-n-\frac{3}{2}} \,.
\end{equation}
In particular, from the commutation relations \eqref{eq:NS_cr}, we have that $\tau$ is a primary vector:
\begin{equation}  \label{eq:tau_is_primary}
L_1\tau =-[G_{-\frac{3}{2}},L_1]\Omega=2G_{-\half}\Omega=0 
\,,\quad
L_2\tau =-[G_{-\frac{3}{2}},L_2]\Omega=\frac{5}{2}G_\half\Omega=0 \,.
\end{equation}
$L_0$ produces a $\half\Z$-grading of the vector space given by the following action on the ordered elements: for any positive integers $s,r$, any integers $m_s\leq\cdots\leq m_1\leq-2$ and any semi-integers $n_r\leq\cdots\leq n_1\leq-\frac{3}{2}$
\begin{equation}
L_0L_{m_s}\cdots L_{m_1}G_{n_r}\cdots G_{n_1}\Omega
=\left(-\sum_{j=1}^s m_j-\sum_{j=1}^r n_j\right)
L_{m_s}\cdots L_{m_1}G_{n_r}\cdots G_{n_1}\Omega
\,.
\end{equation}
Moreover, by \cite{FQS85}, \cite[Section 4]{GKO86}, \cite[Theorem 5.1]{KW86}, if 
\begin{equation}  \label{eq:values_cc_unitary_NS}
\mbox{ either } \quad
c\geq\frac{3}{2}
\quad
\mbox{ or }
\quad
c=\frac{3}{2}\left(
1-\frac{8}{m(m+2)}
\right) 
\quad
\mbox{for}
\quad
m> 2 
\end{equation}
then there exists a scalar product on $V^c(NS)$ with respect to which $Y(\nu,z)$ and $Y(\tau,z)$ are Hermitian fields. Using Proposition \ref{prop:simple_unitary_vosa_gen_by_qp_hermitian_fields}, $V^c(NS)$ has a unitary structure provided that $c$ is as in \eqref{eq:values_cc_unitary_NS}, see also \cite[Section 2.2]{AL17} for a different proof of the unitarity.
It follows from \cite[Eq.s (26) and (27)]{CKL08} and Proposition \ref{prop:energy_boundedness_by_generators} that every $V^c(NS)$ is energy-bounded. 
The family $\A_{\{\nu,\tau\}}$ of von Neumann subalgebras of $\A_{V^c(NS)}$ defined as in \eqref{eq:defin_net_gen_by_subset_F}, coincides with the so called $N=1$ \textbf{Super-Virasoro net} $\mathrm{SVir}_c$, defined in \cite[Eq.\ (32) and Theorem 33]{CKL08}, which satisfies the graded locality, see \cite[Eq.\ (33) and Theorem 33]{CKL08}. 
By Theorem \ref{theo:gen_by_quasi_primary}, we can conclude that for every $c$ as in \eqref{eq:values_cc_unitary_NS}, $V^c(NS)$ is strongly graded-local and $\A_{V^c(NS)}=\mathrm{SVir}_c$.
\end{ex}

The example of the $N=1$ super-Virasoro models enables us to talk about the superconformal structure which we naturally encounter in both the graded-local conformal net and the VOSA setting. In other words, we can wonder whether or not a given VOSA or a graded-local conformal net contains a $N=1$ super-Virasoro model. More precisely:

\begin{defin}    \label{defin:superconformal_net}
(See \cite[Definition 2.11]{CHL15}).
A graded-local conformal net $\A$ with central charge $c\in\C$ is said to be $N=1$ \textbf{superconformal} if it contains $\mathrm{SVir}_c$ as M\"{o}bius covariant subnet and
\begin{equation}
U(\Diff(I)^{(\infty)})\subset\mathrm{SVir}_c(I)\subseteq \A(I) \qquad \forall I\in\J 
\end{equation}
that is, $\mathrm{SVir}_c$ contains the Virasoro subnet $\mathrm{Vir}_c$ generated by the conformal symmetries, see \cite[Example 8.4]{CKLW18}.
\end{defin}
 
\begin{defin}    \label{defin:superconformal_vosa}
(See \cite[Definition 5.9]{Kac01}).
A VOSA $V$ with central charge $c\in\C$ is said to be $N=1$ \textbf{superconformal} if there exists a superconformal vector $\tau$ associated to the conformal vector $\nu$, that is, an odd vector whose operators $G_n$ with $n\in\Z-\half$, given by $Y(\tau,z)=\sum_{n\in\Z-\half}G_nz^{-n-\frac{3}{2}}$, satisfy the $N=1$ super-Virasoro algebra relations \eqref{eq:NS_cr}. 
Furthermore, if $V$ is unitary, then we say that $V$ is a \textbf{unitary $N=1$ superconformal VOSA} if $W(\{\nu,\tau\})$ is a unitary subalgebra.
\end{defin}

 Clearly, if $c$ is in the unitary series \eqref{eq:values_cc_unitary_NS}, then $V^c(NS)$ as in Example \ref{ex:N=1_super-Virasoro} is a simple unitary $N=1$ superconformal VOSA. 

\begin{lem}   \label{lem:automorphisms_NS}
Let $V^c(NS)$ be the $N=1$ super-Virasoro VOSA with central charge $c\in\C$. If $g$ is either a linear or an antilinear VOSA automorphism of $V^c(NS)$, then $g(\tau)=\pm \tau$. If $c$ is in the unitary series \eqref{eq:values_cc_unitary_NS}, then $V^c(NS)$ has a unique unitary structure and the unique PCT operator acts as the identity on the generators. Furthermore, 
$$
\Aut(\mathrm{SVir}_c)=\Aut_{\scalar}(V^c(NS))=\Aut(V^c(NS))\cong\Z_2 \,.
$$
\end{lem}

\begin{proof}
By definition, $V^c(NS)_\frac{3}{2}=\C\tau$, where $\tau=G_{-\frac{3}{2}}\Omega$. This means that there exists $\alpha\in\C\setminus\{0\}$ such that $g(\tau)=\alpha\tau$. By the commutation relations \eqref{eq:NS_cr}, see \cite[Proposition 5.9(i)]{Kac01}, we have that $G_{-\half}\tau=2\nu$ and thus
$$
\alpha^2G_{-\half}\tau=g(G_{-\half}\tau)
=2g(\nu)=2\nu=G_{-\half}\tau
$$
which implies that $\alpha=\pm 1$ as desired. This also implies that any PCT operator $\theta$ satisfies $\theta(\tau)=\tau$ by \eqref{eq:an_cross} and the fact that $Y(\tau,z)$ is a Hermitian field. Any PCT operator must be the antilinear vector space automorphism of $V^c(NS)$ acting as the identity on the generators $L_{m_s}\cdots L_{m_1}G_{n_r}\cdots G_{n_1}\Omega$, see Example \ref{ex:N=1_super-Virasoro}. 
(Cf.\ \cite[Section 2.3]{AL17}.)
Thus, any PCT operator commutes with every $g\in\Aut(V^c(NS))\cong \Z_2$, so that the remaining part follows from Theorem \ref{theo:characterization_aut_group} and from Theorem \ref{theo:AutV_AutNet}.
\end{proof}

Now, we can prove the following correspondence:
\begin{theo}   \label{theo:superconformal_correspondence}
Let $V$ be a simple unitary VOSA. Then we have the following:
\begin{itemize}
	\item[(i)] If $V$ is $N=1$ superconformal, then $V$ is unitary $N=1$ superconformal if and only if the superconformal vector is PCT-invariant. In both cases, $W(\{\nu,\tau\})$ is (up to isomorphism) the $N=1$ super-Virasoro VOSA $V^c(NS)$ for some $c$ in the unitary series \eqref{eq:values_cc_unitary_NS}.
	
	\item[(ii)] If $V$ is strongly graded-local, then $V$ is unitary $N=1$ superconformal if and only if $\A_V$ is $N=1$ superconformal.
\end{itemize}
\end{theo}

\begin{proof}
To prove (i), first suppose that the superconformal vector is PCT-invariant, that is $\theta(\tau)=\tau$. Then $W:=W(\{\nu,\tau\})$ is a unitary subalgebra of $V$ by Proposition \ref{prop:characterisation_unitary_subalgebras}.
Vice versa, suppose that $V$ is a simple unitary $N=1$ superconformal VOSA. If $W$ is a unitary subalgebra of $V$, then it has a structure of simple unitary VOSA by Proposition \ref{prop:unitary_structure_subalgebras}. Accordingly, it must be isomorphic to $V^c(NS)$ of Example \ref{ex:N=1_super-Virasoro} for some $c\in\C\setminus\{0\}$. Moreover, $c$ must be a positive real number because of, see \cite[Theorem 4.10 (a)]{Kac01},
$$
0<(\nu|\nu)=(L_{-2}\Omega|\nu)=(\Omega|L_2\nu)
=\frac{c}{2} \,.
$$
The PCT operator $\theta$ of $V$ restricts to an antilinear automorphism of $W$. By Lemma \ref{lem:automorphisms_NS}, $\theta(\tau)=\tau$, otherwise $\theta(\tau)=-\tau$ and from \eqref{eq:an_cross} applied to $\tau$ and \cite[Proposition 5.9 (ii)]{Kac01}, we have that
$$
0<(\tau|\tau)=(G_{-\frac{3}{2}}\Omega|\tau)
=(\Omega|(\theta (\tau))_\frac{3}{2}\tau)
=-\frac{2c}{3}<0
$$
which is impossible. 
Moreover, it implies that $Y(\tau,z)$ is a Hermitian field of $V$ and thus of $W$ with respect to $\scalar$. Hence, $(W,\scalar)$ realizes a unitary representation of $NS$ and thus $c$ must be in the unitary series \eqref{eq:values_cc_unitary_NS}, see Example \ref{ex:N=1_super-Virasoro} and references therein. 

To prove (ii), suppose that $V$ is a simple strongly graded-local unitary VOSA.
If $V$ is also unitary $N=1$ superconformal, then $W:=W(\{\nu,\tau\})$ is a unitary subalgebra and, by Theorem \ref{theo:covariant_subnets_unitary_subalgebras},
it gives rise to a covariant subnet $\A_W=\mathrm{SVir}_c$ of $\A_V$, containing the Virasoro subnet $\A_{W(\{\nu\})}=\mathrm{Vir}_c$ for some $c$ as in \eqref{eq:values_cc_unitary_NS}. Thus, $\A_V$ is an irreducible $N=1$ superconformal net.
Vice versa, if $\A_V$ is an irreducible $N=1$ superconformal net, then it contains the covariant subnets $\mathrm{Vir}_c\subset \mathrm{SVir}_c$ for some $c$ as in \eqref{eq:values_cc_unitary_NS}, where $\mathrm{Vir}_c=\A_{W(\{\nu\})}$. On the one hand, by Theorem \ref{theo:covariant_subnets_unitary_subalgebras}, there exists a unitary subalgebra $\widetilde{W}$ of $V$ such that $\A_{\widetilde{W}}=\mathrm{SVir}_c$. On the other hand, $\mathrm{SVir}_c= \A_{V^c(NS)}$ as in Example \ref{ex:N=1_super-Virasoro} and thus $\widetilde{W}= V^c(NS)$ by Theorem \ref{theo:unicity_irr_graded-local_cn}. This implies that $\widetilde{W}$ is a unitary subalgebra of $V$ generated by $\nu$ and a superconformal vector $\widetilde{\tau}$ associated to it, that is $V$ is unitary $N=1$ superconformal.
\end{proof}

Now, we look at those automorphisms preserving the $N=1$ superconformal structure in both settings.

\begin{defin}
Let $V$ be a $N=1$ superconformal VOSA. $\Aut^\mathrm{sc}(V)$ is the closed subgroup of $\Aut(V)$ formed by those automorphisms fixing the elements of $W(\{\nu,\tau\})$. If $V$ is also unitary, set 
$$
\Aut^\mathrm{sc}_{\scalar}(V):=\Aut^\mathrm{sc}(V)\cap\Aut_{\scalar}(V) \,.
$$ 
\end{defin}

Therefore, we have a result similar to Theorem \ref{theo:characterization_aut_group}:
\begin{theo}
Let $V$ be a unitary $N=1$ superconformal VOSA. Then $\Aut^\mathrm{sc}_{\scalar}(V)$ is a compact subgroup of $\Aut(V)$. Moreover, if $V$ is simple, then the following are equivalent:
\begin{itemize}  \label{theo:characterization_aut_sc_group}

\item[(i)] $\scalar$ is the unique normalized invariant scalar product which makes $V$ a simple unitary $N=1$ superconformal VOSA;

\item[(ii)] $\scalar$ is the unique normalized invariant scalar product on $V$ such that the corresponding PCT operator $\theta$ fixes the superconformal vector $\tau$;

\item[(iii)] $\Aut^\mathrm{sc}_{\scalar}(V)=\Aut^\mathrm{sc}(V)$;

\item[(iv)] $\theta$ commutes with every $g\in\Aut^\mathrm{sc}(V)$;

\item[(v)] $\Aut^\mathrm{sc}(V)$ is compact;

\item[(vi)] $\Aut^\mathrm{sc}_{\scalar}(V)$ is totally disconnected.
\end{itemize}
\end{theo}

\begin{proof}
The first claim follows from Theorem \ref{theo:characterization_aut_group}.
For the ``TFAE'' part,
we straightforwardly adapt the proof of \cite[Theorem 5.21]{CKLW18}. We prove the following implications: 
$$
\xymatrix{ \\
\mathrm{(i)} \ar@{<=>}[r] &
\mathrm{(ii)} \ar@{=>}[r] & \mathrm{(iii)} \ar@{<=>}[r] \ar@{=>}@/_0.5cm/[rr] \ar@{=>}@/_1cm/[rrr] & \mathrm{(iv)}&
\mathrm{(v)} \ar@{=>}@/_0.5cm/[lll] & \mathrm{(vi)} \,. \ar@{=>}@/_1cm/[llll] 
\\
&
}
$$

(i)$\Leftrightarrow$(ii) follows by (i) of Theorem \ref{theo:superconformal_correspondence}.

(ii)$\Rightarrow$(iii) follows from the fact that if there exists $g\in\Aut^\mathrm{sc}(V)\backslash\Aut^\mathrm{sc}_{\scalar}(V)$, then $(g(\cdot)|g(\cdot))$ defines a new invariant scalar product on $V$, which is different from $\scalar$ and such that the corresponding PCT operator $g^{-1}\theta g$ fixes $\tau$.

Regarding (iii)$\Leftrightarrow$(iv), we have that $g\in\Aut^\mathrm{sc}(V)$ is unitary if and only if $(g\theta a|g b)=(\theta a| b)$ for all $a,b\in V$. On the other hand, $(\theta ga|gb)=(\theta a|b)$ for all $a,b\in V$ thanks to (ii)$\Leftrightarrow$(iii) of Corollary \ref{cor:when_vosa_automorphism_preservs_nu}. It follows that $g$ is unitary if and only if it commutes with $\theta$.

(iii)$\Rightarrow$(v) is immediate from the first part. Checking the construction of the automorphism $h\in\Aut(V)$ of Proposition \ref{prop:uniqueness_unitary_structure} in the proof of \cite[Proposition 5.19]{CKLW18}, we note that $h\in\Aut^\mathrm{sc}(V)$ if we are dealing with PCT operators which preserve $\tau$. Then the proofs of (v)$\Rightarrow$(ii) and (vi)$\Rightarrow$(ii) are as in the proof of \cite[Proposition 5.21]{CKLW18}. Also (iii)$\Rightarrow$(vi) is proved as in \cite[Proposition 5.21]{CKLW18}.
\end{proof}

Here is a criterion for a VOSA to be unitary $N=1$ superconformal:

\begin{prop}  \label{prop:criteria_unitary_N=1}
	Let $V$ be a simple unitary VOSA. Suppose that $V$ is $N=1$ superconformal and $\Aut(V)$ is a finite dimensional Lie group. 
	If there exists a compact subgroup $G$ of $\Aut^\mathrm{sc}(V)$ such that $(V^G)_\frac{3}{2}$ is one-dimensional, then there exists a normalized invariant scalar product $\curlyscalar$ on $V$ such that $G\subseteq \Aut_{\curlyscalar}(V)$ and $V$ is a simple unitary $N=1$ superconformal VOSA (with respect to $\curlyscalar$). 
\end{prop}

\begin{proof}
	Let $\nu$ and $\tau$ be the conformal and the superconformal vector of $V$, so that $(V^G)_\frac{3}{2}=\C\tau$. If $G$ is compact, then there exists a normalized invariant scalar product $\curlyscalar$ on $V$ such that $G\subseteq\Aut_{\curlyscalar}(V)$ by Proposition \ref{prop:unitarize_compact_subgr_aut}. 
	This also implies that every $g\in G$ commutes with the PCT operator $\widetilde{\theta}$ associated to $\curlyscalar$, see the proof of (iii)$\Leftrightarrow$(iv) of Theorem \ref{theo:characterization_aut_sc_group} (note that this fact does not rely on the unitarity of the $N=1$ superconformal structure, cf.\ also the proof of \cite[Theorem 5.21]{CKLW18}).
	Thus, $g(\widetilde{\theta}(\tau))=\widetilde{\theta}(\tau)$ for all $g\in G$ and thus $\widetilde{\theta}(\tau)\in \C\tau$ as $\widetilde{\theta}(\tau)\in V_\frac{3}{2}$ (this is due to the fact that every PCT operator preserves the conformal vector $\nu$ and thus the $L_0$-grading too).  
	By Proposition \ref{prop:characterisation_unitary_subalgebras}, $W(\{\nu,\tau\})$ is a unitary subalgebra of $V$, that is $V$ is a simple unitary $N=1$ superconformal VOSA.
\end{proof}

\begin{defin}
If $\A$ is a $N=1$ superconformal net, denote by $\Aut^\mathrm{sc}(\A)$ the group of those automorphisms of the net fixing the elements of the covariant subnet $\mathrm{SVir}_c$. 
\end{defin}

Note that Lemma \ref{lem:automorphisms_NS} tells us that whenever $c$ is as in \eqref{eq:values_cc_unitary_NS}, then
\begin{equation}
\Aut^\mathrm{sc}(\mathrm{SVir}_c)=\Aut^\mathrm{sc}_{\scalar}(V^c(NS))=\Aut^\mathrm{sc}(V^c(NS))=\{1_{V^c(NS)}\}
\,.
\end{equation}

More generally, proceeding similarly to the proof of Theorem \ref{theo:AutV_AutNet}, we have that:
\begin{theo} \label{theo:AutV_AutNet_superconformal}
Let $V$ be a simple strongly graded-local unitary $N=1$ superconformal VOSA. Then $\Aut^\mathrm{sc}(\A_V)=\Aut^\mathrm{sc}_{\scalar}(V)$. Moreover, if $\Aut^\mathrm{sc}(V)$ is compact, then $\Aut^\mathrm{sc}(\A_V)=\Aut_{\scalar}^\mathrm{sc}(V)=\Aut^\mathrm{sc}(V)$. 
\end{theo}

The following result can be inferred from Theorem \ref{theo:V_strongly_graded-local_iff_V_0_and_v_are} by an argument based on the fact that superconformal vectors satisfy linear energy bounds.

\begin{theo}  \label{theo:superconformal_strongly_graded-local_iff_even_part_is}
	Let $V$ be a simple energy-bounded unitary $N=1$ superconformal VOSA. Then $V$ is strongly graded-local if and only if $V_\parzero$ is strongly local.
\end{theo}

\begin{proof}
	We use the notation as in \eqref{eq:notation_even_odd} and we follow the strategy as in \cite[Lemma 3.6]{CTW22}.
	Our aim is to apply Theorem \ref{theo:V_strongly_graded-local_iff_V_0_and_v_are}
	to the superconformal vector $\tau$ of $V$. By (i) of Theorem \ref{theo:superconformal_correspondence}, $\tau$ is Hermitian and thus $Y(\tau,g)$ is self-adjoint whenever $g\in C_\chi^\infty(S^1,\R)$, see \eqref{eq:smeared_vertex_operator_adjoint}. By \cite[Eq.\ (27)]{CKL08}, $\tau$ satisfies linear energy bounds. 
	Fix some $I\in\J$ and consider $Y(a,f)$ where $a\in V_\parzerone$ and $f\in C^\infty_{p(a)}(S^1)$ with $\mathrm{supp}f\subseteq I$.
	 Let $g\in C_\chi^\infty(S^1,\R)$ be such that $\mathrm{supp}g\subset I'$. Adapting the proof of \cite[Proposition 2.1]{Tol99}, for all $t\in\R$ and all $c\in\mathcal{H}^\infty$, $c(t):=ZY(a,f)Z^*e^{itY(\tau,g)}c$ is a well-defined vector of $\mathcal{H}^\infty$. Moreover, $\R\ni t\mapsto c(t)\in\mathcal{H}^\infty$ is differentiable by Lemma \ref{lem:common_core}, Lemma \ref{lem:h_infty_is_invariant} and \cite[Corollary 2.2]{Tol99}. Recalling that $[Y(\tau,g),Y(a,f)]=0$, it is not difficult to see that $c(t)$ satisfies the Cauchy problem on $\mathcal{H}^\infty$:
	 $$
	 \left\{
	 \begin{array}{l}
	 	\frac{\mathrm{d}}{\mathrm{d}t}c(t)=iY(\tau,g)c(t) \\
	 	c(0)=ZY(a,f)Z^*c
	 \end{array} \right.
	 $$
	 which has $e^{itY(\tau,g)}ZY(a,f)Z^*c$ as unique solution. Therefore,
	 $$
	 ZY(a,f)Z^*e^{itY(\tau,g)}c= c(t)
	 =e^{itY(\tau,g)}ZY(a,f)Z^*c \qquad\forall t\in\R
	 $$
	 which implies that $\A_{\{\tau\}}(I')\subseteq Z\A_V(I)'Z^*\subseteq Z\A_{V_\parzero\cup \{\tau\}}(I)'Z^*$, so that we can apply Theorem \ref{theo:V_strongly_graded-local_iff_V_0_and_v_are}.
\end{proof}

Now, we can continue our exposition of classical examples accompanying it with the superconformal theory just developed.

\begin{ex}[\textit{Supercurrent algebra models}]
\label{ex:supercurrent_algebra_models}
Let $\g$ be a simple complex Lie algebra of finite dimension $n$. If $h^\vee$ is the \textit{dual Coxeter number} associated to $\g$, see \cite[Chapter 6]{Kac95}, then for every $k\in\C\backslash\{-h^\vee\}$, the simple quotient $V^k(\hat{\g})$ of the Verma module from the affinization $\hat{\g}$ of $\g$ has a structure of simple VOA with central charge $c_k=\frac{kn}{k+h^\vee}$, see \cite[Theorem 5.7]{Kac01}. For specific values of $k$, $V^k(\hat{\g})$ owns a unitary structure and they turn out to be strongly local, see \cite[Example 8.7]{CKLW18} and references therein. The corresponding irreducible conformal nets are usually denoted by $\A_{G_k}$ where $G$ is the compact connected simply connected real Lie group with simple Lie algebra the compact real form $\g_\R$. Therefore, the tensor products of $V^k(\hat{\g})$ with the unitary $d$ free fermion VOSA $F^d$ of Example \ref{ex:d_free_fermions} are simple strongly graded-local unitary VOSAs by Corollary \ref{cor:strongly_graded-local_tensor_product}. The corresponding irreducible graded-local conformal nets $\A_{G_k}\hat{\otimes} \mathcal{F}^d$ are known as \textbf{supercurrent algebra nets}. The special cases where $d=n$ are called \textbf{Kac-Todorov models}, see \cite{KT85}, \cite[Section 2.1]{KW94} and \cite[Section 5.9]{Kac01}. The corresponding graded-local conformal nets are the ones presented in \cite[Section 6]{CHL15}, see also \cite[Example 5.14]{CH17}. They are $N=1$ superconformal nets as proved in \cite[Proposition 6.2]{CHL15}, which means that $V^k(\hat{\g})\hat{\otimes} F^n$ are simple strongly graded-local unitary $N=1$ superconformal VOSAs by Theorem \ref{theo:superconformal_correspondence} (cf.\  \cite[Theorem 5.9]{Kac01} for an explicit determination of the superconformal structure). We remark that the simple VOSAs $V^k(\hat{\g})\hat{\otimes} F^d$ can be equivalently realized from Lie superalgebras as done in \cite[Section 4.3]{Li96}. Of course, they give rise to a class of irreducible graded-local conformal nets, which are respectively isomorphic to the ones constructed above by Theorem \ref{theo:unicity_irr_graded-local_cn}. 
\end{ex}

\begin{ex}[$N=2$ \textit{super-Virasoro models}]  \label{ex:N2_superVirasoro}

The \textit{Neveu-Schwarz $N=2$ super-Virasoro algebra} $N2$ is the Lie superalgebra
$$
N2:=
\overbrace{
\bigoplus_{m\in\Z}\C L_m\oplus
\bigoplus_{m\in\Z}\C J_m
\oplus
\C C
}^\mathrm{even}
\oplus
\overbrace{
\bigoplus_{n\in\Z-\half}\C G^+_n \oplus
\bigoplus_{n\in\Z-\half}\C G^-_n
}^\mathrm{odd}
$$
with the commutation relations:
\begin{equation}  \label{eq:N2_cr}
\begin{split}
[L_m,L_n]  &:=
(m-n)L_{m+n}+\frac{m^3-m}{12}\delta_{m,-n}C
\qquad
\forall m,n\in\Z \\
[L_m,G^\pm_n] &:= 
\left(\frac{m}{2}-n\right) G^\pm_{m+n}
\qquad
\forall m\in\Z
\,\,\,
\forall n\in\Z-\half \\
[L_m,J_n] &:=-nJ_{m+n}
\qquad
\forall m,n\in\Z \\
[G^+_m,G^-_n] &:=
2L_{m+n}+(m-n)J_{m+n}+\frac{1}{3}\left(m^2-\frac{1}{4}\right)\delta_{m,-n}C
\qquad
\forall m,n\in\Z-\half \\
[G^\pm_m,J_n] &:=\mp G^\pm_{m+n}
\qquad
\forall m\in\Z-\half
\,\,\,
\forall n\in\Z \\
[G^+_m,G^+_n] &:=0=:[G^-_m,G^-_n]
\qquad
\forall m,n\in\Z-\half \\
[J_m,J_n] &:=\frac{c}{3}m\delta_{m,-n} 
\qquad\forall m,n\in\Z\\
[N2,C] &:= 0 \,.
\end{split}
\end{equation}
For every $c\in\C$, we construct the $N2$-module:
$$
\widetilde{V}^c(N2):=U(N2)/
\langle
\bigoplus_{m\in\Zpluseq-1}\C L_m\oplus
\bigoplus_{m\in\Zpluseq}\C J_m
\oplus
\bigoplus_{n\in\Zpluseq-\half}\C G^+_n \oplus
\bigoplus_{n\in\Zpluseq-\half}\C G^-_n
\,,\,\,\,
C-c
\rangle
$$
where $U(N2)$ is the universal enveloping algebra of $N2$. The simple quotient modules $V^c(N2)$ have a structure of simple VOSA with central charge $c$, see \cite[p.\ 182]{Kac01}, called $N=2$ \textbf{super-Virasoro VOSAs}. In particular, the vertex operators $Y(\nu,z)$, $Y(J,z)$ and $Y(G^j,z)$ for all $j\in\{1,2\}$ generate $V^c(NS)$, where
\begin{equation}
	J:=J_{-1}\Omega \,,\quad
	G^1:=\frac{G^++G^-}{\sqrt{2}}
	\quad\mbox{ and }\quad
	G^2:=-i\frac{G^+-G^-}{\sqrt{2}}
	\quad\mbox{ with }\quad
	G^\pm:=G^\pm_{-\frac{3}{2}}\Omega \,.
\end{equation}
Moreover, $V^c(NS)$ has a scalar product for the following values of the central charge $c$:
\begin{equation}  \label{eq:values_cc_unitary_N2}
\mbox{either}
\quad
c\geq 3
\quad
\mbox{ or }
\quad
c=\frac{3n}{n+2} \qquad \forall n> 0
\end{equation}
making $Y(\nu,z)$, $Y(J,z)$ and $Y(G^j,z)$ for all $j\in\{1,2\}$ Hermitian, see \cite[Theorem 3.2]{CHKLX15} and references therein. Thus, $V^c(NS)$ has a unitary structure by Proposition \ref{prop:simple_unitary_vosa_gen_by_qp_hermitian_fields} and it is energy-bounded by \cite[Eq.\  (3.1)]{CHKLX15} and Proposition \ref{prop:energy_boundedness_by_generators}. 
The family of the von Neumann subalgebras of $\A_{V^c(NS)}$ defined as in \eqref{eq:defin_net_gen_by_subset_F} from the four quasi-primary generators $\nu$, $J$ and $G^j$ of $V^c(N2)$, coincides with the so called 
$N=2$ \textbf{super-Virasoro net} $\mathrm{SVir2}_c$, defined in \cite[Eq.\ (3.2) and Theorem 3.3]{CHKLX15}. The latter satisfies the graded locality, see the proof of \cite[Theorem 3.3]{CHKLX15}. By Theorem \ref{theo:gen_by_quasi_primary}, we have that for every $c$ as in \eqref{eq:values_cc_unitary_N2}, $V^c(N2)$ is strongly graded-local and $\A_{V^c(N2)}=\mathrm{SVir2}_c$.
Note that every $V^c(N2)$ is a simple $N=1$ superconformal VOSA because it contains (at least) two copies of the $N=1$ super-Virasoro VOSA of Example \ref{ex:N=1_super-Virasoro}, given by the generators $G^j$ for all $j\in\{1,2\}$. According to Theorem \ref{theo:superconformal_correspondence}, see also \cite[Definition 3.5]{CHKLX15} and there below, if $c$ is as in \eqref{eq:values_cc_unitary_N2}, then $V^c(N2)$ is a simple unitary $N=1$ superconformal VOSA and $\mathrm{SVir2}_c$ is a $N=1$ superconformal net. 
Now, let $V_{L_{2N}}$ be the simple unitary VOA associated to the even rank-one lattice $L_{2N}$, see e.g.\ \cite[Section 4.4]{DL14}, \cite[Example 5.9]{CKLW18}, \cite[Section II.B]{CGH19} and references therein, and let $V^k(\hat{\g})$ be the simple unitary VOA constructed from the Lie algebra $\g:=\mathfrak{sl}(2,\C)$ at level $k$, see e.g.\ \cite[Section 4.2]{DL14}, \cite[Example 5.7]{CKLW18} and references therein. 
The interesting fact is that, see \cite[Theorem 3.2]{CHKLX15} and references therein, for all $c=\frac{3n}{n+2}$ with $n\in\Zplus$, the simple unitary VOSA $V^c(N2)$ can be constructed as the coset of the unitary subalgebra $V_{L_{2(n+2)}}$ of the tensor product $V^n(\hat{\g})\hat{\otimes} F^2$, where $F^2$ is the charged free fermion of Example \ref{ex:d_free_fermions}. 
$V^n(\hat{\g})$ as well as $V_{L_{2N}}$ are strongly local, see \cite[Example 8.7 and Example 8.8]{CKLW18} respectively, and thus by Corollary \ref{cor:strongly_graded-local_tensor_product} and by Proposition \ref{prop:correspondence_coset_constructions}, we have that
$$
\mathrm{SVir2}_\frac{3n}{n+2}=\A_{V_{L_{2(n+2)}}}^c=\A_{\operatorname{U}(1)_{2(n+2)}}^c \subset \A_{\operatorname{SU}(2)_n}\hat{\otimes}\mathcal{F}^2
\qquad
\forall n\in\Zplus  \,.
$$ 
\end{ex}

\begin{rem}  \label{rem:N=2_superconformal_structures}
In complete analogy with the $N=1$ superconformal structures introduced in Definition \ref{defin:superconformal_net} and in Definition \ref{defin:superconformal_vosa}, there exists an obvious notion of \textbf{$N=2$ superconformal nets} and of \textbf{$N=2$ superconformal VOSAs}, see \cite[Definition 3.5]{CHKLX15} and \cite[Section 5.9]{Kac01} respectively. 
Accordingly, we say a $N=2$ superconformal VOSA $V$ be \textbf{unitary $N=2$ superconformal} if it is unitary as VOSA and if $W(\{\nu,J, G^+, G^-\})$ is a unitary subalgebra of $V$. If $V$ is also simple, this is equivalent to say that $W(\{\nu,J, G^+, G^-\})$ is (up to equivalence) $V^c(N2)$, for some $c$ as in the unitary series \eqref{eq:values_cc_unitary_N2}. 
Also in this case, we have that a simple strongly graded-local unitary VOSA $V$ is unitary $N=2$ superconformal if and only if $\A_V$ is a $N=2$ superconformal net.
It is worthwhile to note that an analogue of the first statement of (i) of Theorem \ref{theo:superconformal_correspondence} for unitary $N=2$ superconformal VOSAs is not possible as the unitary structure on $V^c(N2)$ is unique only up to equivalence. Indeed, for all real number $\beta>0$, there are non-unitary VOSA automorphisms $g_\beta$ of $V^c(N2)$ defined by: $g_\beta(\nu)=\nu$, $g_\beta(J)=-J$, $g_\beta(G^+)=\beta G^-$ and $g_\beta(G^-)=\beta^{-1}G^+$. Then for all $1\not=\beta>0$, $g_\beta$ can be used to model a new unitary structure on $V^c(N2)$ from the one constructed in Example \ref{ex:N2_superVirasoro} and equivalent to it, see Proposition \ref{prop:uniqueness_unitary_structure}. Of course, if $\theta$ is the PCT operator arising from Example \ref{ex:N2_superVirasoro}, then the superconformal vectors $G^{1,\beta}:=\frac{G^+ +(g_\beta^{-1}\theta g_\beta)(G^+)}{\sqrt{2}}$ and $G^{2,\beta}:=-i\frac{G^+ -(g_\beta^{-1}\theta g_\beta)(G^+)}{\sqrt{2}}$ are Hermitian with respect to this new corresponding unitary structure.
\end{rem}

It is known that we can extract a unique, up to isomorphism, simple VOSA structure from any finite rank positive-definite integral lattice as stated by \cite[Theorem 5.5 and Proposition 5.5]{Kac01}, see also \cite[Section 5]{TZ12}. Furthermore, these VOSAs are unitary thanks to \cite[Theorem 2.9]{AL17}.
In \cite[Theorem 5.7]{Gui21}, the author shows that all these unitary VOAs are strongly local, proving \cite[Conjecture 8.17]{CKLW18}. Thus, the corresponding conformal nets are the ones constructed in \cite[Section 3]{DX06}, see \cite[Theorem 2.7.11]{Gui}. 
Here below, we prove that all the lattice type unitary VOSAs are strongly graded-local, heavily relying on \cite[Theorem 5.7]{Gui21} and \cite[Theorem 2.7.9]{Gui}, which make massively use of the theory of \textit{intertwining operators}, see e.g.\ \cite[Section 5.4]{FHL93}, the representation theory of VOAs, see \cite{HL95I, HL95II, HL95III, Hua95, Hua08} (see \cite[Section 2]{HKL15}  for a brief review, cf.\ also \cite[Section 2.4]{Gui19I} or \cite[Section 4.1]{Gui21}) and its unitary version \cite{Gui19I, Gui19II}. See also \cite[Definition 1.8]{Gui22} for the related notion of \textit{completely unitary} VOAs. See e.g.\ \cite{EGNO15} and \cite{NT13} for the general tensor categorical setting.

\begin{theo}[\textit{Lattice type models}]  \label{theo:lattices_strong_graded_locality}
	Let $V_L$ be the simple unitary VOSA associated to any finite rank positive-definite integral lattice $L$. Then $V_L$ is strongly graded-local, so that to any such lattice $L$ is associated an irreducible graded-local conformal net $\A_L:=\A_{V_L}$.
\end{theo}

\begin{proof}  
	We use the notation as in \eqref{eq:notation_even_odd}.
	Set $V:=V_L$ with state-field correspondence $Y$ and use the notation as in Definition \ref{defin:hilbet_space_from_V}. By Example \ref{ex:even_simple_odd_irreducible}, $V_\parzero$ is a simple unitary VOA and $V_\parone$ is an irreducible $V_\parzero$-module. In particular, the latter is generated, as $V_\parzero$-module, by any odd vector of $V$. Choose any odd quasi-primary vector $w$ of $V$. Hence, our goal is to apply Theorem \ref{theo:V_strongly_graded-local_iff_V_0_and_v_are} with $w$.
	
	It is easy to verify that $V_\parzero$ is a simple unitary lattice VOA, see e.g.\ \cite[Table 1]{CGGH23} for definition and properties. By \cite[Theorem 5.8]{Gui21}, these kinds of VOAs are completely unitary and thus the $V_\parzero$-modules form a \textit{unitary modular tensor category} $\Repu(V_\parzero)$. From now onward, we use $\boxtimes$ to denote the \textit{categorical} tensor product of VOA modules to not confuse it with the regular one $\otimes$.
	Therefore, $V_\parone$ can be considered as a \textit{$\Z_2$-simple current} in $\Repu(V)$, see \cite[Theorem 3.1 and Remark A.2]{CKLR19} and \cite[Remark 4.3]{CGGH23}
	(see also \cite{DLM96} for simple current extensions of VOAs).
	Accordingly, we implement the following identifications of unitary $V_\parzero$-modules: $V_p\boxtimes V_q\equiv V_{p+q}$ for all $p,q\in\{\parzero,\parone\}$.
	Note also that $V_\parone$ has \textit{twist} $\omega(V_\parone)=-\boldsymbol{1}_{V_\parone}$ as it is given by the action of $e^{i2\pi L_0}$ on the vector space $V_\parone$. Moreover, it satisfies the \textit{braiding} relations 
	\begin{equation}  \label{eq:braiding_relations_even_odd_parts}
		b_{V_\parzero, V_\parone}=\boldsymbol{1}_{V_\parone}=b_{V_\parone, V_\parzero}
		\quad\mbox{and}\quad
		b_{V_\parone,V_\parone}=-\boldsymbol{1}_{V_\parzero}
	\end{equation}
	as a consequence of the skew-symmetry \eqref{eq:skew-symmetry} of $V$, cf.\ also \cite[Remark 2.20]{CGGH23}. 
	For all $a\in V_\parzero$ and all $b\in V_\parone$, define 
	\begin{equation} \label{eq:intertwining_operators_parity}
		\begin{split}
			Y_\parzerozero(a,z):=Y(a,z)\restriction_{V_\parzero} 
			\,\,\,\mbox{ and }\,\,\,
			Y_\parzerone(a,z):=Y(a,z)\restriction_{V_\parone}
			\,,\\
			Y_\paronezero(b,z):=Y(b,z)\restriction_{V_\parzero}
			\,\,\,\mbox{ and }\,\,\, Y_\paroneone(b,z):=Y(b,z)\restriction_{V_\parone} \,.
		\end{split}
	\end{equation}
	By the Borcherds identity \eqref{eq:borcherds_id} for the state field-correspondence $Y$ of $V$, it follows that they define intertwining operators of $V_\parzero$ of type $\binom{\parzero}{\parzero \,\,\, \parzero}$ and $\binom{\parone}{\parzero \,\,\, \parone}$, $\binom{\parone}{\parone \,\,\, \parzero}$ and $\binom{\parzero}{\parone \,\,\, \parone}$ respectively. In particular, $Y_\parzerozero$ and $Y_\parzerone$ are vertex operators of the $V_\parzero$-module $V_\parzero$ and $V_\parone$ respectively.
	It follows from \cite[Theorem 5.7]{Gui21}, see also \cite[Theorem 2.7.9]{Gui}, that $V_\parzero$ is \textit{completely energy-bounded}, see the beginning of \cite[Section 2.1]{Gui} for the general definition. 
	This implies that the intertwining operators in \eqref{eq:intertwining_operators_parity} satisfy energy bounds (this is a generalization of the usual notion of energy bounds for vertex operators, see \cite[Definition 3.1]{Gui19I}). Noting that for every homogeneous vector $c\in V$, the vertex operator $Y(c,z)$ can be written as the direct sum of two intertwining operators in \eqref{eq:intertwining_operators_parity}, we can conclude that also $Y(c,z)$ satisfies energy bounds (as vertex operator). Thus, we can assume that $V$ is energy-bounded.  
	
	Define $\mathcal{H}_\parzero$, $\mathcal{H}_\parone$ and $\mathcal{H}$ as the Hilbert space completions of $V_\parzero$, $V_\parone$ and $V$ respectively with respect to the scalar product on $V$. Therefore, $\mathcal{H}=\mathcal{H}_\parzero\oplus \mathcal{H}_\parone$ and $\mathcal{H}^\infty=\mathcal{H}_\parzero^\infty\oplus \mathcal{H}_\parone^\infty$, where $\mathcal{H}_\parone^\infty$ is the subspace of $\mathcal{H}_\parone$ of smooth vectors for $L_0^{V_\parone}=L_0\restriction_{V_\parone}$ as defined at p.\ \pageref{defin:smooth_vectors_L_0}.
	The complete unitarity of $V_\parzero$ fixes the invariant scalar products on the $V_\parzero$-modules $V_p\boxtimes V_q$ for all $p,q\in\{\parzero,\parone\}$ and the corresponding Hilbert space closures can be naturally identified with $\mathcal{H}_{pq}$ respectively.
	For all $I\in\J$, define the continuous function $\mathrm{arg}_I: I\to \R$ by $\mathrm{arg}_I(z)=x$ where $z=e^{ix}$ for the unique $x\in(-\pi,\pi]$. Then for any homogeneous vector $c\in V$ and $f\in C^\infty_{p(c)}(S^1)$ with $\mathrm{supp}f\subset I$ for some $I\in\J$, we have that 
	$$
	Y_{p,q}(c,\tilde{f})=Y(c,f)\restriction_{\mathcal{H}_q^\infty}
	\qquad\forall p,q\in\{\parzero,\parone\}
	$$
	where $Y_{p,q}(c,\tilde{f})$ is a \textit{smeared} intertwining operator in the sense of \cite[Section 3.2]{Gui19I} and $\tilde{f}$ is the \textit{arg-valued} function $(f\cdot z^{d_c-1}, \mathrm{arg}_I)$ in the sense of \cite[p.\ 815]{Gui21}. By \cite[Theorem 5.7]{Gui21}, see also \cite[Theorem 2.7.9]{Gui}, the unitary intertwining operators of $V$ satisfies the \textit{strong intertwining property} \cite[Section 2.3]{Gui}, cf.\ \cite[Definition 4.10]{Gui21}, and the \textit{strong braiding} \cite[Section 2.5]{Gui}, cf.\ \cite[Defition 4.13]{Gui21}. Of course, also the intertwining operators in \eqref{eq:intertwining_operators_parity} satisfy those properties. 
	
	First, we prove that the strong intertwining property implies the strong commutativity of $Y(a,f)$ and $Y(w,g)$ whenever $a\in V_\parzero$, $f\in C^\infty(S^1)$ and $g\in C^\infty_\chi(S^1)$ with $\mathrm{supp}f\subset I$ and $\mathrm{supp}g\subset I'$ for any $I\in\J$. Fix $a$, $f$ and $g$ as just explained. The strong intertwining property says that the two diagrams
	\begin{equation}
		\xymatrixcolsep{2cm}\xymatrixrowsep{2cm}\xymatrix{ \mathcal{H}_\parzero^\infty \ar[r]^{Y_\parzerozero(a,\tilde{f})} \ar[d]_{Y_\paronezero(w,\tilde{g})} & \mathcal{H}_\parzero^\infty \ar[d]_{Y_\paronezero(w,\tilde{g})}\\		
		\mathcal{H}_\parone^\infty \ar[r]^{Y_\parzerone(a,\tilde{f})} & \mathcal{H}_\parone^\infty }
	\qquad
	\xymatrixcolsep{2cm}\xymatrixrowsep{2cm}\xymatrix{ \mathcal{H}_\parone^\infty \ar[r]^{Y_\parzerone(a,\tilde{f})} \ar[d]_{Y_\paroneone(w,\tilde{g})} & \mathcal{H}_\parone^\infty \ar[d]_{Y_\paroneone(w,\tilde{g})}\\		
		\mathcal{H}_\parzero^\infty \ar[r]^{Y_\parzerozero(a,\tilde{f})} & \mathcal{H}_\parzero^\infty } 
	\end{equation}
	\textit{commute strongly} in the sense of \cite[Definition 1.2.6]{Gui}, cf.\ \cite[Definition 4.6]{Gui21}.
	This means that the diagram on the left implies that the closures of the following two pre-closed operators on the Hilbert space $\mathcal{H}_1:=\mathcal{H}_\parzero\oplus \mathcal{H}_\parzero\oplus \mathcal{H}_\parone\oplus \mathcal{H}_\parone$ commute strongly (now in the meaning given at p.\ \pageref{defin:operators_commute_strongly}):
	\begin{equation}
		R_1:=\left(\begin{array}{cccc}
			0 & 0 & 0 & 0 \\
			0 & 0 & 0 & 0 \\
			Y_\paronezero(w,\tilde{g}) & 0 & 0 & 0 \\
			0 & Y_\paronezero(w,\tilde{g}) & 0 & 0
		\end{array}
		\right)
		\qquad
		S_1:=\left(\begin{array}{cccc}
			0 & 0 & 0 & 0 \\
			Y_\parzerozero(a,\tilde{f}) & 0 & 0 & 0 \\
			0 & 0 & 0 & 0 \\
			0 & 0 & Y_\parzerone(a,\tilde{f}) & 0
		\end{array}
		\right) 
	\end{equation}
	with domains $\mathcal{D}(R_1):=\mathcal{H}_\parzero^\infty\oplus \mathcal{H}_\parzero^\infty\oplus \mathcal{H}_\parone\oplus \mathcal{H}_\parone$  and $\mathcal{D}(S_1):=\mathcal{H}_\parzero^\infty\oplus \mathcal{H}_\parzero\oplus \mathcal{H}_\parone^\infty\oplus \mathcal{H}_\parone$ respectively.
Similarly, the diagram on the right implies that the closures of the following two pre-closed operators on the Hilbert space $\mathcal{H}_2:=\mathcal{H}_\parone\oplus \mathcal{H}_\parone\oplus \mathcal{H}_\parzero\oplus \mathcal{H}_\parzero$ commute strongly:
\begin{equation}
	R_2:=\left(\begin{array}{cccc}
		0 & 0 & 0 & 0 \\
		0 & 0 & 0 & 0 \\
		Y_\paroneone(w,\tilde{g}) & 0 & 0 & 0 \\
		0 & Y_\paroneone(w,\tilde{g}) & 0 & 0
	\end{array}
	\right)
	\qquad
	S_2:=\left(\begin{array}{cccc}
		0 & 0 & 0 & 0 \\
		Y_\parzerone(a,\tilde{f}) & 0 & 0 & 0 \\
		0 & 0 & 0 & 0 \\
		0 & 0 & Y_\parzerozero(a,\tilde{f}) & 0
	\end{array}
	\right) 
\end{equation}
 with domains $\mathcal{D}(R_2):=\mathcal{H}_\parone^\infty\oplus \mathcal{H}_\parone^\infty\oplus \mathcal{H}_\parzero\oplus \mathcal{H}_\parzero$  and $\mathcal{D}(S_2):=\mathcal{H}_\parone^\infty\oplus \mathcal{H}_\parone\oplus \mathcal{H}_\parzero^\infty\oplus \mathcal{H}_\parzero$ respectively.
	It follows from \cite[Lemma 5.6]{Gui21} that the closures of the following two pre-closed operators on the Hilbert space $\mathcal{H}_1\oplus \mathcal{H}_2$ commute strongly:
	\begin{equation}
		R_{12}:=\left(\begin{array}{cc}
			R_1 & 0 \\
			0 & R_2 
		\end{array}
		\right)
		\qquad
		S_{12}:=\left(\begin{array}{cc}
			S_1 & 0 \\
			0 & S_2 
		\end{array}
		\right) \,.
	\end{equation} 
Now, note that we can write
\begin{equation}
	Y(a,f)\restriction_{\mathcal{H}^\infty}=\left(\begin{array}{cc}
		Y_\parzerozero(a,\tilde{f}) & 0 \\
		0 & Y_\parzerone(a,\tilde{f}) 
	\end{array}
	\right)
	\qquad
	Y(w,g)\restriction_{\mathcal{H}^\infty}:=\left(\begin{array}{cc}
		0 & Y_\paroneone(w,\tilde{g}) \\
		Y_\paronezero(w,\tilde{g}) & 0 
	\end{array}
	\right) \,.
\end{equation} 
Then it is not difficult to find a unitary operator $\mathcal{U}$ from $\mathcal{H}_1\oplus \mathcal{H}_2$ to $\mathcal{H}\oplus\mathcal{H}\oplus\mathcal{H}\oplus\mathcal{H}$ such that:
\begin{equation}
	\begin{split}
	\mathcal{U}R_{12}\mathcal{U}^*   &=\left(\begin{array}{cccc}
		0 & 0 & 0 & 0 \\
		0 & 0 & 0 & 0 \\
		Y(w,g)\restriction_{\mathcal{H}^\infty} & 0 & 0 & 0 \\
		0 & Y(w,g)\restriction_{\mathcal{H}^\infty} & 0 & 0
	\end{array}
	\right)
	\\
	\mathcal{U}S_{12}\mathcal{U}^*   &=\left(\begin{array}{cccc}
		0 & 0 & 0 & 0 \\
		Y(a,f)\restriction_{\mathcal{H}^\infty} & 0 & 0 & 0 \\
		0 & 0 & 0 & 0 \\
		0 & 0 & Y(a,f)\restriction_{\mathcal{H}^\infty} & 0
	\end{array}
	\right) \,.
	\end{split}
\end{equation}  
Using the polar decomposition as in \cite[Proposition B.5]{Gui19I}, it is not difficult to check that the strong commutativity of the closures of $\mathcal{U}R_{12}\mathcal{U}^*$ and $\mathcal{U}S_{12}\mathcal{U}^*$ implies the strong commutativity of $Y(w,g)$ and $Y(a,f)$, as desired.
	
	Second, we prove that the strong braiding implies the strong commutativity of $ZY(w,f)Z^*$ and $Y(w,g)$ whenever $f,g\in C^\infty_\chi(S^1)$ with $\mathrm{supp}f\subset I$ and $\mathrm{supp}g\subset I'$ for any $I\in\J$. Fix $f$ and $g$ as just explained. 
	Let $\mathcal{L}_\parone\restriction_\parzero$ and $\mathcal{L}_\parone\restriction_\parone$ be the intertwining operators of type $\binom{\parone}{\parone \,\,\, \parzero}$ and $\binom{\parzero}{\parone \,\,\, \parone}$ respectively as defined in \cite[Eq.s (4.20)--(4.22)]{Gui21}, see also the beginning of \cite[Section 2.5]{Gui}. 
	Following  \cite[Eq.s (4.23)]{Gui21}, define two intertwining operators of type $\binom{\parone}{\parone \,\,\, \parzero}$ and $\binom{\parzero}{\parone \,\,\, \parone}$ by
	\begin{equation}
		\mathcal{R}_\parone\restriction_\parzero(v,z)a:=b_{V_\parone, V_\parzero}\mathcal{L}_\parone\restriction_\parzero(v,z)a
		\quad\mbox{and}\quad
		\mathcal{R}_\parone\restriction_\parone(v,z)u:=b_{V_\parone, V_\parone}\mathcal{L}_\parone\restriction_\parone(v,z)u
	\end{equation}
	respectively for any $a\in V_\parzero$ and any $u,v\in V_\parone$.
	Since $\binom{\parone}{\parone \,\,\, \parzero}=\C Y_\paronezero$ and $\binom{\parzero}{\parone \,\,\, \parone}=\C Y_\paroneone$, there exist unique $\alpha, \beta\in\C\backslash\{0\}$ such that $\mathcal{L}_\parone\restriction_\parzero=\alpha Y_\paronezero$ and $\mathcal{L}_\parone\restriction_\parone= \beta Y_\paroneone$. 
	Then using the braiding relations \eqref{eq:braiding_relations_even_odd_parts},
	we also have that 
	$\mathcal{R}_\parone\restriction_\parzero=\alpha Y_{\parone,\parzero}$ and $\mathcal{R}_\parone\restriction_\parone=-\beta Y_{\parone,\parone}$.
	Then the previous identities together with the strong braiding, as defined e.g.\ in \cite[Definition 4.13]{Gui21}, implies that the two diagrams
	\begin{equation}
		\xymatrixcolsep{2cm}\xymatrixrowsep{2cm}\xymatrix{ \mathcal{H}_\parzero^\infty \ar[r]^{Y_\paronezero(w,\tilde{f})} \ar[d]_{Y_\paronezero(w,\tilde{g})} & \mathcal{H}_\parone^\infty \ar[d]_{Y_\paroneone(w,\tilde{g})}\\		
			\mathcal{H}_\parone^\infty \ar[r]^{-Y_\paroneone(w,\tilde{f})} & \mathcal{H}_\parzero^\infty }
		\qquad
		\xymatrixcolsep{2cm}\xymatrixrowsep{2cm}\xymatrix{ \mathcal{H}_\parone^\infty \ar[r]^{-Y_\paroneone(w,\tilde{f})} \ar[d]_{Y_\paroneone(w,\tilde{g})} & \mathcal{H}_\parzero^\infty \ar[d]_{Y_\paronezero(w,\tilde{g})}\\		
			\mathcal{H}_\parzero^\infty \ar[r]^{Y_\paronezero(w,\tilde{f})} & \mathcal{H}_\parone^\infty } 
	\end{equation}
	commute strongly, again in the sense of \cite[Definition 1.2.6]{Gui}, cf.\ \cite[Definition 4.6]{Gui21}.
	Thus, from the diagram on the left, we have that the closures of the following two pre-closed operators on the Hilbert space $\mathcal{H}_3:=\mathcal{H}_\parzero\oplus \mathcal{H}_\parone\oplus \mathcal{H}_\parone\oplus \mathcal{H}_\parzero$ commute strongly:
	\begin{equation}
		R_3:=\left(\begin{array}{cccc}
			0 & 0 & 0 & 0 \\
			0 & 0 & 0 & 0 \\
			Y_\paronezero(w,\tilde{g}) & 0 & 0 & 0 \\
			0 & Y_\paroneone(w,\tilde{g}) & 0 & 0
		\end{array}
		\right)
		\qquad
		S_3:=\left(\begin{array}{cccc}
			0 & 0 & 0 & 0 \\
			Y_\paronezero(w,\tilde{f}) & 0 & 0 & 0 \\
			0 & 0 & 0 & 0 \\
			0 & 0 & -Y_\paroneone(w,\tilde{f}) & 0
		\end{array}
		\right) 
	\end{equation}
	with domains $\mathcal{D}(R_3):=\mathcal{H}_\parzero^\infty\oplus \mathcal{H}_\parone^\infty\oplus \mathcal{H}_\parone\oplus \mathcal{H}_\parzero$  and $\mathcal{D}(S_3):=\mathcal{H}_\parzero^\infty\oplus \mathcal{H}_\parone\oplus \mathcal{H}_\parone^\infty\oplus \mathcal{H}_\parzero$ respectively.
Similarly, from the one on the right, we get that the closures of the following two pre-closed operators on the Hilbert space $\mathcal{H}_4:=\mathcal{H}_\parone\oplus \mathcal{H}_\parzero\oplus \mathcal{H}_\parzero\oplus \mathcal{H}_\parone$ commute strongly:
\begin{equation}
	R_4:=\left(\begin{array}{cccc}
		0 & 0 & 0 & 0 \\
		0 & 0 & 0 & 0 \\
		Y_\paroneone(w,\tilde{g}) & 0 & 0 & 0 \\
		0 & Y_\paronezero(w,\tilde{g}) & 0 & 0
	\end{array}
	\right)
	\qquad
	S_4:=\left(\begin{array}{cccc}
		0 & 0 & 0 & 0 \\
		-Y_\paroneone(w,\tilde{f}) & 0 & 0 & 0 \\
		0 & 0 & 0 & 0 \\
		0 & 0 & Y_\paronezero(w,\tilde{f}) & 0
	\end{array}
	\right) 
\end{equation}
with domains $\mathcal{D}(R_4):=\mathcal{H}_\parone^\infty\oplus \mathcal{H}_\parzero^\infty\oplus \mathcal{H}_\parzero\oplus \mathcal{H}_\parone$  and $\mathcal{D}(S_4):=\mathcal{H}_\parone^\infty\oplus \mathcal{H}_\parzero\oplus \mathcal{H}_\parzero^\infty\oplus \mathcal{H}_\parone$ respectively.
Using again \cite[Lemma 5.6]{Gui21}, we get that the closures of the following two pre-closed operators on the Hilbert space $\mathcal{H}_3\oplus \mathcal{H}_4$ commute strongly:
\begin{equation}
	R_{34}:=\left(\begin{array}{cc}
		R_3 & 0 \\
		0 & R_4 
	\end{array}
	\right)
	\qquad
	S_{34}:=\left(\begin{array}{cc}
		S_3 & 0 \\
		0 & S_4 
	\end{array}
	\right) \,.
\end{equation} 
Now, note that we can write
\begin{equation}
	\begin{split}
	ZY(w,f)\restriction_{\mathcal{H}^\infty}Z^*    &=\left(\begin{array}{cc}
		0 & -iY_\paroneone(w,\tilde{f}) \\
		iY_\paronezero(w,\tilde{f}) & 0
	\end{array}
	\right)
	\\
	Y(w,g)\restriction_{\mathcal{H}^\infty}   &=\left(\begin{array}{cc}
		0 & Y_\paroneone(w,\tilde{g}) \\
		Y_\paronezero(w,\tilde{g}) & 0 
	\end{array}
	\right) \,.
	\end{split}
\end{equation} 
As done before, it is not difficult to find a unitary operator $\mathcal{V}$ from $\mathcal{H}_3\oplus \mathcal{H}_4$ to $\mathcal{H}\oplus\mathcal{H}\oplus\mathcal{H}\oplus\mathcal{H}$ such that:
\begin{equation}
	\begin{split}
	\mathcal{V}R_{34}\mathcal{V}^*  
	&=
	\left(\begin{array}{cccc}
		0 & 0 & 0 & 0 \\
		0 & 0 & 0 & 0 \\
		Y(w,g)\restriction_{\mathcal{H}^\infty}  & 0 & 0 & 0 \\
		0 & Y(w,g)\restriction_{\mathcal{H}^\infty}  & 0 & 0
	\end{array}
	\right)
	\\
	\mathcal{V}iS_{34}\mathcal{V}^*  
	&=
	\left(\begin{array}{cccc}
		0 & 0 & 0 & 0 \\
		ZY(w,f)\restriction_{\mathcal{H}^\infty} Z^* & 0 & 0 & 0 \\
		0 & 0 & 0 & 0 \\
		0 & 0 & ZY(w,f)\restriction_{\mathcal{H}^\infty} Z^* & 0
	\end{array}
	\right) \,.
	\end{split}
\end{equation}  
Using again the polar decomposition as in \cite[Proposition B.5]{Gui19I}, we obtain the strong commutativity of $Y(w,g)$ and $ZY(w,f)Z^*$ from the one of the closures of $\mathcal{V}R_{34}\mathcal{V}^*$ and $\mathcal{V}iS_{34}\mathcal{V}^*$.

The argument above implies that $\A_{\{w\}}(I')\subseteq Z\A_{V_\parzero\cup\{w\}}(I)'Z^*$ for all $I\in\J$. Therefore, all the hypotheses of Theorem \ref{theo:V_strongly_graded-local_iff_V_0_and_v_are} are satisfied, so that $V$ is strongly graded-local.
\end{proof}

The proof above is a rather simple illustration of the fact that the strong braiding and the strong intertwining property for the smeared intertwining operators can be 
a very powerful tool for proving the strong (graded) locality of VO(S)A extensions. Indeed, this method is generalized to the case of arbitrary simple current VOSA extensions in \cite[Lemma 2.39]{Gau25}. A more sophisticated analysis shows that in fact this strategy works for many important examples, cf.\ \cite{Guib}.

\begin{rem}[\textit{Rank-one lattice type models}]
	For all $M\in\Zplus$, let $L_M$ be any of the rank-one positive-definite integral lattices on $\R$ such that $L_M\cong \sqrt{M}\Z$.
	We can prove the strong graded locality of all unitary rank-one lattice type VOSAs $V_{L_M}$, without relying on the representation theory of VOAs as done in the proof of Theorem \ref{theo:lattices_strong_graded_locality}, in the following way. 
	By \cite[Example 5.5a]{Kac01}, $V_{L_1}$ and the charged free fermion $F^2$ of Example \ref{ex:d_free_fermions} are isomorphic VOSAs. 
	Therefore, $V_{L_1}$ is strongly graded-local and by Theorem \ref{theo:unicity_irr_graded-local_cn}, $\A_{L_1}:=\A_{V_{L_1}}$ is an irreducible graded-local conformal net isomorphic to $\mathcal{F}^2$. 
	Note that $L_M$ is isomorphic to a sublattice of the lattice $\Z^M$, sending $\sqrt{M}1\in L_M$ to $(1,\dots, 1)\in \Z^M$. 
	This map induces an isomorphism between $V_{L_M}$ and a unitary subalgebra of $F^{2M}$. 
	Hence, Theorem \ref{theo:covariant_subnets_unitary_subalgebras} and Theorem \ref{theo:unicity_irr_graded-local_cn} imply that $V_{L_M}$ is strongly graded-local and that $\A_{L_M}:=\A_{V_{L_M}}$ is a covariant subnet of $\mathcal{F}^{2M}$. 
	Note that this gives another proof of the strong locality of the unitary VOAs $V_{L_M}$ with $M\in 2\Z$, firstly given in \cite[Example 8.8]{CKLW18}.
	It is also interesting to note that $V_{L_3}$ and $V^1(N2)$ are isomorphic VOSAs by \cite[Example 5.9c]{Kac01}, so that by Theorem \ref{theo:unicity_irr_graded-local_cn}, $\mathrm{SVir2}_1\cong \A_{L_3}$, which can then be considered as a covariant subnet of $\mathcal{F}^6$.
\end{rem}

In \cite[Theorem 4.5]{Dun07}, the author constructs a simple $N=1$ superconformal VOSA $\prescript{}{\C}{A^{f\natural}}$ as a complexification of a real VOSA. Moreover, it is proved that $\Aut^\mathrm{sc}(\prescript{}{\C}{A^{f\natural}})\cong \mathrm{Co}_1$, that is \textit{Conway's largest sporadic group}, see \cite[Theorem 4.11]{Dun07}. 
Actually, $\prescript{}{\C}{A^{f\natural}}$ is isomorphic by a boson-fermion correspondence \cite[Section 5.2]{Kac01} to the (complex) odd lattice VOSA $V_{D_{12}^+}$, see the proof of \cite[Proposition 4.1]{Dun07} and \cite[Section 4.1]{DM15} for details. 
Thus, $\prescript{}{\C}{A^{f\natural}}$ is a simple strongly graded-local unitary VOSA by Theorem \ref{theo:lattices_strong_graded_locality} and the irreducible graded-local conformal net $\A^{f\natural}:=\A_{\prescript{}{\C}{A^{f\natural}}}$ is isomorphic to $\A_{D_{12}^+}$ by Theorem \ref{theo:unicity_irr_graded-local_cn}. 
In \cite{DM15}, $\prescript{}{\C}{A^{f\natural}}$ is denoted by $V^{f\natural}$ and we will keep this notation in the following \footnote{Note that the same symbol is used to denote the real form of a different VOSA $\prescript{}{\C}{V^{f\natural}}$ in \cite[Section 6]{Dun07}, which is anyway isomorphic to $\prescript{}{\C}{A^{f\natural}}$.}. 
We call $V^{f\natural}$ the \textbf{super-Moonshine VOSA} as it can be considered as a super version of the famous \textbf{Moonshine VOA} $V^\natural$, see \cite{FLM88} and \cite{Miy04}, see \cite[Theorem 5.4]{KL06} and \cite[Theorem 8.15]{CKLW18} for the conformal net version. 
Therefore, we are able to prove the following result:

\begin{theo}[\textit{The super-Moonshine model}] \label{theo:super-moonshine}
	The super-Moonshine VOSA $V^{f\natural}$ is a simple strongly graded-local unitary $N=1$ superconformal VOSA. Consequently, there exists an irreducible $N=1$ superconformal net $\A^{f\natural}:=\A_{V^{f\natural}}$ such that $\Aut^\mathrm{sc}(\A^{f\natural})$ is isomorphic to the Conway's largest sporadic group $\mathrm{Co}_1$.
\end{theo}

\begin{proof}
	We have already proved that $V^{f\natural}$ is a simple strongly graded-local unitary VOSA and we have already observed that it has a $N=1$ superconformal structure, which we choose in the following way. 
	By \cite[Proposition 4.4]{DM15}, for a given choice of the isomorphism $\Aut^\mathrm{sc}(V^{f\natural})\cong\mathrm{Co}_1$, there is a unique one-dimensional subspace of $(V^{f\natural})_\frac{3}{2}$ fixed by $\Aut^\mathrm{sc}(V^{f\natural})$. Moreover, a suitable scaled vector, say $\tau$, in this one-dimensional subspace is a superconformal vector for $V^{f\natural}$.
	$V$ is finitely-generated and thus $\Aut(V)$ is a finite dimensional Lie algebra, see Remark \ref{rem:finitely_gen_VOSAs}. Then by Proposition \ref{prop:criteria_unitary_N=1}, we can choose the normalized invariant scalar product $\scalar$ on $V^{f\natural}$ in such a way that $\Aut^\mathrm{sc}_{\scalar}(V^{f\natural})=\Aut^\mathrm{sc}(V^{f\natural})\cong\mathrm{Co}_1$ and that $V^{f\natural}$ is a simple unitary $N=1$ superconformal VOSA.
	Therefore, by (ii) of Theorem \ref{theo:superconformal_correspondence}, $\A^{f\natural}$ is a $N=1$ superconformal net and $\Aut^\mathrm{sc}(\A^{f\natural})=\Aut_{\scalar}^\mathrm{sc}(V^{f\natural})=\Aut^\mathrm{sc}(V^{f\natural})\cong\mathrm{Co}_1$ by Theorem \ref{theo:AutV_AutNet_superconformal}.
\end{proof}

Note that the $N=1$ superconformal net $\A^{f\natural}$ constructed above is isomorphic to the one described in \cite[Theorem 4.2]{Kaw10} and we call it the \textbf{super-Moonshine conformal net}.
We also highlight that the $N=1$ superconformal VOSA $V^{f\natural}$ is characterized as the unique, up to isomorphism, \textit{strongly rational} (that is, self-dual, $C_2$-cofinite, rational and of CFT type, see e.g.\ \cite{ABD04, HA15, DH12}, see also \cite[Section 2.3]{HM} for a more general brief review, and references therein) $N=1$ superconformal VOSA with central charge $c=12$ and trivial subspace of weight $\half$, see \cite[Theorem 5.15]{Dun07}. Furthermore, the strongly rational \textit{holomorphic} (that is, the only irreducible VOSA module is the VOSA itself) VOSAs with $c\leq 12$ are completely classified in \cite[Theorem 3.1]{CDR18}. They are: the $d$ free fermions VOSAs $F^d$ for $0\leq d\leq 24$, which are strongly graded-local by Example \ref{ex:d_free_fermions}; the tensor product of the lattice VOA $V_{E_8}$ with the $d$ free fermions VOSA $F^d$ for all $0\leq d\leq 8$, which are strongly graded-local by \cite[Theorem 5.7]{Gui21}, Example \ref{ex:d_free_fermions} and Corollary \ref{cor:strongly_graded-local_tensor_product}; the super-Moonshine VOSA $V^{f\natural}$, whose strong graded locality has been established in the discussion just before Theorem \ref{theo:super-moonshine}. Thus, we have the following result:

\begin{theo}  \label{theo:c_leq_12}
	All the strongly rational holomorphic VOSAs with central charge $c\leq 12$ are strongly graded-local.
\end{theo}

The super-Moonshine VOSA $V^{f\natural}$ of Theorem \ref{theo:super-moonshine} has been shown in \cite[Section 5, Theorem]{JoF21} to be part of the family of unitary VOSAs with a $N=1$ superconformal structure and such that they contain no vectors of conformal weight $\half$ and their even part is equal to a unitary affine VOA for some semisimple (finite dimensional) Lie algebra $\g$ (note that these VOSAs are automatically strongly rational, see references given above). These VOSAs are classified in \cite[Section 1, Theorem]{JoF21}, up to some special cases in the $E$-series for $\g$. Thanks to our results, we are able to prove that the unitarity condition is automatic. Moreover, all the classified VOSAs are strongly graded-local. 
This gives in turn a different proof of the strong graded locality of $V^{f\natural}$, relying on the properties of affine VOAs instead of even lattice type ones, other than new examples of irreducible $N=1$ superconformal nets with the corresponding groups of automorphisms preserving the $N=1$ superconformal structure. 

\begin{theo}[$N=1$ \textit{superconformal extensions of affine VOAs}]  \label{theo:J-F_N=1_models}
	Let $V$ be a simple VOSA such that $V_\parzero$ is an affine VOA associated to a semisimple Lie algebra. 
	Then $V$ is unitary and energy-bounded. 
	Suppose further that $V$ is $N=1$ superconformal with no vectors of conformal weight $\half$ and that $V_\parzero$ is not one of the affine VOAs $V^2(\hat{E_7})$, $V^1(\hat{E_7})\otimes V^1(\hat{E_7})$ or $V^2(\hat{E_8})$. Then $V$ is a strongly graded-local unitary $N=1$ superconformal VOSA. 
	Then there exists an irreducible $N=1$ superconformal net $\A_V$ associated to $V$, whose $\Aut^\mathrm{sc}(\A_V)$ is equal to the finite group $\Aut^\mathrm{sc}(V)$, classified in \cite[Section 1, Theorem]{JoF21}.
\end{theo}
 
\begin{proof}
	First, note that $V_\parzero$ is simple and $V_\parone$ is an irreducible $V_\parzero$-module by Example \ref{ex:even_simple_odd_irreducible}.
	Recall that any semisimple Lie algebra is a direct sum of simple Lie algebras and that the corresponding VOA is a tensor product of those associated to the simple summands, see \cite[Remark 5.7(c)]{Kac01}. Similarly, every irreducible module for $V_\parzero$ is a tensor product of irreducible modules for its VOA components arising from the corresponding simple Lie algebras, see \cite[Theorem 4.7.4]{FHL93}.
	It follows that $V_\parone$ is a unitary $V_\parzero$-module by \cite[Theorem 4.8 and Porposition 2.10]{DL14}, implying that $V$ is unitary by Theorem \ref{theo:unitarity_determind_by_even_part}.
	Recall also that $V_\parzero$ is strongly local by \cite[Example 8.7 and Corollary 8.2]{CKLW18}.
	Moreover, $V_\parzero$ is generated by vectors of conformal weight $1$, satisfying linear energy bounds, see the proof of Proposition \ref{prop:gen_by_V1/2_V1} and references therein. By the proof of Proposition \ref{prop:energy_boundedness_by_generators}, cf.\ also \cite[Proposition 3.14]{CT23}, $V_\parzero$ is a energy-bounded unitary subalgebra of $V$. Hence, $V$ is energy-bounded by Theorem \ref{theo:V_energy-bounded_iff_V_0_is}.
	
	Now, assume that $V$ is $N=1$ superconformal with no vectors of conformal weight $\half$ and that $V_\parzero$ is not one of the affine VOAs $V^2(\hat{E_7})$, $V^1(\hat{E_7})\otimes V^1(\hat{E_7})$ or $V^2(\hat{E_8})$.
	Then the existence and the uniqueness of $V$ are proved and all possible VOSAs of this type are listed in \cite[Section 1, Theorem]{JoF21}, where their groups of automorphisms preserving the $N=1$ superconformal structure are established. To do so, the author constructs the $N=1$ structure of $V$, choosing a superconformal vector $\tau$ in a real form for $V$, see Remark \ref{rem:real_form}, on which the normalized invariant scalar product of $V$ is positive-definite, see the beginning of \cite[Section 3]{JoF21}. Then $\tau$ turns out to be $\theta$-invariant. By (i) of Theorem \ref{theo:superconformal_correspondence}, $V$ is unitary $N=1$ superconformal and thus strongly graded-local by Theorem \ref{theo:superconformal_strongly_graded-local_iff_even_part_is}.
	Alternatively, the fact that $V$ is unitary $N=1$ superconformal can be deduced by Proposition \ref{prop:criteria_unitary_N=1}, noting that $V$ is finitely generated, see Remark \ref{rem:finitely_gen_VOSAs}, and that $\Aut^\mathrm{sc}(V)$ is a finite group, which preserves a unique one-dimensional subspace of $V_\frac{3}{2}$, see \cite[Section 4 and Section 5]{JoF21}.
	
	By (ii) of Theorem \ref{theo:superconformal_correspondence} and Theorem \ref{theo:AutV_AutNet_superconformal}, there exists an irreducible $N=1$ superconformal net $\A_V$ associated to $V$ such that $\Aut^\mathrm{sc}(\A_V)=\Aut^\mathrm{sc}_{\scalar}(V)=\Aut^\mathrm{sc}(V)$, whose possible occurrences are listed in \cite[Section 1, Theorem]{JoF21}.
\end{proof} 
 
Now, we move to the proof of the unitarity and of the strong graded locality of the \textbf{shorter Moonshine VOSA} $VB^\natural$, constructed for the first time in the Ph.D.\ thesis \cite{Hoe95}, see also \cite[Section 4]{Yam05} and \cite[Section 1]{Hoe10}. 
Similarly, we call the irreducible graded-local conformal net $\A_{VB^\natural}$ arising from it, the \textbf{shorter Moonshine conformal net}.
The proof uses the methodology developed in the proof of Theorem \ref{theo:lattices_strong_graded_locality} and thus we refer to the references given there for non-explained mathematical notions.

\begin{theo}[\textit{The shorter Moonshine model}]  \label{theo:baby_moonshine}
	The shorter Moonshine VOSA $VB^\natural$ is a simple strongly graded-local unitary VOSA. Then there exists an irreducible graded-local conformal net $\A_{VB^\natural}$ such that $\Aut(\A_{VB^\natural})$ is isomorphic to the direct product of the baby Monster $\mathbb{B}$ with a cyclic group of order two.
\end{theo}

\begin{proof}
Consider the even shorter Moonshine VOA $VB^\natural_\parzero$. 
As well explained in e.g.\ \cite[Section 4]{Yam05}, this is a vertex subalgebra of the Moonshine VOA $V^\natural$, constructed as the coset of a vertex subalgebra isomorphic to the Virasoro subalgebra $L(\half, 0)$, see \cite[Example 8.4]{CKLW18}. Recall that $V^\natural$ is unitary by \cite[Theorem 4.15]{DL14}. Moreover, it is possible to choose $L(\half, 0)$ in such a way that it turns to be a unitary subalgebra of $V^\natural$. 
Indeed, $V^\natural$ is constructed in \cite{Miy04} from a real form $V^\natural_\R$, see Remark \ref{rem:real_form}, which is obtained as extension of a tensor product of 48 copies of a real form of $L(\half,0)$. In particular, it is proved that $V^\natural_\R$ has a real invariant bilinear form, which is also positive-definite. Then $V^\natural$ has a unitary structure inherited from $V_\R$, making the 48 copies of $L(\half,0)$, obtained as complexification of the corresponding real forms, unitary subalgebras of $V^\natural$, see \cite[Example 5.33]{CKLW18}. 
By Proposition \ref{prop:properties_coset_subalgebra}, cf.\ also \cite[Example 5.33]{CKLW18}, $VB^\natural_\parzero$ is unitary.

Now, it is known that $V^\natural$ is a strongly rational holomorphic VOA, which is also strongly local by \cite[Theorem 8.15]{CKLW18}. Thus, $V^\natural$ satisfies \cite[Condition B]{Gui} and it is also completely energy-bounded, see the beginning of \cite[Section 2.1]{Gui}. By \cite[Theorem 8.1]{Gui19II}, $L(\half,0)$ is completely unitary. Furthermore, $VB^\natural_\parzero$ is regular by \cite[Lemma 4.2]{Yam05}. Hence, we fulfill the hypotheses of \cite[Theorem 2.6.6 and Corollary 2.6.7]{Gui}, which implies that $VB^\natural_\parzero$ satisfies \cite[Condition B]{Gui}, it is completely energy-bounded and the unitary intertwining operators satisfy the strong intertwining property and the strong braiding. By \cite[Theorem 2.4.1]{Gui}, $VB^\natural_\parzero$ is completely unitary. Note that $VB^\natural_\parzero$ is also \textit{strongly unitary}, that is every irreducible $VB^\natural_\parzero$-module can be equipped with a unitary structure, see \cite{Ten18}.

By Example \ref{ex:even_simple_odd_irreducible}, we have that $VB^\natural_\parone$ is an irreducible $VB^\natural_\parzero$-module and thus it is also unitary by the strong unitarity of $VB^\natural_\parzero$. By Theorem \ref{theo:unitarity_determind_by_even_part}, we have that $VB^\natural$ has a unitary structure.
To prove the strong graded locality of $VB^\natural$, we can therefore proceed as in the proof of Theorem \ref{theo:lattices_strong_graded_locality}. Indeed, all we need is the fact that $VB^\natural$ is a simple unitary VOSA, obtained by a $\Z_2$-simple current $VB^\natural_\parone$ of a strongly local, completely unitary and completely energy-bounded VOA  $VB^\natural_\parzero$, whose intertwining operators among unitary $VB^\natural_\parzero$-modules satisfy the strong intertwining property and the strong braiding.
Finally, by Theorem \ref{theo:AutV_AutNet} and \cite[Theorem 1]{Hoe10}, we can conclude that $\Aut(\A_{VB^\natural})=\Aut_{\scalar}(VB^\natural)=\Aut(VB^\natural)\cong \mathbb{B}\times \Z_2$.
\end{proof}

We highlight that we can produce far more examples picking graded tensor products of the simple strongly graded-local unitary VOSAs considered in the examples above. A large exposition of this kind of examples of graded-local conformal nets and of $N=1$ superconformal nets are given in \cite[Section 5.2]{CH17}.

We finish with a useful theorem with important consequences about the unitarity of the $N=1$ and $N=2$ superconformal structures of some VOSA models. We work in the setting of \cite[Section 4]{CGGH23} and thus we refer to it and references therein for unexplained notations and definitions.

\begin{theo} \label{theo:unitarity_VOSA_extensions}
	Let $U$ be a simple unitary VOSA and let $V$ be a vertex subalgebra with the same conformal vector. Assume that: $V_\parzero$ is strongly rational; non-zero vectors of the irreducible $V_\parzero$-modules have non-negative real conformal weights; $V_\parzero$ is the only irreducible $V_\parzero$-module with a non-zero vector of conformal weight zero. Then $V$ is a unitary subalgebra of $U$ if and only if $V_\parzero$ is.
\end{theo}   

\begin{proof}
	If $V$ is a unitary subalgebra of $U$, then so is $V_\parzero$. 
	Vice versa, suppose that $V_\parzero$ is a unitary subalgebra of $U$.
	Let $Y$ be the state-field correspondence of $U$ and $\theta$ be its PCT operator. Consider the $V_\parzero$-module given by the vector space $\theta(V_\parone)$ with the usual scalar multiplication and state-field correspondence $Y^{\theta(V_\parone)}:=Y$. Then $\theta$ induces an isomorphism between $\theta(V_\parone)$ and the conjugate module $\overline{V_\parone}$, see Example \ref{ex:conjugate_module}, which is in turn isomorphic to $V_\parone$, see \eqref{eq:iso_V0-modules_pct}. Thus, $\theta(V_\parone)$ and $\overline{V_\parone}$ are equivalent $V_\parzero$-submodules of $U$.
	
	Suppose by contradiction that $\theta(V_\parone)\not= V_\parone$. Then it must be $\dim \Hom_{\Rep(V_\parzero)}(V_\parone, U)>1$. 
	On the one hand, 
	$$
	\dim\Hom_{\Rep(U_\parzero)}(U_\parzero\boxtimes V_\parone, U)=\dim\Hom_{\Rep(V_\parzero)}(V_\parone, U)
	$$ 
	by \cite[Theorem 1.6 (2)]{KO02}, cf.\ also \cite[Lemma 2.61]{CKM24} 
	(this can be considered as a VOA version of the $\alpha\sigma$-reciprocity for conformal nets, see e.g.\ \cite[Theorem 3.21]{BE98}).
	On the other hand, 
	$$
	\dim\Hom_{\Rep(U_\parzero)}(U_\parzero\boxtimes V_\parone, U)=\dim\Hom_{\Rep(U_\parzero)}(U_\parone, U)
	$$ 
	because $U_\parzero\boxtimes V_\parone\cong U_\parone$ as $U_\parzero$-modules by \cite[Proposition 4.23]{CGGH23}. Then
	$$
	\dim\Hom_{\Rep(V_\parzero)}(V_\parone, U)=\dim\Hom_{\Rep(U_\parzero)}(U_\parzero\boxtimes V_\parone, U)
	=\dim\Hom_{\Rep(U_\parzero)}(U_\parone, U)=1
	$$
	which leads to the contradiction that concludes the proof.
\end{proof}

\begin{cor}  \label{cor:unitarity_N=1_N=2_superconformal_extensions}
	The simple CFT type VOSA extensions of the (simple) $N=1$ and $N=2$ super-Virasoro VOSAs with central charge in the respective unitary discrete series, classified in \cite[Section 6]{CGGH23}, are simple unitary $N=1$ and $N=2$ superconformal VOSAs respectively.
\end{cor}

\section{Back to VOSAs}
\label{section:back}

First, we show how to explicitly reconstruct the VOSA structure of a simple strongly graded-local unitary VOSA $V$ from its irreducible graded-local conformal net $\A_V$. This can be considered as a different proof of the ``if'' part of Theorem \ref{theo:unicity_irr_graded-local_cn}.
In this way, we also introduce a more general theory, which allows us to
give sufficient conditions in order that a given irreducible graded-local conformal net arises from a simple strongly graded-local unitary VOSA.

The fundamental idea for the theory we are going to present in the following is to construct certain \textit{quantum fields} starting from an irreducible graded-local M\"{o}bius covariant net. This construction relies on the idea developed in \cite{FJ96} by Fredenhagen and J\"{o}r{\ss}, who
construct certain pointlike localized fields associated to an irreducible M\"{o}bius covariant net by
a scaling limit procedure.
Instead, we use and naturally generalize the alternative construction based on
the Tomita-Takesaki modular theory for von Neumann algebras (see e.g.\  \cite[Chapter 2.5]{BR02}) presented in \cite[Chapter 9]{CKLW18}. 

For the reader's convenience, we recall the following standard argument.

\begin{rem}
 \label{rem:modular_theory_for_smeared_vertex_ops}

Let $\mathcal{M}$ be a von Neumann algebra on a Hilbert space $\mathcal{H}$ with a cyclic and separating vector $\Omega\in\mathcal{H}$ for $\mathcal{M}$. As usual, we denote by $S, \Delta$ and $J$ the Tomita operator, the modular operator and the modular conjugation associated to the pair $(\mathcal{M},\Omega)$ respectively. Recall that $S$ is an antilinear operator such that $SA\Omega=A^*\Omega$ for all $A\in\mathcal{M}$, $S^*A\Omega=A^*\Omega$ for all $A\in\mathcal{M}'$ and $S=J\Delta^\half$. For an arbitrary vector $a\in\mathcal{H}$, we define the following linear operator with dense domain:
\begin{equation}
\mathscr{L}_a^0:\mathcal{M}'\Omega\longrightarrow
\mathcal{H}
\,,\quad
A\Omega\longmapsto Aa \,.
\end{equation}
If $a$ is in the domain $\mathcal{D}(S)$ of $S$, then
\begin{equation}
\begin{split}
(\mathscr{L}_{Sa}^0A\Omega|B\Omega)
 &=
(ASa|B\Omega)
=(a|S^*A^*B\Omega) \\
 &=
(B^*A\Omega|a)
=(A\Omega|Ba) \\
 &=
(A\Omega|\mathscr{L}_a^0B\Omega)
\qquad\forall A, B\in \mathcal{M}'\Omega 
\end{split}
\end{equation}
that is, $\mathscr{L}_{Sa}^0\subseteq(\mathscr{L}_a^0)^*$. Hence, $\mathscr{L}_a^0$ and $\mathscr{L}_{Sa}^0$ are both closable operators and their respective closures $\mathscr{L}_a$ and $\mathscr{L}_{Sa}$ satisfy $\mathscr{L}_{Sa}\subseteq\mathscr{L}_a^*$ for all $a\in\mathcal{D}(S)$. We also have that for all $a\in\mathcal{D}(S)$, the operators $\mathscr{L}_a$ and $\mathscr{L}_{Sa}$ are affiliated with $\mathcal{M}$.
\end{rem} 

The operator $\mathscr{L}_a$ with $a\in\mathcal{D}(S)$ of Remark \ref{rem:modular_theory_for_smeared_vertex_ops} can be interpreted in certain situations, see \cite{Car05}, as abstract analogue of smeared vertex operators, see also \cite{BBS01}. We explain this point of view in the following, constructing a family of localized fields associated to a given irreducible graded-local M\"{o}bius covariant net $\A$ on $S^1$. Let $\mathcal{H}$ be the vacuum Hilbert space of $\A$. 
We call an eigenvector $a\in\mathcal{H}$ of the grading unitary $\Gamma$ \textbf{even} or \textbf{odd} if it has eigenvalue $1$ or $-1$ respectively. Accordingly, the \textbf{parity} $p(a)\in\Z_2$ of $a$ is $\parzero$ or $\parone$ respectively.  
Let $U$ be the positive-energy strongly continuous unitary representation of $\Mob(S^1)^{(\infty)}$. Recall that the conformal Hamiltonian $L_0$, that is, the infinitesimal generator of the one-parameter subgroup of $U(\Mob(S^1)^{(\infty)})$ of rotations, is a positive self-adjoint operator on $\mathcal{H}$. 
By the Vacuum Spin-Statistic theorem \eqref{eq:spin-statistics_theorem}, $U$ factors through a representation of $\Mob(S^1)^{(2)}$, which we denote by the same symbol. Accordingly,
the spectrum of $L_0$ is contained in $\half\Zpluseq$ and we set the usual definition
\begin{equation}
\mathcal{H}^\mathrm{fin}:=
\bigoplus_{n\in\half\Zpluseq}
\mathrm{Ker}(L_0-n1_\mathcal{H})
\subseteq\mathcal{H}^\infty\subseteq \mathcal{H} 
\end{equation}
where $\mathcal{H}^\infty$ is the dense subspace of smooth vectors for $L_0$, see p.\ \pageref{defin:smooth_vectors_L_0} and references therein. 
Therefore, $U$ differentiates uniquely to a positive-energy unitary representation of the Virasoro algebra $\Vir$ on $\mathcal{H}^\mathrm{fin}$ and thus it restricts to a unitary representation of $\mathfrak{sl}(2,\C)$ on $\mathcal{H}^\mathrm{fin}$, the complexification of $\mathfrak{sl}(2,\R)$, see Section \ref{subsection:diff_group} and references therein. The latter representation is spanned by the operators $L_{-1}, L_0$ and $L_1$ on $\mathcal{H}$, satisfying $L_1\subseteq L_{-1}^*$ and the usual commutation relations 
\begin{equation}   \label{eq:cr_sl(2,C)}
[L_1,L_{-1}]  = 2L_0
\,,\qquad
[L_1,L_0]  = L_1
\,,\qquad
[L_{-1},L_0]  = -L_{-1} \,.
\end{equation} 
A vector $a\in\mathcal{H}$ 
is called \textbf{homogeneous} of \textbf{conformal weight} $d_a\in\half\Z$, if $L_0a=d_aa$, so that $a\in\mathcal{H}^\mathrm{fin}$.
By the Vacuum Spin-Statistics theorem \eqref{eq:spin-statistics_theorem}, if $d_a\in\Z$, then $a$ is even, otherwise it is odd.
Moreover, a homogeneous vector $a\in\mathcal{H}$  is said to be \textbf{quasi-primary}, if $L_1a=0$. 
Fix a quasi-primary vector $a\in\mathcal{H}$. With a standard induction argument, we can prove that the commutation relations \eqref{eq:cr_sl(2,C)} imply that the vectors
\begin{equation}
a^n:=\frac{L_{-1}^na}{n!}\in\mathcal{H}^\mathrm{fin}
\qquad \forall n\in\Zpluseq
\end{equation}
satisfy
\begin{equation}
L_0a^n=(n+d_a)a^n
\qquad \forall n\in\Zpluseq
\,,\qquad
L_1 a^n=(2d_a+n-1)a^{n-1}
\qquad\forall n\in\Zplus \,.
\end{equation}
Moreover, it is not difficult to prove that
\begin{equation}  \label{eq:norm_a^n}
\norm{a^n}^2=(a^n|a^n)=
\binom{2d_a+n-1}{n}\norm{a}^2
\qquad\forall n\in\Zpluseq\,.
\end{equation}
By \eqref{eq:norm_a^n}, we can prove that for all $f\in C_{p(a)}^\infty(S^1)$, see \eqref{eq:notation_even_odd} for the notation, the series 
$\sum_{n\in\Zpluseq} \widehat{f}_{-n-d_a}a^n$ converges in $\mathcal{H}^\infty$ to an element, which we denote by $a(f)$. Furthermore, $f\mapsto a(f)$ defines a continuous linear map from $C_{p(a)}^\infty(S^1)$ to $\mathcal{H}^\infty$ and thus proceeding as done to prove Proposition \ref{prop:mob_covariance_qp}, we have that:
\begin{prop}   \label{prop:mob_covariance_a(f)}
Let $a$ be a quasi-primary vector in $\mathcal{H}$. Then for all $I\in\J$, we have that $U(\gamma)a(f)=a(\iota_{d_a}(\gamma)f)$ for all $\gamma\in\Mob(S^1)^{(2)}$ and all $f\in C_{p(a)}^\infty(S^1)$ with $\mathrm{supp}f\subset I$.
\end{prop}

Now, consider for every interval $I\in\J$, the Tomita operator $S_I=J_I\Delta_I^\half$ associated to the von Neumann algebra $\A(I)$. By the M\"{o}bius covariance of the net, we get that
\begin{equation}  \label{eq:covariance_modular_theory}
U(\gamma)S_IU(\gamma)^*=S_{\dot{\gamma} I}
\,,\qquad
U(\gamma)J_IU(\gamma)^*=J_{\dot{\gamma} I}
\,,\qquad
U(\gamma)\Delta_IU(\gamma)^*=\Delta_{\dot{\gamma} I}
\end{equation}
for all $I\in\J$ and all
$\gamma\in\Mob(S^1)^{(2)}$. Furthermore, by the Bisognano-Wichmann property \eqref{eq:B-W_dilations}, we have that
\begin{equation} 
 \label{eq:properties_modular_op_upper_semicircle}
\Delta_{S^1_+}^{it}=e^{iKt}
\quad\forall t\in\R
\,,\qquad
\Delta_{S^1_+}^\half=e^\frac{K}{2}
\,,\qquad
K:=i\pi\overline{(L_1-L_{-1})}
\end{equation}
where $K$ is the infinitesimal generator of the one-parameter subgroup of dilations. We set $\theta:=ZJ_{S^1_+}$ and we note that $\theta$ commutes with $L_{-1}, L_0$ and $L_1$. (Indeed, $\theta$ will be the PCT operator of the unitary VOSA we are going to construct, cf.\  the proof of Theorem \ref{theo:gen_by_quasi_primary}.)

It should be not difficult for the reader to recognize that an equivalent of Theorem \ref{theo:delta_one_half} can be derived for $a(f)$ here in place of $Y(a,f)\Omega$ there. Indeed, although it is stated in the energy-bounded unitary VOSA setting, $a(f)$ shares with $Y(a,f)\Omega$ all the properties necessary to the proof of Theorem \ref{theo:delta_one_half}. Hence, we state:
\begin{theo}  \label{theo:delta_one_half_a(f)}
Let $a$ be a quasi-primary vector in $\mathcal{H}$ and let $f\in C_c^\infty(S^1 \backslash \left\{-1\right\})$ with $\mathrm{supp}f\subset S^1_+$. Then $a(f)$ is in the domain of the operator $e^\frac{K}{2}$ and 
\begin{equation}
e^\frac{K}{2}a(f)=i^{2d_a}a(f\circ j)
\end{equation}
where $j(z)=\overline{z}=z^{-1}$ for all $z\in S^1$. 
\end{theo}
As an application, we get:
\begin{cor}
Let $a$ be a quasi-primary vector in $\mathcal{H}$ and let $I\in\J$.
Then $a(f)$ is in the domain of the operator $S_I$ for all $f\in C_{p(a)}^\infty(S^1)$ with $\mathrm{supp}f\subset I$  and
$$
S_Ia(f)
=(-1)^{2d_a^2+d_a}(\theta(a))(\overline{f}) \,.
$$
\end{cor}

\begin{proof}
By Theorem \ref{theo:delta_one_half_a(f)} and \eqref{eq:properties_modular_op_upper_semicircle}, we have that
\begin{equation}
\Delta_{S^1_+}^\half a(f)=e^\frac{K}{2}a(f)=i^{2d_a}a(f\circ j)
\qquad
\forall f\in C_c^\infty(\pSone)
\,\,\,\mbox{ with }\,\,\,
\mathrm{supp}f\subset S^1_+ \,.
\end{equation} 
Thus, $a(f)$ is in the domain of $S_{S^1_+}$ and
\begin{equation}
S_{S^1_+}a(f)
=J_{S^1_+}\Delta_{S^1_+}^\half a(f)
=(-1)^{2d_a^2+d_a}(\theta(a))(\overline{f})
\end{equation}
for all $f\in C_c^\infty(S^1 \backslash \left\{-1\right\})$ with $\mathrm{supp}f\subset S^1_+$.
Let $I\in\J$ and choose $\gamma\in\Mob(S^1)^{(2)}$ such that $\dot{\gamma}I=S^1_+$.
By the covariance of the modular theory \eqref{eq:covariance_modular_theory} and Proposition \ref{prop:mob_covariance_a(f)}, we get that $a(f)$ is in the domain of $S_I$ for all $f\in C_{p(a)}^\infty(S^1)$ with $\mathrm{supp}f\subset I$ because
\begin{equation}
\begin{split}
S_Ia(f)
 &=
U(\gamma)^*S_{S^1_+} U(\gamma)a(f)  
 =
U(\gamma)^* S_{S^1_+}a(\iota_{d_a}(\gamma)f)  \\
 &=
(-1)^{2d_a^2+d_a}U(\gamma)^*(\theta(a))(\overline{\iota_{d_a}(\gamma)f})  
 =
(-1)^{2d_a^2+d_a}(\theta(a))(\overline{f}) \,.
\end{split}
\end{equation}
This also concludes the proof.
\end{proof}

By Remark \ref{rem:modular_theory_for_smeared_vertex_ops}, we can construct the closed operator $\mathscr{L}_{a(f)}^I$ affiliated to $\A(I)$ for all $I\in\J$, where $a\in\mathcal{H}$ is a quasi-primary vector and $f\in C_{p(a)}^\infty(S^1)$ with $\mathrm{supp}f\subset I$.

\begin{defin}  \label{defin:FJ_smeared_vertex_op}
For every quasi-primary vector $a\in\mathcal{H}$, every $I\in\J$ and every $f\in C_{p(a)}^\infty(S^1)$ with $\mathrm{supp}f\subset I$, $\mathscr{L}_{a(f)}^I$ is called a \textbf{Fredenhagen-J\"{o}r{\ss}}, shortly \textbf{FJ}, \textbf{smeared vertex operator} and it is denoted by $Y_I(a,f)$.
\end{defin}

The notation introduced by Definition \ref{defin:FJ_smeared_vertex_op} is justified by the following properties of FJ smeared vertex operators, which follow directly from the properties of the operators $\mathscr{L}_{a(f)}^I$ constructed above.

\begin{prop}  \label{prop:FJ_smeared_vertex_ops_properties}
Let $a$ be a quasi-primary vector in $\mathcal{H}$. Then the FJ smeared vertex operators have the following properties:
\begin{itemize}
\item[(i)] 
Fix $I\in\J$. Then for every $f\in C_{p(a)}^\infty(S^1)$, $\A(I)'\Omega$ is a core for $Y_I(a,f)$ and $Y_I(a,f)\Omega=a(f)$. Furthermore, for every $b\in \A(I)'\Omega$, the map 
\begin{equation} \label{eq:FJ_operator_valued_distr}
\{f\in C_{p(a)}^\infty(S^1)\mid \mathrm{supp}f\subset I\}
\ni f
\longmapsto 
Y_I(a,f)b\in\mathcal{H}
\end{equation}
is an operator-valued distribution, that is, it is linear and continuous.

\item[(ii)] 
If $I_1\subseteq I_2$ in $\J$, then 
$Y_{I_1}(a,f)\subseteq  Y_{I_2}(a,f)$ for all $f\in C_{p(a)}^\infty(S^1)$ with $\mathrm{supp}f\subset I_1$, so that the operator-valued distribution $f\mapsto Y_{I_1}(a,f)$ as defined in \eqref{eq:FJ_operator_valued_distr} extends to the one 
$f\mapsto Y_{I_2}(a,f)$.

\item[(iii)] 
The following M{\"o}bius covariance property holds:
\begin{equation}
U(\gamma) Y_I(a,f) U(\gamma)^*
=Y_{\dot{\gamma}I}(a,\iota_{d_a}(\gamma)f)
\end{equation}
for all $\gamma\in\Mob(S^1)^{(2)}$, all $I\in\J$ and all $f\in C_{p(a)}^\infty(S^1)$ with $\mathrm{supp}f\subset I$.

\item[(iv)]
The following hermiticity formula holds:
\begin{equation}
(-1)^{2d_a^2+d_a} Y_I(\theta(a),\overline{f})
\subseteq Y_I(a,f)^* 
\end{equation}
for all $I\in\J$ and all $f\in C_{p(a)}^\infty(S^1)$ with $\mathrm{supp}f\subset I$.
\end{itemize}
\end{prop}

Thanks to (i) of Proposition \ref{prop:FJ_smeared_vertex_ops_properties}, we can use the usual formal notation of distributions for FJ smeared vertex operators:
\begin{equation}
Y_I(a,f)=\oint_{S^1} Y_I(a,z)f(z) z^{d_a}\frac{dz}{2\pi iz}
\end{equation}
for every quasi-primary vector $a\in\mathcal{H}$ of conformal weight $d_a$, every $I\in\J$ and every $f\in C_{p(a)}^\infty(S^1)$ with $\mathrm{supp}f\subset I$.
Accordingly, we have families of type $\{Y_I(a,z)\mid I\in\J\}$, which we call \textbf{FJ vertex operators} or \textbf{FJ fields}. As for the local case, it is unknown whether the FJ smeared vertex operators admit a common invariant domain and this prevents us to extend a FJ vertex operator $\{Y_I(a,z)\mid I\in\J\}$ to a unique distribution $\widetilde{Y}(a,z)$. This also implies that the FJ fields cannot be considered as quantum fields in the sense of Wightman, see \cite[Chapter 3]{SW64}.

\begin{prop}
Let $\A$ be an irreducible graded-local M\"{o}bius covariant net. Then $\A$ is generated by its FJ smeared vertex operators, which means that for all $I\in\J$, 
\begin{equation} 
 \label{eq:net_gen_by_FJ_smeared_vertex_ops}
\A(I)=
W^*\left(
\left\{
Y_{I_1}(a,f) \left| 
\begin{array}{l}
a\in\bigcup_{n\in\half\Zpluseq}(\Ker(L_0-n1_\mathcal{H}))
\,,\,\, L_1a=0 \,, \\ f\in C_{p(a)}^\infty(S^1)
\,,\,\,\mathrm{supp}f\subset I_1\subseteq I \,,\,\,\, I_1\in\J
\end{array}
\right.
\right\}
\right)\,.
\end{equation}
\end{prop}

\begin{proof}
We trivially adapt the proof of \cite[Proposition 9.1]{CKLW18}.
For every $I\in\J$, call $\B(I)$ the right hand side of \eqref{eq:net_gen_by_FJ_smeared_vertex_ops}, so defying a graded-local M\"{o}bius covariant subnet $\B$ of $\A$. By (i) of Proposition \ref{prop:FJ_smeared_vertex_ops_properties}, for every quasi-primary vector $a\in\mathcal{H}$ and every $f\in C_{p(a)}^\infty(S^1)$, $a(f)$ belongs to the vacuum Hilbert space $\mathcal{H}_\B=\overline{\B(S^1)\Omega}$ of $\B$. Since the representation $U$ of $\Mob(S^1)^{(2)}$ on $\mathcal{H}$ is completely reducible, all the vectors of type $a(f)$ linearly span a dense subset of $\mathcal{H}$. Consequently, $\mathcal{H}=\mathcal{H}_\B$ and thus $\A=\B$, see Remark \ref{rem:cyclicity_covariant_subnets}. 
\end{proof}

Now, we can prove the first of the two main results of this section. As previously anticipated, we show that the FJ smeared vertex operators of an irreducible graded-local conformal net arisen from a simple strongly graded-local unitary VOSA are actually the ordinary smeared vertex operators. Note that we obtain the ``if'' part of Theorem \ref{theo:unicity_irr_graded-local_cn} as a corollary.

\begin{theo}   \label{theo:back}
Let $V$ be a simple strongly graded-local unitary VOSA and consider the corresponding irreducible graded-local conformal net $\A_V$. Then for every quasi-primary vector $a\in V$, we have that $Y_I(a,f)=Y(a,f)$ for all $I\in\J$ and all $f\in C_{p(a)}^\infty(S^1)$ with $\mathrm{supp}f\subset I$. In particular, one can recover the VOSA structure on $V=\mathcal{H}^\mathrm{fin}$ from the graded-local conformal net $\A_V$.
\end{theo}

\begin{proof}
We easily adapt the proof of \cite[Theorem 9.2]{CKLW18}.
Note that $Y(a,f)$ is affiliated with $\A_V(I)$ and thus its domain contains $\A_V(I)'\Omega$, which contains $Z\A_V(I')\Omega\cap\mathcal{H}^\infty$ by twisted Haag duality \eqref{eq:twisted_haag_duality}. By a slight modification of \cite[Proposition 7.3]{CKLW18}, the latter is a core for $Y(a,f)$, implying that also the former is a core for it. By (i) of Proposition \ref{prop:FJ_smeared_vertex_ops_properties}, $\A_V(I)'\Omega$ is a core for every FJ smeared vertex operator $Y_I(a,f)$. \cite[Eq.\  (4.1.2)]{Kac01} says that $Y(b,z)\Omega=e^{zL_{-1}}b$ for all $b\in V$. Thus, $Y(a,f)\Omega=a(f)$ and 
$$
Y(a,f)A\Omega=AY(a,f)\Omega=Aa(f)=Y_I(a,f)A\Omega
\qquad\forall A\in\A_V(I)'\,.
$$
It follows that the closed operators $Y(a,f)$ and $Y_I(a,f)$ coincide on a common core and thus they must be the same.
\end{proof}

Now, we can move to the last part of this section, which consists of the reconstruction of a simple strongly graded-local unitary VOSA $V$ from a given irreducible graded-local conformal net $\A$ such that $\A=\A_V$. To reach this goal, we note that a necessary condition is that for every quasi-primary vector $a\in\mathcal{H}$, the corresponding FJ vertex operator $\{Y_I(a,z)\mid I\in\J\}$ satisfies \textbf{energy bounds}, i.e., there exists non-negative real numbers $M,s$ and $k$ such that
\begin{equation}   \label{eq:energy_bounds_FJ}
\norm{Y_I(a,f)b}\leq 
M\norm{f}_s \norm{(L_0+1_\mathcal{H})^kb}
\end{equation}
for all $I\in\J$, all $f\in C_{p(a)}^\infty(S^1)$ with $\mathrm{supp}f\subset I$ and all $b\in \A(I)'\Omega\cap\mathcal{H}^\infty$. We are going to prove that the energy bound condition for the FJ smeared vertex operators is actually a sufficient condition.

We say that a family of quasi-primary vectors $\F\subset\mathcal{H}$ generates $\A$ if the corresponding FJ smeared vertex operators generate the graded-local von Neumann algebras, that is, if
\begin{equation}
\A(I)=
W^*\left(\left\{
Y_{I_1}(a,f) \mid
a\in \F
\,,\,\, f\in C_{p(a)}^\infty(S^1)
\,,\,\,\mathrm{supp}f\subset I_1\subseteq I
\,,\,\, I_1\in\J
\right\}\right)
\quad\forall I\in\J \,.
\end{equation}

Next, we present the second main result of this section:
\begin{theo}  \label{theo:back_energy_bounds}
Let $\A$ be an irreducible graded-local conformal net. Suppose that $\A$ is generated by a family of quasi-primary vectors $\F$. Suppose further that $\F$ is $\theta$-invariant and that the corresponding FJ vertex operators satisfy energy bounds. Furthermore, assume that $\mathrm{Ker}(L_0-n1_\mathcal{H})$ is finite-dimensional for all $n\in\half\Zpluseq$. Then the complex vector space $V:=\mathcal{H}^\mathrm{fin}$ has a structure of simple strongly graded-local unitary VOSA such that $\A=\A_V$.
\end{theo}

\begin{proof}
The following proof retraces the one of \cite[Theorem 9.3]{CKLW18} with very few adjustments where necessary.

As for the smeared vertex operators in Section \ref{subsection:energy_bounds}, we can prove that the energy bound condition implies that $\mathcal{H}^\infty$ is a common invariant core for the operators $Y_I(a,f)$ with $I\in\J$, $a\in\F$ and $f\in C_{p(a)}^\infty(S^1)$ with $\mathrm{supp}f\subset I$. Let $\{I_1, I_2\}\subset \J$ be a cover of $S^1$ and let $\{\varphi_1,\varphi_2\}\subset C^\infty(S^1,\R)$ be a partition of unity on $S^1$ subordinate to $\{I_1, I_2\}$, that is, $\mathrm{supp}\varphi_j\subset I_j$ for all $j\in\{1,2\}$ and $\sum_{j=1}^2\varphi_j(z)=1$ for all $z\in S^1$. Then for every $a\in\F$ and every $f\in C_{p(a)}^\infty(S^1)$, we define the operators $\widetilde{Y}(a,f)$ on $\mathcal{H}$ with domain $\mathcal{H}^\infty$ by
\begin{equation} \label{eq:defin_general_FJ_smeared_vo}
\widetilde{Y}(a,f)b=\sum_{j=1}^2 Y_{I_j}(a,\varphi_j f)b
\qquad b\in\mathcal{H}^\infty \,.
\end{equation}
Let $\{\widetilde{I_1},\widetilde{I_2}\}\subset\J$ be a second cover of $S^1$ and let $\{\widetilde{\varphi_1},\widetilde{\varphi_2} \}\subset C^\infty(S^1,\R)$ be a partition of unity on $S^1$ subordinate to $\{\widetilde{I_1},\widetilde{I_2}\}$. Thus, we have that
\begin{equation}
\begin{split}
\widetilde{Y}(a,f)b
 &=
\sum_{j=1}^2 Y_{I_j}(a,\varphi_j f)b
=\sum_{k,j=1}^2 Y_{I_j}(a,\widetilde{\varphi_k}\varphi_j f)b \\
 &= 
\sum_{k,j=1}^2 Y_{\widetilde{I_k}}(a,\widetilde{\varphi_k}\varphi_j f)b
= \sum_{k=1}^2 Y_{\widetilde{I_k}}(a,\widetilde{\varphi_k} f)b
\qquad\forall b\in\mathcal{H}^\infty
\end{split}
\end{equation}
where we have used (ii) of Proposition \ref{prop:FJ_smeared_vertex_ops_properties} for the third equality. This means that the definition \eqref{eq:defin_general_FJ_smeared_vo} of $\widetilde{Y}(a,f)$ is independent of the choice of the cover of $S^1$ and of the partition of unity subordinate to it. As a consequence, we have that $\widetilde{Y}(a,f)=Y_I(a,f)$ for all $a\in\F$, all $I\in \J$ and all $f\in C_{p(a)}^\infty(S^1)$ with $\mathrm{supp}f\subset I$. This implies the state-field correspondence $\widetilde{Y}(a,f)\Omega=a(f)$ for all $a\in\F$ and all $f\in C_{p(a)}^\infty(S^1)$, see (i) of Proposition \ref{prop:FJ_smeared_vertex_ops_properties}. Moreover, we get the M\"{o}bius covariance property and the adjoint relation for the operators $\widetilde{Y}(a,f)$ with $f\in C_{p(a)}^\infty(S^1)$ as stated in (iii) and (iv) of Proposition \ref{prop:FJ_smeared_vertex_ops_properties} respectively.

By hypothesis, for every $a\in\F$, the FJ vertex operator $\{Y_I(a,z)\mid I\in\J\}$ satisfies energy bounds for some non-negative real numbers $M,s$ and $k$.
By the convolution theorem for the Fourier series, for every $\varphi\in C^\infty(S^1)$ and every
$f\in C_{p(a)}^\infty(S^1)$, it is not difficult to prove that
$$
\abs{\widehat{(\varphi f)}_n}\leq 
\sum_{j\in \Z-d_a}\abs{\widehat{f}_j}\abs{\widehat{\varphi}_{n-j}}
\qquad\forall n\in\Z-d_a \,.
$$ 
Hence, 
\begin{equation}
\begin{split}
\norm{\varphi f}_s 
 &=
\sum_{n\in\Z-d_a}
(\abs{n}+1)^s \abs{\widehat{(\varphi f)}_n} 
 \leq
\sum_{n\in\Z-d_a} \sum_{j\in\Z-d_a}
(\abs{n}+1)^s 
\abs{\widehat{f}_j}\abs{\widehat{\varphi}_{n-j}} \\
 &= 
\sum_{m\in\Z} \sum_{j\in\Z-d_a} 
(\abs{m+j}+1)^s 
\abs{\widehat{f}_j}\abs{\widehat{\varphi}_{m}} 
 \leq 
\norm{\varphi}_s\norm{f}_s 
\end{split}
\end{equation}
which implies that
\begin{equation}
\norm{\widetilde{Y}(a,f)b}
=\norm{\sum_{j=1}^2 Y_{I_j}(a,\varphi_j f)b}
\leq M \left( \sum_{j=1}^2\norm{\varphi_j}_s \right)
\norm{f}_s \norm{(L_0+1_\mathcal{H})^kb}
\end{equation}
for all $f\in C_{p(a)}^\infty(S^1)$ and all $b\in \mathcal{H}^\infty$. This means that the operators $\widetilde{Y}(a,f)$ with $f\in C_{p(a)}^\infty(S^1)$ satisfy energy bounds for some $s$ and $k$ and the non-negative real number $\widetilde{M}:=M\left( \sum_{j=1}^2\norm{\varphi_j}_s \right)$.

For every $n\in\Z$ and every $m\in\Z-\half$, define $e_n\in C^\infty(S^1)$ and $e_m\in C_\chi^\infty(S^1)$ by $e_n(z):=z^n$ and $e_m(z):=\chi(z)z^{m-\half}$ with $z\in S^1$ respectively. For any $a\in\F$ and any $n\in\Z-d_a$, set $a_n:=\widetilde{Y}(a,e_n)$. Thus, we have that
$$
\norm{a_nb}
\leq \widetilde{M}(\abs{n}+1)^s\norm{(L_0+1_\mathcal{H})^kb}
\qquad\forall n\in\Z-d_a
\,\,\,\forall b\in\mathcal{H}^\infty \,.
$$
By definition, we get that $a_{-d_a}\Omega=a(e_{-d_a})=a$ for all $a\in\F$.
The M\"{o}bius covariance property implies that $e^{itL_0}a_ne^{-itL_0}=e^{-int}a_n$ for all $a\in\F$, all $t\in \R$ and all $n\in\Z-d_a$. This implies that $[L_0,a_n]b=-na_nb$ for all $a\in\F$, all $n\in\Z-d_a$ and all $b\in\mathcal{H}^\infty$ and thus $a_n$ preserves $\mathcal{H}^\mathrm{fin}$ for all $a\in\F$ and all $n\in\Z-d_a$. By the M\"{o}bius covariance property, we also have the remaining commutation relations $[L_{-1},a_n]b=(-n-d_a+1)a_{n-1}b$ and $[L_1,a_n]b=(-n+d_a-1)a_{n+1}b$ for all $a\in\F$, all $n\in\Z-d_a$ and all $b\in\mathcal{H}^\infty$. Define the linear span
\begin{equation}
V:=\langle \Omega \,,\,\,
a^1_{n_1}\cdots a^l_{n_l}\Omega \mid
l\geq 1 \,,\,\,\,
a^j\in\F \,,\,\,\, n_j\in\Z-d_{a^j}
\,\,\,
\forall j\in\{1,\cdots, l\} 
\rangle
\subseteq \mathcal{H}^\mathrm{fin}
\end{equation}
We want to show that $V=\mathcal{H}^\mathrm{fin}$. Let $\mathcal{H}_V$ be the closure of $V$ in $\mathcal{H}$ and $e_V$ be the orthogonal projection onto $\mathcal{H}_V$. For any $f\in C_{p(a)}^\infty(S^1)$, $\sum_{n\in\Z-d_a}\widehat{f}_ne_n$ converges to $f$ in $C_{p(a)}^\infty(S^1)$ and hence
$$
\widetilde{Y}(a,f)b=\sum_{n\in\Z-d_a}\widehat{f}_na_nb
\qquad\forall a\in\F
\,\,\,
\forall f\in C_{p(a)}^\infty(S^1)
\,\,\,
\forall b\in\mathcal{H}^\infty \,.
$$
It follows that $\widetilde{Y}(a,f)b$ and $\widetilde{Y}(a,f)^*b$ are in $\mathcal{H}_V$ for all $a\in\F$, all $f\in C_{p(a)}^\infty(S^1)$ and all $b\in\mathcal{H}^\mathrm{fin}$. Considering that $\mathrm{Ker}(L_0-n1_\mathcal{H})$ is finite-dimensional for all $n\in\half\Zpluseq$, we get that $e_V\mathcal{H}^\mathrm{fin}=V$. Consequently, we have that $[e_V,\widetilde{Y}(a,f)]b=0$ for all $a\in\F$, all $f\in C_{p(a)}^\infty(S^1)$ and all $b\in V$. But $\mathcal{H}^\mathrm{fin}$ is a core for every FJ smeared vertex operator, which means that $e_V Y_I(a,f)\subseteq Y_I(a,f)e_V$ for all $a\in\F$, all $I\in\J$ and all $f\in C_{p(a)}^\infty(S^1)$ with $\mathrm{supp}f\subset I$. Thus, using that $\F$ generates $\A$, it must be $e_V=1_\mathcal{H}$ by the irreducibility of $\A$. This obviously means that $V=\mathcal{H}^\mathrm{fin}$.

What we have just proved says us that the formal series
\begin{equation}
\Phi_a(z):=\sum_{n\in\Z-d_a}a_nz^{-n-d_a}
\qquad \forall a\in\F
\end{equation}
are translation covariant (with respect to the even endomorphism $T:=L_{-1}$) parity-preserving fields on $V$, which are also mutually local (in the vertex superalgebra sense) thanks to the graded-locality of $\A$ and Proposition \ref{prop:wightman_locality}. Hence, we have a unique vertex superalgebra structure on the vector superspace $(V,\Gamma\restriction_V)$ thanks to the existence theorem for vertex superalgebra \cite[Theorem 4.5]{Kac01} with vertex operators satisfying $Y(a,z)=\Phi_a(z)$ for all $a\in\F$. Moreover, we have a unitary representation of the Virasoro algebra on $V$ by operators $L_n$ with $n\in\Z$ by differentiating the representation $U$ of $\Diff^+(S^1)^{(\infty)}$ associated to $\A$, see Theorem \ref{theo:representations_Diff_Vir}. Consequently, $L(z)=\sum_{n\in\Z}L_nz^{-n-2}$ is a local field on $V$, which is also mutually local with respect to every $Y(a,z)$ with $a\in V$ thanks to the graded-locality of $\A$. Furthermore, $L(z)\Omega=e^{zL_{-1}}L_{-2}\Omega$ and thus $L(z)=Y(\nu,z)$ with $\nu:=L_{-2}\Omega$ by the uniqueness theorem for vertex superalgebras \cite[Theorem 4.4.]{Kac01}. Clearly, $\nu$ is a conformal vector which makes $V$ a VOSA.

The scalar product on $\mathcal{H}$ restricts to a normalized scalar product on $V$ having unitary M\"{o}bius symmetry as defined in (i) of Definition \ref{defin:unitary_mob_symmetry}. Furthermore, we have that
\begin{equation}
Y(a,z)^+=(-1)^{2d_a^2+d_a}Y(\theta(a),z)
\qquad\forall a\in\F
\end{equation}
where $Y(a,z)^+$ is the adjoint vertex operator as defined in \eqref{eq:defin_adjoint_field}. This means that every $Y(a,z)^+$ with $a\in\F$ is local and mutually local with respect to every vertex operator of $V$. Set
\begin{equation}
\F_+:=\left\{
\frac{a+(-1)^{2d_a^2+d_a}\theta(a)}{2} \,\,\Bigg|\,\, a\in\F
\right\}
\,,\quad
\F_-:=\left\{
-i\frac{a-(-1)^{2d_a^2+d_a}\theta(a)}{2} \,\,\Bigg|\,\, a\in\F
\right\} \,.
\end{equation}
Note that $V$ is generated by the family of Hermitian quasi-primary fields $\{Y(a,z)\mid a\in\F_+\cup \F_-\}$ and thus $V$ is unitary by Proposition \ref{prop:simple_unitary_vosa_gen_by_qp_hermitian_fields}. Furthermore, $V$ is simple by (iv) of Proposition \ref{prop:properties_bilinear_forms} because $V_0=\C\Omega$. $V$ is also energy-bounded thanks to Proposition \ref{prop:energy_boundedness_by_generators} because it is generated by the family $\F$ of vectors satisfying energy bounds. $\A_\F$ as defined in \eqref{eq:defin_net_gen_by_subset_F} coincides with $\A$ by hypothesis and thus we can conclude that $V$ is strongly graded-local and that $\A_V=\A$ by Theorem \ref{theo:gen_by_quasi_primary}.
\end{proof}

We end with the following:
\begin{conj}  
	\label{conj:graded-local_conformal_net_has_vosa}
	For every irreducible graded-local conformal net $\A$, there exists a simple strongly graded-local unitary VOSA $V$ such that $\A=\A_V$. 
\end{conj}

\appendix

\section{Vertex superalgebra locality and Wightman locality}
\label{appendix:wightman_locality}

In the current section, we investigate the relationship between the locality axiom in the vertex superalgebra framework and in the Wightman one. Indeed, as it is well-explained e.g.\ in \cite[Chapter 1]{Kac01}, the vertex superalgebra axioms are motivated by the Wightman axioms for a QFT, see \cite[Section 3.1]{SW64}. Hence, it is likely to have an equivalence between the locality axioms in the two frameworks at least under certain assumptions. Those assumptions are represented by a suitable energy bound condition for fields in the spirit of Section \ref{subsection:energy_bounds}.

Let $L_0$ be a self-adjoint operator on a Hilbert space $\mathcal{H}$ with spectrum contained in $\half\Zpluseq$. Call $V$ the algebraic direct sum
$$
\mathcal{H}^\mathrm{fin}=\bigoplus_{n\in\half\Zpluseq} \mathrm{Ker}(L_0-n1_\mathcal{H}) 
$$
which is a dense subspace of $\mathcal{H}$. As usual, denote by $\mathcal{H}^\infty$ the dense subspace of $\mathcal{H}$ given by the smooth vectors of $L_0$, see p.\ \pageref{defin:smooth_vectors_L_0} with references therein.
Let $a:=\{a_n\mid n\in\Z\}$ and $b:=\{b_n\mid n\in\Z-\half\}$ be two families of operators on $\mathcal{H}$ with $V$ as a common domain and such that
\begin{equation}  \label{eq:rot_covariance_Wightman}
\begin{split}
e^{itL_0}a_ne^{-itL_0}&=e^{-int}a_n
\qquad\forall t\in\R\,\,\,\forall n\in\Z \\
e^{itL_0}b_ne^{-itL_0}&=e^{-int}b_n
\qquad\forall t\in\R\,\,\,\forall n\in\Z-\half\,.
\end{split}
\end{equation}
It follows that
\begin{align}
a_n\mathrm{Ker} (L_0-k1_\mathcal{H})&\subseteq \mathrm{Ker} (L_0-(k-n)1_\mathcal{H})
\qquad\forall n\in\Z\,\,\,\forall k\in\half\Zpluseq \\
b_n\mathrm{Ker} (L_0-k1_\mathcal{H})&\subseteq \mathrm{Ker} (L_0-(k-n)1_\mathcal{H})
\qquad\forall n\in\Z-\half\,\,\,\forall k\in\half\Zpluseq 
\,.
\end{align}
Consequently, all the operators $a_n$ and $b_n$ restrict to endomorphisms of $V$ with the property that for all $c\in V$ there exists $N>0$ such that $a_nc=b_nc=0$ for all $n> N$. In other words, the formal series
\begin{equation}
\Phi_a(z):=\sum_{n\in\Z}a_nz^{-n}
\,, \qquad
\Phi_b(z):=\sum_{n\in\Z-\half}b_nz^{-n-\half}
\end{equation}
are fields on $V$ as defined in \cite[Section 3.1]{Kac01}, see also Section \ref{subsection:basic_definitions_VOSA}. Moreover, assume that $\Phi_a(z)$ and $\Phi_b(z)$ satisfy the following \textbf{energy bounds}: there exist non-negative real numbers $M$, $s$ and $k$ such that
\begin{align}
\norm{a_nc} &\leq M(1+\abs{n})^s\norm{(L_0+1_\mathcal{H})^kc}
\qquad\forall n\in\Z \,\,\,\forall c\in V \\
\norm{b_nc} &\leq M(1+\abs{n})^s\norm{(L_0+1_\mathcal{H})^kc}
\qquad\forall n\in\Z-\half \,\,\,\forall c\in V \,.
\end{align} 
Thus, we define the following operators:
\begin{equation}
\Phi_a(f)c:=\sum_{n\in\Z}\widehat{f}_n a_nc
\,, \qquad
\Phi_b(f)c:=\sum_{n\in\Z-\half}\widehat{f}_\frac{2n-1}{2} b_nc
\qquad \forall f\in C^\infty(S^1) \,\,\, \forall c\in V \,.
\end{equation}
The closures of the above operators, which we still denote by the same symbols, have $\mathcal{H}^\infty$ as common invariant domain, cf.\  Section \ref{subsection:energy_bounds}. We call $\Phi_a(f)$ and $\Phi_b(f)$ \textbf{smeared fields}.
We say that $a, \Phi_a(z)$ and $\Phi_a(f)$ are \textbf{even}, whereas $b, \Phi_b(z)$ and $\Phi_b(f)$ are \textbf{odd}. Accordingly, we say that $a$ has \textbf{parity} $p(a)=\parzero$, whereas that $b$ has parity $p(b)=\parone$.

Now, let $a^1$ and $a^2$ be two families of operators with given parities as above. Let $[\cdot,\cdot]$ be the \textbf{graded commutator}, that is, 
\begin{equation} \label{eq:defin_graded_commutator_fields_locality}
[\Phi_{a^1}(\cdot),\Phi_{a^2}(\cdot)]:=\Phi_{a^1}(\cdot)\Phi_{a^2}(\cdot)-(-1)^{p(a^1)p(a^2)}\Phi_{a^2}(\cdot)\Phi_{a^1}(\cdot) \,.
\end{equation}
where, with an abuse of notation, we are using $(-1)^{p(a^1)p(a^2)}$ as in \eqref{eq:defin_graded_commutator_fields_locality} to denote $(-1)^{p_{a^1}p_{a^2}}$, where $p_{a^1},p_{a^2}\in\{0,1\}$ are representatives of the remainder class of $p(a^1)$ and $p(a^2)$ in $\Z_2$ respectively.
We say that the fields $\Phi_{a^1}(z)$ and $\Phi_{a^2}(z)$ are, accordingly with Section \ref{subsection:basic_definitions_VOSA}, \textbf{mutually local in the vertex superalgebra sense} if there exists an integer $N\geq0$ such that
\begin{equation}
(z-w)^n[\Phi_{a^1}(z),\Phi_{a^2}(w)]c=0
\qquad
\forall n\geq N\,\,\,\forall c\in V 
\end{equation}
whereas, we say that the fields $\Phi_{a^1}(z)$ and $\Phi_{a^2}(z)$ are \textbf{mutually local in the Wightman sense} if 
\begin{equation}
[\Phi_{a^1}(f_1),\Phi_{a^2}(f_2)]c=0
\qquad\forall c\in V 
\end{equation}
whenever $f_1,f_2\in C^\infty(S^1)$ are such that $\mathrm{supp}f_1\subset I$ and $\mathrm{supp}f_2\subset I'$ for any $I\in\J$.
Therefore, we can state the correspondence between these two definitions of locality:

\begin{prop}   \label{prop:wightman_locality}
Let $a^1$ and $a^2$ be two families of operators on $\mathcal{H}$ with given parities and with $V$ as common domain. Assume further that $a^1$ and $a^2$ satisfy energy bounds. Then the fields $\Phi_{a^1}(z)$ and $\Phi_{a^2}(z)$ are mutually local in the vertex superalgebras sense if and only if they are mutually local in the Wightman sense.
\end{prop}

\begin{proof}
The following is an adaptation of the proof of \cite[Proposition A.1]{CKLW18}.

For every $c,d\in V$, define the following formal series in the two formal variables $z$ and $w$:
$$
\varphi_{c,d}(z,w):=
(d|[\Phi_{a^1}(z),\Phi_{a^2}(w)]c)
=\sum_{\substack{n\in\Z-p_1 \\ m\in\Z-p_2}}
(d|[a^1_n,a^2_m]c)z^{-n-p_1}w^{-m-p_2}
$$
where $p_k$ are either $0$ if $p(a^k)=\parzero$ or $\half$ if $p(a^k)=\parone$ for $k\in\{1,2\}$. $\varphi_{c,d}(z,w)$ can be considered as a formal distribution on $S^1\times S^1$ as in \cite[Section 2.1]{Kac01}, that is, $\varphi_{c,d}$ is a linear functional on the complex vector space of trigonometric polynomials in two variables. Thanks to the energy bounds, $\varphi_{c,d}$ extends to a unique ordinary distribution on $S^1\times S^1$ by continuity, that is, a continuous linear functional on the Fréchet space $C^\infty(S^1\times S^1)$, which we call $\varphi_{c,d}(f)$ for $f\in C^\infty(S^1\times S^1)$. (Recall that $C^\infty(S^1\times S^1)$ is isomorphic to the completion of the algebraic tensor product $C^\infty(S^1)\otimes C^\infty(S^1)$.) The energy bounds assure us that there exists a $N>0$ such that for every $c,d\in V$, there exists $M_{c,d}>0$ such that
\begin{equation}   \label{eq:varphi_cd_is_orderN}
\abs{\varphi_{c,d}(f)}\leq M_{c,d}
\norm{f}_N
\,,\qquad
\max_{\substack{\abs{s}\leq N \\ x,y\in\R}}
 \abs{\partial^sf(e^{ix},e^{iy})} 
 \qquad
 \forall f\in C^\infty(S^1\times S^1)
\end{equation}
where, as usual, $s=(s^1,s^2)$, $s^j\in\Zpluseq$ is a multi-index with $\abs{s}=s^1+s^2$ and $\partial^s:=(\partial_x)^{s^1}(\partial_y)^{s^2}$.

Suppose that the fields $\Phi_{a^1}(z)$ and $\Phi_{a^2}(z)$ are mutually local in the vertex superalgebra sense. By \cite[Theorem 2.3(i)]{Kac01}, the ordinary distribution $\varphi_{c,d}$ has support in the diagonal 
$$
D:=\{(z,w)\in S^1\times S^1\mid z-w=0\} .
$$
It follows that $(d|[\Phi_{a^1}(f),\Phi_{a^2}(g)]c)=0$ for all $c,d\in V$, whenever $f,g\in C^\infty(S^1)$ with $\mathrm{supp}f\subset I\in\J$ and $\mathrm{supp}g\subset I'$. Due to the energy bounds, we can extend the equality above to all $c\in\mathcal{H}^\infty$ (cf.\  Lemma \ref{lem:common_core}) and thus the fields $\Phi_{a^1}(z)$ and $\Phi_{a^2}(z)$ are mutually local in the Wightman sense.

Conversely, suppose that $\Phi_{a^1}(z)$ and $\Phi_{a^2}(z)$ are mutually local in the Wightman sense, then the formal distribution $\varphi_{c,d}$ has support in the diagonal $D$ of $S^1\times S^1$ as defined above. Moreover, \eqref{eq:varphi_cd_is_orderN} says us that $\varphi_{c,d}$ is a distribution of \textit{order $N$} in the sense of \cite[p.\  156]{Rud91}. We would like to apply \cite[Theorem 6.25]{Rud91}. To this aim, consider a finite open cover $\{U_j\}$ of the diagonal $D$ of the torus $S^1\times S^1$ in such a way that every $U_j$ is diffeomorphic to the open square $(-\pi,\pi)\times(-\pi,\pi)$ of $\R^2$ and such that $D\cap U_j$ is diffeomorphic to a diagonal of such a square. Complete $\{U_j\}$ to a finite open cover of $S^1\times S^1$, which we denote by the same symbols. Let $\{h_j\}$ be a partition of unity of $S^1\times S^1$ with respect to the finite open cover $\{U_j\}$ and set $\varphi_{c,d}^j:=h_j\varphi_{c,d}$, so that $\varphi_{c,d}=\sum_j\varphi_{c,d}^j$. Every $\varphi_{c,d}^j$ can be considered as a distribution in the angular variables $(x,y)\in(-\pi,\pi)\times(-\pi,\pi)$ and can be easily extended to a distribution on $\R^2$. Applying the change of variable $l:=x-y$ and $m:=x+y$, we obtain distributions $\phi_{c,d}^j$ on $\R^2$ with support in $\{(l,m)\in \R^2\mid l=0\}$. Smearing a distribution $\phi_{c,d}^j$ with a test function $g$ in the variable $m\in\R$, we obtain a new distribution in the variable $l\in\R$ with support in $\{0\}$ only. Thus, we are in the case of \cite[Theorem 6.25]{Rud91}, which says us that $\phi_{c,d}^j(l,g)=\sum_{s=1}^N\widetilde{k}_s^j(g)\partial^s\delta(l)$. One can check that for every $s$ and $j$, $\widetilde{k}_s^j$ is a distribution in the variable $m\in\R$. Thus, for every $j$, we can rewrite $\varphi_{c,d}(z,w)=\sum_{s=1}^Nk_s^j(w)\partial_w^s\delta(z-w)$ for some distributions $k_s^j$ in the variable $w$. By \cite[Theorem 2.3]{Kac01}, we can deduce that $(z-w)^N\varphi_{c,d}(z,w)=0$ for all $c,d\in V$, that is, the fields $\Phi_{a^1}(z)$ and $\Phi_{a^2}(z)$ are mutually local in the vertex superalgebra sense.
\end{proof}

\bigskip
\noindent
{\small
	{\bf Acknowledgements.}
	S.C.\ would like to thank Bin Gui for some useful discussions.
	S.C.\ acknowledges support from the GNAMPA-INDAM project \emph{Operator algebras and infinite quantum systems}, CUP E53C23001670001 and from the University of Rome ``Tor Vergata'' funding \emph{OAQM}, CUP E83C22001800005. 
	S.C.\ and T.G.\ acknowledge support from the \emph{MUR Excellence Department Project   
	MatMod@TOV} awarded to the Department of Mathematics, University of Rome ``Tor Vergata'', CUP E83C23000330006 and from the Tor Vergata University of Rome funding OANGQS, CUP E83C25000580005.
	Part of this work was done while T.G.\ was Senior Research Associate, Leverhulme Trust at the School of Mathematical Sciences, Lancaster University, supported by the Leverhulme Trust Research Project Grant RPG-2021-129. 
	R.H.\ would like to thank the Department of Mathematics at the University of Rome ``Tor Vergata'' for hospitality received during an extended research visit. 
	The authors thank the referee for useful comments.
}

\end{document}